\documentclass[12pt]{report}
\usepackage{mathptmx}
\usepackage[T1]{fontenc}
\usepackage[latin9]{inputenc}
\usepackage{array}
\usepackage{verbatim}
\usepackage{units}
\usepackage{amsthm}
\usepackage{amsmath}
\usepackage{amssymb}
\usepackage{graphicx}
\usepackage[unicode=true,pdfusetitle,
 bookmarks=true,bookmarksnumbered=false,bookmarksopen=false,
 breaklinks=true,pdfborder={0 0 0},backref=false,colorlinks=false]
 {hyperref}

\makeatletter

\providecommand{\tabularnewline}{\\}

\usepackage{enumitem}		
\theoremstyle{plain}
\newtheorem{thm}{\protect\theoremname}[section]
  \theoremstyle{definition}
  \newtheorem{example}[thm]{\protect\examplename}
  \theoremstyle{definition}
  \newtheorem{defn}[thm]{\protect\definitionname}
  \theoremstyle{plain}
  \newtheorem{prop}[thm]{\protect\propositionname}
  \theoremstyle{remark}
  \newtheorem*{rem*}{\protect\remarkname}
  \theoremstyle{plain}
  \newtheorem*{thm*}{\protect\theoremname}
  \theoremstyle{definition}
  \newtheorem*{example*}{\protect\examplename}
  \theoremstyle{plain}
  \newtheorem{lem}[thm]{\protect\lemmaname}
 \theoremstyle{definition}
 \newtheorem*{defn*}{\protect\definitionname}
  \theoremstyle{plain}
  \newtheorem{cor}[thm]{\protect\corollaryname}

 \theoremstyle{remark}
 \newtheorem*{rems*}{\protect\remarksname}

\usepackage{suthesis-2e_mod}
\usepackage[style=alphabetic,backend=bibtex,maxbibnames=99]{biblatex}

\usepackage{url}
\usepackage{comment}

\DeclareMathOperator{\im}{im}
\DeclareMathOperator{\ev}{ev}
\DeclareMathOperator{\pack}{pack}

\DeclareMathOperator{\Res}{Res}
\DeclareMathOperator{\Ind}{Ind}
\DeclareMathOperator{\gr}{gr}
\DeclareMathOperator{\sspan}{span}
\DeclareMathOperator{\Des}{Des}
\DeclareMathOperator{\des}{des}
\DeclareMathOperator{\asc}{asc}
\DeclareMathOperator{\peak}{peak}
\DeclareMathOperator{\vall}{vall}
\DeclareMathOperator{\aasc}{aasc}
\DeclareMathOperator{\ddes}{ddes}
\DeclareMathOperator{\inv}{inv}
\DeclareMathOperator{\ides}{ides}
\DeclareMathOperator{\Root}{root}
\DeclareMathOperator{\desc}{desc}
\DeclareMathOperator{\anc}{anc}
\DeclareMathOperator{\Anc}{Anc}

\vfuzz \hfuzz
\newcommand\f{\mathbf{f}}
\newcommand\g{\mathbf{g}}

\newcommand\calh{\mathcal{H}}
\newcommand\calhn{\mathcal{H}_n}
\newcommand\calhdual{\mathcal{H}^*}
\newcommand\calhndual{\mathcal{H}_n^*}
\newcommand\barcalh{\bar{\mathcal{H}}}

\newcommand\calb{\mathcal{B}}
\newcommand\calbn{\mathcal{B}_n}
\newcommand\hatcalb{\check{\mathcal{B}}}
\newcommand\hatcalbn{\check{\mathcal{B}}_n}
\newcommand\calbdual{\mathcal{B}^*}
\newcommand\calbndual{\mathcal{B}_n^*}
\newcommand\barcalb{\bar{\mathcal{B}}}
\newcommand\hatbarcalb{\check{\bar{\mathcal{B}}}}
\newcommand\sg{\bar{\calsh^*}_G}

\newcommand\hatx{\check{x}}
\newcommand\haty{\check{y}}
\newcommand\hatr{\check{R}}
\newcommand\backw{\overleftarrow{w}}

\newcommand\kan{K_{a,n}}
\newcommand\hatk{\check{K}}
\newcommand\hatkan{\check{K}_{a,n}}
\newcommand\bark{\bar K}

\newcommand\hatbark{\check{\bar{K}}}

\newcommand\bard{\bar\Delta}

\newcommand\proda{m^{[a]}}
\newcommand\coproda{\Delta^{[a]}}

\newcommand\calc{\mathcal{C}}
\newcommand\cald{\mathcal{D}}
\newcommand\cale{\mathcal{E}}
\newcommand\caly{\mathcal{Y}}
\newcommand\calp{\mathcal{P}}
\newcommand\barcalg{\bar{\mathcal{G}}}
\newcommand\calg{\mathcal{G}}
\newcommand\calsh{\mathcal{S}}
\newcommand\sym{\mathbf{Sym}}
\newcommand\sn{\mathfrak{S}_n}
\newcommand\sm{\mathfrak{S}_m}
\newcommand\sk{\mathfrak{S}_k}
\newcommand\skk{\mathfrak{S}_K}
\newcommand\Sl{\mathfrak{S}_l}

\newcommand\barx{\bar{x}}
\newcommand\bary{\bar{y}}
\newcommand\barc{\bar{c}}
\newcommand\barb{\bar{b}}
\newcommand\barv{\bar{V}}
\newcommand\barf{\bar{f}}
\newcommand\barg{\bar{g}}

\let\amalg=\undefined
\let\coprod=\undefined
\DeclareSymbolFont{cmsymbols}{OMS}{cmsy}{m}{n}
\DeclareSymbolFont{cmlargesymbols}{OMX}{cmex}{m}{n}
\DeclareMathSymbol{\amalg}{\mathbin}{cmsymbols}{"71}
\DeclareMathSymbol{\coprod}{\mathop}{cmlargesymbols}{"60}

\newlength\cellsize 
\savebox2{%
\begin{picture}(12,12)
\put(0,0){\line(1,0){12}}
\put(0,0){\line(0,1){12}}
\put(12,0){\line(0,1){12}}
\put(0,12){\line(1,0){12}}
\end{picture}}

\newcommand\cellify[1]{\def\thearg{#1}\def\nothing{}%
\ifx\thearg\nothing
\vrule width0pt height\cellsize depth0pt\else
\hbox to 0pt{\usebox2\hss}\fi%
\vbox to 12\unitlength{
\vss
\hbox to 12\unitlength{\hss$#1$\hss}
\vss}}

\newcommand\tableau[1]{\vtop{\let\\=\cr
\setlength\baselineskip{-12000pt}
\setlength\lineskiplimit{12000pt}
\setlength\lineskip{0pt}
\halign{&\cellify{##}\cr#1\crcr}}}

\newcommand{\e}{\mbox{}}

 \providecommand{\remarksname}{Remarks}

\makeatother

  \providecommand{\corollaryname}{Corollary}
  \providecommand{\definitionname}{Definition}
  \providecommand{\examplename}{Example}
  \providecommand{\lemmaname}{Lemma}
  \providecommand{\propositionname}{Proposition}
  \providecommand{\remarkname}{Remark}
  \providecommand{\theoremname}{Theorem}
\providecommand{\theoremname}{Theorem}

\begin{document}

\newpage

\title{Hopf Algebras and Markov Chains}
\author{Chung Yin Amy Pang}
\dept{Mathematics} 

\beforepreface

Because of time constraints, I did NOT submit this version to Stanford.
This version differs from the submitted version in that the chapters
are in a different order, and there are additional results. I prefer
that you cite this version (http://arxiv.org/abs/1412.8221), or one
of the related papers. Consult the table on my webpage for which papers
contain which sections of this thesis.

Unless there is a major mathematical error, the version of this thesis
on arXiv will not be updated. However, I aim to keep an updated version
on my webpage, so please alert me to typos and confusing parts. Below
is the list of major changes since the arXiv version; minor typographical
corrections are not listed.

\prefacesection{Abstract}

This thesis introduces a way to build Markov chains out of Hopf algebras.
The transition matrix of a \emph{Hopf-power Markov chain} is (the
transpose of) the matrix of the coproduct-then-product operator on
a \emph{combinatorial Hopf algebra} with respect to a suitable basis.
These chains describe the breaking-then-recombining of the combinatorial
objects in the Hopf algebra. The motivating example is the famous
Gilbert-Shannon-Reeds model of riffle-shuffling of a deck of cards,
which arises in this manner from the shuffle algebra.

The primary reason for constructing Hopf-power Markov chains, or for
rephrasing familiar chains through this lens, is that much information
about them comes simply from translating well-known facts on the underlying
Hopf algebra. For example, there is an explicit formula for the stationary
distribution (Theorem \ref{thm:hpmc-stationarydistribution}), and
constructing quotient algebras show that certain statistics on a Hopf-power
Markov chain are themselves Markov chains (Theorem \ref{thm:hpmc-projection}).
Perhaps the pinnacle is Theorem \ref{thm:diagonalisation}, a collection
of algorithms for a full left and right eigenbasis in many common
cases where the underlying Hopf algebra is commutative or cocommutative.
This arises from a cocktail of the Poincare-Birkhoff-Witt theorem,
the Cartier-Milnor-Moore theorem, Reutenauer's structure theory of
the free Lie algebra, and Patras's Eulerian idempotent theory.

Since Hopf-power Markov chains can exhibit very different behaviour
depending on the structure of the underlying Hopf algebra and its
distinguished basis, one must restrict attention to certain styles
of Hopf algebras in order to obtain stronger results. This thesis
will focus respectively on a free-commutative basis, which produces
\textquotedbl{}independent breaking\textquotedbl{} chains, and a cofree
basis; there will be both general statements and in-depth examples.
\prefacesection{Acknowledgement}

First, I must thank my advisor Persi Diaconis for his patient guidance
throughout my time at Stanford. You let me roam free on a landscape
of algebra, probability and combinatorics in whichever direction I
choose, yet are always ready with a host of ideas the moment I feel
lost. Thanks in addition for putting me in touch with many algebraic
combinatorialists.

The feedback from my thesis committee - Dan Bump, Tom Church, Eric
Marberg - was invaluable to improving both the mathematics and the
exposition in this thesis. Thanks also for a very enjoyable discussion
at the defense - I doubt I'd have another chance to take an hour-long
conversation solely on my work.

Thanks to my coauthor Arun Ram for your observations about card-shuffling
and the shuffle algebra, from which grew the theory in this thesis.
I also greatly appreciate your help in improving my mathematical writing
while we worked on our paper together.

I'm grateful to Marcelo Aguiar for pointing me to Patras's work, which
underlies a key part of this thesis, and for introducing me to many
of the combinatorial Hopf algebras I describe here.

To the algebraic combinatorics community: thanks for taking me in
as a part of your family at FPSAC and other conferences. Special mentions
go to Sami Assaf and Aaron Lauve for spreading the word about my work;
it makes a big difference to know that you are as enthusiastic as
I am about my little project.

I'd like to thank my friends, both at Stanford and around the world,
for mathematical insights, practical tips, and just cheerful banter.
You've very much brightened my journey in the last five years. Thanks
especially to Agnes and Janet for taking time from your busy schedules
to hear me out during my rough times; because of you, I now have a
more positive look on life.

Finally, my deepest thanks go to my parents. You are always ready
to share my excitement at my latest result, and my frustrations at
another failed attempt, even though I'm expressing it in way too much
technical language. Thank you for always being there for me and supporting
my every decision. 

\afterpreface

\chapter{Introduction\label{chap:Introduction}}

Sections \ref{sec:Markov-chains-intro} and \ref{sec:Hopf-algebras-intro}
briefly summarise, respectively, the basics of the two worlds that
this thesis bridges, namely Markov chains and Hopf algebras. Section
\ref{sec:Hopf-power-Markov-chains-intro} introduces the motivating
example of riffle-shuffling of a deck of cards, and outlines the main
themes in the thesis.

\section{Markov chains\label{sec:Markov-chains-intro}}

A friendly introduction to this topic is Part I of the textbook \cite{markovmixing}.

A (discrete time) Markov chain is a simple model of the evolution
of an object over time. The key assumption is that the state $X_{m}$
of the object at time $m$ only depends on $X_{m-1}$, its state one
timestep prior, and not on earlier states. Writing $P\{A|B\}$ for
the probability of the event $A$ given the event $B$, this \emph{Markov
property} translates to
\[
P\{X_{m}=x_{m}|X_{0}=x_{0},X_{1}=x_{1},\dots,X_{m-1}=x_{m-1}\}=P\{X_{m}=x_{m}|X_{m-1}=x_{m-1}\}.
\]
Consequently,
\begin{align*}
 & P\{X_{0}=x_{0},X_{1}=x_{1},\dots,X_{m}=x_{m}\}\\
= & P\{X_{0}=x_{0}\}P\{X_{1}=x_{1}|X_{0}=x_{0}\}\dots P\{X_{m}=x_{m}|X_{m-1}=x_{m-1}\}.
\end{align*}
The set of all possible values of the $X_{m}$ is the \emph{state
space} - in this thesis, this will be a finite set, and will be denoted
$S$ or $\calb$, as it will typically be the basis of a vector space.

All Markov chains in this thesis are time-invariant, so $P\{X_{m}=y|X_{m-1}=x\}=P\{X_{1}=y|X_{0}=x\}$.
Thus a chain is completely specified by its \emph{transition matrix}
\[
K(x,y):=P\{X_{1}=y|X_{0}=x\}.
\]
It is clear that $K(x,y)\geq0$ for all $x,y\in S$, and $\sum_{y\in S}K(x,y)=1$
for each $x\in S$. Conversely, any matrix $K$ satisfying these two
conditions defines a Markov chain. So this thesis will use the term
``transition matrix'' for any matrix with all entries non-negative
and all row sums equal to 1. (A common equivalent term is \emph{stochastic
matrix}).

Note that
\begin{align*}
P\{X_{2}=y|X_{0}=x\} & =\sum_{z\in S}P\{X_{2}=y|X_{1}=z\}P\{X_{1}=z|X_{0}=x\}\\
 & =\sum_{z\in S}K(z,y)K(x,z)=K^{2}(x,y);
\end{align*}
similarly, $K^{m}(x,y)=P\{X_{m}=y|X_{0}=x\}$ - the powers of the
transition matrix contain the transition probabilities after many
steps. 
\begin{example}
\label{ex:shuffle-intro}The process of card-shuffling is a Markov
chain: the order of the cards after $m$ shuffles depends only on
their order just before the last shuffle, not on the orders prior
to that. The state space is the $n!$ possible orderings of the deck,
where $n$ is the number of cards in the deck.

The most well-known model for card-shuffling, studied in numerous
ways over the last 25 years, is due to Gilbert, Shannon and Reeds
(GSR): first, cut the deck binomially (i.e. take $i$ cards off the
top of an $n$-card deck with probability $2^{-n}\binom{n}{i}$),
then drop one by one the bottommost card from one of the two piles,
chosen with probability proportional to the current pile size. Equivalently,
all interleavings of the two piles which keep cards from the same
pile in the same relative order are equally likely. This has been
experimentally tested to be an accurate model of how the average person
shuffles. Section \ref{sec:Riffle-Shuffling} is devoted to this example,
and contains references to the history and extensive literature.
\end{example}
After many shuffles, the deck is almost equally likely to be in any
order. This is a common phenomenon for Markov chains: under mild conditions,
the probability of being in state $x$ after $m$ steps tends to a
limit $\pi(x)$ as $m\rightarrow\infty$. These limiting probabilities
must satisfy $\sum_{x}\pi(x)K(x,y)=\pi(y)$, and any probability distribution
satisfying this equation is known as a \emph{stationary distribution}.
With further mild assumptions (see \cite[Prop. 1.14]{markovmixing}),
$\pi(x)$ also describes the proportion of time the chain spends in
state $x$.

The purpose of shuffling is to put the cards into a random order,
in other words, to choose from all orderings of cards with equal probability.
Similarly, Markov chains are often used as ``random object generators'':
thanks to the Markov property, running a Markov chain is a computationally
efficient way to sample from $\pi$. Indeed, there are schemes such
as Metropolis \cite[Chap. 3]{markovmixing} for constructing Markov
chains to converge to a desired stationary distribution. For these
sampling applications, it is essential to know roughly how many steps
to run the chain. The standard way to measure this rigorously is to
equip the set of probability distributions on $S$ with a metric,
such as \emph{total variation} or \emph{separation distance}, and
find a function $m(\epsilon)$ for which $||K^{m}(x_{0},\cdot)-\pi(\cdot)||<\epsilon$.
Such convergence rate bounds are outside the scope of this thesis,
which simply views this as motivation for studying high powers of
the transition matrix.

One way to investigate high powers of a matrix is through its spectral
information.
\begin{defn}
Let $\{X_{m}\}$ be a Markov chain on the state space $S$ with transition
matrix $K$. Then 
\begin{itemize}
\item A function $\g:S\rightarrow\mathbb{R}$ is a \emph{left eigenfunction}
of the chain $\{X_{m}\}$ of eigenvalue $\beta$ if $\sum_{x\in S}\g(x)K(x,y)=\beta\g(y)$
for each $y\in S$. 
\item A function $\f:S\rightarrow\mathbb{R}$ is a \emph{right eigenfunction}
of the chain $\{X_{m}\}$ of eigenvalue $\beta$ if $\sum_{y\in S}K(x,y)\f(y)=\beta\f(x)$
for each $x\in S$. 
\end{itemize}
\end{defn}
(It may be useful to think of $\g$ as a row vector, and $\f$ as
a column vector.) Observe that a stationary distribution $\pi$ is
a left eigenfunction of eigenvalue 1. \cite[Sec. 2.1]{hopfpowerchains}
lists many applications of both left and right eigenfunctions, of
which two feature in this thesis. Chapter \ref{chap:free-commutative-basis}
and Section \ref{sec:Riffle-Shuffling} employ their Use A: the expected
value of a right eigenfunction $\f$ with eigenvalue $\beta$ is 
\[
E\{\f(X_{m})|X_{0}=x_{0}\}:=\sum_{s\in S}K^{m}(x_{0},s)\f(s)=\beta^{m}\f(x_{0}).
\]
The Proposition below records this, together with two simple corollaries.
\begin{prop}[Expectation estimates from right eigenfunctions]
\label{prop:expectationrightefns} Let $\{X_{m}\}$ be a Markov chain
with state space $S$, and $\f_{i}$ some right eigenfunctions with
eigenvalue $\beta_{i}$. 
\begin{enumerate}
\item For each $\f_{i}$, 
\[
E\{\f_{i}(X_{m})|X_{0}=x_{0}\}=\beta_{i}^{m}\f_{i}(x_{0}).
\]

\item Suppose $\f:S\rightarrow\mathbb{R}$ is such that, for each $x\in S$,
\[
\sum_{i}\alpha_{i}\f_{i}(x)\leq\f(x)\leq\sum_{i}\alpha'_{i}\f_{i}(x)
\]
for some non-negative constants $\alpha_{i},\alpha'_{i}$. Then 
\[
\sum_{i}\alpha_{i}\beta_{i}^{m}\f_{i}(x_{0})\leq E\{\f(X_{m})|X_{0}=x_{0}\}\leq\sum_{i}\alpha'_{i}\beta_{i}^{m}\f_{i}(x_{0}).
\]

\item Let $S'$ be a subset of the state space $S$. Suppose the right eigenfunction
$\f_{i}$ is non-negative on $S'$ and zero on $S\backslash S'$.
Then \textup{
\[
\frac{\beta_{i}^{m}\f_{i}(x_{0})}{\max_{s\in S'}\f_{i}(s)}\leq P\{X_{m}\in S'|X_{0}=x_{0}\}\leq\frac{\beta_{i}^{m}\f_{i}(x_{0})}{\min_{s\in S'}\f_{i}(s)}.
\]
}
\end{enumerate}
\end{prop}
\begin{proof}
Part i is immediate from the definition of right eigenfunction. Part
ii follows from the linearity of expectations. To see Part iii, specialise
to $\f=\mathbf{1}_{S'}$, the indicator function of being in $S'$.
Then it is true that 
\[
\frac{\f_{i}(x)}{\max_{s\in S'}\f_{i}(s)}\leq\mathbf{1}_{S'}(x)\leq\frac{\f_{i}(x)}{\min_{s\in S'}\f_{i}(s)}
\]
and the expected value of an indicator function is the probability
of the associated event.
\end{proof}
A modification of \cite[Sec. 2.1, Use H]{hopfpowerchains} occurs
in Corollary \ref{cor:probdescentset}. Here is the basic, original
version: 
\begin{prop}[]
\label{prop:efnsdiagonalisation}Let $K$ be the transition matrix
of a Markov chain $\{X_{m}\}$, and let $\{\f_{i}\}$, $\{\g_{i}\}$
be dual bases of right and left eigenfunctions for $\{X_{m}\}$ -
that is, \textup{$\sum_{j}\f_{i}(j)\g_{i'}(j)=0$} if \textup{$i\neq i'$},
and\textup{ $\sum_{j}\f_{i}(j)\g_{i}(j)=1$}. Write $\beta_{i}$ for
the common eigenvalue of $\f_{i}$ and $\g_{i}$. Then \textup{
\[
P\{X_{m}=y|X_{0}=x\}=K^{m}(x,y)=\sum_{i}\beta_{i}^{m}\f_{i}(x)\g_{i}(y).
\]
 }\end{prop}
\begin{proof}
Let $D$ be the diagonal matrix of eigenvalues (so $D(i,i)=\beta_{i}$).
Put the right eigenfunctions $\f_{j}$ as columns into a matrix $F$
(so $F(i,j)=\f_{j}(i)$), and the left eigenfunctions $\g_{i}$ as
rows into a matrix $G$ (so $G(i,j)=\g_{i}(j)$). The duality means
that $G=F^{-1}$. So, a simple change of coordinates gives $K=FDG$,
hence $K^{m}=FD^{m}G$. Note that $D^{m}$ is diagonal with $D^{m}(i,i)=\beta_{i}^{m}$.
So 
\begin{align*}
K^{m}(x,y) & =(FD^{m}G)(x,y)\\
 & =\sum_{i,j}F(x,i)D^{m}(i,j)G(j,y)\\
 & =\sum_{i}F(x,i)\beta_{i}^{m}G(i,y)\\
 & =\sum_{i}\beta_{i}^{m}\f_{i}(x)\g_{i}(y).
\end{align*}

\end{proof}
For general Markov chains, computing a full basis of eigenfunctions
(a.k.a. ``diagonalising'' the chain) can be an intractable problem;
this strategy is much more feasible when the chain has some underlying
algebraic or geometric structure. For example, the eigenvalues of
a random walk on a group come directly from the representation theory
of the group \cite[Chap. 3E]{randomwalksongroups}. Similarly, there
is a general formula for the eigenvalues and right eigenfunctions
of a random walk on the chambers of a hyperplane arrangement \cite{hyperplanewalk,shuffleefns2}.
The purpose of this thesis is to carry out the equivalent analysis
for Markov chains arising from Hopf algebras.

\section{Hopf algebras\label{sec:Hopf-algebras-intro}}

A graded, connected Hopf algebra is a graded vector space $\calh=\bigoplus_{n=0}^{\infty}\calh_{n}$
equipped with two linear maps: a product $m:\calh_{i}\otimes\calh_{j}\rightarrow\calh_{i+j}$
and a \emph{coproduct} $\Delta:\calh_{n}\to\bigoplus_{j=0}^{n}\calh_{j}\otimes\calh_{n-j}$.
The product is associative and has a unit which spans $\calh_{0}$.
The corresponding requirements on the coproduct are \textit{coassociativity}:
$(\Delta\otimes\iota)\Delta=(\iota\otimes\Delta)\Delta$ (where $\iota$
denotes the identity map) and the counit axiom: $\Delta(x)-1\otimes x-x\otimes1\in\bigoplus_{j=1}^{n-1}\calh_{j}\otimes\calh_{n-j}$,
for $x\in\calh_{n}$. The product and coproduct satisfiy the compatibility
axiom $\Delta(wz)=\Delta(w)\Delta(z)$, where multiplication on $\calh\otimes\calh$
is componentwise. This condition may be more transparent in \emph{Sweedler
notation}: writing $\sum_{(x)}x_{(1)}\otimes x_{(2)}$ for $\Delta(x)$,
the axiom reads $\Delta(wz)=\sum_{(w),(z)}w_{(1)}z_{(1)}\otimes w_{(2)}z_{(2)}$.
This thesis will use Sweedler notation sparingly.

The definition of a general Hopf algebra, without the grading and
connectedness assumptions, is slightly more complicated (it involves
an extra \emph{antipode} map, which is automatic in the graded case);
the reader may consult \cite{sweedler}. However, that reference (like
many other introductions to Hopf algebras) concentrates on finite-dimensional
Hopf algebras, which are useful in representation theory as generalisations
of group algebras. These behave very differently from the infinite-dimensional
Hopf algebras in this thesis.
\begin{example}[Shuffle algebra]
\label{ex:shufflealg-intro}The shuffle algebra $\calsh$, as a vector
space, has basis the set of all words in the letters $\{1,2,\dots\}$.
Write these words in parantheses to distinguish them from integers.
The degree of a word is its number of letters, or \emph{length}. The
product of two words is the sum of all their interleavings (with multiplicity),
and the coproduct is by deconcatenation; for example:
\[
m((13)\otimes(52))=(13)(52)=(1352)+(1532)+(1523)+(5132)+(5123)+(5213);
\]
\[
m((15)\otimes(52))=(15)(52)=2(1552)+(1525)+(5152)+(5125)+(5215);
\]
\[
\Delta((336))=\emptyset\otimes(336)+(3)\otimes(36)+(33)\otimes(6)+(336)\otimes\emptyset.
\]
 (Here, $\emptyset$ denotes the empty word, which is the unit of
$\calsh$.) 
\end{example}
More examples of Hopf algebras are in Section \ref{sec:Combinatorial-Hopf-algebras}.
This thesis will concentrate on Hopf algebras satisfying at least
one of the following two symmetry conditions: $\calh$ is \emph{commutative}
if $wz=zw$ for all $w,z\in\calh$, and $\calh$ is \emph{cocommutative}
if $\sum_{(x)}x_{(1)}\otimes x_{(2)}=\sum_{(x)}x_{(2)}\otimes x_{(1)}$
for all $x\in\calh$. In other words, if $\tau:\calh\otimes\calh\rightarrow\calh\otimes\calh$
is the linear map satisfying $\tau(w\otimes z)=z\otimes w$ for all
$w,z\in\calh$, then cocommutativity is the condition $\tau(\Delta(x))=\Delta(x)$
for all $x$.

Hopf algebras first appeared in topology, where they describe the
cohomology of a topological group or loop space. Cohomology is always
an algebra under cup product, and the group product or the concatenation
of loops induces the coproduct structure. Nowadays, the Hopf algebra
is an indispensable tool in many parts of mathematics, partly due
to structure theorems regarding abstract Hopf algebras. To give a
flavour, a theorem of Hopf \cite[Th. A49]{hopfthmref} states that
any finite-dimensional, graded-commutative and graded-cocommutative
Hopf algebra over a field of characteristic 0 is isomorphic as an
algebra to a free exterior algebra with generators in odd degrees.
More relevant to this thesis is the Cartier-Milnor-Moore theorem \cite[Th. 3.8.1]{cmm}:
any cocommutative and conilpotent Hopf algebra $\calh$over a field
of characteristic zero is the universal enveloping algebra of its
\emph{primitive} subspace $\{x\in\calh|\Delta(x)=1\otimes x+x\otimes1\}$.
That such a Hopf algebra is completely governed by its primitives
will be important for Theorem \ref{thm:diagonalisation}.B, one of
the algorithms diagonalising the Markov chains in this thesis.

\section{Hopf-power Markov chains\label{sec:Hopf-power-Markov-chains-intro}}

To see the connection between the shuffle algebra and the GSR riffle-shuffle
Markov chain, identify a deck of cards with the word whose $i$th
letter denotes the value of the $i$th card, counting the cards from
the top of the deck. So $(316)$ describes a three-card deck with
the card labelled 3 on top, card 1 in the middle, and card 6 at the
bottom. Then, the probability that shuffling a deck $x$ of $n$ cards
results in a deck $y$ is 
\[
K(x,y)=\mbox{coefficient of }y\mbox{ in }2^{-n}m\Delta(x).
\]
In other words, the transition matrix of the riffle-shuffle Markov
chain for decks of $n$ cards is the transpose of the matrix of the
linear map $2^{-n}m\Delta$ with respect to the basis of words of
length $n$. Thus diagonalising the riffle-shuffle chain amounts to
the completely algebraic problem of finding an eigenbasis for $m\Delta$,
the coproduct-then-product operator, on the shuffle algebra. Chapter
\ref{chap:hpmc-Diagonalisation} and Section \ref{sec:Riffle-Shuffling}
achieve this; although the resulting eigenfunctions are not dual in
the sense of Proposition \ref{prop:efnsdiagonalisation}, this is
the first time that full eigenbases for riffle-shuffling have been
determined.

The subject of this thesis is to analogously model the breaking-then-recombining
of other combinatorial objects. As described in Section \ref{sec:Combinatorial-Hopf-algebras},
the literature contains numerous constructions of \emph{combinatorial
Hopf algebras}, which encode how to assemble and take apart combinatorial
objects. For example, in the Hopf algebras of graphs (Example \ref{ex:graph}),
the product is disjoint union, and the coproduct sends a graph to
pairs of induced subgraphs on a subset of the vertices and on its
complement. Then one can product a ``graph-breaking'' model by defining
the transition probabilities $K(x,y)$ to be the coefficient of $y$
in $2^{-n}m\Delta(x)$, where $n$ is the number of vertices of the
graphs $x$ and $y$. Then each step of the chain chooses a subset
of the vertex set and severs all edges with exactly one endpoint in
the chosen subset. Since this transition matrix is the matrix of the
linear operator $2^{-n}m\Delta$, its eigenfunctions again come from
the eigenvectors of $2^{-n}m\Delta$.

The obstacle to making the same definition on other Hopf algebras
is that the coefficients of $2^{-n}m\Delta$ need not always sum to
one. Fortunately, a clean workaround exists in the form of the \emph{Doob
transform}. Theorem \ref{thm:doob-transform} describes this very
general method of building a transition matrix out of most non-negative
linear operators, by rescaling the basis. 

Since the transition matrix of such a \emph{Hopf-power Markov chain}
is the matrix of the coproduct-then-product operator $m\Delta$ (albeit
with a rescaling of basis), many questions about these chains can
be translated from probability into algebra. As previously mentioned,
the eigenfunctions of the chain are the eigenvectors of $m\Delta$;
this applies in particular to their stationary distributions. Reversibility
of a Hopf-power Markov chain is equivalent to self-duality of the
underlying Hopf algebra (Theorem \ref{thm:hpmc-reversible}), and
the Projection Theorem (Theorem \ref{thm:hpmc-projection}) explains
how Markov statistics arise from certain maps between Hopf algebras.
For example, Theorem \ref{thm:shuffletoqsym} constructs a Hopf-morphism
which sends a deck of distinct cards to its \emph{descent set} (the
positions where a card has greater value than its immediate successor).
Consequently, tracking the descent set under riffle-shuffling of distinct
cards is itself a Markov chain. In other words, the descent set after
one shuffle only depends on the current descent set, not on the precise
ordering of the deck, an observation originally due to Stanley.

The Hopf-power Markov chain is a very general construction - it can
exhibit various different behaviour depending on the structure of
the underlying Hopf algebra, i.e. on the interplay of the breaking
and combining rules. For example, the stationary distribution of the
edge-removal chain is concentrated at the graph with no edges, whilst
riffle-shuffling has a uniform distribution. In fact, for all combinatorial
families with a ``deterministic combining rule'', their chains are
absorbing, and there is a standard procedure for approximating how
close they are to absorption (Proposition \ref{prop:probboundsrightefns}).

The organisation of the thesis is as follows: Chapter \ref{chap:hpmc-Diagonalisation}
derives some results on the eigenvectors of $m\Delta$, which will
be useful both in constructing and diagonalising Hopf-power Markov
chains. It does not involve any probability. Chapter \ref{chap:linearoperators}
is independent of Chapter \ref{chap:hpmc-Diagonalisation} and describes
the properties of the Doob transform under very general hypotheses,
without reference to Hopf algebras. Chapter \ref{chap:hpmc-construction}
is the centerpiece of the thesis - it contains the construction of
Hopf-power Markov chains, and the theorems regarding their stationary
distribution, reversibility, and Markov statistics. Chapter \ref{chap:free-commutative-basis}
opens with additional theory for chains with a ``deterministic combining
rule'', then illustrates this in detail on the examples of rock-breaking
and tree-pruning. Chapter \ref{chap:cofree} is devoted to the initial
example of riffle-shuffling - Section \ref{sec:Riffle-Shuffling}
derives a full left and right eigenbases and some associated probability
estimates, and Section \ref{sec:Descent-Sets} interprets the left
and right eigenbases of the descent set chain.
\begin{rem*}
An earlier version of the Hopf-power Markov chain framework, restricted
to free-commutative or free state space bases, appeared in \cite{hopfpowerchains}.
Table \ref{tab:sectionmatching} pairs up the results and examples
of that paper and their improvements in this thesis. (I plan to update
this table on my website, as the theory advances and more examples
are available.) In addition, a summary of Section \ref{sec:Descent-Sets},
on the descent set Markov chain under riffle-shuffling, appeared in
\cite{hpmccompositions}.

\begin{table}
\begin{centering}
\begin{tabular}{|c||c|c|}
\hline 
\multicolumn{1}{|c|}{} & \cite{hopfpowerchains} & thesis\tabularnewline
\hline 
\hline 
construction & 3.2 & 4.2,4.3\tabularnewline
\hline 
stationary distribution & 3.7.1 & 4.5\tabularnewline
\hline 
reversibility &  & 4.6\tabularnewline
\hline 
projection &  & 4.7\tabularnewline
\hline 
 &  & \tabularnewline
\hline 
diagonalisation &  & \tabularnewline
\hline 
general & 3.5 & 2\tabularnewline
\hline 
algorithm for free-commutative basis & Th. 3.15 & Th. 2.5.1.A\tabularnewline
\hline 
algorithm for basis of primitives &  & Th. 2.5.1.B\tabularnewline
\hline 
algorithm for shuffle basis &  & Th. 2.5.1.A'\tabularnewline
\hline 
algorithm for free basis & Th. 3.16 & Th. 2.5.1.B'\tabularnewline
\hline 
 &  & \tabularnewline
\hline 
unidirectionality for free-commutative basis & 3.3 & 5.1.2\tabularnewline
\hline 
right eigenfunctions for free-commutative basis & 3.6 & 5.1.3\tabularnewline
\hline 
link to terminality of $QSym$ & 3.7.2 & 5.1.4\tabularnewline
\hline 
 &  & \tabularnewline
\hline 
examples &  & \tabularnewline
\hline 
rock-breaking & 4 & 5.2\tabularnewline
\hline 
tree-pruning &  & 5.3\tabularnewline
\hline 
riffle-shuffling%
\footnote{Due to the limitations of the early Hopf-power Markov chain theory,
\cite[Sec. 5]{hopfpowerchains}studied inverse riffle-shuffling, while
the present Section \ref{sec:Riffle-Shuffling} analyses forward riffle-shuffling.%
} & 5 & 6.1\tabularnewline
\hline 
descent sets under riffle-shuffling &  & 6.2\tabularnewline
\hline 
\end{tabular}
\par\end{centering}

\centering{}\caption{Corresponding sections of \cite{hopfpowerchains} and the present
thesis}
\label{tab:sectionmatching}
\end{table}
\end{rem*}

\chapter{Diagonalisation of the Hopf-power map\label{chap:hpmc-Diagonalisation}}

\chaptermark{Diagonalisation}

This chapter collects together some results on the eigenvectors of
the Hopf-power map; these will be useful in subsequent chapters for
constructing and diagonalising Hopf-power Markov chains. These results
do not require any probability, and may be of interest independently
of Hopf-power Markov chains.

Section \ref{sec:The-Hopf-power-Map} introduces the Hopf-power map
and its dual. The next three sections build towards Theorem \ref{thm:diagonalisation},
a set of four explicit algorithms for full eigenbases of the Hopf-power
map $\Psi^{a}$ on a commutative or cocommutative (graded connected)
Hopf algebra. These allow explicit computations of left and right
eigenbases of the associated Markov chains. Each algorithm follows
the same general two-step principle: first, produce the eigenvectors
of smallest eigenvalue, using the Eulerian idempotent (Section \ref{sec:The-Eulerian-Idempotent}),
then, combine these into eigenvectors of higher eigenvalue, following
Section \ref{sec:Higher-Eigenvalue}. Section \ref{sec:Lyndon-Words}
explains the Lyndon word terminology necessary to implement Theorems
\ref{thm:diagonalisation}.A$'$ and \ref{thm:diagonalisation}.B$'$;
these extended algorithms are useful when the information required
for Theorems \ref{thm:diagonalisation}.A and \ref{thm:diagonalisation}.B
are not readily available. Section \ref{sec:Algorithms-for-Eigenbasis}
contains all four algorithms and their proofs. 

Section \ref{sec:topeigenspace} drops the assumptions of commutativity
or cocommutativity, and proves that the eigenbases algorithms still
hold, in some sense, for the highest eigenvalue. This last result
encodes the stationary distributions for any Hopf-power Markov chain
(Theorem \ref{thm:hpmc-stationarydistribution}), and offers some
explanation as to why certain bases cannot produce Markov chains through
the Doob transform (end of Section \ref{sec:better-definition-of-hpmc}).

\section{The Hopf-power Map\label{sec:The-Hopf-power-Map}}

The Markov chains in this thesis are built from the \emph{$a$th Hopf-power
map} $\Psi^{a}:\calh\rightarrow\calh$, defined to be the \emph{$a$-fold
coproduct} followed by the \emph{$a$-fold produc}t: $\Psi^{a}:=\proda\coproda$.
Here $\coproda:\calh\rightarrow\calh^{\otimes a}$ is defined inductively
by $\coproda:=(\iota\otimes\dots\otimes\iota\otimes\Delta)\Delta^{[a-1]}$,
$\Delta^{[1]}=\iota$ (recall $\iota$ denotes the identity map),
and $\proda:\calh^{\otimes a}\rightarrow\calh$ by $\proda:=m(m^{[a-1]}\otimes\iota)$,
$m^{[1]}=\iota$. So the Hopf-square is coproduct followed by product:
$\Psi^{2}:=m\Delta$. Observe that, on a graded Hopf algebra, the
Hopf-powers preserve degree: $\Psi^{a}:\calhn\rightarrow\calhn$. 

The Hopf-power map first appeared in \cite{tateoort} in the study
of group schemes. The notation $\Psi^{a}$ comes from \cite{patras};
\cite{kashina} writes $[a]$, and \cite{diagonalisingusinggrh} writes
$\iota^{*a}$, since it is the $a$th convolution power of the identity
map. \cite{montgomery} denotes $\Psi^{a}(x)$ by $x^{[a]}$; they
study this operator on finite-dimensional Hopf algebras as a generalisation
of group algebras. The nomenclature ``Hopf-power'' comes from the
fact that these operators exponentiate the basis elements of a group
algebra; in this special case, $\Psi^{a}(g)=g^{a}$. Since this thesis
deals with graded, connected Hopf algebras, there will be no elements
satisfying $\Psi^{a}(g)=g^{a}$, other than multiples of the unit.
However, the view of $\Psi^{a}$ as a power map is still helpful:
on commutative or cocommutative Hopf algebras, the \emph{power rule}
$\Psi^{a}\Psi^{a'}=\Psi^{aa'}$ holds. Here is a simple proof \cite[Lem. 4.1.1]{kashina},
employing Sweedler notation:
\begin{align*}
 & \Psi^{a'}\Psi^{a}(x)\\
= & \sum_{(x)}\Psi^{a'}(x_{(1)}\dots x_{(a)})\\
= & \sum\left[(x_{(1)})_{(1)}(x_{(2)})_{(1)}\dots(x_{(a)})_{(1)}\right]\left[(x_{(1)})_{(2)}\dots(x_{(a)})_{(2)}\right]\dots\left[(x_{(1)})_{(a')}\dots(x_{(a)})_{(a')}\right]\\
= & \sum\left[x_{(1)}x_{(a'+1)}\dots x_{(a'(a-1)+1)}\right]\left[x_{(2)}\dots x_{(a'(a-1)+2)}\right]\dots\left[x_{(a')}\dots x_{(aa')}\right]\\
= & \sum x_{(1)}x_{(2)}\dots x_{(aa')}=\Psi^{aa'}(x).
\end{align*}
(The third equality uses coassociativity, and the fourth uses commutativity
or cocommutativity.)

The Hopf-power Markov chains of this thesis arise from applying the
Doob transform to the Hopf-power map $\Psi^{a}:\calhn\rightarrow\calhn$.
As Theorem \ref{thm:doob-transform} will explain, the Doob transform
requires a special eigenvector of the dual map to $\Psi^{a}$. This
dual map is in fact also a Hopf-power map, but on the dual Hopf algebra,
as defined below. 
\begin{defn}
\label{defn: dualhopfalg}Let $\calh=\bigoplus_{n\geq0}\calhn$ be
a graded, connected Hopf algebra over $\mathbb{R}$ with basis $\calb=\amalg_{n}\calbn$.
The \emph{(graded) dual }of $\calh$ is $\calhdual:=\oplus_{n\geq0}\calhndual$,
where $\calhndual$ is the set of linear functionals on $\calhn$.
(This is the dual of $\calhn$ in the sense of vector spaces, as described
at the start of Chapter \ref{chap:linearoperators}.) The product
and coproduct on $\calhdual$ are defined by 
\[
m(f\otimes g)(x):=(f\otimes g)(\Delta x);\quad\Delta(f)(w\otimes z)=f(wz)
\]
for $x,z,w\in\calh$ and $f,g\in\calhdual$. (Here, $(f\otimes g)(a\otimes b)=f(a)g(b)$.) 
\end{defn}
The symmetry of the Hopf axioms ensures that $\calhdual$ is also
a (graded, connected) Hopf algebra. Note that, for $x\in\calh$ and
$f\in\calhdual$, 
\[
(\proda\coproda f)(x)=(\coproda f)(\coproda x)=f(\proda\coproda x)
\]
so the $a$th Hopf-power of $\calhndual$ is the dual map (in the
linear algebraic sense) to the $a$th Hopf-power on $\calhn$.
\begin{example}
\label{ex:shufflealg-dual}The dual of the shuffle algebra $\calsh$
is the \emph{free associative algebra} $\calsh^{*}$, whose basis
is also indexed by words in the letters $\{1,2,\dots\}$. The product
in $\calsh^{*}$ is concatenation, for example: 
\[
m((12)\otimes(231))=(12231)
\]
and the coproduct is ``deshuffling'': 
\[
\Delta(w_{1}\dots w_{n})=\sum_{S\subseteq\{1,2,\dots,N\}}\prod_{i\in S}w_{i}\otimes\prod_{i\notin S}w_{i}.
\]
For example, 
\begin{align*}
\Delta((316)) & =\emptyset\otimes(316)+(3)\otimes(16)+(1)\otimes(36)+(6)\otimes(31)\\
 & \phantom{=}+(31)\otimes(6)+(36)\otimes(1)+(16)\otimes(3)+(316)\otimes\emptyset.
\end{align*}
Observe that the free associative algebra is noncommutative and cocommutative.
In general, the dual of a commutative algebra is cocommutative, and
vice versa.
\end{example}

\section{The Eulerian Idempotent\label{sec:The-Eulerian-Idempotent}}

The first step in building an eigenbasis for the Hopf-power map $\Psi^{a}$
is to use the Eulerian idempotent map to produce eigenvectors of smallest
eigenvalue. Defining this map requires the notion of the\emph{ reduced
coproduct} $\bard(x):=\Delta(x)-1\otimes x-x\otimes1$. It follows
from the counit axiom that $\bard(x)$ consists precisely of the terms
of $\Delta(x)$ where both tensor-factors have strictly positive degree.
Define inductively the $a$-fold reduced coproduct: $\bard^{[1]}:=\iota$,
and $\bard^{[a]}:=(\iota\otimes\cdots\otimes\iota\otimes\bard)\bard^{[a-1]}$,
which picks out the terms in $\Delta^{[a]}(x)$ with all $a$ tensor-factors
having strictly positive degree. This captures the notion of breaking
into $a$ non-trivial pieces. Note that $\bard^{[2]}=\bard$.
\begin{defn}[Eulerian idempotent]
\label{defn:eulerianidem}\cite[Def. 2.2]{patras} Let $\calh$ be
a Hopf algebra over a field of characteristic zero which is\textit{
conilpotent} (i.e. for each $x$, there is some $a$ with $\bard^{[a]}x=0$).
Then the \textit{(first) Eulerian idempotent} map $e:\calh\rightarrow\calh$
is given by 
\[
e(x)=\sum_{r\geq1}\frac{(-1)^{r-1}}{r}m^{[r]}\bar{\Delta}^{[r]}(x).
\]
(Conilpotence ensures this sum is finite).
\end{defn}
Clearly, graded Hopf algebras are conilpotent: if $x\in\calh_{n}$,
then $\bar{\Delta}^{[r]}(x)=0$ whenever $r>n$.

Patras proved that, if $\calh$ is commutative or cocommutative, then
the image of $e$ is the eigenspace for $\Psi^{a}$ of eigenvalue
$a$. Furthermore, if $\calh$ is cocommutative, \cite[Th. 9.4]{incidencehopfalg}
shows that this image is the subspace of \emph{primitive} elements
of $\calh$, defined to be $\{x\in\calh|\Delta(x)=1\otimes x+x\otimes1\}$.
Note that this subspace is precisely the kernel of the reduced coproduct
map $\bard$. A brief explanation of these properties of $\im(e)$
is at the end of this section, after an example of calculating $e(x)$.
\begin{example}
\label{ex:shufflealg-eulerianidem} Work in the shuffle algebra $\calsh$
of Example \ref{ex:shufflealg-intro}, where the product is interleaving
and the coproduct is deconcatenation. 
\begin{align*}
e((12)) & =(12)-\frac{1}{2}m\bard(12)\\
 & =(12)-\frac{1}{2}(1)(2)\\
 & =(12)-\frac{1}{2}\left[(12)+(21)\right]\\
 & =\frac{1}{2}\left[(12)-(21)\right].
\end{align*}
Observe that 
\[
\bard\left(\frac{1}{2}\left[(12)-(21)\right]\right)=\frac{1}{2}\left[(1)\otimes(2)-(2)\otimes(1)\right],
\]
so, by commutativity, $m\bard e((12))=0$, but $\bard e((12))\neq0$.
Thus $e((12))$ is an eigenvector for $\Psi^{a}$ of eigenvalue $a$,
but is not primitive.

Here is one more demonstration of the Eulerian idempotent:
\begin{align*}
e((123)) & =(123)-\frac{1}{2}m\bard(123)+\frac{1}{3}m^{[3]}\bard^{[3]}(123)\\
 & =(123)-\frac{1}{2}\left[(12)(3)+(1)(23)\right]+\frac{1}{3}(1)(2)(3)\\
 & =(123)-\frac{1}{2}\left[2(123)+(132)+(312)+(213)+(231)\right]\\
 & \hphantom{=}+\frac{1}{3}\left[(123)+(132)+(312)+(213)+(231)+(321)\right]\\
 & =\frac{1}{6}\left[2(123)-(132)-(312)-(213)-(231)+2(321)\right].
\end{align*}

\end{example}
The idea of the Eulerian idempotent came independently from Reutenauer
and from Patras: Reutenauer analysed it on the free associative algebra
$\calsh^{*}$ (see Example \ref{ex:shufflealg-dual}), and Patras
derived the same properties for a general commutative or cocommutative
conilpotent algebra. They both define the Eulerian idempotent as the
logarithm of the identity map in the algebra (under convolution product)
of endomorphisms of $\calh$. To obtain the explicit formula of Definition
\ref{defn:eulerianidem} above, use the Taylor expansion of $\log(1+x)$
with $x$ being $\iota-1$, where $1$ is projection to $\calh_{0}$
(or, more generally, the counit followed by unit). From the familiar
identity 
\[
y^{a}=e^{a\log y}=\sum_{i=0}^{\infty}\frac{a^{i}}{i!}(\log y)^{i}
\]
applied to the identity map, Patras concludes in his Proposition 3.2
that $\Psi^{a}=\sum_{i=0}^{\infty}a^{i}e_{i}$ where the $e_{i}$
are his \emph{higher Eulerian idempotents}, the $i$th convolution
power of $e$ scaled by $i!$:
\[
e_{i}:=\frac{1}{i!}m^{[i]}(e\otimes\dots\otimes e)\Delta^{[i]}.
\]
Hence the usual Eulerian idempotent $e$ is $e_{1}$. Recall from
Section \ref{sec:The-Hopf-power-Map} that, if $\calh$ is commutative
or cocommutative, then the power law holds: $\Psi^{a}\Psi^{a'}=\Psi^{aa'}$
(the left hand side is the composition of two Hopf-powers). In terms
of Eulerian idempotents, this says 
\[
\sum_{i,j=0}^{\infty}a^{i}e_{i}a{}^{\prime j}e_{j}=\sum_{k=0}^{\infty}(aa'){}^{k}e_{k}.
\]
Equating coefficients of $aa'$ then shows that the $e_{i}$ are orthogonal
idempotents under composition: $e_{i}e_{i}=e_{i}$ and $e_{i}e_{j}=0$
for $i\neq j$. Combining this knowledge with the expansion $\Psi^{a}=\sum_{i=0}^{\infty}a^{i}e_{i}$
concludes that $e_{i}$ is the orthogonal projection of $\calh$ onto
the $a^{i}$-eigenspace of $\Psi^{a}$.

\section{Eigenvectors of Higher Eigenvalue\label{sec:Higher-Eigenvalue}}

As just discussed, on a commutative or cocommutative graded Hopf algebra,
Patras's higher Eulerian idempotent maps $e_{k}$ are projections
to the $a^{k}$-eigenspaces for the $a$th Hopf-power. However, this
thesis chooses instead to build the $a^{k}$-eigenspace out of $k$-tuples
of eigenvectors of eigenvalue $a$.

First, consider the case where $\calh$ is commutative. Then, as noted
in \cite{patras}, the power-map $\Psi^{a}$ is an algebra homomorphism:
\begin{align*}
\Psi^{a}(xy) & =m^{[a]}\sum_{(x),(y)}x_{(1)}y_{(1)}\otimes\cdots\otimes x_{(a)}y_{(a)}\\
 & =\sum_{(x),(y)}x_{(1)}y_{(1)}\dots x_{(a)}y_{(a)}=\sum_{(x),(y)}x_{(1)}\dots x_{(a)}y_{(1)}\dots y_{(a)}=\Psi^{a}(x)\Psi^{a}(y).
\end{align*}
Then it follows easily that:
\begin{thm}
Work in a commutative Hopf algebra. Suppose $x_{1},x_{2},\dots,x_{k}$
are eigenvectors of $\Psi^{a}$ of eigenvalue $a$. Then $x_{1}x_{2}\dots x_{k}$
is an eigenvector of $\Psi^{a}$ with eigenvalue $a^{k}$. \qed
\end{thm}
If $\calh$ is not commutative, then a strikingly similar construction
holds, if one restricts the $x_{i}$ to be primitive rather than simply
eigenvectors of eigenvalue $a$. The reasoning is completely different:
\begin{thm}[Symmetrisation Lemma]
\label{thm:symlemma} Let $x_{1},x_{2},\dots,x_{k}$ be primitive
elements of any Hopf algebra, then $\sum_{\sigma\in S_{k}}x_{\sigma(1)}x_{\sigma(2)}\dots x_{\sigma(k)}$
is an eigenvector of $\Psi^{a}$ with eigenvalue $a^{k}$. \end{thm}
\begin{proof}
The proof is essentially a calculation. For concreteness, take $a=2$.
Then 
\begin{align*}
 & m\Delta\left(\sum_{\sigma\in S_{k}}x_{\sigma(1)}x_{\sigma(2)}\dots x_{\sigma(k)}\right)\\
= & m\left(\sum_{\sigma\in S_{k}}\left(\Delta x_{\sigma(1)}\right)\left(\Delta x_{\sigma(2)}\right)\dots\left(\Delta x_{\sigma(k)}\right)\right)\\
= & m\left(\sum_{\sigma\in S_{k}}\left(x_{\sigma(1)}\otimes1+1\otimes x_{\sigma(1)}\right)\dots\left(x_{\sigma(k)}\otimes1+1\otimes x_{\sigma(k)}\right)\right)\\
= & m\left(\sum_{A_{1}\amalg A_{2}=\{1,2,\dots,k\}}\sum_{\sigma\in S_{k}}\prod_{i\in A_{1}}x_{\sigma(i)}\otimes\prod_{j\in A_{2}}x_{\sigma(j)}\right)\\
= & \left|\left\{ \left(A_{1},A_{2}\right)|A_{1}\amalg A_{2}=\{1,2,\dots,k\}\right\} \right|\sum_{\sigma\in S_{k}}x_{\sigma(1)}\dots x_{\sigma(k)}\\
= & 2^{k}\sum_{\sigma\in S_{k}}x_{\sigma(1)}\dots x_{\sigma(k)}.
\end{align*}
For higher $a$, the same argument shows that 
\begin{align*}
 & \Psi^{a}\left(\sum_{\sigma\in S_{k}}x_{\sigma(1)}x_{\sigma(2)}\dots x_{\sigma(k)}\right)\\
= & \proda\left(\sum_{A_{1}\amalg\dots\amalg A_{a}=\{1,2,\dots,k\}}\sum_{\sigma\in S_{k}}\left(\prod_{i\in A_{1}}x_{\sigma(i)}\right)\otimes\dots\otimes\left(\prod_{i\in A_{a}}x_{\sigma(i)}\right)\right)\\
= & a^{k}\sum_{\sigma\in S_{k}}x_{\sigma(1)}\dots x_{\sigma(k)}.
\end{align*}

\end{proof}

\section{Lyndon Words\label{sec:Lyndon-Words}}

The previous two sections show that, for commutative or cocommutative
$\calh$, (symmetrised) products of images under the Eulerian idempotent
map are eigenvectors of the Hopf-power maps $\Psi^{a}$. A natural
question follows: to which elements of $\calh$ should one apply the
Eulerian idempotent map in order for this process to output a basis?
One possible answer is ``the generators of $\calh$'', in a sense
which Theorems \ref{thm:diagonalisation}.A and \ref{thm:diagonalisation}.B
will make precise. Such generators can sometimes be conveniently determined,
but in many cases it is easier to first relate the combinatorial Hopf
algebra to the shuffle algebra or the free associative algebra, and
then use the structure theory of these two famous algebras to pick
out the required generators. This is the main idea of Theorems \ref{thm:diagonalisation}.A$'$
and \ref{thm:diagonalisation}.B$'$ respectively, and this section
explains, following \cite[Sec. 5.1]{lothaire}, the Lyndon word terminology
necessary for this latter step.
\begin{defn}[Lyndon word]
\label{defn:lyndonword} A word is \emph{Lyndon} if it is lexicographically
strictly smaller than its cyclic rearrangements.
\end{defn}
For example, $(11212)$ is Lyndon, as it is lexicographically strictly
smaller than $(12121)$, $(21211)$, $(12112)$ and $(21121)$. The
word $(1212)$ is not Lyndon as it is equal to one of its cyclic rearrangements.
$(31421)$ is also not Lyndon - for example, it does not begin with
its minimal letter. 
\begin{defn}[Lyndon factorisation]
\label{defn:lyndonfact} The \emph{Lyndon factorisation} $u_{1}\cdot\dots\cdot u_{k}$
of $w$ is obtained by taking $u_{k}$ to be the lexicographically
smallest tail of $w$, then $u_{k-1}$ is the lexicographically smallest
tail of $w$ with $u_{k}$ removed, and so on. Throughout this thesis,
$k(w)$ will always denote the number of Lyndon factors in $w$. 
\end{defn}
Observe that $w$ is the concatenation of its Lyndon factors, not
the product of these factors in the sense of the shuffle algebra.
Indeed, all this terminology is independent of the product on the
shuffle algebra. 

\cite[Th. 5.1.5, Prop. 5.1.6]{lothaire} asserts that such $u_{i}$
are each Lyndon - indeed, this is the only way to deconcatenate $w$
into Lyndon words with $u_{1}\geq u_{2}\geq\dots\geq u_{k}$ in lexicographic
order. It follows from this uniqueness that each unordered $k$-tuple
of Lyndon words (possibly with repeats) is the Lyndon factorisation
of precisely one word, namely their concatenation in decreasing lexicographic
order. 
\begin{example}
\label{ex:lyndonfact}Let $w=(31421)$. The tails of $w$ are $(1)$,
$(21)$, $(421)$, $(1421)$ and $(31421)$, and the lexicographically
smallest of these is $(1)$. The lexicographically smallest tail of
$(3142)$ is $(142)$. So $k(w)=3$ and the Lyndon factors of $w$
are $u_{1}=(3)$, $u_{2}=(142)$ and $u_{3}=(1)$.\end{example}
\begin{defn}[Standard factorisation]
\label{defn:stdfact}A Lyndon word $u$ of length greater than 1
has \emph{standard factorisation} $u_{1}\cdot u_{2}$, where $u_{2}$
is the longest Lyndon tail of $u$ that is not $u$ itself, and $u_{1}$
is the corresponding head. By \cite[Prop. 5.1.3]{lothaire}, the head
$u_{1}$ is also Lyndon.\end{defn}
\begin{example}
\label{ex:stdfact}The Lyndon word $u=(1323)$ has two tails which
are Lyndon (and are not $u$ itself): $(3)$ and $(23)$. The longer
Lyndon tail is $(23)$, so the standard factorisation of $u$ is $(1323)=(13\cdot23)$
\end{example}
When using Theorems \ref{thm:diagonalisation}.A$'$ and \ref{thm:diagonalisation}.B$'$
below, it will be more convenient to work with an alphabet of combinatorial
objects rather than the positive integers - all the above notions
are well-defined for ``words'' whose letters are drawn from any
totally-ordered set. In addition, if this set is graded, then one
can assign the degree of a word to be the sum of the degree of its
letters.
\begin{example}
\label{ex:alphabet-lyndonfact-stdfact-degree}If $\bullet$ comes
before $x$ in an alphabet, then the word $\bullet x\bullet$ has
Lyndon factorisation $\bullet x\cdot\bullet$, and the Lyndon word
$\bullet\bullet x$ has standard factorisation $\bullet\cdot\bullet x$.
If $\deg(x)=2$, then both $\bullet x\bullet$ and $\bullet\bullet x$
have degree 4. Example \ref{ex:freebasis-efncalculation} below will
demonstrate the eigenvector corresponding to $\bullet x\bullet$.
\end{example}

\section{Algorithms for a Full Eigenbasis\label{sec:Algorithms-for-Eigenbasis}}

Theorem \ref{thm:diagonalisation} below collects together four algorithms
for a full eigenbasis of the Hopf-power $\Psi^{a}$. Immediately following
are calculations illustrating Parts A$'$ and B$'$, before the proofs
of all four algorithms. These algorithms will be useful in Chapters
\ref{chap:free-commutative-basis} and \ref{chap:cofree} to compute
eigenfunctions of Hopf-power Markov chains.

One more ingredient is necessary to state Part A of Theorem \ref{thm:diagonalisation}:
the dual Cartier-Milnor-Moore theorem \cite[Th. 3.8.3]{cmm} states
that every graded connected commutative Hopf algebra $\calh$ (over
a field $\mathbb{F}$ of characteristic 0) is a polynomial algebra,
i.e. $\calh=\mathbb{F}[c_{1},c_{2},\dots]$ for homogeneous elements
$c_{i}$. $\{c_{1},c_{2},\dots\}$ is then called a \emph{free generating
set} for $\calh$. (The usual Cartier-Milnor-Moore theorem, for cocommutative
Hopf algebras, also plays a role in the eigenbasis algorithms; see
the proof of Part B).

\begin{thm}[Eigenbasis algorithms]
\label{thm:diagonalisation}In all four parts below, $\calh=\bigoplus_{n\geq0}\calhn$
is a graded connected Hopf algebra over $\mathbb{R}$ with each $\calhn$
finite-dimensional.
\begin{description}
\item [{(A)}] Suppose $\calh$ is commutative, and let $\calc$ be a free
generating set for $\calh$. Then $\left\{ e(c_{1})\dots e(c_{k})|k\in\mathbb{N},\left\{ c_{1},\dots,c_{k}\right\} \mbox{ a multiset in }\calc\right\} $
is an eigenbasis for $\Psi^{a}$ on $\calh$, and the eigenvector
$e(c_{1})\dots e(c_{k})$ has eigenvalue $a^{k}$. So the multiplicity
of the eigenvalue $a^{k}$ in $\mathcal{H}_{n}$ is the coefficient
of $x^{n}y^{k}$ in $\prod_{c\in\calc}\left(1-yx^{\deg c}\right)^{-1}$. 
\item [{(B)}] Suppose $\calh$ is cocommutative, and let $\calp$ be a
basis of its primitive subspace. Then $\left\{ \frac{1}{k!}\sum_{\sigma\in\sk}p_{\sigma(1)}\dots p_{\sigma(k)}|k\in\mathbb{N},\left\{ p_{1},\dots,p_{k}\right\} \mbox{ a multiset in }\calp\right\} $
is an eigenbasis for $\Psi^{a}$ on $\calh$, and the eigenvector
$\frac{1}{k!}\sum_{\sigma\in\sk}p_{\sigma(1)}\dots p_{\sigma(k)}$
has eigenvalue $a^{k}$. So the multiplicity of the eigenvalue $a^{k}$
in $\mathcal{H}_{n}$ is the coefficient of $x^{n}y^{k}$ in $\prod_{p\in\calp}\left(1-yx^{\deg p}\right)^{-1}$. 
\item [{(A$'$)}] Suppose $\calh$ is isomorphic, as a non-graded algebra
only, to the shuffle algebra, and write $P_{w}$ for the image in
$\calh$ of the word $w$ under this isomorphism. (So $\left\{ P_{w}\right\} $
is a basis of $\calh$ indexed by words such that $P_{w}P_{w'}=\sum_{v}P_{v}$,
summing over all interleavings $v$ of $w$ and $w'$ with multiplicity.)
For each word $w$, define $g_{w}\in\calh$ recursively to be:
\begin{alignat*}{2}
g_{w} & :=e(P_{w}) & \quad & \mbox{if }w\mbox{ is a Lyndon word};\\
g_{w} & :=g_{u_{1}}\dots g_{u_{k}} & \quad & \mbox{if }w\mbox{ has Lyndon factorisation }w=u_{1}\cdot\dots\cdot u_{k}.
\end{alignat*}
Then $\{g_{w}\}$ is an eigenbasis for $\Psi^{a}$ on $\calh$, and
the eigenvector $g_{w}$ has eigenvalue $a^{k(w)}$, where $k(w)$
is the number of factors in the Lyndon factorisation of $w$. So the
multiplicity of the eigenvalue $a^{k}$ in $\mathcal{H}_{n}$ is the
coefficient of $x^{n}y^{k}$ in ${\displaystyle \prod_{w\mbox{ Lyndon}}\left(1-yx^{\deg P_{w}}\right)^{-1}}$. 
\item [{(B$'$)}] Suppose $\calh$ is cocommutative, and is isomorphic,
as a non-graded algebra only, to the free associative algebra $\mathbb{R}\langle S_{1},S_{2},\dots\rangle$.
For each word $w=w_{1}\dots w_{l}$, where each $w_{i}$ is a letter,
write $S_{w}$ for $S_{w_{1}}\dots S_{w_{l}}$, so $\left\{ S_{w}\right\} $
is a \emph{free basis} with concatenation product. For each word $w$,
define $g_{w}\in\calh$ recursively by:
\hspace*{-1cm}\vbox{ \begin{alignat*}{2} g_{w} & :=e(S_{w}) & \quad  & \mbox{if }w\mbox{ is a single letter};\\ g_{w} & :=\left[g_{u_{1}},g_{u_{2}}\right]:=g_{u_{1}}g_{u_{2}}-g_{u_{2}}g_{u_{1}} & \quad  & \mbox{if }w\mbox{ is Lyndon with standard factorisation }w=u_{1}u_{2};\\ g_{w} & :=\frac{1}{k!}\sum_{\sigma\in\sk}g_{u_{\sigma(1)}}\dots g_{u_{\sigma(k)}} & \quad  & \mbox{if }w\mbox{ has Lyndon factorisation }w=u_{1}\cdot\dots \cdot u_{k}. \end{alignat*} }
Then $\{g_{w}\}$ is an eigenbasis for $\Psi^{a}$ on $\calh$, and
the eigenvector $g_{w}$ has eigenvalue $a^{k(w)}$, where $k(w)$
is the number of factors in the Lyndon factorisation of $w$. So the
multiplicity of the eigenvalue $a^{k}$ in $\mathcal{H}_{n}$ is the
coefficient of $x^{n}y^{k}$ in ${\displaystyle \prod_{w\mbox{ Lyndon}}\left(1-yx^{\deg S_{w}}\right)^{-1}}$. 
\end{description}
\end{thm}
\begin{rems*}
$ $

\begin{enumerate}[label=\arabic*.]
\item  The notation $P$ and $S$ for the bases in Parts A$'$ and B$'$
are intentionally suggestive of dual power sums and complete noncommutative
symmetric functions respectively, see Section \ref{sec:Descent-Sets}.
\item Part A does not imply that the map $x_{i}\rightarrow e(c_{i})$ is
a Hopf-isomorphism from the polynomial algebra $\mathbb{R}[x_{1},x_{2},\dots]$
to any graded connected commutative Hopf algebra, as the $e(c_{i})$
need not be primitive. This map is only a Hopf-isomorphism if the
Hopf algebra in question is cocommutative in addition to being commutative.
See Section \ref{sec:Tree-Pruning} on the tree-pruning process for
a counterexample. Similarly, Part A$'$ does not imply that the shuffle
algebra is Hopf-isomorphic to any Hopf algebra with a shuffle product
structure via the map $w\rightarrow e(P_{w})$ for Lyndon $w$; even
if all the $e(P_{i})$ were primitive, $\bard(e(P_{12}))$ might not
be $e(P_{1})\otimes e(P_{2})$. In short, the presence of a shuffle
product structure is not sufficiently restrictive on the coproduct
structure to uniquely determine the Hopf algebra. 
\item In contrast, the map $i\rightarrow e(S_{i})$ in Part B$'$ does construct
a (non-graded) Hopf-isomorphism from the free associative algebra
$\calsh^{*}$ to any cocommutative Hopf algebra with a free basis.
This is because the image under $e$ of a cocommutative Hopf algebra
is primitive. In fact, the eigenvectors $g_{w}$ are simply the images
of an eigenbasis for the free associative algebra $\calsh^{*}$ under
this isomorphism. Hence the approach of this thesis is as follows:
Section \ref{sub:rightefns-Shuffling} uses Part B$'$ above to generate
an eigenbasis for $\calsh^{*}$, and writes these, up to scaling,
as 
\[
\sum_{w'\in\calsh_{\deg(w)}}\f_{w}^{\calsh}(w')w'.
\]
(The notation $\f_{w}^{\calsh}$ comes from these being the right
eigenfunctions of riffle-shuffling.) It explains a method to calculate
them in terms of decreasing Lyndon hedgerows. Thereafter, the thesis
will ignore Part B$'$ and simply use 
\[
g_{w}=\sum_{w'}\f_{w}^{\calsh}(w')e(S_{w'_{1}})\dots e(S_{w'_{l}})
\]
to obtain the necessary eigenvectors in Section \ref{sub:rightefns-qsym},
taking advantage of the graphical way to calculate $\f_{w}^{\calsh}$.
Here the sum runs over all $w'$ containing the same letters as $w$,
and $w'_{i}$ denotes the $i$th letter of $w'$. This alternative
expression differs from the $g_{w}$ in Part B$'$ above by a scaling
factor, but for the probability applications in this thesis, this
alternative scaling is actually more convenient. 
\item Each part of the Theorem closes with the generating function for the
multiplicities of each eigenvalue on subspaces of each degree. These
are simple generalisations of the generating function for partitions,
since each eigenvector of eigenvalue $a^{k}$ corresponds to a $k$-tuple
(unordered, possibly with repeats) of generators (Part A), primitives
(Part B), or Lyndon words (Parts A$'$ and B$'$). See \cite[Th. 3.14.1]{genfn}.
All four generating functions hold for Hopf algebras that are multigraded
- simply replace all $x$s, $n$s and degrees by tuples, and read
the formula as multi-index notation. For example, for a bigraded commutative
algebra $\calh$ with free generating set $\calc$ (so Part A applies),
the multiplicity of the $a^{k}$-eigenspace in $\calh_{m,n}$ is the
coefficient of $x_{1}^{m}x_{2}^{n}y^{k}$ in $\prod_{c\in\calc}\left(1-yx_{1}^{\deg_{1}c}x_{2}^{\deg_{2}c}\right)^{-1}$,
where $\deg(c)=(\deg_{1}(c),\deg_{2}(c))$. This idea will be useful
in Section \ref{sec:Riffle-Shuffling} for the study of riffle-shuffling. 
\item To analyse Markov chains, one ideally wants expressions for left and
right eigenfunctions of the transition matrix that are ``dual'',
in the sense of Proposition \ref{prop:efnsdiagonalisation}. For Hopf-power
Markov chains, Proposition \ref{prop:efns} below translates this
goal to an eigenbasis for the Hopf-power $\Psi^{a}$ on $\calh$ and
the dual eigenbasis for $\Psi^{a}$ on $\calhdual$. Thus it would
be best to apply the above algorithms to $\calh$ and $\calhdual$
in such a way that the results interact nicely. Theorem \ref{thm:ABdual}
achieves this when a free-commutative basis of $\calh$ is explicit,
using Part A on $\calh$ and Part B on $\calhdual$.
\end{enumerate}
\end{rems*}
\begin{example}
\label{ex:shufflalg-efncalculation} Theorem \ref{thm:diagonalisation}.A$'$
applies to the shuffle algebra, with $P_{w}=w$ for each word $w$.
Take $w=(3141)$, which has Lyndon factorisation $(3\cdot14\cdot1)$.
Then the associated eigenvector $g_{w}$, which has eigenvalue $a^{3}$,
is 
\begin{align*}
 & e((3))e((14))e((1))\\
= & (3)\left[(14)-\frac{1}{2}(1)(4)\right](1)\\
= & (3)\left[\frac{1}{2}(14)-\frac{1}{2}(41)\right](1)\\
= & (3)\frac{1}{2}\left[(141)+2(114)-2(411)-(141)\right]\\
= & (3114)+(1314)+(1134)+(1143)-(3411)-(4311)-(4131)-(4113).
\end{align*}

\end{example}

\begin{example}
\label{ex:freebasis-efncalculation} Consider applying Theorem \ref{thm:diagonalisation}.B$'$
to a Hopf algebra with a free basis to find the eigenvector corresponding
to the word $\bullet x\bullet$, where $\bullet$ and $x$ are letters
with $\deg(\bullet)=1$, $\deg(x)=2$, and $\bullet$ coming before
$x$ in ``alphabetical order''. (This would, for example, construct
a right eigenfunction for the Markov chain of the descent set under
riffle-shuffling corresponding to the composition $(1,2,1)$, if $x$
were $S^{(2)}$. See Example \ref{ex:rightefns-qsym}.) As noted in
Example \ref{ex:alphabet-lyndonfact-stdfact-degree}, the Lyndon factorisation
of $\bullet x\bullet$ is $\bullet x\cdot\bullet$, so, according
to Theorem \ref{thm:diagonalisation}.B$'$ 
\[
g_{\bullet x\bullet}=\frac{1}{2!}\left(g_{\bullet x}g_{\bullet}+g_{\bullet}g_{\bullet x}\right).
\]
The first Lyndon factor $\bullet x$ has standard factorisation $\bullet\cdot x$,
so 
\[
g_{\bullet x}=g_{\bullet}g_{x}-g_{x}g_{\bullet}=e(\bullet)e(x)-e(x)e(\bullet).
\]
As $\deg(\bullet)=1$, it follows that $e(\bullet)=\bullet$. Hence
\begin{align*}
g_{\bullet x\bullet} & =\frac{1}{2!}\left((\bullet e(x)-e(x)\bullet)\bullet+\bullet(\bullet e(x)-e(x)\bullet)\right)\\
 & =\frac{1}{2}(\bullet\bullet e(x)-e(x)\bullet\bullet).
\end{align*}

Alternatively, use the formulation in Remark 3 above,
\[
g_{\bullet x\bullet}=\sum_{w'}\f_{\bullet x\bullet}^{\calsh}(w')e(S_{w'_{1}})\dots e(S_{w'_{l}}).
\]
 summing over all words $w'$ whose letters are $\bullet,x,\bullet$.
Thus 
\begin{align*}
g_{\bullet x\bullet} & =\left[\f_{\bullet x\bullet}^{\calsh}(\bullet\bullet x)\right]e(\bullet)e(\bullet)e(x)+\left[\f_{\bullet x\bullet}^{\calsh}(\bullet x\bullet)\right]e(\bullet)e(x)e(\bullet)+\left[\f_{\bullet x\bullet}^{\calsh}(x\bullet\bullet)\right]e(\bullet)e(\bullet)e(x)\\
 & =\left[\f_{\bullet x\bullet}^{\calsh}(\bullet\bullet x)\right]\bullet\bullet e(x)+\left[\f_{\bullet x\bullet}^{\calsh}(\bullet x\bullet)\right]\bullet e(x)\bullet+\left[\f_{\bullet x\bullet}^{\calsh}(x\bullet\bullet)\right]e(x)\bullet\bullet.
\end{align*}
The graphical calculation of $\f_{\bullet x\bullet}^{\calsh}$ then
shows that $\f_{\bullet x\bullet}^{\calsh}(\bullet\bullet x)=1$,
$\f_{\bullet x\bullet}^{\calsh}(\bullet x\bullet)=0$ and $\f_{\bullet x\bullet}^{\calsh}(x\bullet\bullet)=1$,
so this gives twice the eigenvector found before. As $\bullet x\bullet$
has two Lyndon factors, the eigenvector $g_{\bullet x\bullet}$ has
eigenvalue $a^{2}$.\end{example}
\begin{proof}[Proof of Theorem \ref{thm:diagonalisation}.A]
As explained in Sections \ref{sec:The-Eulerian-Idempotent} and \ref{sec:Higher-Eigenvalue}
respectively, $e(c_{i})$ is an eigenvector of $\Psi^{a}$ with eigenvalue
$a$, and the product of eigenvectors is again an eigenvector, with
the product eigenvalue. Hence $e(c_{1})\dots e(c_{k})$ is an eigenvector
of eigenvalue $a^{k}$. 

To deduce that $\left\{ e(c_{1})\dots e(c_{k})|k\in\mathbb{N},\left\{ c_{1},\dots,c_{k}\right\} \mbox{ a multiset in }\calc\right\} $
is a basis, it suffices to show that the matrix changing $\left\{ e(c_{1})\dots e(c_{k})|k\in\mathbb{N},\left\{ c_{1},\dots,c_{k}\right\} \mbox{ a multiset in }\calc\right\} $
to $\left\{ c_{1}\dots c_{k}|k\in\mathbb{N},\left\{ c_{1},\dots,c_{k}\right\} \mbox{ a multiset in }\calc\right\} $
is uni-triangular, under any ordering which refines the length $k$.
By definition of the Eulerian idempotent map, $e(c_{i})=c_{i}+\mbox{products}$.
So 
\[
e(c_{1})\dots e(c_{k})=c_{1}\dots c_{k}+\mbox{products of at least }k+1\mbox{ factors.}
\]

Expanding these products in terms of the free generating set $\calc$
requires at least $k+1$ $c$'s in each summand.
\end{proof}

\begin{proof}[Proof of Theorem \ref{thm:diagonalisation}.B]
The Symmetrisation Lemma (Theorem \ref{thm:symlemma}) asserts that,
if $x_{1},\dots,x_{k}$ are all primitive, then $\sum_{\sigma\in\sk}x_{\sigma(1)}\dots x_{\sigma(k)}$
is an eigenvector of $\Psi^{a}$ of eigenvalue $a^{k}$. That these
symmetrised products give a basis follows directly from the following
two well-known theorems on the structure of Hopf algebras (recall
from Section \ref{sec:The-Eulerian-Idempotent} that a graded Hopf
algebra is conilpotent because $\bard^{[\deg x+1]}(x)=0$):
\begin{thm*}[Cartier-Milnor-Moore]
 \cite[Th. 3.8.1]{cmm} A connected, conilpotent and cocommutative
Hopf algebra $\calh$ (over a field of characteristic 0) is isomorphic
to $\mathcal{U}(\mathfrak{g})$, the universal enveloping algebra
of a Lie algebra $\mathfrak{g}$, where $\mathfrak{g}$ is the Lie
algebra of primitive elements of $\calh$. 
\end{thm*}

\begin{thm*}[Poincare-Birkhoff-Witt, symmetrised version]
 \cite[Prop. 3.23]{pbwref} If $\{x_{1},x_{2},...\}$ is a basis
for a Lie algebra $\mathfrak{g}$, then the symmetrised products $\sum_{\sigma\in S_{k}}x_{i_{\sigma(1)}}x_{i_{\sigma(2)}}...x_{i_{\sigma(k)}}$,
for $1\leq i_{1}\leq i_{2}\leq\dots\leq i_{k}$, form a basis for
$\mathcal{U}(\mathfrak{g})$.
\end{thm*}
\end{proof}

\begin{proof}[Proof of Theorem \ref{thm:diagonalisation}.A$'$]
 Apply Theorem \ref{thm:diagonalisation}.A, the eigenbasis algorithm
for commutative Hopf algebras, with $\left\{ P_{w}|w\mbox{ Lyndon}\right\} $
as the free generating set $\calc$, since \cite[Th. 6.1.i]{freeliealgs}
asserts that the Lyndon words generate the shuffle algebra freely
as a commutative algebra.
\end{proof}

\begin{proof}[Proof of Theorem \ref{thm:diagonalisation}.B$'$]
 \cite[Prop. 22]{superchar3} shows that $\left\{ g_{i}|i\mbox{ a single letter}\right\} $
generates $\calh$ freely. Since each $g_{i}=e(S_{i})$ is primitive,
the map $i\rightarrow g_{i}$ is a Hopf-isomorphism from the free
associative algebra to $\calh$. Now, by \cite[Th. 5.3.1]{lothaire},
the ``standard bracketing'' of Lyndon words is a basis for the primitive
subspace of the free associative algebra, and its image under this
Hopf-isomorphism is precisely $\left\{ g_{w}|w\mbox{ Lyndon}\right\} $.
So applying Theorem \ref{thm:diagonalisation}.B to $\calp=\left\{ g_{w}|w\mbox{ Lyndon}\right\} $
gives the result.

Here is a second proof employing length-triangularity arguments similar
to those in the proof of Theorem \ref{thm:diagonalisation}.A. First
observe that, if $x,y$ are primitive, then so is $[x,y]=xy-yx$:
\begin{align*}
\Delta(xy-yx) & =\Delta(x)\Delta(y)-\Delta(y)\Delta(x)\\
 & =(1\otimes x+x\otimes1)(1\otimes y+y\otimes1)-(1\otimes y+y\otimes1)(1\otimes x+x\otimes1)\\
 & =1\otimes xy+y\otimes x+x\otimes y+xy\otimes1-(1\otimes yx+x\otimes y+y\otimes x+yx\otimes1)\\
 & =1\otimes xy+xy\otimes1-1\otimes yx-yx\otimes1\\
 & =1\otimes(xy-yx)+(xy-yx)\otimes1.
\end{align*}
Applying this argument recursively shows that, for Lyndon $w$, the
vector $g_{w}$ as defined in the Theorem is indeed primitive. So,
by the Symmetrisation Lemma (Theorem \ref{thm:symlemma}), the $g_{w}$
for general $w$, which are the symmetrised products of the primitive
$g_{w}$, are indeed eigenvectors of $\Psi^{a}$.

To deduce that these give a basis for $\calh$, it suffices to show
that the matrix changing $\left\{ g_{w}\right\} $ to the basis $\left\{ s[w]\right\} $
of \cite[Th. 5.2]{descentalg} is uni-triangular, under any ordering
which refines the length $l(w)$. (Recall that the length $l(w)$
is the number of letters in $w$). The $\left\{ s[w]\right\} $ basis
is defined recursively as follows: 
\begin{alignat*}{2}
s[w] & :=S_{w} & \quad & \mbox{if }w\mbox{ is a single letter};\\
s[w] & :=s[u_{1}]s[u_{2}]-s[u_{2}]s[u_{1}] & \quad & \mbox{if }w\mbox{ is Lyndon with standard factorisation }w=u_{1}u_{2};\\
s[w] & :=\frac{1}{k!}\sum_{\sigma\in\sk}s[u_{\sigma(1)}]\dots s[u_{\sigma(k)}] & \quad & \mbox{if }w\mbox{ has Lyndon factorisation }w=u_{1}\cdot\dots\cdot u_{k}.
\end{alignat*}
(For a Lyndon word $w$, the expression $s[w]$ is known as its \emph{standard
bracketing}.) For single-letter words $w$, $g_{w}=e(w)=S_{w}+\mbox{products}$,
by definition of the Eulerian idempotent map. The recursive definition
of both $g_{w}$ and $s[w]$ show that 
\[
g_{w}=s[w]+\mbox{products of at least }l(w)+1\mbox{ factors.}
\]
As in the proof of Theorem \ref{thm:diagonalisation}.A, expressing
these products in the basis $\left\{ S_{w}\right\} $ involves words
of length at least $l(w)+1$. It is clear from the definition of $s[w]$
that all $S_{u}$ appearing in the $S$-expansion of $s[v]$ have
$l(u)=l(v)$, so all $s[v]$ in the $s$-expansion of these products
have $l(v)\geq l(w)+1$.
\end{proof}

\section{Basis for the Eigenspace of Largest Eigenvalue\label{sec:topeigenspace}}

What are the eigenvectors and eigenvalues of the Hopf-power map $\Psi^{a}$
on a Hopf algebra that is neither commutative nor cocommutative? The
power rule need not hold in this case, so the Eulerian idempotent
map may not produce eigenvectors. By the Symmetrisation Lemma (Theorem
\ref{thm:symlemma}), the symmetrised products of $k$ primitives
are eigenvectors of eigenvalue $a^{k}.$ Appealing to the Poincare-Birkhoff-Witt
theorem on the universal enveloping algebra of the primitives, these
symmetrised products can be made linearly independent, but, without
cocommutativity, these will in general not span the eigenspace.

Recently \cite{diagonalisingusinggrh} found the eigenvalues of $\Psi^{a}$
and their algebraic multiplicities (i.e. the exponents of the factors
in the characteristic polynomial) by passing to $\gr(\calh)$, the
associated graded Hopf algebra of $\calh$ with respect to the coradical
filtration. The key to their argument is a simple linear algebra observation:
the eigenvalues and algebraic multiplicities of $\Psi^{a}$ are the
same for $\calh$ as for $\gr(\calh)$. By \cite[Prop. 1.6]{grhiscommutative},
$\gr(\calh)$ is commutative, so the eigenbasis algorithm in Theorem
\ref{thm:diagonalisation}.A above applies. So the last assertion
of the algorithm gives the following formula:
\begin{thm}
\label{thm:algebraicmultiplicities} \cite[Th. 4 and remark in same section]{diagonalisingusinggrh}
Let $\calh=\bigoplus_{n\geq0}\calhn$ be a graded connected Hopf algebra
over $\mathbb{R}$, and write $b_{i}$ for the number of degree $i$
elements in a free generating set of $\gr(\calh)$. In other words,
$b_{i}$ are the numbers satisfying $\prod_{i}\left(1-x^{i}\right)^{-b_{i}}=\sum_{n}\dim\gr(\calh)_{n}x^{n}=\sum_{n}\dim\calhn x^{n}$.
Then the algebraic multiplicity of the eigenvalue $a^{k}$ for $\Psi^{a}:\mathcal{H}_{n}\rightarrow\calhn$
is the coefficient of $x^{n}y^{k}$ in $\prod_{i}\left(1-yx^{i}\right)^{-b_{i}}$.
Equivalently, this multiplicity is the number of ways to choose $k$
elements, unordered and possibly with repetition, out of $b_{i}$
elements in degree $i$, subject to the condition that their degrees
sum to $n$.\qed
\end{thm}
\begin{rems*}
$ $

\begin{enumerate}[label=\arabic*.]
\item The proof in \cite{diagonalisingusinggrh} applies the Poincare-Birkhoff-Witt
theorem to the dual of $\gr(\calh)$, instead of appealing to the
eigenbasis algorithm on commutative Hopf algebras.
\item Explicit calculations on $FQSym$, the Malvenuto-Reutenauer Hopf algebra
of permutations \cites[Th. 3.3]{duality}{mrstructure} show that $\Psi^{a}$
need not be diagonalisable on a noncommutative, noncocommutative Hopf
algebra - in other words, there are non-trivial Jordan blocks.
\end{enumerate}
\end{rems*}

Happily, in the special case $k=n$ (corresponding to the largest
eigenvalue), this multiplicity formula implies that the Symmetrisation
Lemma indeed builds all eigenvectors of eigenvalue $a^{n}$, provided
$\calh_{1}\neq\emptyset$:
\begin{thm}
\label{thm:topeigenspace}Let $\calh=\bigoplus_{n\geq0}\calhn$ be
a graded connected Hopf algebra over $\mathbb{R}$. Suppose $\calh_{1}\neq\emptyset$,
and let $\calb_{1}$ be a basis of $\calh_{1}$. Then $a^{n}$ is
the largest eigenvalue of the Hopf-power map $\Psi^{a}$ on $\calhn$,
and the corresponding eigenspace has basis 
\[
\cale:=\left\{ \sum_{\sigma\in\sn}c_{\sigma(1)}\dots c_{\sigma(n)}|\left\{ c_{1},\dots,c_{n}\right\} \mbox{ a multiset in }\calb_{1}\right\} .
\]

\end{thm}
As Theorem \ref{thm:hpmc-stationarydistribution} below shows, this
identifies all stationary distributions of a Hopf-power Markov chain.
\begin{proof}
For each monomial $x^{n}y^{k}$ in the generating function $\prod_{i}\left(1-yx^{i}\right)^{-b_{i}}$of
Theorem \ref{thm:algebraicmultiplicities}, it must be that $k\leq n$.
Hence all eigenvalues $a^{k}$ of $\Psi^{a}$ on $\calhn$ necessarily
have $k\leq n$, and thus $a^{n}$ is the largest possible eigenvalue.

Next observe that, since the $c_{i}$ each have degree 1, they are
necessarily primitive. So $\sum_{\sigma\in\sn}c_{\sigma(1)}\dots c_{\sigma(n)}$
is a symmetrised product of $n$ primitives, which the Symmetrisation
Lemma (Theorem \ref{thm:symlemma}) asserts is an eigenvector of $\Psi^{a}$
of eigenvalue $a^{n}$. Working in the universal enveloping algebra
of $\calh_{1}$, the Poincare-Birkhoff-Witt theorem gives linear independence
of $\left\{ \sum_{\sigma\in\sn}c_{\sigma(1)}\dots c_{\sigma(n)}\right\} $
across all choices of multisets $\left\{ c_{1},\dots,c_{n}\right\} \subseteq\calb_{1}$.

To conclude that the set $\cale$ of symmetrised products span the
$a^{n}$-eigenspace, it suffices to show that $|\cale|$ is equal
to the algebraic multiplicity of the eigenvalue $a^{n}$ as specified
by Theorem \ref{thm:algebraicmultiplicities}. Clearly $|\cale|=\binom{|\calb_{1}|+n-1}{n}$,
the number of ways to choose $n$ unordered elements, allowing repetition,
from $\calb_{1}$. On the other hand, the algebraic multiplicity is
$\binom{b_{1}+n-1}{n}$, since choosing $n$ elements whose degrees
sum to $n$ constrains each element to have degree 1. By equating
the coefficient of $x$ in the equality $\prod_{i}\left(1-x^{i}\right)^{-b_{i}}=\sum_{n}\dim\calhn x^{n}$,
it is clear that $b_{1}=\dim\calh_{1}=|\calb_{1}|$. So $|\cale|$
is indeed the algebraic multiplicity of the eigenvalue $a^{n}$. 
\end{proof}
The condition $\calh_{1}\neq\emptyset$ is satisfied for the vast
majority of combinatorial Hopf algebras, so this thesis will not require
the analogous, clumsier, result for general $\calh$, though I include
it below for completeness. To determine the highest eigenvalue, first
define the sets $\cald:=\{d>0|\calh_{d}\neq\emptyset\}$, and $\cald'=\{d\in\cald|d\neq d_{1}+d_{2}\mbox{ with }d_{1},d_{2}\in\cald\}$.
In the familiar case where $\cald=\{1,2,3,\dots\},$ the set $\cald'$
is $\{1\}$. It is possible to build Hopf algebras with $\cald$ being
any additively-closed set - for example, take a free associative algebra
with a generator in degree $d$ for each $d\in\cald$, and let all
these generators be primitive. The reason for considering $\cald'$
is that $\bigoplus_{d\in\cald'}\calh_{d}$ consists solely of primitives:
for $x\in\calh_{d}$, the counit axiom mandates that $\bard(x)\in\bigoplus_{d_{1}+d_{2}=d}\calh_{d_{1}}\otimes\calh_{d_{2}}$,
and this direct sum is empty if $d\in\cald'$. However, there may
well be primitives in higher degrees.

For a fixed degree $n\in\cald$, define a \emph{$\cald'$-partition
of $n$ }to be an unordered tuple $\lambda:=(\lambda_{1},\dots,\lambda_{l(\lambda)})$
such that each $\lambda_{i}\in\cald'$ and $\lambda_{1}+\dots+\lambda_{l(\lambda)}=n$.
The \emph{parts} $\lambda_{i}$ need not be distinct. Then $l(\lambda)$
is the \emph{length} of $\lambda$. (The analogous notion of a $\cald$-partition
will be useful in the proof of Theorem \ref{thm:topeigenspace2}.)
\begin{example}
\label{ex:topeigenspace}Suppose $\cald=\{5,6,7,9,10,11,\dots\}=\mathbb{N}\backslash\{1,2,3,4,8\}$,
so $\cald'=\{5,6,7,9\}$. There are four $\cald'$-partitions of 23:
$(6,6,6,5)$, $(7,6,5,5)$, $(9,7,7)$, $(9,9,5)$. These have length
$4,4,3,3$ respectively.\end{example}
\begin{thm}
\label{thm:topeigenspace2}Let $\calh=\bigoplus_{n\in\cald}\calhn$
be a graded connected Hopf algebra over $\mathbb{R}$. Then the highest
eigenvalue of the Hopf-power map $\Psi^{a}$ on $\calhn$ is $a^{K(n)}$,
where $K(n)$ denotes the maximal length of a $\cald'$-partition
of $n$. A basis for the corresponding eigenspace is 
\[
\cale:=\left\{ \sum_{\sigma\in\skk}c_{\sigma(1)}\dots c_{\sigma(K)}|\left\{ c_{1},\dots,c_{K}\right\} \mbox{ a multiset in }\calb\mbox{ with }\deg c_{1}+\dots+\deg c_{K}=n\right\} .
\]
More explicitly, for each $\cald'$-partition $\lambda$ of $n$ of
the maximal length $K$, set 
\[
\cale_{\lambda}:=\left\{ \sum_{\sigma\in\skk}c_{\sigma(1)}\dots c_{\sigma(K)}\left|\begin{array}{c}
\left\{ c_{1},\dots,c_{m_{1}}\right\} \mbox{ a multiset in }\calb_{1},\\
\left\{ c_{m_{1}+1},\dots,c_{m_{1}+m_{2}}\right\} \mbox{ a multiset in }\calb_{2},\dots
\end{array}\right.\right\} ,
\]
where $m_{i}$ is the number of parts of size $i$ in $\lambda$.
Then $\cale=\amalg\cale_{\lambda}$, over all $\cald'$-partitions
$\lambda$ of $n$ having length $K$.\end{thm}
\begin{example}
Continue from Example \ref{ex:topeigenspace}. In degree 23, the highest
eigenvalue of $\Psi^{a}$ is $a^{4}$, and its corresponding eigenspace
has basis $\cale_{(6,6,6,5)}\amalg\cale_{(7,6,5,5)}$, where
\[
\cale_{(6,6,6,5)}:=\left\{ \sum_{\sigma\in\mathfrak{S}_{4}}c_{\sigma(1)}c_{\sigma(2)}c_{\sigma(3)}c_{\sigma(4)}\left|\begin{array}{c}
c_{1}\in\calb_{5},\\
\left\{ c_{2},c_{3},c_{4}\right\} \mbox{ a multiset in }\calb_{6}
\end{array}\right.\right\} ,
\]
\[
\cale_{(7,6,5,5)}:=\left\{ \sum_{\sigma\in\mathfrak{S}_{4}}c_{\sigma(1)}c_{\sigma(2)}c_{\sigma(3)}c_{\sigma(4)}\left|\begin{array}{c}
\left\{ c_{1},c_{2}\right\} \mbox{ a multiset in }\calb_{5},\\
c_{3}\in\calb_{6},c_{4}\in\calb_{7}
\end{array}\right.\right\} .
\]
\end{example}
\begin{proof}
The argument below is essentially a more careful version of the proof
of Theorem \ref{thm:topeigenspace}.

By Theorem \ref{thm:algebraicmultiplicities}, $a^{k}$ is an eigenvalue
of $\Psi^{a}:\calhn\rightarrow\calhn$ if and only if there are $k$
elements in $\calh$ whose degrees sum to $n$. In other words, $a^{k}$
is an eigenvalue precisely when there is a $\cald$-partition of $n$
of length $k$. Note that a $\cald$-partition of $n$ with maximal
length must be a $\cald'$-partition: if a part $\lambda_{i}$ of
$\lambda$ is not in $\cald'$, then $\lambda_{i}=d_{1}+d_{2}$ with
$d_{1},d_{2}\in\cald$, and replacing $\lambda_{i}$ with two parts
$d_{1},d_{2}$ in $\lambda$ creates a longer partition. Hence the
largest eigenvalue of $\Psi^{a}:\calhn\rightarrow\calhn$ corresponds
to the maximal length of a $\cald'$-partition of $n$.

As observed earlier, every element of $\bigoplus_{d\in\cald'}\calh_{d}$
is primitive, by degree considerations. So each element in $\cale$
is a symmetrised product of $K$ primitives; by the Symmetrisation
Lemma (Theorem \ref{thm:symlemma}), they are eigenvectors of $\Psi^{a}$
of eigenvalue $a^{K}$. As before, applying the Poincare-Birkhoff-Witt
theorem to the universal enveloping algebra of $\bigoplus_{d\in\cald'}\calh_{d}$
gives linear independence of $\cale$. 

It remains to show that $\cale$ spans the $a^{K}$-eigenspace. The
dimension-counting argument which closes the proof of Theorem \ref{thm:topeigenspace}
will function, so long as $b_{i}=|\calb_{i}|$ for each $i\in\cald'$.
Recall that $b_{i}$ is defined by $\prod_{i}\left(1-x^{i}\right)^{-b_{i}}=\sum_{d\in\cald}|\calb_{d}|x^{d}$.
Equating coefficients of $x^{d}$ for $d\not\in\cald$ shows that
$b_{d}=0$ for $d\not\in\cald$, so the left hand side is $\prod_{i\in\cald}\left(1-x^{i}\right)^{-b_{i}}$.
Now, for each $i\in\cald'$, there is no $d_{1},d_{2}\in\cald$ with
$i=d_{1}+d_{2}$, so the coefficient of $x^{i}$ in $\prod_{i\in\cald}\left(1-x^{i}\right)^{-b_{i}}$
is $b_{i}$.
\end{proof}

\chapter{Markov chains from linear operators\label{chap:linearoperators}}

As outlined previously in Section \ref{sec:Hopf-power-Markov-chains-intro},
one advantage of relating riffle-shuffling to the Hopf-square map
on the shuffle algebra is that Hopf algebra theory supplies the eigenvalues
and eigenvectors of the transition matrix. Such a philosophy applies
whenever the transition matrix is the matrix of a linear operator.
Although this thesis treats solely the case where this operator is
the Hopf-power, some arguments are cleaner in the more general setting,
as presented in this chapter. The majority of these results have appeared
in the literature under various guises.

Section \ref{sec:Construction} explains how the Doob transform normalises
a linear operator to obtain a transition matrix. Then Sections \ref{sec:Diagonalisation},
\ref{sec:Reversibility}, \ref{sec:Projection} connect the eigenbasis,
stationary distribution and time-reversal, and projection of this
class of chains respectively to properties of its originating linear
map.

A few pieces of notation: in this chapter, all vector spaces are finite-dimensional
over $\mathbb{R}$. For a linear map $\theta:V\rightarrow W$, and
bases $\calb,\calb'$ of $V,W$ respectively, $\left[\theta\right]_{\calb,\calb'}$
will denote the matrix of $\theta$ with respect to $\calb$ and $\calb'$.
In other words, the entries of $\left[\theta\right]_{\calb,\calb'}$
satisfy
\[
\theta(v)=\sum_{w\in\calb'}\left[\theta\right]_{\calb,\calb'}(w,v)w
\]
for each $v\in\calb$. When $V=W$ and $\calb=\calb'$, shorten this
to $\left[\theta\right]_{\calb}$. The \emph{transpose} of a matrix
$A$ is given by $A^{T}(x,y):=A(y,x)$. The \emph{dual vector space}
to $V$, written $V^{*}$, is the set of linear functions from $V$
to $\mathbb{R}$. If $\calb$ is a basis for $V$, then the natural
basis to use for $V^{*}$ is $\calbdual:=\left\{ x^{*}|x\in\calb\right\} $,
where $x^{*}$ satisfies $x^{*}(x)=1$, $x^{*}(y)=0$ for all $y\in\calb$,
$y\neq x$. In other words, $x^{*}$ is the linear extension of the
indicator function on $x$. When elements of $V$ are expressed as
column vectors, it is often convenient to view these functions as
row vectors, so that evaluation on an element of $V$ is given by
matrix multiplication. The \emph{dual map} to $\theta:V\rightarrow W$
is the linear map $\theta^{*}:W^{*}\rightarrow V^{*}$ satisfying
$(\theta^{*}f)(v)=f(\theta v)$. Note that $\left[\theta^{*}\right]_{\calb'^{*},\calbdual}=\left[\theta\right]_{\calb,\calb'}^{T}$.

\section{Construction\label{sec:Construction}}

The starting point is as follows: $V$ is a vector space with basis
$\calb$, and $\Psi:V\rightarrow V$ is a linear map. Suppose the
candidate transition matrix $K:=\left[\Psi\right]_{\calb}^{T}$ has
all entries non-negative, but its rows do not necessarily sum to 1.

One common way to resolve this is to divide each entry of $K$ by
the sum of the entries in its row. This is not ideal for the present
situation since the outcome is no longer a matrix for $\Psi$. For
example, an eigenbasis of $\Psi$ will not give the eigenfunctions
of the resulting matrix.

A better solution comes in the form of Doob's $h$-transform. This
is usually applied to a transition matrix with the row and column
corresponding to an absorbing state removed, to obtain the transition
matrix of the chain conditioned on non-absorption. Hence some of the
references listed in Theorem \ref{thm:doob-transform} below assume
that $K$ is \emph{sub-Markovian} (i.e. $\sum_{y}K(x,y)<1$), but,
as the calculation in the proof shows, that is unnecessary. 

The Doob transform works in great generality, for continuous-time
Markov chains on general state spaces. In the present discrete case,
it relies on an eigenvector $\eta$ of the dual map $\Psi^{*}$, that
takes only positive values on the basis $\calb$. Without imposing
additional constraints on $\Psi$ (which will somewhat undesirably
limit the scope of this theory), the existence of such an eigenvector
$\eta$ is not guaranteed. Even when $\eta$ exists, it may not be
unique in any reasonable sense, and different choices of $\eta$ will
in general lead to different Markov chains. However, when $\Psi$
is a Hopf-power map, there is a preferred choice of $\eta$, given
by Definition \ref{defn:eta}. Hence this thesis will suppress the
dependence of this construction on the eigenvector $\eta$. 
\begin{thm}[Doob $h$-transform for non-negative linear maps]
\cites[Sec. XIII.6.1]{perronfrob}[Def. 8.11, 8.12]{doobtransformbook}[Lemma 4.4.1.1]{zhou}[Sec.17.6.1]{markovmixing}[Lem. 2.7]{doobnotes}
\label{thm:doob-transform} Let $V$ be a vector space with basis
$\calb$, and $\Psi:V\rightarrow V$ be a non-zero linear map for
which $K:=\left[\Psi\right]_{\calb}^{T}$ has all entries non-negative.
Suppose there is an eigenvector $\eta$ of the dual map $\Psi^{*}$
taking only positive values on $\calb$, and let $\beta$ be the corresponding
eigenvalue. Then 
\[
\hatk(x,y):=\frac{1}{\beta}K(x,y)\frac{\eta(y)}{\eta(x)}
\]
defines a transition matrix. Equivalently, $\hatk:=\left[\frac{\Psi}{\beta}\right]_{\hatcalb}^{T}$,
where $\hatcalb:=\left\{ \hatx:=\frac{x}{\eta(x)}|x\in\calb\right\} $.
\end{thm}
Call the resulting chain a\emph{ $\Psi$-Markov chain on $\calb$}
(neglecting the dependence on its \emph{rescaling function} $\eta$
as discussed previously). See Example \ref{ex:schurfn-chain} for
a numerical illustration of this construction.
\begin{proof}
First note that $K:=\left[\Psi^{*}\right]_{\calbdual}$, so $\Psi^{*}\eta=\beta\eta$
translates to $\sum_{y}K(x,y)\eta(y)=\beta\eta(x)$. (Functions satisfying
this latter condition are called \emph{harmonic}, hence the name $h$-transform.)
Since $\eta(y)>0$ for all $y$, $K(x,y)\geq0$ for all $x,y$ and
$K(x,y)>0$ for some $x,y$, the eigenvalue $\beta$ must be positive.
So $\hatk(x,y)\ge0$. It remains to show that the rows of $\hatk$
sum to 1: 
\[
\sum_{y}\hatk(x,y)=\frac{\sum_{y}K(x,y)\eta(y)}{\beta\eta(x)}=\frac{\beta\eta(x)}{\beta\eta(x)}=1.
\]

\end{proof}
\begin{rems*}
$ $

\begin{enumerate}[label=\arabic*.]
\item $\beta$, the eigenvalue of $\eta$, is necessarily the largest eigenvalue
of $\Psi$. Here's the reason: by the Perron-Frobenius theorem for
non-negative matrices \cite[Ch. XIII Th. 3]{perronfrob}, there is
an eigenvector $\xi$ of $\Psi$, with largest eigenvalue $\beta_{\max}$,
whose components are all non-negative. As $\eta$ has all components
positive, the matrix product $\eta^{T}\xi$ results in a positive
number. But $\beta\eta^{T}\xi=(\Psi^{*}\eta)^{T}\xi=\eta^{T}(\Psi\xi)=\beta_{\max}\eta^{T}\xi$,
so $\beta=\beta_{\max}$.
\item Rescaling the basis $\calb$ does not change the chain: suppose $\calb'=\left\{ x':=\alpha_{x}x|x\in\calb\right\} $
for some non-zero constants $\alpha_{x}$. Then, since $\eta$ is
a linear function, 
\[
\check{x'}:=\frac{x'}{\eta(x')}=\frac{\alpha_{x}x}{\alpha_{x}\eta(x)}=\hatx.
\]
Hence the transition matrix for both chains is the transpose of the
matrix of $\Psi$ with respect to the same basis. This is used in
Theorem \ref{thm: reversibility} to give a condition under which
the chain is reversible.
\item In the same vein, if $\eta'$ is a multiple of $\eta$, then both
eigenvectors $\eta'$ and $\eta$ give rise to the same $\Psi$-Markov
chain, since the transition matrix depends only on the ratio $\frac{\eta(y)}{\eta(x)}$.
\end{enumerate}
\end{rems*}

\section{Diagonalisation\label{sec:Diagonalisation}}

Recall that the main reason for defining the transition matrix $\hatk$
to be the transpose of a matrix for some linear operator $\Psi$ is
that it reduces the diagonalisation of the Markov chain to identifying
the eigenvectors of $\Psi$ and its dual $\Psi^{*}$. Proposition
\ref{prop:efns} below records precisely the relationship between
the left and right eigenfunctions of the Markov chain and these eigenvectors;
it is immediate from the definition of $\hatk$ above.
\begin{prop}[Eigenfunctions of $\Psi$-Markov chains]
\cites[Lemma 4.4.1.4]{zhou}[Lem. 2.11]{doobnotes}\label{prop:efns}
Let $V$ be a vector space with basis $\calb$, and $\Psi:V\rightarrow V$
be a linear operator allowing the construction of a $\Psi$-Markov
chain (whose transition matrix is $\hatk:=\left[\frac{\Psi}{\beta}\right]_{\hatcalb}^{T}$,
where $\hatcalb:=\left\{ \hatx:=\frac{x}{\eta(x)}|x\in\calb\right\} $).
Then:
\begin{description}
\item [{(L)}] Given a function $\g:\calb\rightarrow\mathbb{R}$, define
a vector $g\in V$ by 
\[
g:=\sum_{x\in\calb}\frac{\g(x)}{\eta(x)}x.
\]
Then $\g$ is a left eigenfunction, of eigenvalue $\beta'$, for this
$\Psi$-Markov chain if and only if $g$ is an eigenvector, of eigenvalue
$\beta\beta'$, of $\Psi$. Consequently, given a basis $\left\{ g_{i}\right\} $
of $V$ with $\Psi g_{i}=\beta_{i}g_{i}$, the set of functions 
\[
\left\{ \g_{i}(x):=\mbox{coefficient of }x\mbox{ in }\eta(x)g_{i}\right\} 
\]
is a basis of left eigenfunctions for the $\Psi$-Markov chain, with
$\sum_{x}\hatk(x,y)\g(x)=\frac{\beta_{i}}{\beta}\g(y)$ for all $y$.
\item [{(R)}] Given a function $\f:\calb\rightarrow\mathbb{R}$, define
a vector $f$ in the dual space $V^{*}$ by 
\[
f:=\sum_{x\in\calb}\f(x)\eta(x)x^{*}.
\]
Then $\f$ is a right eigenfunction, of eigenvalue $\beta'$, for
this $\Psi$-Markov chain if and only if $f$ is an eigenvector, of
eigenvalue $\beta\beta'$, of the dual map $\Psi^{*}$. Consequently,
given a basis $\left\{ f_{i}\right\} $ of $V^{*}$ with $\Psi^{*}f_{i}=\beta_{i}f_{i}$,
the set of functions 
\[
\left\{ \f_{i}(x):=\frac{1}{\eta(x)}f_{i}(x)\right\} 
\]
is a basis of right eigenfunctions for the $\Psi$-Markov chain, with
$\sum_{x}\hatk(x,y)\f(y)=\frac{\beta_{i}}{\beta}\f(x)$ for all $x$.\qed
\end{description}
\end{prop}
\begin{rem*}
In the Markov chain literature, the term ``left eigenvector'' is
often used interchangeably with ``left eigenfunction'', but this
thesis will be careful to make a distinction between the eigenfunction
$\g:\calb\rightarrow\mathbb{R}$ and the corresponding eigenvector
$g\in V$ (and similarly for right eigenfunctions).
\end{rem*}

\section{Stationarity and Reversibility\label{sec:Reversibility}}

Recall from Section \ref{sec:Markov-chains-intro} that, for a Markov
chain with transition matrix $K$, a \emph{stationary distribution}
$\pi(x)$ is one which satisfies $\sum_{x}\pi(x)K(x,y)=\pi(y)$, or,
if written as a row vector, $\pi K=\pi$. So it is a left eigenfunction
of eigenvalue 1. These are of interest as they include all possible
limiting distributions of the chain. The following Proposition is
essentially a specialisation of Proposition \ref{prop:efns}.L to
the case $\beta'=1$:
\begin{prop}[Stationary Distributions of $\Psi$-Markov chains]
\label{prop: stationarydistribution}\cites[Lemma 4.4.1.2]{zhou}[Lem. 2.16]{doobnotes}
Work in the setup of Theorem \ref{thm:doob-transform}. The stationary
distributions $\pi$ of a $\Psi$-Markov chain are in bijection with
the eigenvectors $\xi=\sum_{x\in\calb}\xi_{x}x$ of the linear map
$\Psi$ of eigenvalue $\beta$, which have $\xi_{x}\geq0$ for all
$x\in\calb$, and are scaled so $\eta(\xi)=\sum_{x}\eta(x)\xi_{x}=1$.
The bijection is given by $\pi(x)=\eta(x)\xi_{x}$.\qed
\end{prop}
Observe that a stationary distribution always exists: as remarked
after Theorem \ref{thm:doob-transform}, $\beta$ is the largest eigenvalue
of $\Psi$, and the Perron-Frobenius theorem guarantees a corresponding
eigenvector with all entries non-negative. Rescaling this then gives
a $\xi$ satisfying the conditions of the Proposition.

For the rest of this section, assume that $\beta$ has multiplicity
1 as an eigenvalue of $\Psi$, so there is a unique stationary distribution
$\pi$ and corresponding eigenvector $\xi$ of the linear map $\Psi$.
(Indeed, Proposition \ref{prop: stationarydistribution} above asserts
that $\beta$ having multiplicity 1 is also the necessary condition.)
Assume in addition that $\pi(x)>0$ for all $x\in\calb$. Then, there
is a well-defined notion of the Markov chain run backwards; that is,
one can construct a stochastic process $\{X_{m}^{*}\}$ for which
\[
P\{X_{0}^{*}=x_{i},X_{1}^{*}=x_{i-1},\dots,X_{i}^{*}=x_{0}\}=P\{X_{0}=x_{0},X_{1}=x_{1},\dots,X_{i}=x_{i}\}
\]
for every $i$. As \cite[Sec. 1.6]{markovmixing} explains, if the
original Markov chain \emph{started in stationarity} (i.e. $P(X_{0}=x)=\pi(x)$),
then this reversed process is also a Markov chain - the formal \emph{time-reversal}
chain - with transition matrix 
\[
K^{*}(x,y)=\frac{\pi(y)}{\pi(x)}K(y,x).
\]

Theorem \ref{thm: dual} below shows that, if the forward chain is
built from a linear map via the Doob transform, then its time-reversal
corresponds to the dual map.
\begin{thm}[Time-reversal of a $\Psi$-Markov chain]
\label{thm: dual}Work in the framework of Theorem \ref{thm:doob-transform}.
If the time-reversal of a $\Psi$-Markov chain is defined, then it
arises from applying the Doob transform to the linear-algebraic-dual
map $\Psi^{*}:V^{*}\rightarrow V^{*}$ with respect to the dual basis
$\calbdual$.\end{thm}
\begin{proof}
Let $K^{*}$ denote the transpose of the matrix of $\Psi^{*}$ with
respect to the basis $\calbdual$. Then $K^{*}(x^{*},y^{*})=K(y,x)$.
By definition, the transition matrix of a $\Psi^{*}$-Markov chain
is 
\[
\check{K^{*}}(x^{*},y^{*})=\frac{K^{*}(x^{*},y^{*})}{\beta^{*}}\frac{\eta^{*}(y^{*})}{\eta^{*}(x^{*})},
\]
where $\eta^{*}$ is an eigenvector of the dual map to $\Psi^{*}$
with $\eta^{*}(x^{*})>0$ for all $x^{*}\in\calbdual$, and $\beta^{*}$
is its eigenvalue. Identify the dual map to $\Psi^{*}$ with $\Psi$;
then $\xi$ is such an eigenvector, since the condition $\pi(x)>0$
for the existence of a time-reversal is equivalent to $\xi(x^{*})=\xi_{x}>0$.
Then $\beta^{*}=\beta$, so 
\begin{align*}
\check{K^{*}}(x^{*},y^{*}) & =\frac{K^{*}(x^{*},y^{*})}{\beta}\frac{\xi_{y}}{\xi_{x}}\\
 & =\frac{K(y,x)}{\beta}\frac{\xi_{y}\eta(y)}{\xi_{x}\eta(x)}\frac{\eta(x)}{\eta(y)}\\
 & =\frac{\pi(y)}{\pi(x)}\frac{K(y,x)}{\beta}\frac{\eta(x)}{\eta(y)}\\
 & =\frac{\pi(y)}{\pi(x)}\hatk(y,x).
\end{align*}
\end{proof}
\begin{rem*}
This time-reversed chain is in fact the only possible $\Psi^{*}$-Markov
chain on $\calbdual$; all possible rescaling functions $\eta^{*}$
give rise to the same chain. Here is the reason: as remarked after
Theorem \ref{thm:doob-transform}, a consequence of the Perron-Frobenius
theorem is that all eigenvectors with all coefficients positive must
correspond to the largest eigenvalue. Here, the existence of a time-reversal
constrains this eigenvalue to have multiplicity 1, so any other choice
of $\eta^{*}$ must be a multiple of $\xi$, hence defining the same
$\Psi^{*}$-Markov chain on $\calbdual$.
\end{rem*}
Markov chains that are \emph{reversible}, that is, equal to their
own time-reversal, are particularly appealing as they admit more tools
of analysis. It is immediate from the definition of the time-reversal
that the necessary and sufficient conditions for reversibility are
$\pi(x)>0$ for all $x$ in the state space, and the \emph{detailed
balance equation} $\pi(x)K(x,y)=\pi(y)K(y,x)$. Thanks to Theorem
\ref{thm: dual}, a $\Psi$-Markov chain is reversible if and only
if $\left[\Psi\right]_{\calb}=\left[\Psi^{*}\right]_{\calbdual}$.
As the right hand side is $\left[\Psi\right]_{\calb}^{T}$, this equality
is equivalent to $\left[\Psi\right]_{\calb}$ being a symmetric matrix.
A less coordinate-dependent rephrasing is that $\Psi$ is self-adjoint
with respect to some inner product where the basis $\calb$ is orthonormal.
Actually, it suffices to require that the vectors in $\calb$ are
pairwise orthogonal; the length of the vectors are unimportant since,
as remarked after Theorem \ref{thm:doob-transform}, all rescalings
of a basis define the same chain. To summarise:
\begin{thm}
\label{thm: reversibility}Let $V$ be a vector space with an inner
product, and $\calb$ a basis of $V$ consisting of pairwise orthogonal
vectors. Suppose $\Psi:V\rightarrow V$ is a self-adjoint linear map
admitting the construction of a $\Psi$-Markov chain on $\calb$,
and that this chain has a unique stationary distribution, which happens
to take only positive values. Then this chain is reversible.\qed
\end{thm}

\section{Projection\label{sec:Projection}}

Sometimes, one is interested only in one particular feature of a Markov
chain. A classic example from \cite{shufflenondistinct} is shuffling
cards for a game of Black-Jack, where the suits of the cards are irrelevant.
In the same paper, they also study the position of the ace of spades.
In situations like these, it makes sense to study the \emph{projected}
process $\{\theta(X_{m})\}$ for some function $\theta$ on the state
space, rather than the original chain $\{X_{m}\}$. Since $\theta$
effectively merges several states into one, the process $\{\theta(X_{m})\}$
is also known as the \emph{lumping} of $\{X_{m}\}$ under $\theta$.

Since the projection $\{\theta(X_{m})\}$ is entirely governed by
$\{X_{m}\}$, information about $\{\theta(X_{m})\}$ can shed some
light on the behaviour of $\{X_{m}\}$. For example, the convergence
rate of $\{\theta(X_{m})\}$ is a lower bound for the convergence
rate of $\{X_{m}\}$. So, when $\{X_{m}\}$ is too complicated to
analyse, one may hope that some $\{\theta(X_{m})\}$ is more tractable
- after all, its state space is smaller. For chains on algebraic structures,
quotient structures often provide good examples of projections. For
instance, if $\{X_{m}\}$ is a random walk on a group, then $\theta$
can be a group homomorphism. Section \ref{sec:hpmc-projection} will
show that the same applies to Hopf-power Markov chains. 

In the ideal scenario, the projection $\{\theta(X_{m})\}$ is itself
a Markov chain also. As explained in \cite[Sec. 6.3]{lumping}, $\{\theta(X_{m})\}$
is a Markov chain for any starting distribution if and only if the
sum of probabilities $\sum_{y:\theta(y)=\bary}K(x,y)$ depends only
on $\theta(x)$, not on $x$. This condition is commonly known as
\emph{Dynkin's criterion}. (Weaker conditions suffice if one desires
$\{\theta(X_{m})\}$ to be Markov only for particular starting distributions,
see \cite[Sec. 6.4]{lumping}.) Writing $\barx$ for $\theta(x)$,
the chain $\left\{ \theta(X_{m})\right\} $ then has transition matrix
\[
\bark(\barx,\bary)=\sum_{y:\theta(y)=\bary}K(x,y)\quad\mbox{ for any }x\mbox{ with }\theta(x)=\barx.
\]
Equivalently, as noted in \cite[Th. 6.3.4]{lumping}, if $R$ is the
matrix with 1 in positions $x,\theta(x)$ for all $x$, and 0 elsewhere,
then $KR=R\bark$.

To apply this to chains from linear maps, take $\theta:V\rightarrow\barv$
to be a linear map and suppose $\theta$ sends the basis $\calb$
of $V$ to a basis $\barcalb$ of $\barv$. ($\theta$ must be surjective,
but need not be injective - several elements of $\calb$ may have
the same image in $\barv$, as long as the distinct images are linearly
independent.) Then the matrix $R$ above is $\left[\theta\right]_{\calb,\barcalb}^{T}$.
Recall that $K=\left[\Psi\right]_{\calb}^{T}$, and let $\bark=\left[\bar{\Psi}\right]_{\barcalb}^{T}$
for some linear map $\bar{\Psi}:\barv\rightarrow\barv$. Then the
condition $KR=R\bark$ is precisely $\left[\theta\Psi\right]_{\calb,\barcalb}^{T}=\left[\bar{\Psi}\theta\right]_{\calb,\barcalb}^{T}$.
A $\theta$ satisfying this type of relation is commonly known as
an \emph{intertwining map}. So, if $K,\bark$ are transition matrices,
then $\theta\Psi=\bar{\Psi}\theta$ guarantees that the chain built
from $\Psi$ lumps to the chain built from $\bar{\Psi}$. 

When $K$ is not a transition matrix, so the Doob transform is non-trivial,
an extra hypothesis is necessary:
\begin{thm}
\label{thm:projection}Let $V,\barv$ be vector spaces with bases
$\calb,\barcalb$, and let $\Psi:V\rightarrow V,\bar{\Psi}:\barv\rightarrow\barv$
be linear maps allowing the Markov chain construction of Theorem \ref{thm:doob-transform},
using dual eigenvectors $\eta,\bar{\eta}$ respectively. Let $\theta:V\rightarrow\barv$
be a linear map with $\theta(\calb)=\barcalb$ and $\theta\Psi=\bar{\Psi}\theta$.
Suppose in addition that at least one of the following holds:
\begin{enumerate}
\item all entries of $\left[\Psi\right]_{\calb}$ are positive;
\item the largest eigenvalue of $\Psi$ has multiplicity 1;
\item for all $x\in\calb$, $\bar{\eta}(\theta(x))=\alpha\eta(x)$ for some
constant $\alpha\neq0$
\end{enumerate}

Then $\theta$ defines a projection of the $\Psi$-Markov chain to
the $\bar{\Psi}$-Markov chain.

\end{thm}
\begin{rem*}
Condition iii is the weakest of the three hypotheses, and the only
one relevant to the rest of the thesis, as there is an easy way to
verify it on Hopf-power Markov chains. This then leads to Theorem
\ref{thm:hpmc-projection}, the Projection Theorem of Hopf-power Markov
chains. Hypotheses i and ii are potentially useful when there is no
simple expression for $\eta(x)$. \end{rem*}
\begin{proof}
Let $\beta,\bar{\beta}$ be the largest eigenvalues of $\Psi,\bar{\Psi}$
respectively. The equality $\left[\theta\Psi\right]_{\hatcalb,\hatbarcalb}^{T}=\left[\bar{\Psi}\theta\right]_{\hatcalb,\hatbarcalb}^{T}$
gives $\left(\beta\hatk\right)\hatr=\hatr(\bar{\beta}\hatbark)$,
where $\hatr=\left[\theta\right]_{\hatcalb,\hatbarcalb}^{T}$. The
goal is to recover $\hatk R=R\hatbark$ from this: first, show that
$\beta=\bar{\beta}$, then, show that $\hatr=\alpha R$. 

To establish that the top eigenvalues are equal, appeal to \cite[Thms. 1.3.1.2, 1.3.1.3]{johnpike},
which in the present linear-algebraic notation reads: (the asterisks
denote taking the linear-algebraic dual map)
\begin{prop}
\label{prop:projectefns}$ $
\begin{enumerate}
\item If $\barf$ is an eigenvector of $\bar{\Psi}^{*}$ with eigenvalue
$\beta'$, then $f:=\theta^{*}\barf$ (i.e. $f(x)=\barf(\barx)$),
if non-zero, is an eigenvector of $\Psi^{*}$ with eigenvalue $\beta'$.
\item If $g$ is an eigenvector of $\Psi$ with eigenvalue $\beta''$, then
$\barg:=\theta g$ (i.e. $\barg_{\barx}=\sum_{x|\theta(x)=\barx}g_{x}$),
if non-zero, is an eigenvector of $\bar{\Psi}$ with eigenvalue $\beta''$.\qed
\end{enumerate}
\end{prop}

So it suffices to show that $\theta^{*}\barf\neq0$ for at least one
eigenvector $\barf$ of $\bar{\Psi}^{*}$ with eigenvalue $\bar{\beta}$,
and $\theta g\neq0$ for at least one eigenvector $g$ of $\Psi$
with eigenvalue $\beta$. Since $\barf$ is non-zero, it is clear
that $f(x)=\barf(\barx)\neq0$ for some $x$. As for $g$, the Perron-Frobenius
theorem guarantees that each component of $g$ is non-negative, and
since some component of $g$ is non-zero, $\barg_{\barx}=\sum_{x|\theta(x)=\barx}g_{x}$
is non-zero for some $\bar{x}$.

Now show $\hatr=\alpha R$. Recall that $\hatr=\left[\theta\right]_{\hatcalb,\hatbarcalb}^{T}$,
so its $x,\theta(x)$ entry is $\frac{\bar{\eta}(\barx)}{\eta(x)}$.
The corresponding entries of $R$ are all 1, and all other entries
of both $\hatr$ and $R$ are zero. So hypothesis iii exactly ensures
that $\hatr=\alpha R$. Hypothesis i clearly implies hypothesis ii
via the Perron-Frobenius theorem. To see that hypothesis ii implies
hypothesis iii, use Proposition \ref{prop:projectefns}.i in the above
paragraph: the composite function $\bar{\eta}\theta$, sending $x$
to $\bar{\eta}(\barx)$, is a non-zero eigenvector of $\Psi^{*}$
with eigenvalue $\bar{\beta}=\beta$; as this eigenvalue has multiplicity
1, it must be some multiple of $\eta$.
\end{proof}

\chapter{Construction and Basic Properties of Hopf-power Markov Chains\label{chap:hpmc-construction}}

\chaptermark{Construction and Basic Properties}

This chapter covers all theory of Hopf-power Markov chains that do
not involve diagonalisation, and does not require commutativity or
cocommutativity. The goal is the following routine for initial analysis
of a Hopf-power Markov chain:
\begin{itemize}
\item (Definition \ref{defn:statespacebasis}) discern whether the given
Hopf algebra $\calh$ and basis $\calb$ are suitable for building
a Hopf-power Markov chain (whether $\calb$ satisfies the conditions
of a state space basis);
\item (Definition \ref{defn: better-defition-of-hpmc}) build the Hopf-power
Markov chain;
\item (Definition \ref{defn:eta}) calculate the rescaling function $\eta$;
\item (Theorem \ref{thm:threestep}) describe the chain combinatorially
without using the Hopf algebra structure;
\item (Theorem \ref{thm:hpmc-stationarydistribution}) obtain its stationary
distributions;
\item (Theorem \ref{thm:hpmc-dual}) describe the time-reversal of this
process.
\end{itemize}
Two examples will be revisited throughout Sections \ref{sec:better-definition-of-hpmc}-\ref{sec:hpmc-Reversibility}
to illustrate the main theorems, building the following two blurbs
step by step.
\begin{example*}[Riffle-shuffling]
The shuffle algebra $\calsh$ has basis $\calb$ consisting of words.
The product of two words is the sum of their interleavings, and the
coproduct is deconcatenation (Example \ref{ex:shufflealg}). The rescaling
function is the constant function 1; in other words, no rescaling
is necessary to create the associated Markov chain (Example \ref{ex:shuffle-eta}).
The $a$th Hopf-power Markov chain is the Bayer-Diaconis $a$-handed
generalisation of the GSR riffle-shuffle (Example \ref{ex:shuffle-threestep}):
\begin{enumerate}[label=\arabic*.]
\item Cut the deck multinomially into $a$ piles.
\item Interleave the $a$ piles with uniform probability.
\end{enumerate}

Its stationary distribution is the uniform distribution (Example \ref{ex:schurfn-stationarydistribution}).
Its time-reversal is inverse-shuffling (Example \ref{ex:inverseshuffle}):
\begin{enumerate}[label=\arabic*.]
\item With uniform probability, assign each card to one of $a$ piles,
keeping the cards in the same relative order.
\item Place the first pile on top of the second pile, then this combined
pile on top of the third pile, etc.
\end{enumerate}
\end{example*}

\begin{example*}[Restriction-then-induction]
Let $\calh$ be the vector space spanned by representations of the
symmetric groups $\sn$, over all $n\in\mathbb{N}$. Let $\calb$
be the basis of irreducible representations. The product of representations
of $\sn$ and $\sm$ is the induction of their external product to
$\mathfrak{S}_{n+m}$, and the coproduct of a representation of $\sn$
is the sum of its restrictions to $\mathfrak{S}_{i}\times\mathfrak{S}_{n-i}$
for $0\leq i\leq n$ (Example \ref{ex:schurfn}). For any irreducible
representation $x$, the rescaling function $\eta(x)$ evaluates to
its dimension $\dim x$ (Example \ref{ex:schurfn-eta}). One step
of the $a$th Hopf-power Markov chain, starting from an irreducible
representation $x$ of $\sn$, is the following two-fold process (Example
\ref{ex:schurfn-twostep}):
\begin{enumerate}[label=\arabic*.]
\item Choose a Young subgroup $\mathfrak{S}_{i_{1}}\times\dots\times\mathfrak{S}_{i_{a}}$
multinomially.
\item Restrict the starting state $x$ to the chosen subgroup, induce it
back up to $\sn$, then pick an irreducible constituent with probability
proportional to the dimension of its isotypic component.
\end{enumerate}

The stationary distribution of this chain is the famous Plancherel
measure (Example \ref{ex:schurfn-stationarydistribution}). This chain
is reversible (Example \ref{ex:schurfn-reversible}).

\end{example*}
Section \ref{sec:Combinatorial-Hopf-algebras} reviews the literature
on combinatorial Hopf algebras. Section \ref{sec:first-defition-of-hpmc}
gives a rudimentary construction of Hopf-power Markov chains, which
is improved in Section \ref{sec:better-definition-of-hpmc}, using
the Doob transform of Section \ref{sec:Construction}. Section \ref{sec:threestep}
derives an interpretation of these chains as a breaking step followed
by a combining step. Section \ref{sec:stationarydistribution} gives
a complete description of the stationary distributions. Sections \ref{sec:hpmc-Reversibility}
and \ref{sec:hpmc-projection} employ the theory of Sections \ref{sec:Reversibility}
and \ref{sec:Projection} respectively to deduce that the time-reversal
of a Hopf-power chain is that associated to its dual algebra, and
that the projection of a Hopf-power chain under a Hopf-morphism is
the Hopf-power chain on the target algebra.

\section{Combinatorial Hopf algebras\label{sec:Combinatorial-Hopf-algebras}}

Recall from Section \ref{sec:Hopf-algebras-intro} the definition
of a graded connected Hopf algebra: it is a vector space $\calh=\bigoplus_{n=0}^{\infty}\calh_{n}$
with a product map $m:\calh_{i}\otimes\calh_{j}\rightarrow\calh_{i+j}$
and a coproduct map $\Delta:\calh_{n}\to\bigoplus_{j=0}^{n}\calh_{j}\otimes\calh_{n-j}$
satisfying $\Delta(wz)=\Delta(w)\Delta(z)$ and some other axioms.
To construct the Markov chains in this thesis, the natural Hopf algebras
to use are \emph{combinatorial Hopf algebras}, where the product and
coproduct respectively encode how to combine and split combinatorial
objects. These easily satisfy the non-negativity conditions required
to define the associated Markov chain, which then has a natural interpretation
in terms of breaking an object and then reassembling the pieces. A
motivating example of a combinatorial Hopf algebra is:
\begin{example}[Shuffle algebra]
\label{ex:shufflealg}The shuffle algebra $\calsh(N)$, as defined
in \cite{shufflealg}, has as its basis the set of all words in the
letters $\{1,2,\dots,N\}$. The number of letters $N$ is usually
unimportant, so we write this algebra simply as $\calsh$. These words
are notated in parantheses to distinguish them from integers.

The product of two words is the sum of all their interleavings, with
multiplicity. For example,
\[
m((13)\otimes(52))=(13)(52)=(1352)+(1532)+(1523)+(5132)+(5123)+(5213),
\]
\resizebox{\textwidth}{!}{  $(12)(231)=2(12231)+(12321)+(12312)+(21231)+(21321)+(21312)+(23121)+2(23112). $}
\cite[Sec. 1.5]{freeliealgs} shows that deconcatenation is a compatible
coproduct. For example,
\[
\Delta((316))=\emptyset\otimes(316)+(3)\otimes(16)+(31)\otimes(6)+(316)\otimes\emptyset.
\]
 (Here, $\emptyset$ denotes the empty word, which is the unit of
$\calsh$.)

The associated Markov chain is the GSR riffle-shuffle of Example \ref{ex:shuffle-intro};
below Example \ref{ex:shuffle-threestep} will deduce this connection
from Theorem \ref{thm:threestep}.
\end{example}
The idea of using Hopf algebras to study combinatorial structures
was originally due to Joni and Rota \cite{jonirota}. The concept
enjoyed increased popularity in the late 1990s, when \cite{cktrees}
linked a combinatorial Hopf algebra on trees (see Section \ref{sec:Tree-Pruning}
below) to renormalisation in theoretical physics. Today, an abundance
of combinatorial Hopf algebras exists; see the introduction of \cite{foissychalist}
for a list of references to many examples. An instructive and entertaining
overview of the basics and the history of the subject is in \cite{hopfzoo}.
\cite{chastructure} gives structure theorems for these algebras analogous
to the Poincare-Birkhoff-Witt theorem (see Section \ref{sec:Algorithms-for-Eigenbasis}
above) for cocommutative Hopf algebras. 

A particular triumph of this algebrisation of combinatorics is \cite[Th. 4.1]{abs},
which claims that $QSym$, the algebra of quasisymmetric functions
(Example \ref{ex:qsym} below) is the terminal object in the category
of combinatorial Hopf algebras with a multiplicative linear functional
called a \emph{character}. Their explicit map from any such algebra
to $QSym$ unifies many ways of assigning polynomial invariants to
combinatorial objects, such as the chromatic polynomial of graphs
and Ehrenboug's quasisymmetric function of a ranked poset. Section
\ref{sub:Absorption} makes the connection between these invariants
and the probability of absorption of the associated Hopf-power Markov
chains.

There is no universal definition of a combinatorial Hopf algebra in
the literature; each author considers Hopf algebras with slightly
different axioms. What they do agree on is that it should have a distinguished
basis $\calb$ indexed by ``combinatorial objects'', such as permutations,
set partitions, or trees, and it should be graded by the ``size''
of these objects. The Hopf algebra is connected since the empty object
is the only object of size 0.

For $x,y,z_{1},\dots,z_{a}\in\calb$, define\emph{ structure constants}
$\xi_{z_{1,}\dots,z_{a}}^{y},\eta_{x}^{z_{1},\dots,z_{a}}$ by 
\[
z_{1}\dots z_{a}=\sum_{y\in\calb}\xi_{z_{1,}\dots,z_{a}}^{y}y,\quad\Delta^{[a]}(x)=\sum_{z_{1},\dots,z_{a}\in\calb}\eta_{x}^{z_{1},\dots,z_{a}}z_{1}\otimes\dots\otimes z_{a}.
\]
Note that, by the inductive definitions of $\proda$ and $\coproda$,
all structure constants are determined by $\xi_{w,z}^{y}$ and $\eta_{x}^{w,z}$
(see the proof of Lemma \ref{lem: nonnegative-coeffs}). Shorten these
to $\xi_{wz}^{y}$ and $\eta_{x}^{wz}$, without the comma in between
$w$ and $z$. In a combinatorial Hopf algebra, these two numbers
should have interpretations respectively as the (possibly weighted)
number of ways to combine $w,z$ and obtain $y$, and the (possibly
weighted) number of ways to break $x$ into $w,z$. Then, the compatibility
axiom $\Delta(wz)=\Delta(w)\Delta(z)$ translates roughly into the
following: suppose $y$ is one possible outcome when combining $w$
and $z$; then every way of breaking $y$ comes (bijectively) from
a way of breaking $w$ and $z$ separately. The axioms $\deg(wz)=\deg(w)+\deg(z)$
and $\Delta(x)\in\bigoplus_{i=0}^{\deg(x)}\calh_{i}\otimes\calh_{\deg(x)-i}$
simply say that the ``total size'' of an object is conserved under
breaking and combining.

These are the minimal conditions for a combinatorial Hopf algebra,
and will be sufficient for this thesis. For interest, a common additional
hypothesis is the existence of an internal product $\calhn\otimes\calhn\rightarrow\calhn$,
and perhaps also an internal coproduct. Note that commutativity of
a combinatorial Hopf algebra indicates a symmetric assembling rule,
and a symmetric breaking rule induces a cocommutative\emph{ }Hopf
algebra.

Many families of combinatorial objects have a single member of size
1, so $\calh_{1}$ is often one-dimensional. For example, there is
only one graph on one vertex, and only one partition of total size
1. In such cases, $\bullet$ will denote this sole object of size
1, so $\calb_{1}=\{\bullet\}$. A larger $\calb_{1}$ may be the sign
of a disconnected state space. That is, the associated Markov chain
may separate into two (or more) chains running on disjoint subsets
of the state space. For example, the usual grading on the shuffle
algebra is by the length of the words. Then $\calsh_{3}$ contains
both permutations of $\{1,2,3\}$ and permutations of $\{1,1,2\}$,
but clearly no amount of shuffling will convert from one set to the
other. To study these two Markov chains separately, refine the degree
of a word $w$ to be a vector whose $i$th component is the number
of occurrences of $i$ in $w$. (Trailing $0$s in this vector are
usually omitted.) So summing the components of this multidegree gives
the old notion of degree. Now $\calsh_{(1,1,1)}$ contains the permutations
of $\{1,2,3\}$, whilst $\calsh_{(2,1)}$ contains the permutations
of $\{1,1,2\}$. 

As Proposition \ref{prop:multigrading} below will show, there is
often an analogous multigrading on any combinatorial Hopf algebra
with $|\calb_{1}|>1$. The catch is that elements of the basis $\calb$
may not be homogeneous in this multigrading, that is, $\calb$ might
not be the disjoint union of bases $\calb_{\nu}$ for each degree
$\nu$ subspace $\calh_{\nu}$. (Currently, I do not know of any examples
of such non-homogeneous bases.) In the case where $\calb=\amalg_{\nu}\calb_{\nu}$,
Theorem \ref{thm:hpmc-stationarydistribution}.ii shows that the stationary
distribution of the associated Markov chains (on each subspace $\calh_{\nu}$)
is unique.
\begin{prop}
\label{prop:multigrading}Let $\calh=\bigoplus_{n\geq0}\calhn$ be
a graded connected Hopf algebra over $\mathbb{R}$. Suppose $\calb_{1}:=\{\bullet_{1},\bullet_{2},\dots,\bullet_{|\calb_{1}|}\}$
is a basis of $\calh_{1}$. For each $\nu=(\nu_{1},\dots,\nu_{|\calb_{1}|})\in\mathbb{N}^{|\calb_{1}|}$,
set $c_{1}=c_{2}=\dots=c_{\nu_{1}}=\bullet_{1}$, $c_{\nu_{1}+1}=\dots=c_{\nu_{1}+\nu_{2}}=\bullet_{2}$,
etc., and define 
\[
\calh_{\nu}:=\{x\in\calh|\bard^{[|\nu|]}(x)\in\sspan\{c_{\sigma(1)}\otimes\dots\otimes c_{\sigma(|\nu|)}|\sigma\in\mathfrak{S}_{|\nu|}\}.
\]
If $\calh_{n}=\bigoplus_{|\nu|=n}\calh_{\nu}$, then this gives a
multigrading on $\calh$ refining the $\mathbb{N}$-grading. This
is the unique multigrading satisfying $\deg(\bullet_{1})=(1,0,\dots,0),\deg(\bullet_{2})=(0,1,0,\dots0),\dots,\deg(\bullet_{|\calb_{1}|})=(0,\dots,0,1)$.\end{prop}
\begin{proof}
Comultiplication respects this notion of degree as coassociativity
implies $\Delta^{[i+j]}(x)=(\Delta^{[i]}\otimes\Delta^{[j]})(\Delta x)$.

It is trickier to see the product respecting the degree. Take $z\in\calh_{i},w\in\calh_{j}$.
Then $\Delta^{[i+j]}(zw)=\Delta^{[i+j]}(z)\Delta^{[i+j]}(w)$. Since
$\deg(z)=i$, at least $j$ tensor-factors in each term of $\Delta^{[i+j]}(z)$
are in $\calh_{0}$, and the same is true for at least $i$ tensor-factors
in each term of $\Delta^{[i+j]}(w)$. Hence a term in $\bard^{[i+j]}(zw)$
must arise from terms in $\Delta^{[i+j]}(z),\Delta^{[i+j]}(w)$ which
have exactly $j$ and $i$ tensor-factors respectively in $\calh_{0}$,
in complementary positions. A term of $\Delta^{[i+j]}(z)$ with $j$
tensor-factors in degree 0 must have the remaining $i$ tensor-factors
in degree 1, hence it corresponds to a term in $\bard^{[i]}(z)$,
and similarly for $w$. So there is a bijection 
\begin{align*}
\left\{ \begin{array}{c}
\mbox{terms in}\\
\bard^{[i]}(z)
\end{array}\right\}  & \times & \left\{ \begin{array}{c}
\mbox{terms in}\\
\bard^{[j]}(w)
\end{array}\right\}  & \times & \left\{ \begin{array}{c}
\mbox{subsets of }\\
\{1,2,\dots,i+j\}\\
\mbox{of size }i
\end{array}\right\}  & \leftrightarrow\left\{ \begin{array}{c}
\mbox{terms in}\\
\bard^{[i+j]}(zw)
\end{array}\right\} \\
c_{1}\otimes\dots\otimes c_{i} & , & c'_{1}\otimes\dots\otimes c'_{j} & , & k_{1}<\dots<k_{i} & \rightarrow k_{r}\text{th tensor-factor is }c_{r},\\
 &  &  &  &  & \phantom{\rightarrow}c'_{1},\dots,c'_{j}\mbox{ in other tensor-factors}.
\end{align*}
And so the multidegree $\deg(zw)$ is $\deg(z)+\deg(w)$.

As for uniqueness: suppose $\bard^{[|\nu|]}(x)\in\sspan\{c_{\sigma(1)}\otimes\dots\otimes c_{\sigma(|\nu|)}|\sigma\in\mathfrak{S}_{|\nu|}\}$.
Then, since the coproduct respects the multigrading, it must be that
$\deg(x)=\deg c_{1}+\dots+\deg c_{|\nu|}=\nu$.
\end{proof}
The rest of this section is a whistle-stop tour of three sources of
combinatorial Hopf algebras. A fourth important source is operads
\cite{operads}, but that theory is too technical to cover in detail
here.

\subsection{Species-with-Restrictions\label{sub:Species}}

This class of examples is especially of interest in this thesis, as
the associated Markov chains have two nice properties. Firstly, constructing
these chains does not require the Doob transform (Definition \ref{defn: first-defition-of-hpmc}).
Secondly, the natural bases of these Hopf algebras are \emph{free-commutative}
in the sense of Chapter \ref{chap:free-commutative-basis}, so additional
tools are available to study the associated Markov chains. For instance,
these chains are absorbing, and Section \ref{sub:Altrightefns} provides
bounds for the probability of being ``far from absorption''.

The theory of species originated in \cite{species}, as an abstraction
of common manipulations of generating functions. Loosely speaking,
a species is a type of combinatorial structure which one can build
on sets of ``vertices''. Important examples include (labelled) graphs,
trees and permutations. The formal definition of a species is as a
functor from the category of sets with bijections to the same category.
In this categorical language, the species of graphs maps a set $V$
to the set of all graphs whose vertices are indexed by $V$. There
are operations on species which correspond to the multiplication,
composition and differentiation of their associated generating functions;
these are not so revelant to the present Markov chain construction,
so the reader is referred to \cite{speciesbook} for further details.

Schmitt \cite{schmitt} first makes the connection between species
and Hopf algebras. He defines a species-with-restrictions, or $R$-species,
to be a functor from sets with coinjections to the category of functions.
(A \emph{coinjection} is a partially-defined function whose restriction
to where it's defined is a bijection; an example is $f:\{1,2,3,4\}\rightarrow\{7,8\}$
with $f(1)=8$, $f(3)=7$ and $f(2),f(4)$ undefined.) Intuitively,
these are combinatorial structures with a notion of restriction to
a subset of their vertex set; for example, one can restrict a graph
to a subset of its vertices by considering only the edges connected
to this subset (usually known as the \emph{induced subgraph}). Schmitt
fashions from each such species a Hopf algebra which is both commutative
and cocommutative; Example \ref{ex:graph} below explains his construction
via the species of graphs.
\begin{example}[The Hopf algebra of graphs]
\cites[Sec. 12]{incidencehopfalg}[Sec. 3.2]{fisher}  \label{ex:graph}
Let $\bar{\calg}$ be the vector space with basis the set of simple
graphs (no loops or multiple edges). The vertices of such graphs are
unlabelled, so these may be considered the isomorphism classes of
graphs. Define the degree of a graph to be its number of vertices.
The product of two graphs is their disjoint union, and the coproduct
is 
\[
\Delta(G)=\sum G_{S}\otimes G_{S^{\calc}}
\]
where the sum is over all subsets $S$ of vertices of $G$, and $G_{S},G_{S^{\calc}}$
denote the subgraphs that $G$ induces on the vertex set $S$ and
its complement. As an example, Figure \ref{fig:coproduct-graphs}
calculates the coproduct of $P_{3}$, the path of length 3. Writing
$P_{2}$ for the path of length 2, and $\bullet$ for the unique graph
on one vertex, this calculation shows that 
\[
\Delta(P_{3})=P_{3}\otimes1+2P_{2}\otimes\bullet+\bullet^{2}\otimes\bullet+2\bullet\otimes P_{2}+\bullet\otimes\bullet^{2}+1\otimes P_{3}.
\]
As mentioned above, this Hopf algebra, and analogous constructions
from other species-with-restrictions, are both commutative and cocommutative.
\begin{figure}
\begin{centering}
\includegraphics[scale=0.4]{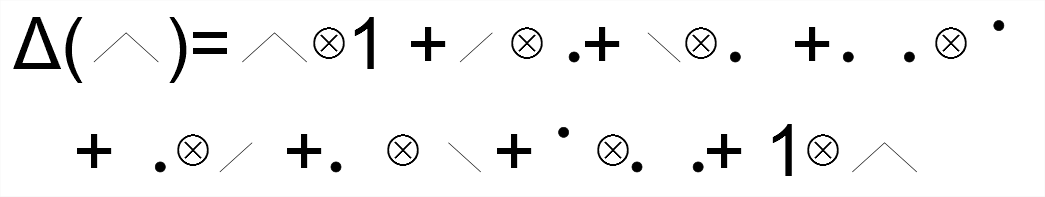}
\par\end{centering}

\caption{An example coproduct calculation in $\barcalg$, the Hopf algebra
of graphs}
\label{fig:coproduct-graphs}
\end{figure}

As Example \ref{ex:chain-graph} will describe, the Hopf-power Markov
chain on $\bar{\calg}$ models the removal of edges: at each step,
colour each vertex independently and uniformly in one of $a$ colours,
and disconnect edges between vertices of different colours. This chain
will act as the running example in Section \ref{sec:freecommutative},
to illustrate general results concerning a Hopf-power Markov chain
on a free-commutative basis. However, because the concept of graph
is so general, it is hard to say anything specific or interesting
without restricting to graphs of a particular structure. For example,
restricting to unions of complete graphs gives the rock-breaking chain
of Section \ref{sec:Rock-breaking}. I aim to produce more such examples
in the near future.
\end{example}
Recently, Aguiar and Mahajan \cite{hopfmonoids} extended vastly this
construction to the concept of a \emph{Hopf monoid in species}, which
is a finer structure than a Hopf algebra. Their Chapter 15 gives two
major pathways from a species to a Hopf algebra: the Bosonic Fock
functor, which is essentially Schmitt's original idea, and the Full
Fock functor. (Since the product and coproduct in the latter involves
``shifting'' and ``standardisation'' of labels, the resulting
Hopf algebras lead to rather contrived Markov chains, so this thesis
will not explore the Full Fock functor in detail.) In addition there
are decorated and coloured variants of these two constructions, which
allow the input of parameters. Many popular combinatorial Hopf algebras,
including all examples in this thesis, arise from Hopf monoids; perhaps
this is an indication that the Hopf monoid is the ``correct'' setting
to work in. The more rigid set of axioms of a Hopf monoid potentially
leads to stronger theorems.

In his masters' thesis, Pineda \cite{hopfmonoidchains} transfers
some of the Hopf-power Markov chain technology of this thesis to the
world of Hopf monoids, building a Markov chain on faces of a permutohedra.
His chain has many absorbing states, a phenomenon not seen in any
of the chains in this thesis. This suggests that a theory of Markov
chains from Hopf monoids may lead to a richer collection of examples.

\subsection{Representation rings of Towers of Algebras\label{sub:Representation-rings}}

The ideas of this construction date back to Zelevinsky \cite[Sec. 6]{pshclassification},
which the lecture notes \cite[Sec. 4]{vicreinernotes} retell in modern
notation. The archetype is as follows:
\begin{example}[Representations of symmetric groups]
\label{ex:schurfn} Let $\calb_{n}$ be the irreducible representations
of the symmetric group $\sn$, so $\calh_{n}$ is the vector space
spanned by all representations of $\sn$. The product of representations
$w,z$ of $\sn$, $\sm$ respectively is defined using induction:
\[
m(w\otimes z)=\Ind_{\sn\times\sm}^{\mathfrak{S}_{n+m}}w\times z,
\]
and the coproduct of $x$, a representation of $\sn$, is the sum
of its restrictions:
\[
\Delta(x)=\bigoplus_{i=0}^{n}\Res_{\mathfrak{S}_{i}\times\mathfrak{S}_{n-i}}^{\sn}x.
\]
Mackey theory ensures these operations satisfy $\Delta(wz)=\Delta(w)\Delta(z)$.
This Hopf algebra is both commutative and cocommutative, as $\sn\times\sm$
and $\sm\times\sn$ are conjugate in $\mathfrak{S}_{n+m}$; however,
the general construction need not have either symmetry. The associated
Markov chain describes the restriction then induction of representations,
see Example \ref{ex:schurfn-twostep}. 
\end{example}
It's natural to attempt this construction with, instead of $\{\sn\}$,
any series of algebras $\{A_{n}\}$ where an injection $A_{n}\otimes A_{m}\subseteq A_{n+m}$
allows this outer product of its modules. For the result to be a Hopf
algebra, one needs some additional hypotheses on the algebras $\{A_{n}\}$;
this leads to the definition of a \emph{tower of algebras} in \cite{towersofalgs}.
In general, two Hopf algebras can be built this way: one using the
finitely-generated modules of each $A_{n}$, and one from the finitely-generated
projective modules of each $A_{n}$. (For the above example of symmetric
groups, these coincide, as all representations are semisimple.) These
are graded duals in the sense of Section \ref{sec:The-Hopf-power-Map}.
For example, \cite[Sec. 5]{ncsym4} takes $A_{n}$ to be the 0-Hecke
algebra, then the Hopf algebra of finitely-generated modules is $QSym$,
the Hopf algebra of quasisymmetric functions. Example \ref{ex:qsym}
below will present $QSym$ in a different guise that does not require
knowledge of Hecke algebras. The Hopf algebra of finitely-generated
projective modules of the 0-Hecke algebras is $\sym$, the algebra
of noncommutative symmetric functions of Section \ref{sub:qsym-sym}.
Further developments regarding Hopf structures from representations
of towers of algebras are in \cite{dggandcha}. 

It will follow from Definition \ref{defn: better-defition-of-hpmc}
of a Hopf-power Markov chain that, as long as every irreducible representation
of $A_{n}$ has a non-zero restriction to some proper subalgebra $A_{i}\otimes A_{n-i}$
($1\leq i\leq n$), one can build a Markov chain on the irreducible
representations of the tower of algebras $\{A_{n}\}$. (Unfortunately,
when $A_{n}$ is the group algebra of $GL_{n}$ over a finite field,
the cuspidal representations violate this hypothesis.) These chains
should be some variant of restriction-then-induction. It is highly
possible that the precise description of the chain is exactly as in
Example \ref{ex:schurfn-twostep}: starting at an irreducible representation
of $A_{n}$, pick $i\in[0,n]$ binomially, restrict to $A_{i}\otimes A_{n-i}$,
then induce back to $A_{n}$ and pick an irreducible representation
with probability proportional to the dimension of the isotypic component.

Interestingly, it is sometimes possible to tell a similar story with
the basis $\calbn$ being a set of reducible representations, possibly
with slight tweaks to the definitions of product and coproduct. In
\cite{superchar1,superchar2,superchar3,superchar4}, $\calbn$ is
a \emph{supercharacter} theory of various matrix groups over finite
fields. This means that the matrix group can be partitioned into \emph{superclasses},
which are each a union of conjugacy classes, such that each supercharacter
(the characters of the representations in $\calbn$) is constant on
each superclass, and each irreducible character of the matrix group
is a consituent of exactly one supercharacter. \cite{superchar} gives
a unified method to build a supercharacter theory on many matrix groups;
this is useful as the irreducible representations of these groups
are extremely complicated.

\subsection{Subalgebras of Power Series\label{sub:Polyrealisation}}

The starting point for this approach is the algebra of symmetric functions,
widely considered as the first combinatorial Hopf algebra in history,
and possibly the most extensively studied. Thorough textbook introductions
to its algebra structure and its various bases are \cite[Chap. 1]{macdonald}
and \cite[Chap. 7]{stanleyec2}.
\begin{example}[Symmetric functions]
\label{ex:symmetricfn}  Work in the algebra $\mathbb{R}[[x_{1},x_{2},\dots]]$
of power series in infinitely-many commuting variables $x_{i}$, graded
so $\deg(x_{i})=1$ for all $i$. The \emph{algebra of symmetric functions}
$\Lambda$ is the subalgebra of power series of finite degree invariant
under the action of the infinite symmetric group $\mathfrak{S}_{\infty}$
permuting the variables. (These elements are often called ``polynomials''
due to their finite degree, even though they contain infinitely-many
monomial terms.) 

An obvious basis of $\Lambda$ is the sum of monomials in each $\mathfrak{S}_{\infty}$
orbit; these are the \emph{monomial symmetric functions}: 
\[
m_{\lambda}:=\sum_{\substack{(i_{1},\dots,i_{l})\\
i_{j}\mbox{ distinct}
}
}x_{i_{1}}^{\lambda_{1}}\dots x_{i_{l}}^{\lambda_{l}}.
\]
Here, $\lambda$ is a partition of $\deg(m_{\lambda})$: $\lambda_{1}+\dots+\lambda_{l(\lambda)}=\deg(m_{\lambda})$
with $\lambda_{1}\geq\dots\geq\lambda_{l(\lambda)}$. For example,
the three monomial symmetric functions of degree three are:
\begin{align*}
m_{(3)} & =x_{1}^{3}+x_{2}^{3}+\dots;\\
m_{(2,1)} & =x_{1}^{2}x_{2}+x_{1}^{2}x_{3}+\dots+x_{2}^{2}x_{1}+x_{2}^{2}x_{3}+x_{2}^{2}x_{4}+\dots;\\
m_{(1,1,1)} & =x_{1}x_{2}x_{3}+x_{1}x_{2}x_{4}+\dots+x_{1}x_{3}x_{4}+x_{1}x_{3}x_{5}+\dots+x_{2}x_{3}x_{4}+\dots.
\end{align*}

It turns out \cite[Th. 7.4.4, Cor. 7.6.2]{stanleyec2} that $\Lambda$
is isomorphic to a polynomial ring in infinitely-many variables: $\Lambda=\mathbb{R}[h_{(1)},h_{(2)},\dots]$,
where
\[
h_{(n)}:=\sum_{i_{1}\leq\dots\leq i_{n}}x_{i_{1}}\dots x_{i_{n}}.
\]
(This is often denoted $h_{n}$, as it is standard to write the integer
$n$ for the partition $(n)$ of single part.) For example, 
\[
h_{(2)}=x_{1}^{2}+x_{1}x_{2}+x_{1}x_{3}+\dots+x_{2}^{2}+x_{2}x_{3}+\dots.
\]
So, setting $h_{\lambda}:=h_{(\lambda_{1})}\dots h_{(\lambda_{l(\lambda)})}$
over all partitions $\lambda$ gives another basis of $\Lambda$,
the \emph{complete symmetric functions}.

Two more bases are important: the \emph{power sums} are $p_{(n)}:=\sum_{i}x_{i}^{n}$,
$p_{\lambda}:=p_{(\lambda_{1})}\dots p_{(\lambda_{l(\lambda)})}$;
and the \emph{Schur functions} $\{s_{\lambda}\}$ are the image of
the irreducible representations under the \emph{Frobenius characteristic
isomorphism} from the representation rings of the symmetric groups
(Example \ref{ex:schurfn}) to $\Lambda$ \cite[Sec. 7.18]{stanleyec2}.
This map is defined by sending the indicator function of an $n$-cycle
of $\sn$ to the scaled power sum $\frac{p_{(n)}}{n}$. (I am omitting
the elementary basis $\{e_{\lambda}\}$, as it has similar behaviour
as $\{h_{\lambda}\}$.)

The coproduct on $\Lambda$ comes from the ``alphabet doubling trick''.
This relies on the isomorphism between the power series algebras $\mathbb{R}[[x_{1},x_{2},\dots,y_{1},y_{2},\dots]]$
and $\mathbb{R}[[x_{1},x_{2},\dots]]\otimes\mathbb{R}[[y_{1},y_{2},\dots]]$,
which simply rewrites the monomial $x_{i_{1}}\dots x_{i_{k}}y_{j_{1}}\dots y_{j_{l}}$
as $x_{i_{1}}\dots x_{i_{k}}\otimes y_{j_{1}}\dots y_{j_{l}}$. To
calculate the coproduct of a symmetric function $f$, first regard
$f$ as a power series in two sets of variables $x_{1},x_{2},\dots,y_{1},y_{2},\dots$;
then $\Delta(f)$ is the image of $f(x_{1},x_{2},\dots y_{1,}y_{2,}\dots)$
in $\mathbb{R}[[x_{1},x_{2},\dots]]\otimes\mathbb{R}[[y_{1},y_{2},\dots]]$
under the above isomorphism. Because $f$ is a symmetric function,
the power series $f(x_{1},x_{2},\dots,y_{1},y_{2},\dots)$ is invariant
under the permutation of the $x_{i}$s and $y_{i}$s separately, so
$\Delta(f)$ is in fact in $\Lambda\otimes\Lambda$. For example,
\begin{align*}
h_{(2)}(x_{1},x_{2},\dots y_{1},y_{2}\dots) & =x_{1}^{2}+x_{1}x_{2}+x_{1}x_{3}+\dots+x_{1}y_{1}+x_{1}y_{2}+\dots\\
 & \phantom{=}+x_{2}^{2}+x_{2}x_{3}+\dots+x_{2}y_{1}+x_{2}y_{2}+\dots\\
 & \phantom{=}+\dots\\
 & \phantom{=}+y_{1}^{2}+y_{1}y_{2}+y_{1}y_{2}+\dots\\
 & \phantom{=}+y_{2}^{2}+y_{2}y_{3}+\dots\\
 & \phantom{=}+\dots\\
 & =h_{(2)}(x_{1},x_{2},\dots)+h_{(1)}(x_{1},x_{2},\dots)h_{(1)}(y_{1},y_{2},\dots)+h_{(2)}(y_{1},y_{2},\dots),
\end{align*}
so $\Delta(h_{(2)})=h_{(2)}\otimes1+h_{(1)}\otimes h_{(1)}+1\otimes h_{(2)}$.
In general, $\Delta(h_{(n)})=\sum_{i=0}^{n}h_{(i)}\otimes h_{(n-i)}$,
with the convention $h_{(0)}=1$. (This is Geissenger's original definition
of the coproduct \cite{symfnscoproduct}.) Note that $\Delta(p_{(n)})=1\otimes p_{(n)}+p_{(n)}\otimes1$;
this property is the main reason for working with the power sum basis.

The Hopf-power Markov chain on $\{h_{\lambda}\}$ describes an independent
multinomial rock-breaking process, see Section \ref{sec:Rock-breaking}.
\end{example}
The generalisation of $\Lambda$ is easier to see if the $\mathfrak{S}_{\infty}$
action is rephrased in terms of a function to a fundamental domain.
Observe that each orbit of the monomials, under the action of the
infinite symmetric group permuting the variables, contains precisely
one term of the form $x_{1}^{\lambda_{1}}\dots x_{l}^{\lambda_{l}}$
for some partition $\lambda$. Hence the set $\mathcal{D}:=\left\{ x_{1}^{\lambda_{1}}\dots x_{l}^{\lambda_{l}}|l,\lambda_{i}\in\mathbb{N},\lambda_{1}\geq\lambda_{2}\geq\dots\geq\lambda_{l}>0\right\} $
is a fundamental domain for this $\mathfrak{S}_{\infty}$ action.
Define a function $f$ sending a monomial to the element of $\mathcal{D}$
in its orbit; explicitly, 
\[
f\left(x_{j_{1}}^{i_{1}}\dots x_{j_{l}}^{i_{l}}\right)=x_{1}^{i_{\sigma(1)}}\dots x_{l}^{i_{\sigma(l)}},
\]
where $\sigma\in\mathfrak{S}_{l}$ is such that $i_{\sigma(1)}\geq\dots\geq i_{\sigma(l)}$.
For example, $f(x_{1}x_{3}^{2}x_{4})=x_{1}^{2}x_{2}x_{3}$. It is
clear that the monomial symmetric function $m_{\lambda}$, previously
defined to be the sum over $\mathfrak{S}_{\infty}$orbits, is the
sum over preimages of $f$:
\[
m_{\lambda}:=\sum_{f(x)=x^{\lambda}}x,
\]
where $x^{\lambda}$ is shorthand for $x_{1}^{\lambda_{1}}\dots x_{l}^{\lambda_{l}}$.
Summing over preimages of other functions can give bases of other
Hopf algebras. Again, the product is that of power series, and the
coproduct comes from alphabet doubling. Example \ref{ex:qsym}, essentially
a simplified, commutative, version of \cite[Sec. 2]{polynomialrealisation},
builds the algebra of quasisymmetric functions using this recipe.
This algebra is originally due to Gessel \cite{qsym}, who defines
it in terms of $P$-partitions. 
\begin{example}[Quasisymmetric functions]
\label{ex:qsym}  Start again with $\mathbb{R}[[x_{1},x_{2},\dots]]$,
the algebra of power series in infinitely-many commuting variables
$x_{i}$. Let $\pack$ be the function sending a monomial $x_{j_{1}}^{i_{1}}\dots x_{j_{l}}^{i_{l}}$
(assuming $j_{1}<\dots<j_{l}$) to its \emph{packing} $x_{1}^{i_{1}}\dots x_{l}^{i_{l}}$.
For example, $\pack(x_{1}x_{3}^{2}x_{4})=x_{1}x_{2}^{2}x_{3}$. A
monomial is \emph{packed} if it is its own packing, in other words,
its constituent variables are consecutive starting from $x_{1}$.
Let $\mathcal{D}$ be the set of packed monomials, so $\mathcal{D}:=\left\{ x_{1}^{i_{1}}\dots x_{l}^{i_{l}}|l,i_{j}\in\mathbb{N}\right\} $.
Writing $I$ for the \emph{composition} $(i_{1},\dots,i_{l})$ and
$x^{I}$ for $x_{1}^{i_{1}}\dots x_{l}^{i_{l}}$, define the \emph{monomial
quasisymmetric functions} to be:
\[
M_{I}:=\sum_{\pack(x)=x^{I}}x=\sum_{j_{1}<\dots<j_{l(I)}}x_{j_{1}}^{i_{1}}\dots x_{j_{l(I)}}^{i_{l(I)}}.
\]
For example, the four monomial quasisymmetric functions of degree
three are:
\begin{align*}
M_{(3)} & =x_{1}^{3}+x_{2}^{3}+\dots;\\
M_{(2,1)} & =x_{1}^{2}x_{2}+x_{1}^{2}x_{3}+\dots+x_{2}^{2}x_{3}+x_{2}^{2}x_{4}+\dots+x_{3}^{2}x_{4}+\dots;\\
M_{(1,2)} & =x_{1}x_{2}^{2}+x_{1}x_{3}^{2}+\dots+x_{2}x_{3}^{2}+x_{2}x_{4}^{2}+\dots+x_{3}x_{4}^{2}+\dots;\\
M_{(1,1,1)} & =x_{1}x_{2}x_{3}+x_{1}x_{2}x_{4}+\dots+x_{1}x_{3}x_{4}+x_{1}x_{3}x_{5}+\dots+x_{2}x_{3}x_{4}+\dots.
\end{align*}
$QSym$, the \emph{algebra of quasisymmetric functions}, is then the
subalgebra of $\mathbb{R}[[x_{1},x_{2},\dots]]$ spanned by the $M_{I}$.

Note that the monomial symmetric function $m_{(2,1)}$ is $M_{(2,1)}+M_{(1,2)}$;
in general, $m_{\lambda}=\sum M_{I}$ over all compositions $I$ whose
parts, when ordered decreasingly, are equal to $\lambda$. Thus $\Lambda$
is a subalgebra of $QSym$.

The basis of $QSym$ with representation-theoretic significance, analogous
to the Schur functions of $\Lambda$, are the fundamental quasisymmetric
functions:
\[
F_{I}=\sum_{J\geq I}M_{J}
\]
where the sum runs over all compositions $J$ refining $I$ (i.e.
$I$ can be obtained by gluing together some adjacent parts of $J$).
For example, 
\[
F_{(2,1)}=M_{(2,1)}+M_{(1,1,1)}=\sum_{j_{1}\leq j_{2}<j_{3}}x_{j_{1}}x_{j_{2}}x_{j_{3}}.
\]
The fundamental quasisymmetric functions are sometimes denoted $L_{I}$
or $Q_{I}$ in the literature. They correspond to the irreducible
modules of the 0-Hecke algebra \cite[Sec. 5]{ncsym4}. The analogue
of power sums are more complex (as they natually live in the dual
Hopf algebra to $QSym$), see Section \ref{sub:qsym-sym} for a full
definition.

The Hopf-power Markov chain on the basis of fundamental quasisymmetric
functions $\{F_{I}\}$ is the change in descent set under riffle-shuffling,
which Section \ref{sec:Descent-Sets} analyses in detail.
\end{example}
In the last decade, a community in Paris have dedicated themselves
\cite{fqsyminncsymbook,polynomialrealisation,polynomialrealisation2}
to recasting familiar combinatorial Hopf algebras in this manner,
a process they call \emph{polynomial realisation}. They usually start
with power series in noncommuting variables, so the resulting Hopf
algebra is not constrained to be commutative. The least technical
exposition is probably \cite{polynomialrealisationtalk}, which also
provides a list of examples. The simplest of these is $\sym$, a noncommutative
analogue of the symmetric functions; its construction is explained
in Section \ref{sub:qsym-sym} below. For a more interesting example,
take $M_{T}$ to be the sum of all noncommutative monomials with $Q$-tableau
equal to $T$ under the Robinson-Schensted-Knuth algorithm \cite[Sec. 7.11]{stanleyec2};
then their span is $\mathbf{FSym}$, the Poirier-Reutenauer Hopf algebra
of tableaux \cite{fsym}. \cite[Th. 31]{hivertcspolynomialrealisation}
and \cite[Th. 1]{latticeofchas} give sufficient conditions on the
functions for this construction to produce a Hopf algebra. One motivation
for this program is to bring to light various bases that are free
(like $h_{\lambda}$), interact well with the coproduct (like $p_{\lambda}$)
or are connected to representation theory (like $s_{\lambda}$), and
to carry over some of the vast amount of machinery developed for the
symmetric functions to analyse these combinatorial objects in new
ways. Indeed, Joni and Rota anticipated in their original paper \cite{jonirota}
that ``many an interesting combinatorial problem can be formulated
algebraically as that of transforming this basis into another basis
with more desirable properties''.

\section{First Definition of a Hopf-power Markov Chain\label{sec:first-defition-of-hpmc}}

Recall from Section \ref{sec:Markov-chains-intro} the GSR riffle-shuffle
of a deck of cards: cut the deck into two piles according to a symmetric
binomial distribution, then drop the cards one by one from the bottom
of the piles, chosen with probability proportional to the current
pile size. As mentioned in Section \ref{sec:Hopf-power-Markov-chains-intro},
a direct calculation shows that, for words $x,y$ of length $n$ in
the shuffle algebra of Example \ref{ex:shufflealg}, the coefficient
of $y$ in $2^{-n}m\Delta(x)$ is the probability of obtaining a deck
of cards in order $y$ after applying a GSR riffle-shuffle to a deck
in order $x$: 
\begin{equation}
2^{-n}m\Delta(x)=\sum_{y}K(x,y)y.\label{eq:first-defn-of-hpmc}
\end{equation}
(Here, identify the word $x_{1}x_{2}\dots x_{n}$ in the shuffle algebra
with the deck whose top card has value $x_{1}$, second card has value
$x_{2}$, and so on, so $x_{n}$ is the value of the bottommost card.)
In other words, the matrix of the linear operator $2^{-n}m\Delta$
on $\calhn$, with respect to the basis of words, is the transpose
of the transition matrix of the GSR shuffle. Furthermore, the matrix
of the $a$th Hopf-power map $a^{-n}\Psi^{a}:=a^{-n}\proda\coproda$
on $\calhn$ (with respect to the basis of words) is the transpose
of the transition matrix of an \emph{$a$-handed shuffle} of \cite{bd};
this will follow from Theorem \ref{thm:threestep} below. An $a$-handed
shuffle is a straightforward generalisation of the GSR shuffle: cut
the deck into $a$ piles according to the symmetric multinomial distribution,
then drop the cards one by one from the bottom of the pile, where
the probability of dropping from any particular pile is proportional
to the number of cards currently in that pile. This second step is
equivalent to all interleavings of the $a$ piles being equally likely;
more equivalent views are in \cite[Chap. 3]{bd}.

This relationship between $a$-handed shuffles and the $a$th Hopf-power
map on the shuffle algebra motivates the question: for which graded
Hopf algebras $\calh$ and bases $\calb$ does Equation \ref{eq:first-defn-of-hpmc}
(and its analogue for $a>2$) define a Markov chain? In other words,
what conditions on $\calh$ and $\calb$ guarantee that the coefficients
of $a^{-n}\Psi^{a}(x)$ are non-negative and sum to 1? Achieving a
sum of 1 is the subject of the next section; as for non-negativity,
one solution is to mandate that the product and coproduct structure
constants are non-negative:
\begin{lem}
\label{lem: nonnegative-coeffs}Let $\calh$ be a Hopf algebra over
$\mathbb{R}$ with basis $\calb$ such that:
\begin{enumerate}
\item for all $w,z\in\calb$, $wz=\sum_{y\in\calb}\xi_{wz}^{y}y$ with $\xi_{wz}^{y}\geq0$
(non-negative product structure constants);
\item for all $x\in\calb$, $\Delta(x)=\sum_{w,z\in\calb}\eta_{x}^{wz}w\otimes z$
with $\eta_{x}^{wz}\geq0$ (non-negative coproduct structure constants).
\end{enumerate}

Then, for all $x,y\in\calb$, the coefficient of $y$ in $\Psi^{a}(x)$
is non-negative, for all $a$.

\end{lem}
\begin{proof}
In the notation for structure constants at the start of Section \ref{sec:Combinatorial-Hopf-algebras},
the coefficient of $y$ in $\Psi^{a}(x)$ is $\sum_{z_{1},\dots,z_{n}}\xi_{z_{1},\dots,z_{a}}^{y}\eta_{x}^{z_{1},\dots,z_{a}}$.
By definition of $a$-fold multiplication and comultiplication, 
\[
\xi{}_{z_{1},\dots,z_{a}}^{y}=\sum_{z}\xi_{zz_{a}}^{y}\xi_{z_{1},\dots,z_{a-1}}^{z},\quad\eta_{x}^{z_{1},\dots,z_{a}}=\sum_{z}\eta_{x}^{zx_{a}}\eta_{z}^{z_{1},\dots,z_{a-1}},
\]
so, by induction on $a$ (the base case of $a=2$ being the hypothesis),
both $\xi{}_{z_{1},\dots,z_{a}}^{y}$ and $\eta_{x}^{z_{1},\dots,z_{a}}$
are non-negative.
\end{proof}
So the following indeed specifies a Markov chain:
\begin{defn}[First definition of Hopf-power Markov chain]
\label{defn: first-defition-of-hpmc}Let $\calh=\bigoplus_{n\geq0}\calhn$
be a graded connected Hopf algebra over $\mathbb{R}$, with each $\calhn$
finite-dimensional. Let $\calb=\amalg_{n\geq0}\calbn$ be a basis
of $\calh$ with non-negative structure constants (i.e. satisfying
conditions i, ii of Lemma \ref{lem: nonnegative-coeffs} above). Assume
in addition that, for all $x\in\calbn$, the coefficients (with respect
to $\calbn$) of $a^{-n}\Psi^{a}(x)$ sum to 1. Then the \emph{$a$th
Hopf-power Markov chain on $\calbn$} has transition matrix $\kan:=\left[a^{-n}\Psi^{a}\right]_{\calbn}^{T}$,
the transpose of the matrix of $a^{-n}\Psi^{a}$ with respect to the
basis $\calbn$. 
\end{defn}
Observe that, if $\calh$ comes from a species-with-restrictions in
the method of Section \ref{sub:Species}, then the coefficients of
$a^{-n}\Psi^{a}(x)$ sum to 1, for all $a$ and all $n$. This is
because the terms in $\coproda(x)$ correspond to the $a^{n}$ ways
of partitioning the underlying set into $a$ (possibly trivial) subsets
(the order of the subsets matter), and each such term gives only a
single term under $\proda$.
\begin{example}
\label{ex:chain-graph}Take $\calh=\barcalg$, the algebra of graphs
of Example \ref{ex:graph}. Recall that the product of two graphs
is their disjoint union, and the coproduct gives the induced subgraphs
on two complimentary subsets of the vertex set. Thus one step of the
associated $a$th Hopf-power Markov chain is the following: independently
assign to each vertex one of $a$ colours, each with an equal probability
of $\frac{1}{a}$. Then remove all edges between vertices of different
colours. As an example, take $a=2$ and start at $P_{3}$, the path
of length 3. Write $P_{2}$ for the two-vertex graph with a single
edge. By Figure \ref{fig:coproduct-graphs}, 
\[
\Delta(P_{3})=P_{3}\otimes1+2P_{2}\otimes\bullet+\bullet^{2}\otimes\bullet+2\bullet\otimes P_{2}+\bullet\otimes\bullet^{2}+1\otimes P_{3}.
\]
Hence $\Psi^{2}(P_{3})=2P_{3}+4P_{2}\bullet+2\bullet^{3}$. So, starting
at $P_{3}$, the chain stays at $P_{3}$ with probability $\frac{2}{2^{3}}=\frac{1}{4}$,
or moves to $P_{2}\bullet$ with probability $\frac{4}{2^{3}}=\frac{1}{2}$,
or moves to the disconnected graph with probability $\frac{2}{2^{3}}=\frac{1}{4}$. 
\end{example}

\section{General Definition of a Hopf-power Markov Chain\label{sec:better-definition-of-hpmc}}

One would like to remove from Definition \ref{defn: first-defition-of-hpmc}
above the restrictive condition that the sum of the coefficients of
$a^{-n}\Psi^{a}(x)$ is 1. In other words, it would be good to build
a Markov chain out of $\Psi^{a}$ even when the matrix $\kan:=\left[a^{-n}\Psi^{a}\right]_{\calbn}^{T}$
does not have every row summing to 1. Lemma \ref{thm:doob-transform},
the Doob $h$-transform for linear maps, gives one possible answer:
instead of $\calbn$, work with the basis $\hatcalb_{n}:=\left\{ \hatx:=\frac{x}{\eta_{n}(x)}|x\in\calbn\right\} $,
where $\eta_{n}\in\calhdual_{n}$ is a ``positive'' eigenvector
for the map dual to $\Psi^{a}$. Recall from Section \ref{sec:The-Hopf-power-Map}
that this dual map is again a Hopf-power map $\Psi^{a}$, but on the
(graded) dual Hopf algebra $\calhdual$. On a combinatorial Hopf algebra,
one choice of $\eta_{n}$ has a remarkably simple description as ``the
number of ways to break into singletons'', and is usually a well-investigated
number. The first two definitions of $\eta_{n}$ below are more intuitive,
as they avoid direct reference to $\calhdual$, whilst the third streamlines
the proofs.
\begin{defn}
\label{defn:eta}Three equivalent definitions of the \emph{rescaling
functions} $\eta_{n}:\calbn\rightarrow\mathbb{R}$ are:
\begin{enumerate}
\item $\eta_{n}(x)$ is the sum of coproduct structure constants (over all
ordered $n$-tuples, possibly with repetition of the $c_{i}$):
\[
\eta_{n}(x):=\sum_{c_{1},c_{2},\dots,c_{n}\in\calb_{1}}\eta_{x}^{c_{1},\dots,c_{n}};
\]

\item $\eta_{n}(x)$ is the sum of the coefficients of $\bard^{[n]}(x)$,
the\emph{ }$n$-fold reduced coproduct of $x$, when expanded in the
basis $\calb^{\otimes n}$. (Recall from Section \ref{sec:The-Eulerian-Idempotent}
that $\bard(x):=\bard^{[2]}(x):=\Delta(x)-1\otimes x-x\otimes1$,
and $\bard^{[n]}:=(\iota\otimes\cdots\otimes\iota\otimes\bard)\bard^{[n-1]}$,
so $\bard^{[n]}\in\calh_{1}^{\otimes n}$.)
\item Let $\bullet^{*}\in\calhdual_{1}$ be the linear function on $\calh$
taking value 1 on each element of $\calb_{1}$ and 0 on all other
basis elements. (In the dual basis notation from the start of Chapter
\ref{chap:linearoperators}, $\bullet^{*}:=\sum_{c\in\calb_{1}}c^{*}$;
in particular, if $\calb_{1}=\left\{ \bullet\right\} $ then this
agrees with the dual basis notation.) Then set $\eta_{n}:=(\bullet^{*})^{n}$.
In other words, $\eta_{n}(x):=(\bullet^{*}\otimes\dots\otimes\bullet^{*})\Delta^{[n]}(x)$. 
\end{enumerate}
\end{defn}
Since, for each $n\in\mathbb{N}$, the rescaling function $\eta_{n}$
has a different domain (namely $\calhn$), no confusion arises from
abbreviating $\eta_{\deg x}(x)$ by $\eta(x)$. Observe though that
such a function $\eta$ is not an element of the (graded) dual $\calhdual$,
as it is an infinite sum of linear functions on the subspaces $\calhn$.
However, the variant $\frac{\eta_{\deg x}(x)}{\deg x!}$ is a \emph{character}
in the sense of \cite{abs}, as it is multiplicative; see Lemma \ref{lem:etaproduct}.
\begin{example}
\label{ex:schurfn-eta}Recall from Example \ref{ex:schurfn} the Hopf
algebra of representations of the symmetric groups, with product arising
from induction and coproduct from restriction. Its distinguished basis
$\calb$ is the set of irreducible representations. So $\calb_{1}$
consists only of the trivial representation $\bullet$, thus, by the
first of the equivalent definitions above, $\eta(x)=\eta_{x}^{\bullet,\dots,\bullet}$.
For an irreducible representation $x$ of $\sn$, $\Res_{\mathfrak{S}_{1}\times\dots\times\mathfrak{S}_{1}}^{\sn}x=\dim x(\bullet\otimes\dots\otimes\bullet)$,
so $\eta(x)=\dim x$.
\end{example}
A simple application of the Symmetrisation Lemma (Theorem \ref{thm:symlemma})
shows that $\eta_{n}$ is an eigenvector of $\Psi^{a}:\calhndual\rightarrow\calhndual$
of eigenvalue $a^{n}$, since $\bullet^{*}$ has degree 1 and is hence
primitive. In order to use $\eta_{n}$ in the Doob transform, we must
ensure that $\eta_{n}(x)>0$ for all $x\in\calbn$. (It suffices to
force $\eta_{n}(x)\neq0$ for all $x\in\calbn$, since, as a sum of
coproduct structure constants, $\eta_{n}$ takes non-negative values
on $\calbn$.) This is the purpose of condition iii in Definition
\ref{defn:statespacebasis} below. This requirement essentially translates
to ``every object of size greater than 1 breaks non-trivially'';
the intuition is that repeatedly applying such non-trivial breaks
to the pieces provides a way to reduce $x$ to singletons. Theorem
\ref{thm:rescaling} below rigorises this heuristic, and explains
why it is necessary to forbid primitive basis elements of degree greater
than one in order to apply the Doob transform to the Hopf-power map,
for all choices of rescaling functions. 
\begin{defn}[State space basis]
\label{defn:statespacebasis}Let $\calh=\bigoplus_{n\geq0}\calhn$
be a graded connected Hopf algebra over $\mathbb{R}$, with each $\calhn$
finite-dimensional. A basis $\calb=\amalg_{n\geq0}\calbn$ of $\calh$
is a \emph{state space basis} if:
\begin{enumerate}
\item for all $w,z\in\calb$, $wz=\sum_{y\in\calb}\xi_{wz}^{y}y$ with $\xi_{wz}^{y}\geq0$
(non-negative product structure constants);
\item for all $x\in\calb$, $\Delta(x)=\sum_{w,z\in\calb}\eta_{x}^{wz}w\otimes z$
with $\eta_{x}^{wz}\geq0$ (non-negative coproduct structure constants);
\item for all $x\in\calb$ with $\deg(x)>1$, it holds that $\Delta(x)\neq1\otimes x+x\otimes1$
(no primitive elements in $\calb$ of degree greater than 1).
\end{enumerate}
\end{defn}
Note that $\calh$ may contain primitive elements of any degree, so
long as those of degree greater than one are not in the basis $\calb$.
Applying the Doob transform to $\Psi^{a}:\calhn\rightarrow\calhn$
(with the rescaling function $\eta$) then creates the family of Markov
chains defined below.
\begin{defn}[General definition of Hopf-power Markov chain]
\label{defn: better-defition-of-hpmc}Let $\calh=\oplus_{n\geq0}\calhn$
be a graded connected Hopf algebra over $\mathbb{R}$, with each $\calhn$
finite-dimensional, and with state space basis $\calb$. Take $\eta_{n}$
according to Definition \ref{defn:eta}. Then the \emph{$a$th Hopf-power
Markov chain on $\calbn$} has transition matrix $\hatkan:=\left[a^{-n}\Psi^{a}\right]_{\hatcalbn}^{T}$,
where $\hatcalbn:=\left\{ \hatx:=\frac{x}{\eta_{n}(x)}|x\in\calbn\right\} $.
In other words,
\[
a^{-n}\Psi^{a}(\hatx)=\sum_{y\in\calbn}\hatkan(x,y)\haty,
\]
or, equivalently, 
\[
a^{-n}\Psi^{a}(x)=\sum_{y\in\calbn}\frac{\eta_{n}(x)}{\eta_{n}(y)}\hatkan(x,y)y.
\]

\end{defn}
Recall that, if $\calh$ is commutative or cocommutative, then the
power law $\Psi^{a}\Psi^{a'}=\Psi^{aa'}$ holds. Thus long term behaviour
of Hopf-power Markov chains may be deduced from increasing the power
further and further: taking $m$ steps of the $a$th Hopf-power chain
is equivalent to a single step of the $a^{m}$th Hopf-power chain.
This will be relevant in Section \ref{sub:Absorption}, on approximations
of absorbing probabilities using quasisymmetric functions.
\begin{example}
\label{ex:shuffle-eta} In the shuffle algebra of Example \ref{ex:shufflealg},
for any word $x$, and any $c_{1},\dots,c_{n}\in\calb_{1}$, the coproduct
structure constant $\eta_{x}^{c_{1},\dots,c_{n}}=0$ unless $x$ is
the concatenation of $c_{1},c_{2},\dots,c_{n}$ in that order, in
which case $\eta_{x}^{c_{1},\dots,c_{n}}=1$. So $\eta(x)=1$ for
all $x\in\calb$, thus no rescaling of the basis is necessary to define
the Hopf-power Markov chain. (No rescaling is necessary whenever $\eta$
is a constant function on each $\calbn$ - this constant may depend
on $n$.)
\end{example}

\begin{example}
\label{ex:schurfn-chain} Take $\calh$ to be the Hopf algebra of
representations of the symmetric groups, as in Example \ref{ex:schurfn}.
$\calb_{3}$ is the set of irreducible representations of $\mathfrak{S}_{3}$,
comprising the trivial representation, the sign representation and
the two-dimensional irreducible representation. From explicit computation
of $m\Delta=\bigoplus_{i=0}^{3}\Ind_{\mathfrak{S}_{i}\times\mathfrak{S}_{3-i}}^{\mathfrak{S}_{3}}\Res_{\mathfrak{S}_{i}\times\mathfrak{S}_{3-i}}^{\mathfrak{S}_{3}}$
for these three representations, it follows that 
\[
K_{2,3}:=[2^{-3}m\Delta]_{\calb_{3}}^{T}=\begin{bmatrix}\frac{1}{2} & 0 & \frac{1}{4}\\
0 & \frac{1}{2} & \frac{1}{4}\\
\frac{1}{4} & \frac{1}{4} & \frac{3}{4}
\end{bmatrix}.
\]
Observe that $(1,1,2)$, the vector of dimensions of these representations,
is a (right) eigenvector of $K_{2,3}$ of eigenvalue 1, as predicted
by Example \ref{ex:schurfn-eta}. So applying the Doob transform to
$K_{2,3}$ is to divide the third row by two and multiply the third
column by 2, giving 
\[
\hatk_{2,3}=\begin{bmatrix}\frac{1}{2} & 0 & \frac{1}{2}\\
0 & \frac{1}{2} & \frac{1}{2}\\
\frac{1}{8} & \frac{1}{8} & \frac{3}{4}
\end{bmatrix}.
\]
This is a transition matrix as its rows sum to 1. Example \ref{ex:schurfn-twostep}
below interprets this Markov chain as restriction-then-induction. 
\end{example}
As promised, here is a check that $\eta$ indeed takes positive values
on a state space basis, and that, assuming $\calh_{1}\neq\emptyset$,
there is no suitable rescaling function for bases which are not state
space bases (i.e. there are primitive basis elements of degree greater
than one.) In this sense, $\eta$ is an optimal rescaling function.
Example \ref{ex:eta-fail} gives a numerical illustration of this
second fact.
\begin{thm}
\label{thm:rescaling}Suppose $\calh=\oplus_{n\geq0}\calhn$ is a
graded connected Hopf algebra over $\mathbb{R}$ with non-negative
coproduct structure constants in the basis $\calb=\amalg_{n\geq0}\calbn$.
Assume also that $\calh_{1}\neq\emptyset$.
\begin{enumerate}
\item If $\Delta(x)\neq1\otimes x+x\otimes1$ for all $x\in\calb$ with
$\deg(x)>1$, then the functions $\eta_{n}$ of Definition \ref{defn:eta}
satisfy $\eta_{\deg x}(x)>0$ for all $x\in\calb$.
\item If $\Delta(x)=1\otimes x+x\otimes1$ for some $x\in\calbn$ with $n>1$,
then $\eta'_{n}(x)=0$ for all eigenvectors $\eta'_{n}$ of $\Psi^{a}:\calhndual\rightarrow\calhndual$
of highest eigenvalue.
\end{enumerate}
\end{thm}
\begin{proof}
Recall that the intuition behind Part i is that ``repeatedly breaking
$x$ non-trivially gives a way to reduce it to singletons''. So proceed
by induction on $\deg x$. If $\deg x=1$, then $\eta_{1}(x)=1$ by
definition. Otherwise, by hypothesis, $\bard(x)\neq0$. Take a term
$w\otimes z$ in $\bard(x)$, so $\eta_{x}^{wz}>0$. Then the counit
axiom forces $\deg w,\deg z<\deg x$. Consequently 
\begin{eqnarray*}
\eta_{\deg x}(x) & = & (\bullet^{*})^{\deg x}(x)\\
 & = & (\bullet^{*})^{\deg w}(\bullet^{*})^{\deg z}(x)\\
 & = & \left[(\bullet^{*})^{\deg w}\otimes(\bullet^{*})^{\deg z}\right](\Delta x)\\
 & = & \left[(\bullet^{*})^{\deg w}\otimes(\bullet^{*})^{\deg z}\right]\left(\sum_{w',z'\in\calb}\eta_{x}^{w'z'}w'\otimes z'\right)\\
 & = & \sum\eta_{x}^{w'z'}\eta_{\deg w}(w')\eta_{\deg z}(z')
\end{eqnarray*}
where the last sum is over all $w'\in\calb_{\deg w},z'\in\calb_{\deg z}$,
because on all other summands, $(\bullet^{*})^{\deg w}\otimes(\bullet^{*})^{\deg z}$
evaluates to 0. The coproduct structure constants $\eta_{x}^{w'z'}$
are non-negative, and, by inductive hypothesis, $\eta_{\deg w}(w'),\eta_{\deg z}(z')>0$.
So all summands above are non-negative and the summand $\eta_{x}^{wz}\eta_{\deg w}(w')\eta_{\deg z}(z')$
is positive, so the sum is positive.

To see Part ii, it suffices to show that $\eta'_{n}(x)=0$ for $\eta'$
belonging to the basis in Theorem \ref{thm:topeigenspace} of the
eigenspace of $\Psi^{a}:\calhndual\rightarrow\calhndual$ of highest
eigenvalue. Such basis eigenvectors have the form $\eta'=\sum_{\sigma\in\sn}c_{\sigma(1)}^{*}\dots c_{\sigma(n)}^{*}$
for some $c_{1}^{*},\dots,c_{n}^{*}\in\calh_{1}^{*}$. Now, because
multiplication in $\calhdual$ is dual to comultiplication in $\calh$,
\begin{align*}
\eta'(x) & =\left(\sum_{\sigma\in\sn}c_{\sigma(1)}^{*}\dots c_{\sigma(n)}^{*}\right)(x)\\
 & =\sum_{\sigma\in\sn}\left(c_{\sigma(1)}^{*}\otimes\dots\otimes c_{\sigma(n)}^{*}\right)(\Delta^{[n]}x)\\
 & =\sum_{\sigma\in\sn}c_{\sigma(1)}^{*}(x)\otimes c_{\sigma(2)}^{*}(1)\otimes\dots\otimes c_{\sigma(n)}^{*}(1)\\
 & \phantom{=\sum_{\sigma\in\sn}}+c_{\sigma(1)}^{*}(1)\otimes c_{\sigma(2)}^{*}(x)\otimes c_{\sigma(3)}^{*}(1)\otimes\dots\otimes c_{\sigma(n)}^{*}(1)+\dots\\
 & \phantom{=\sum_{\sigma\in\sn}}+c_{\sigma(1)}^{*}(1)\otimes\dots\otimes c_{\sigma(n-1)}^{*}(1)\otimes c_{\sigma(n)}^{*}(x)\\
 & =0,
\end{align*}
since $c_{\sigma(i)}^{*}(x)$, $c_{\sigma(i)}^{*}(1)$ are all zero
by degree considerations. (The third equality used that $x$ is primitive.)\end{proof}
\begin{example}
\label{ex:eta-fail} Work in the algebra $\Lambda$ of symmetric functions,
and take $\calb$ to be the power sums, as described in Example \ref{ex:symmetricfn}.
So $\calb_{3}=\{p_{1}^{3},p_{1}p_{2},p_{3}\}$ and $\Delta(p_{n})=1\otimes p_{n}+p_{n}\otimes1$
for each $n$. By explicit computation, 
\[
K_{2,3}:=[2^{-3}m\Delta]_{\calb_{3}}^{T}=\begin{bmatrix}1 & 0 & 0\\
0 & \frac{1}{2} & 0\\
0 & 0 & \frac{1}{4}
\end{bmatrix}.
\]
Simply rescaling the basis $\calb_{3}$ cannot make the rows of this
matrix sum to 1, as rescaling the basis does not change the diagonal
entries, and can only change non-zero non-diagonal entries. 

It is easy to see how this problem generalises: for any primitive
element $x\in\calbn$, it happens that $m\Delta(x)=2x$, so the row
corresponding to $x$ in $K_{2,n}$ is $2^{-n+1}$ in the main diagonal
and zeroes elsewhere. Then this row sum cannot change under basis
rescaling. 
\end{example}
To end this section, here is a brief word on how to modify the above
notions for the case where $\calh_{1}=\emptyset$. As in Section \ref{sec:topeigenspace},
set $\cald:=\{d>0|\calh_{d}\neq\emptyset\}$, $\cald'=\{d\in\cald|d\neq d_{1}+d_{2}\mbox{ with }d_{1},d_{2}\in\cald\}$,
so $\bigoplus_{d\in\cald'}\calh_{d}$ consists solely of primitive
elements. Then define $\calb$ to be a state space basis if it contains
no primitive elements outside of $\bigoplus_{d\in\cald'}\calh_{d}$.
For each $n\in\cald$, let $K(n)$ denote the maximal length of a
$\cald'$-partition of $n$, so, by Theorem \ref{thm:topeigenspace2},
$a^{K(n)}$ is the largest eigenvalue of $\Psi^{a}:\calhn\rightarrow\calhn$.
Then the value of the rescaling function $\eta_{n}(x)$ should be
the sum of the coefficients of $\bard^{[K(n)]}(x)$, and the transition
matrix of the Hopf-power Markov chain is $\hatkan:=\left[a^{-K(n)}\Psi^{a}\right]_{\hatcalbn}^{T}$,
where $\hatcalbn:=\left\{ \hatx:=\frac{x}{\eta_{n}(x)}|x\in\calbn\right\} $.

\section{Description of a Hopf-power Markov chain\label{sec:threestep}}

Definition \ref{defn: better-defition-of-hpmc} gives the exact transition
probabilities of a Hopf-power Markov chain, but this is not very enlightening
without an intuitive description of the chain. Such descriptions can
be very specific to the underlying Hopf algebra (see Theorem \ref{thm:chain-cktrees}
regarding tree-pruning). The starting point to finding these interpretations
is Theorem \ref{thm:threestep}, which separates each timestep of
the chain into breaking (steps 1 and 2) and recombining (step 3).
The probabilities involved in both stages are expressed in terms of
the structure constants of $\calh$ and the rescaling function $\eta$.
\begin{thm}[Three-step description for Hopf-power Markov chains]
\label{thm:threestep}A single step of the $a$th Hopf-power Markov
chain, starting at $x\in\calbn$, is equivalent to the following three-step
process:
\begin{enumerate}[label=\arabic*.]
\item Choose a composition $\left(i_{1},\dots,i_{a}\right)$ of $n$ (that
is, non-negative integers with $i_{1}+\dots+i_{a}=n$) according to
the multinomial distribution with parameter $1/a$. In other words,
choose $\left(i_{1},\dots,i_{a}\right)$ with probability $a^{-n}\binom{n}{i_{1}\dots i_{a}}$.
\item Choose $z_{1}\in\calb_{i_{1}},z_{2}\in\calb_{i_{2}},\dots,z_{a}\in\calb_{i_{a}}$
with probability $\frac{1}{\eta(x)}\eta_{x}^{z_{1},\dots,z_{a}}\eta(z_{1})\dots\eta(z_{a})$.
\item Choose $y\in\calbn$ with probability $\left(\binom{n}{\deg z_{1}\dots\deg z_{a}}\eta(z_{1})\dots\eta(z_{a})\right)^{-1}\xi_{z_{1},\dots,z_{a}}^{y}\eta(y)$.
\end{enumerate}
\end{thm}
\begin{example}
\label{ex:shuffle-threestep}Applying Theorem \ref{thm:threestep}
to the shuffle algebra $\calsh$ recovers the description of the $a$-handed
shuffle at the start of Section \ref{sec:Combinatorial-Hopf-algebras}.
Since the coproduct on $\calsh$ is deconcatenation, the coproduct
structure constant $\eta_{x}^{z_{1},\dots,z_{a}}=0$ unless $x$ is
the concatenation of $z_{1},z_{2},\dots,z_{a}$ in that order, so
there is no choice at step 2. Hence steps 1 and 2 combined correspond
to a multinomially-distributed cut of the deck. As for step 3: $\eta(y)=1$
for all $y$, so $y$ is chosen with probability proportional to $\xi_{z_{1},\dots,z_{a}}^{y}$,
the number of ways to interleave $z_{1},\dots,z_{a}$ to obtain $y$.
Hence all interleavings are equally likely.
\end{example}

\begin{example}
\label{ex:schurfn-twostep}How does Theorem \ref{thm:threestep} interpret
the chain on the irreducible representations of the symmetric groups?
Recall from Example \ref{ex:schurfn} that the product is external
induction and the coproduct is restriction. For simplicity, first
take $a=2$. Then, starting at a representation $x$ of $\sn$, the
first step is to binomially choose an integer $i$ between 0 and $n$.
It turns out that a cleaner description emerges if steps 2 and 3 above
are combined. This merged step is to choose an irreducible representation
$y$ with probability proportional to $\sum\eta_{x}^{z_{1}z_{2}}\xi_{z_{1}z_{2}}^{y}\eta(y)$,
where the sum is over all irreducible representations $z_{1}$ of
$\mathfrak{S}_{i}$, and $z_{2}$ of $\mathfrak{S}_{n-i}$. Now $\sum\eta_{x}^{z_{1}z_{2}}\xi_{z_{1}z_{2}}^{y}$
is the coefficient or the multiplicity of the representation $y$
in $\Ind_{\mathfrak{S}_{i}\times\mathfrak{S}_{n-i}}^{\sn}\Res_{\mathfrak{S}_{i}\times\mathfrak{S}_{n-i}}^{\sn}(x)$,
and Example \ref{ex:schurfn-eta} showed that $\eta(y)=\dim y$. So
the product of these two numbers have a neat interpretation as the
dimension of the $y$ isotypic component.

So, for general $a$, the chain on irreducible representations of
the symmetric group has the following description:
\begin{enumerate}[label=\arabic*.]
\item Choose a Young subgroup $\mathfrak{S}_{i_{1}}\times\dots\times\mathfrak{S}_{i_{a}}$
according to a symmetric multinomial distribution.
\item Restrict the starting state $x$ to the chosen subgroup, induce it
back up to $\sn$, then pick an irreducible constituent with probability
proportional to the dimension of its isotypic component.
\end{enumerate}

A similar interpretation holds for other Hopf-power Markov chains
on Hopf algebras of representations of other towers of algebras. For
this particular case with the symmetric groups, this representation
Hopf algebra is isomorphic to the cohomology of the infinite Grassmannian:
the product is cup product, and the coproduct comes from a product
on the infinite Grassmannian, which is taking direct sums of the subspaces.
This isomorphism sends the basis $\calb$ of irreducible representations
to the Schubert classes. So perhaps the restriction-then-induction
chain on irreducible representations has an alternative interpretation
in terms of decomposing a Schubert variety in terms of smaller Grassmannians,
then taking the intersection.

A variant of this restriction-then-induction chain, where the choice
of Young subgroup is fixed instead of random, appears in \cite{randomcharratios}.
There, it generates central limit theorems for character ratios, via
Stein's method. 

\end{example}
\begin{proof}[Proof of Theorem \ref{thm:threestep}, the three-step description]
First check that the probabilities in step 2 do sum to 1: 
\begin{align*}
 & \sum_{z_{1}\in\calb_{i_{1}},\dots,z_{a}\in\calb_{i_{a}}}\eta_{x}^{z_{1},\dots,z_{a}}\eta(z_{1})\dots\eta(z_{a})\\
= & \left(\left(\bullet^{*}\right)^{i_{1}}\otimes\dots\otimes\left(\bullet^{*}\right)^{i_{a}}\right)\left(\sum_{z_{1}\in\calb_{i_{1}},\dots,z_{a}\in\calb_{i_{a}}}\eta_{x}^{z_{1},\dots,z_{a}}z_{1}\otimes\dots\otimes z_{a}\right)\\
= & \left(\left(\bullet^{*}\right)^{i_{1}}\dots\left(\bullet^{*}\right)^{i_{a}}\right)\left(\Delta^{a}(x)\right)\\
= & \left(\bullet^{*}\right)^{n}(x)\\
= & \eta(x)
\end{align*}
where the first equality uses Definition \ref{defn:eta}.iii of the
rescaling function $\eta_{x}$, the second equality is because $\left(\bullet^{*}\right)^{i}(x_{j})=0$
if $\deg(x_{j})\neq i$, and the third equality is by definition of
the product of $\calhdual$. And similarly for the probabilities in
step 3, the combining step:
\begin{align*}
\sum_{y\in\calbn}\xi_{z_{1},\dots,z_{a}}^{y}\eta(y) & =\left(\bullet^{*}\right)^{n}\left(\sum_{y\in\calbn}\xi_{z_{1},\dots,z_{a}}^{y}y\right)\\
 & =\left(\bullet^{*}\right)^{n}(z_{1}\dots z_{a})\\
 & =\Delta^{a}((\bullet^{*})^{n})(z_{1}\otimes\dots\otimes z_{a})\\
 & =\left(\sum_{i_{1},\dots,i_{n}}\binom{n}{i_{1}\dots i_{a}}\left(\bullet^{*}\right)^{i_{1}}\otimes\dots\otimes\left(\bullet^{*}\right)^{i_{a}}\right)(z_{1}\otimes\dots\otimes z_{a})\\
 & =\binom{n}{\deg z_{1}\dots\deg z_{a}}\eta(z_{1})\dots\eta(z_{a}).
\end{align*}
Finally, the probability of moving from $x$ to $y$ under the three-step
process is 
\begin{align*}
 & \sum_{z_{1}\dots z_{a}}a^{-n}\binom{n}{\deg z_{1}\dots\deg z_{a}}\frac{\eta_{x}^{z_{1},\dots,z_{a}}\eta(z_{1})\dots\eta(z_{a})}{\eta(x)}\frac{\xi_{z_{1},\dots,z_{a}}^{y}\eta(y)}{\binom{n}{\deg z_{1}\dots\deg z_{a}}\eta(z_{1})\dots\eta(z_{a})}\\
= & a^{-n}\sum_{z_{1},\dots,z_{a}}\frac{\eta(y)}{\eta(x)}\xi_{z_{1},\dots,z_{a}}^{x}\eta_{y}^{z_{1},\dots,z_{a}}\\
= & \hatkan(x,y).
\end{align*}

\end{proof}

\section{Stationary Distributions\label{sec:stationarydistribution}}

The theorem below classifies all stationary distributions of a Hopf-power
Markov chain; they have a simple expression in terms of the product
structure constants and the rescaling function $\eta$ of Definition
\ref{defn:eta}. 
\begin{thm}[Stationary distribution of Hopf-power Markov chains]
\label{thm:hpmc-stationarydistribution}Follow the notation of Definition
\ref{defn: better-defition-of-hpmc}. Then, for each multiset $\{c_{1},\dots,c_{n}\}$
in $\calb_{1}$, the function 
\[
\pi_{c_{1},\dots,c_{n}}(x):=\frac{\eta(x)}{n!^{2}}\sum_{\sigma\in\sn}\xi_{c_{\sigma(1)},\dots,c_{\sigma(n)}}^{x}
\]
is a stationary distribution for the $a$th Hopf-power Markov chain
on $\calbn$, and any stationary distribution of this chain can be
uniquely written as a linear combination of these $\pi_{c_{1},\dots,c_{n}}$.
In particular,
\begin{enumerate}
\item if $\calb_{1}=\left\{ \bullet\right\} $, then 
\[
\pi_{n}(x):=\frac{\eta(x)}{n!}\xi_{\bullet,\dots,\bullet}^{x}
\]
is the unique stationary distribution of the chain on $\calb_{n}$;
\item if $\calh$ is multigraded ($\calh=\bigoplus_{\nu}\calh_{\nu}$, $\calb=\amalg_{\nu}\calb_{\nu}$)
and $\calb_{(1,0,\dots,0)}=\left\{ \bullet_{1}\right\} $, $\calb_{(0,1,0,\dots,0)}=\left\{ \bullet_{2}\right\} $
and so on, then
\[
\pi_{\nu}(x):=\frac{\eta(x)}{n!^{2}}\sum_{\sigma\in\sn}\xi_{c_{\sigma(1)},\dots,c_{\sigma(n)}}^{x}
\]
with $c_{1}=c_{2}=\dots=c_{\nu_{1}}=\bullet_{1}$, $c_{\nu_{1}+1}=\dots=c_{\nu_{1}+\nu_{2}}=\bullet_{2}$,
etc. is the unique stationary distribution of the chain on $\calb_{\nu}$;
\end{enumerate}

and these are also necessary conditions.

\end{thm}
Intuitively, the sum of product structure constants $\sum_{\sigma\in\sn}\xi_{c_{\sigma(1)},\dots,c_{\sigma(n)}}^{x}$
counts the ways that $x$ can be assembled from $c_{1},\dots,c_{n}$
in any order. So $\pi_{c_{1},\dots,c_{n}}(x)$ is proportional to
the number of ways to assemble $x$ from $c_{1},\dots,c_{n}$, and
then repeatedly break it down into objects of size 1.
\begin{proof}
First, show that $\pi_{c_{1},\dots,c_{n}}$ is a probability distribution.
As remarked in the proof of Lemma \ref{lem: nonnegative-coeffs},
$\xi_{c_{\sigma(1)},\dots,c_{\sigma(n)}}^{x}\geq0$, so $\pi_{c_{1},\dots,c_{n}}$
is a non-negative function. To see that $\sum_{x\in\calbn}\pi_{c_{1},\dots,c_{n}}(x)=1$,
appeal to the second displayed equation of the proof of Theorem \ref{thm:threestep}.
Taking $a=n$, it shows that, for each $\sigma\in\sn$, 
\[
\sum_{x\in\calbn}\xi_{c_{\sigma(1)},\dots,c_{\sigma(n)}}^{x}\eta(x)=\binom{n}{\deg c_{\sigma(1)}\dots\deg c_{\sigma(n)}}\eta(c_{1})\dots\eta(c_{n})=n!\cdot1\cdot\dots\cdot1.
\]

Next, recall that the stationary distributions are the left eigenfunctions
of the transition matrix of eigenvalue 1. So, by Proposition \ref{prop:efns}.L,
it suffices to show that $\sum_{x\in\calbn}\pi_{c_{1},\dots,c_{n}}(x)\frac{x}{\eta(x)}=\sum_{\sigma\in\sn}c_{\sigma(1)}\dots c_{\sigma(n)}$
is a basis for the $a^{n}$-eigenspace of $\Psi^{a}$. This is precisely
the assertion of \ref{thm:topeigenspace}. 

Finally, the two uniqueness results are immediate by taking the sole
choice of $c_{i}$s. 
\end{proof}
The first example below describes the typical behaviour when $\calb_{1}=\left\{ \bullet\right\} $
and $\calb$ is a free-commutative basis: the unique stationary distribution
is concentrated at a single state. Such a chain is said to be absorbing,
and Sections \ref{sub:Altrightefns} and \ref{sub:Absorption} give
some methods for estimating the probability of absorption after a
given time.
\begin{example}
\label{ex:noncommgraph-stationarydistribution}Continue with the edge-removal
chain of Example \ref{ex:chain-graph}, which arises from the Hopf
algebra $\barcalg$ of graphs. Here, the only element of $\calb_{1}$
is the graph $\bullet$ with a single vertex. So Part i of Theorem
\ref{thm:hpmc-stationarydistribution} applies, and the unique stationary
distribution is 
\[
\pi_{n}(x)=\frac{\eta(x)}{n!}\xi_{\bullet,\dots,\bullet}^{x},
\]
which is the point mass at the graph with no edges. This is because
the structure constant $\xi_{\bullet,\dots,\bullet}^{x}$ is 0 for
all other graphs $x$. Indeed, one would expect after many steps of
this chain, that all edges would be removed.
\end{example}

\begin{example}
\label{ex:schurfn-stationarydistribution}Continuing from Example
\ref{ex:schurfn}, take $\calh$ to be the representation rings of
the symmetric groups. The only irreducible representation of $\mathfrak{S}_{1}$
is the trivial representation, so again Theorem \ref{thm:hpmc-stationarydistribution}.i
above applies. Now $\bullet^{n}$ is the induced representation from
$\mathfrak{S}_{1}\times\dots\times\mathfrak{S}_{1}$ to $\sn$ of
the trivial representation, which gives the regular representation.
So $\xi_{\bullet,\dots,\bullet}^{x}$ is the multiplicity of the irreducible
representation $x$ in the regular representation, which is $\dim x$.
Recall from Example \ref{ex:schurfn-eta} that the rescaling constant
$\eta(x)$ is also $\dim x$. Thus the unique stationary distribution
of this restriction-then-induction chain is $\pi_{n}(x)=\frac{1}{n!^{2}}(\dim x)^{2}$.
This is the well-studied Plancherel measure. It appears as the distribution
of partitions growing one cell at a time under the Plancherel growth
process \cite{plancherelgrowthprocess}. \cite{planchereleigenvalues}
identifies its limit as $n\rightarrow\infty$, suitably rescaled,
with the distribution of eigenvalues of a Gaussian random Hermitian
matrix; the proof involves some combinatorially flavoured topology
and illuminates a connection to the intersection theory on moduli
spaces of curves.
\end{example}

\begin{example}
\label{ex:shufflealg-stationarydistribution}Consider $\calsh_{(1,1,\dots,1)}$,
the degree $(1,1,\dots,1)$ subspace of the shuffle algebra. This
corresponds to riffle-shuffling a distinct deck of cards. Use Theorem
\ref{thm:hpmc-stationarydistribution}.ii with $c_{i}=(i)$. It is
clear that, for each word $x$ in $\calsh_{(1,\dots,1)}$, there is
a unique way to interleave $(1),(2),\dots,(n)$ to obtain $x$. So
$\xi_{(1),\dots,(n)}^{x}=1$, and by commutativity, $\xi_{(\sigma(1)),\dots,(\sigma(n))}^{x}=1$
for all permutations $\sigma$. Recall also that $\eta(x)=1$ for
all words. So the unique stationary distribution for riffle-shuffling
is the uniform distribution $\pi(x)\equiv\frac{1}{n!}$. 
\end{example}
All the chains appearing in this thesis have unique stationary distributions.
For an example of a Hopf-power Markov chain with several absorbing
states, see Pineda's example on the Hopf monoid of permutohedra \cite{hopfmonoidchains}.

\section{Reversibility\label{sec:hpmc-Reversibility}}

Recall from Section \ref{sec:Reversibility} that the time-reversal
of a Markov chain from a linear map is given by the dual map. As observed
in Section \ref{sec:The-Hopf-power-Map}, the dual map to $\Psi^{a}:\calhn\rightarrow\calhn$
is the Hopf-power map on the dual Hopf algebra, $\Psi^{a}:\calhndual\rightarrow\calhndual$.
Thus Theorem \ref{thm: dual} specialises to the following for Hopf-power
chains:
\begin{thm}[Time-reversal of Hopf-power Markov chains]
\label{thm:hpmc-dual}Let $\calh$ be a graded, connected Hopf algebra
over $\mathbb{R}$ with state space basis $\calb$ satisfying $\calb=\left\{ \bullet\right\} $
(or $\calh$ is multigraded and $\calb_{(1,0,\dots,0)}=\left\{ \bullet_{1}\right\} ,\calb_{(0,1,0,\dots,0)}=\left\{ \bullet_{2}\right\} $
and so on). Suppose in addition that, for all $y\in\calb$ with $\deg(y)>1$,
there is some $w,z\in\calb$ with $\deg(w),\deg(z)>0$ such that $\xi_{wz}^{y}\neq0$.
Then the time-reversal of the $a$th Hopf-power Markov chain on $\calbn$
(or $\calb_{\nu}$) is the $a$th Hopf-power Markov chain on the dual
basis $\calbndual$ (or $\calbdual_{\nu}$) of the (graded) dual Hopf
algebra $\calhdual$.\qed
\end{thm}
Note the the condition $\xi_{wz}^{y}\neq0$ is equivalent to $\calbdual$
being a state space basis, since dualising the Hopf algebra simply
exchanges the product and coproduct structure constants:
\[
\xi_{w^{*}z^{*}}^{x^{*}}=\eta_{x}^{wz};\quad\eta_{y^{*}}^{w^{*}z^{*}}=\xi_{wz}^{y}.
\]
Then, applying Theorem \ref{thm:rescaling}.i to $\calhdual$ implies
$\xi_{\bullet,\dots,\bullet}^{y}>0$ for all $y\in\calb$. So the
stationary distribution of the Hopf-power chain on $\calh$ is nowhere
zero, and the time-reversal chain is indeed defined.
\begin{example}
\label{ex:inverseshuffle}Recall from Example \ref{ex:shufflealg-dual}
that the dual of the shuffle algebra $\calsh$ is the free associative
algebra $\calsh^{*}$, with concatenation product and deshuffling
coproduct. Its associated Hopf-square Markov chain has this interpretation
in terms of decks of cards: uniformly and independently assign each
card to the left or right pile, keeping cards which land in the same
pile in the same relative order, then put the left pile on top of
the right pile. This agrees with the description of inverse shuffling
of \cite[Sec. 3]{bd}.
\end{example}
The final result of Section \ref{sec:Reversibility} states that,
if $\Psi$ is self-adjoint with respect to an inner product where
the state space basis is orthogonal, and if a $\Psi$-Markov chain
has a well-defined time-reversal, then this chain is reversible. The
condition that the Hopf-power be self-adjoint is a little odd; a stronger
but more natural hypothesis is that the product and coproduct are
adjoint, in the manner described below.
\begin{thm}[Reversibility of Hopf-power Markov chains]
\label{thm:hpmc-reversible}Let $\calh$ be a graded, connected Hopf
algebra over $\mathbb{R}$ equipped with an inner product $\langle,\rangle$
adjoining product and coproduct, that is, $\langle wz,x\rangle=\langle w\otimes z,\Delta(x)\rangle$.
(Here, $\langle w\otimes z,a\otimes b\rangle=\langle w,a\rangle\langle z,b\rangle$.)
Let $\calb$ be a state space basis of $\calh$ which is orthogonal
under this inner product, with $\calb_{1}=\left\{ \bullet\right\} $
(or $\calh$ is multigraded and $\calb_{(1,0,\dots,0)}=\left\{ \bullet_{1}\right\} ,\calb_{(0,1,0,\dots,0)}=\left\{ \bullet_{2}\right\} $
and so on). Assume in addition that, for all $y\in\calb$ with $\deg(y)>1$,
there is some $w,z\in\calb$ with $\deg(w),\deg(z)>0$ such that $\xi_{wz}^{y}\neq0$.
Then the $a$th Hopf-power Markov chain on $\calbn$ (or $\calb_{\nu}$)
is reversible. \qed
\end{thm}
Zelevinsky's classification \cite[Th. 2.2, 3.1]{pshclassification}
of \emph{positive self-dual} Hopf algebras says that, if one restricts
to Hopf algebras with integral structure constants, then the example
below is essentially the only chain satisfying the hypothesis of Theorem
\ref{thm:hpmc-reversible} above.
\begin{example}
\label{ex:schurfn-reversible}Equip the representation rings of the
symmetric group with the usual inner product where the irreducible
representations are orthonormal. (This is equivalent to the Hall inner
product of symmetric functions, see \cite[Sec. 7.9]{stanleyec2}.)
That this inner product adjoins the multiplication and comultiplication
is simply Frobenius reciprocity: 
\[
\langle\Ind_{\mathfrak{S}_{i}\times\mathfrak{S}_{j}}^{\mathfrak{S}_{i+j}}w\times z,x\rangle=\langle w\otimes z,\Res_{\mathfrak{S}_{i}\times\mathfrak{S}_{j}}^{\mathfrak{S}_{i+j}}x\rangle.
\]
(Note that, if $w,z$ are representations of $\mathfrak{S}_{i},\mathfrak{S}_{j}$
respectively, then $\langle w\otimes z,\Res_{\mathfrak{S}_{k}\times\mathfrak{S}_{i+j-k}}^{\mathfrak{S}_{i+j}}x\rangle=0$
unless $k=i$.) As calculated in Example \ref{ex:schurfn-stationarydistribution},
the associated restriction-then-induction chain has a unique stationary
distribution given by the Plancherel measure $\pi(x)=\frac{\dim x^{2}}{n!}>0$.
So this chain is reversible.
\end{example}

\section{Projection\label{sec:hpmc-projection}}

Recall the mantra of Section \ref{sec:Projection}: intertwining maps
give rise to projections of Markov chains. For Hopf-power Markov chains,
the natural maps to use are Hopf-morphims. A linear map $\theta:\calh\rightarrow\barcalh$
is a \emph{Hopf-morphism} if $\theta(1)=1$, $\deg(\theta(x))=\deg(x)$,
$\theta(w)\theta(z)=\theta(wz)$ and $\Delta(\theta(x))=(\theta\otimes\theta)(\Delta(x))$
for all $x,w,z\in\calh$. Then 
\[
\theta(m\Delta(x))=m(\theta\otimes\theta)(\Delta(x))=m\Delta(\theta(x)),
\]
so $\theta$ intertwines the Hopf-square maps on $\calh$ and on $\barcalh$.
Indeed, a simple (co)associativity argument shows that $\theta\proda=\proda\theta^{\otimes a}$
and $\theta^{\otimes a}\coproda=\coproda\theta$ for all $a$, so
$\theta\Psi_{\calh}^{a}=\Psi_{\barcalh}^{a}\theta$. (Note that $\Psi^{a}$
is not a Hopf-morphism in general.)

Specialising Theorem \ref{thm:projection}, concerning projections
of chains from linear maps, to the Hopf-power map, gives the following: 
\begin{thm}[Projection Theorem for Hopf-power Markov Chains]
\label{thm:hpmc-projection} Let $\calh$, $\barcalh$ be graded,
connected Hopf algebras over $\mathbb{R}$ with bases $\calb$, \textup{$\barcalb$}
respectively. Suppose in addition that $\calb$ is a state space basis.
If $\theta:\calh\rightarrow\barcalh$ is a Hopf-morphism such that
$\theta(\calbn)=\barcalb_{n}$ for some $n$, and $\theta(\calb_{1})\subseteq\barcalb_{1}$,
then the Hopf-power Markov chain on $\barcalb_{n}$ is the projection
via $\theta$ of the Hopf-power Markov chain on $\calbn$.
\end{thm}
\begin{rems*}
$ $

\begin{enumerate}[label=\arabic*.]
\item  As in the more general Theorem \ref{thm:projection}, the condition
$\theta(\calbn)=\barcalb_{n}$ does not mean that the restriction
$\theta:\calhn\rightarrow\barcalh_{n}$ is an isomorphism. Although
$\theta$ must be surjective onto $\barcalb_{n}$, it need not be
injective - the requirement is simply that distinct images of $\calbn$
under $\theta$ are linearly independent.
\item The theorem does not require $\theta(\calbn)=\barcalb_{n}$ to hold
for all $n$. Section \ref{sub:chain-qsym}, regarding the descent
sets under riffle-shuffling, is an important example where the domain
$\calh$ is multigraded, and $\theta(\calb_{\nu})=\barcalb_{|\nu|}$
for only certain values of $\nu$.
\item The proof will show that the weaker assumption $\theta(\calb_{1})\subseteq\alpha\barcalb_{1}:=\left\{ \alpha\bar{c}|\bar{c}\in\barcalb_{1}\right\} $
is sufficient. (Here, $\alpha$ can be any non-zero constant.)
\end{enumerate}
\end{rems*}
\begin{proof}
As discussed before the statement of the theorem, $\theta\Psi_{\calh}^{a}=\Psi_{\barcalh}^{a}\theta$.
So it suffices to show that the condition $\theta(\calb_{1})\subseteq\barcalb_{1}$
guarantees $\eta(x)=\eta(\theta(x))$ for all $x\in\calbn$. Then
Theorem \ref{thm:projection}, concerning projections of chains from
linear maps, applies to give the desired result. 

Let $n=\deg x=\deg(\theta(x))$. Recall that the rescaling function
$\eta(x)$ is the sum of the coefficients of $\bard^{[n]}(x)$ when
expanded in the basis $\calb_{1}^{\otimes n}$:
\[
\eta(x)=\sum_{c_{1},\dots,c_{n}\in\calb_{1}}\eta_{x}^{c_{1},\dots,c_{n}},
\]
so 
\[
\eta(\theta(x))=\sum_{c_{1},\dots,c_{n}\in\barcalb_{1}}\eta_{\theta(x)}^{c_{1},\dots,c_{n}}.
\]
Now expanding the equality $\bard^{[n]}(\theta(x))=\theta^{\otimes n}(\bard^{[n]}(x))$
in the basis $\calb_{1}^{\otimes n}$ gives:
\begin{align*}
\sum_{\barc_{1},\dots,\barc_{n}\in\barcalb_{1}}\eta_{\theta(x)}^{\barc_{1},\dots,\barc_{n}}\barc_{1}\otimes\dots\otimes\barc_{n} & =\theta^{\otimes n}\left(\sum_{c_{1},\dots,c_{n}\in\calb_{1}}\eta_{x}^{c_{1},\dots,c_{n}}c_{1}\otimes\dots\otimes c_{n}\right)\\
 & =\sum_{c_{1},\dots,c_{n}\in\calb_{1}}\eta_{x}^{c_{1},\dots,c_{n}}\theta(c_{1})\otimes\dots\otimes\theta(c_{n})\\
 & =\sum_{\barc_{1},\dots,\barc_{n}\in\barcalb_{1}}\left(\sum_{c_{1},\dots,c_{n},\theta(c_{i})=\barc_{i}}\eta_{x}^{c_{1},\dots,c_{n}}\right)\barc_{1}\otimes\dots\otimes\barc_{n},
\end{align*}
where the last equality uses the assumption $\theta(c_{i})\in\calb_{1}$.
So the coefficient sums of the left and right hand sides are equal,
and these are $\eta(\theta(x))$ and $\eta(x)$ respectively.\end{proof}
\begin{example}
\label{ex:idescent-lumping}Work in $\calsh^{*}$, the free associative
algebra introduced in Example \ref{ex:shufflealg-dual}, where the
product of two words is their concatenation, and the coproduct is
deshuffle. As seen in Example \ref{ex:inverseshuffle}, the associated
Hopf-power Markov chain describes inverse riffle-shuffling: randomly
place each card on the left or right pile, then place the left pile
on top of the right. Let $\bar{\calsh^{*}}$ be the quotient of $\calsh^{*}$,
as an algebra, by the relations $\{ij=ji|\left|i-j\right|>1\}$. Then
$\bar{\calsh^{*}}$ is one example of a free partially commutative
algebra of \cite{tracealg}, based on the free partially commutative
monoids of \cite{tracealg2}. The technical Lemmas \ref{lem:freepartiallycommutativealg}
and \ref{lem:commutationrelationbasis} below prove respectively that
the quotient map $\calsh^{*}\rightarrow\bar{\calsh^{*}}$ is a map
of Hopf algebras, and that this map sends the basis of words of $\calsh^{*}$
to a basis of $\bar{\calsh^{*}}$. Thus this quotient map shows that
inverse riffle-shuffling while forgetting the orders of cards with
nonconsecutive values is a Markov chain. For example, this would identify
$(231124)$ with $(213412)$. When all cards in the deck are distinct,
this amounts to keeping track only of whether card 1 is above or below
card 2, whether card 2 is above or below card 3, etc. This statistic
is known as the \emph{idescent set} (or recoil): 
\[
\ides(w)=\left\{ i|i+1\mbox{ occurs before }i\mbox{ in }w\right\} 
\]
as it is the descent set of the inverse of $w$, when regarding $w$
as a permutation in one-line notation. The projection of inverse riffle-shuffling
by idescent set is studied in \cite[Ex. 5.12.ii]{lumpedhyperplanewalks}. 

The same construction goes through for other sets of commutation relations.
Specifically, let $G$ be a graph with vertex set $\{1,2,\dots\}$
and finitely-many edges, and set $\sg$ to be the quotient of $\calsh^{*}$,
as an algebra, by the relations $\{ij=ji|(i,j)\mbox{ not an edge of }G\}$.
Thus the edges of $G$ indicate noncommuting pairs of letters in $\sg$.
The example above, where only nonconsecutive values commute, corresponds
to a path. The Lemmas below show that, for any graph $G$, the quotient
map $\theta_{G}:\calsh\rightarrow\sg$ satisfies the conditions of
the Projection Theorem, so these maps all give Markov statistics for
inverse shuffling. To interpret these statistics, appeal to \cite[Prop. 2]{tracealg3}.
For a word $w$, let $w_{ij}$ denote the subword of $w$ obtained
by deleting all letters not equal to $i$ or $j$. Thus $(231124)_{12}=(2112)$,
$(231124)_{23}=(232)$. Then their proposition asserts that $\theta_{G}(w)$
is recoverable from the set of $w_{ij}$ over all edges $(i,j)$ of
$G$. To summarise:
\begin{thm}
Let $G$ be a graph with vertex set $\{1,2,\dots\}$ and finitely-many
edges. For a deck of cards $w$, let $w_{ij}$ be the subdeck obtained
by throwing out all cards not labelled $i$ or $j$. Then the set
of all $w_{ij}$ over all edges $(i,j)$ of $G$ is a Markov statistic
under inverse shuffling. \qed
\end{thm}

Below are the promised technical Lemmas necessary to establish this
result.
\begin{lem}
\label{lem:freepartiallycommutativealg} Let $G$ be a graph with
vertex set $\{1,2,\dots\}$ and finitely-many edges. Denote by $I_{G}$
the ideal in the free associative algebra $\calsh^{*}$ generated
by $\{ij-ji|(i,j)\mbox{ not an edge of }G\}$. Then $I_{G}$ is also
a \emph{coideal}, (i.e. $\Delta(I_{G})\subseteq\calsh^{*}\otimes I_{G}+I_{G}\otimes\calsh^{*}),$
so the quotient $\sg:=\calsh^{*}/I_{G}$ is a Hopf algebra.\end{lem}
\begin{proof}
Since $\Delta$ is linear and $\Delta(xy)=\Delta(x)\Delta(y)$, it
suffices to check the coideal condition only on the generators of
$I_{G}$, that is, that $\Delta(ij-ji)\subseteq\calsh^{*}\otimes I_{G}+I_{G}\otimes\calsh^{*}$
whenever $(i,j)$ is not an edge of $G$. Now
\begin{align*}
\Delta(ij-ji) & =\Delta(i)\Delta(j)-\Delta(j)\Delta(i)\\
 & =(1\otimes i+i\otimes1)(1\otimes j+j\otimes1)-(1\otimes j+j\otimes1)(1\otimes i+i\otimes1)\\
 & =1\otimes ij+j\otimes i+i\otimes j+ij\otimes1-(1\otimes ji+i\otimes j+j\otimes i+ji\otimes1)\\
 & =1\otimes ij+ij\otimes1-1\otimes ji-ji\otimes1\\
 & =1\otimes(ij-ji)+(ij-ji)\otimes1\\
 & \subseteq\calsh^{*}\otimes I_{G}+I_{G}\otimes\calsh^{*}.
\end{align*}
\end{proof}
\begin{lem}
\label{lem:commutationrelationbasis}Let $\theta_{G}:\calsh^{*}\rightarrow\sg$
be the quotient map, by the ideal $I_{G}$ in Lemma \ref{lem:freepartiallycommutativealg}
above. Write $\calb$ the basis of words in the free associative algebra
$\calsh^{*}$. Then $\barcalb:=\theta_{G}(\calb)$ is a basis of $\sg$.\end{lem}
\begin{proof}
(The main idea of this proof arose from a discussion with Zeb Brady.)
Clearly $\barcalb$ spans $\sg$, so the only issue is linear independence.
This will follow from 
\[
I_{G}=J:=\left\{ a_{1}b_{1}+\dots+a_{m}b_{m}|b_{i}\in\calb,\sum_{i:\theta_{G}(b_{i})=\barb}a_{i}=0\mbox{ for each }\barb\in\barcalb\right\} .
\]
The quotient map $\theta_{G}$ clearly sends each element of $J$
to 0, so $J\subseteq\ker\theta_{G}=I_{G}$. To see $I_{G}\subseteq J$,
it suffices to show that $J$ is an ideal containing the generators
$ij-ji$ of $I_{G}$. First, $J$ is clearly closed under addition.
$J$ is closed under multiplication by elements of $\calsh^{*}$ because,
for any letter $c$ (i.e. any generator of $\calsh^{*}$), $c(a_{1}b_{1}+\dots+a_{m}b_{m})=a_{1}(cb_{1})+\dots+a_{m}(cb_{m})$
with each $cb_{i}\in\calb$, and, if $\theta_{G}(b_{i})=\theta_{G}(b_{j})$,
then $\theta_{G}(cb_{i})=\theta_{G}(cb_{j})$. Lastly, if $(i,j)$
is not an edge of $G$, then $\theta_{G}(ij)=\theta_{G}(ji)$, so
$ij-ji\in J$.
\end{proof}
\end{example}

\chapter{Hopf-power Markov chains on Free-Commutative Bases \label{chap:free-commutative-basis}}

\chaptermark{Chains on Free-Commutative Bases}

This chapter concentrates on a class of Hopf-power Markov chains whose
behaviour is ``simple'' in two ways, thanks to the additional hypothesis
that the state space basis is free-commutative, as defined below.
\begin{defn*}[Free generating set, free-commutative basis]
Let $\calh$ be a graded connected commutative Hopf algebra over
$\mathbb{R}$. Then the dual Cartier-Milnor-Moore theorem \cite[Th. 3.8.3]{cmm}
states that $\calh$ is isomorphic as an algebra to the polynomial
algebra $\mathbb{R}[c_{1},c_{2},\dots]$ for some elements $c_{i}$,
which may have any degree. (In fact, it suffices that the base field
be of characteristic 0.) The set $\calc:=\{c_{1},c_{2},\dots\}$ is
a \emph{free generating set} for $\calh$, and the basis $\calb=\left\{ c_{1}\dots c_{l}|l\in\mathbb{N},\left\{ c_{1},\dots,c_{l}\right\} \mbox{ a multiset in }\calc\right\} $,
consisting of all products of the $c_{i}$s, is a \emph{free-commutative
basis}.
\end{defn*}
One can think of a free-commutative basis as the basis of monomials
in the $c_{i}$, but this thesis prefers to reserve the terminology
``monomial'' for analogues of the monomial symmetric functions,
which are cofree.

An archetypal chain on a free-commutative basis is the edge-removal
of graphs (or indeed the analogous construction for any species-with-restrictions,
as discussed in Section \ref{sub:Species}). Specialising to disjoint
unions of complete graphs gives the independent multinomial breaking
of rocks, as discussed in Section \ref{sec:Rock-breaking}.
\begin{example*}[Edge-removal of graphs]
\label{ex:graphrecap}Recall from Examples \ref{ex:graph} and \ref{ex:chain-graph}
the Hopf algebra $\barcalg$ of graphs: the degree $\deg(G)$ of a
graph $G$ is its number of vertices, the product of two graphs is
their disjoint union, and the coproduct is 
\[
\Delta(G)=\sum G_{S}\otimes G_{S^{\calc}}
\]
where the sum is over all subsets $S$ of the vertex set of $G$,
and $G_{S},G_{S^{\calc}}$ denote the subgraphs that $G$ induces
on the vertex set $S$ and its complement. The set $\calb$ of all
graphs is a free-commutative basis, and the free generating set $\calc$
consists of the connected graphs.

The $a$th Hopf-power Markov chain describes edge-removal: at each
step, assign uniformly and independently one of $a$ colours to each
vertex, and remove the edges connecting vertices of different colours.
There is no need to rescale the state space basis $\calb$ to define
this chain: for all graphs $G$ with $n$ vertices, the rescaling
function $\eta(G)$ counts the ways to break $G$ into $n$ (ordered)
singletons, of which there are $n!$, irrespective of $G$. An easy
application of Theorem \ref{thm:hpmc-stationarydistribution} shows
that its unique stationary distribution takes value 1 on the graph
with no edges and 0 on all other states, so the chain is \emph{absorbing.} 
\end{example*}
The first ``simplicity'' feature of this edge-removal chain is that
each connected component behaves independently. Section \ref{sub:Independence}
explains the analogous behaviour for all chains on a free-commutative
basis as a consequence of the Hopf-power map $\Psi^{a}$ being an
algebra homomorphism, since the underlying Hopf algebra is commutative.
The second aspect of interest is that the edge-removal chain never
returns to a state it has left. Indeed, at each step the chain either
stays at the same graph or the number of connected components increases.
Section \ref{sub:Unidirectionality} will show that a Hopf-power Markov
chain on a free-commutative state space basis always has a triangular
transition matrix; then, applying Perron-Frobenius to each minor gives
right eigenfunctions that are non-negative in the first few coordinates
and zero in the last coordinates. Section \ref{sub:Altrightefns}
identifies these as the output of Theorem \ref{thm:diagonalisation}.B,
and outlines how they give upper bounds for the probability of being
``far from absorbed''. Section \ref{sub:Absorption} then repackages
the exact probabilities in terms of a ``generalised chromatic quasisymmetric
function'' constructed in \cite{abs}, though this is a theoretical
discussion only as I have no effective way to compute or bound such
functions. This appears to require weaker hypotheses than a free-commutative
state space basis, but it is unclear whether there are non-free-commutative
state space bases that satisfy the weaker hypotheses, nor what the
conclusions mean in this more general setup. 

Sections \ref{sec:Rock-breaking} and \ref{sec:Tree-Pruning} apply
these techniques to a rock-breaking and tree-pruning process respectively,
arising from the algebra of symmetric functions and the Connes-Kreimer
algebra of rooted forests.

\section{General Results\label{sec:freecommutative}}

\subsection{Independence\label{sub:Independence}}

The following theorem converts the fact that $\Psi^{a}$ is an algebra
homomorphism into ``independent breaking'' of the Hopf-power Markov
chain if the starting state is a product. For example, in the Hopf
algebra $\barcalg$ of graphs, a graph is the product of its connected
components, so the associated edge-removal Markov chain behaves independently
on each connected component. As a result, to understand a Hopf-power
Markov chain on a free-commutative basis, it suffices to describe
one step of the chain starting only from the generators, i.e. to apply
Theorem \ref{thm:threestep}, the three-step interpretation, only
to states which are not products.
\begin{thm}
\label{thm:independence}Let $x_{1},x_{2}\in\calb$, a free-commutative
state space basis. Then one step of the $a$th Hopf-power Markov chain
on $\calb$ starting at $x:=x_{1}x_{2}$ is equivalent to the following:
take one step of the $a$th Hopf-power Markov chain from $x_{1}$
and from $x_{2}$, and move to the product of the results.\end{thm}
\begin{proof}
Let $n,n_{1},n_{2}$ be the degrees of $x,x_{1},x_{2}$ respectively.
By definition, the probability of moving from $x$ to $y$ in the
$a$th Hopf-power Markov chain is 
\[
\hatkan(x,y)=y^{*}(a^{-n}\Psi^{a}(x))\frac{\eta(y)}{\eta(x)}.
\]
So the probability of moving from $x$ to $y$ under the composite
process described in the theorem is
\begin{align*}
 & \sum_{y_{1}y_{2}=y}\hatk_{a,n_{1}}(x_{1},y_{1})\hatk_{a,n_{2}}(x_{2},y_{2})\\
= & \sum_{y_{1}y_{2}=y}y_{1}^{*}(a^{-n_{1}}\Psi^{a}(x_{1}))\frac{\eta(y_{1})}{\eta(x_{1})}y_{2}^{*}(a^{-n_{2}}\Psi^{a}(x_{2}))\frac{\eta(y_{2})}{\eta(x_{2})}.
\end{align*}
Since $\calb$ is a free-commutative basis, the structure constant
$\xi_{y_{1}y_{2}}^{y}$ is 1 if $y_{1}y_{2}=y$, and 0 otherwise.
So the above probability is
\begin{align*}
 & \sum_{y_{1}\in\calb_{n_{1}}y_{2}\in\calb_{n_{2}}}\xi_{y_{1}y_{2}}^{y}(y_{1}^{*}\otimes y_{2}^{*})(a^{-n}\Psi^{a}(x_{1})\otimes\Psi^{a}(x_{2}))\frac{\eta(y_{1})}{\eta(x_{1})}\frac{\eta(y_{2})}{\eta(x_{2})}\\
= & \Delta^{*}(y)(a^{-n}\Psi^{a}(x_{1})\otimes\Psi^{a}(x_{2}))\frac{\eta(y_{1})}{\eta(x_{1})}\frac{\eta(y_{2})}{\eta(x_{2})}\\
= & y^{*}(a^{-n}\Psi^{a}(x_{1})\Psi^{a}(x_{2}))\frac{\eta(y_{1})}{\eta(x_{1})}\frac{\eta(y_{2})}{\eta(x_{2})}\\
= & y^{*}(a^{-n}\Psi^{a}(x_{1}x_{2}))\frac{\eta(y_{1})}{\eta(x_{1})}\frac{\eta(y_{2})}{\eta(x_{2})}.
\end{align*}
The last step uses that $\Psi^{a}$ is an algebra homomorphism since
the Hopf algebra is commutative. Lemma \ref{lem:etaproduct} below
shows that $\frac{\eta(y_{1})}{\eta(x_{1})}\frac{\eta(y_{2})}{\eta(x_{2})}=\frac{\eta(y_{1}y_{2})}{\eta(x_{1}x_{2})}=\frac{\eta(y)}{\eta(x)}$,
so this probability is indeed $\hatkan(x,y)$.\end{proof}
\begin{lem}
\label{lem:etaproduct}The rescaling function $\eta$ satisfies 
\[
\eta(x_{1}x_{2})=\binom{\deg(x_{1}x_{2})}{\deg(x_{1})}\eta(x_{1})\eta(x_{2}).
\]
In other words, $\frac{\eta(x)}{(\deg x)!}$ is multiplicative.\end{lem}
\begin{proof}
There is a short proof via $\eta(x)=(\bullet^{*})^{\deg x}$, but
the enumerative argument here is more transparent and more versatile
- similar lines of reasoning lie behind Proposition \ref{prop:rightefnsofproducts}
and (to a lesser extent) Theorems \ref{thm:eta-cktrees} and \ref{thm:rightefn-cktrees}.

Write $n,n_{1},n_{2}$ for the degrees of $x,x_{1},x_{2}$ respectively.
$\eta(x_{1}x_{2})$ is the sum of the coefficients of $\bard^{[n]}(x_{1}x_{2})$.
The Hopf axiom $\Delta(x_{1}x_{2})=\Delta(x_{1})\Delta(x_{2})$ gives
the following bijection:
\[
\left\{ \begin{array}{c}
\mbox{terms in}\\
\bard^{[n]}(x_{1}x_{2})
\end{array}\right\} \leftrightarrow\left\{ \begin{array}{c}
\mbox{terms in}\\
\bard^{[n_{1}]}(x_{1})
\end{array}\right\} \times\left\{ \begin{array}{c}
\mbox{terms in}\\
\bard^{[n_{2}]}(x_{2})
\end{array}\right\} \times\left\{ \begin{array}{c}
\mbox{choices of }n_{1}\mbox{ tensor-factors}\\
\mbox{amongst }n\mbox{ to place}\\
\mbox{ the term from }\bard^{[n]}(x_{1})
\end{array}\right\} .
\]
Taking coefficients of both sides recovers the lemma.
\end{proof}

\subsection{Unidirectionality\label{sub:Unidirectionality}}

Call a Markov chain \emph{unidirectional} if it cannot return to any
state it has left. (The term ``unidirectional'' is a suggestion
from John Pike, since ``monotone'' and ``acyclic'' already have
technical meanings in Markov chain theory.) An equivalent phrasing
is that the state space is a poset under the relation ``is accessible
from''. Yet another characterisation of a unidirectional chain is
that its transition matrix is triangular for some suitable ordering
of the states.

The edge-removal chain at the start of this chapter is unidirectional
as the chain either stays at the current graph, or the number of connected
components increases. Corollary \ref{cor:unidirectional} below shows
that this phenomenon occurs for all Hopf-power Markov chains on a
free-commutative basis. The generalisation of ``number of connected
components'' is the \emph{length}: for $x\in\calb$, its length $l(x)$
is the number of factors in the factorisation of $x$ into generators.
Lemma \ref{lem:length} below explains the way the length changes
under product and coproduct. It requires one more piece of notation:
define $x\rightarrow x'$ for $x,x'\in\calb$ if $x'$ appears in
$\Psi^{a}(x)$ (when expanded in the basis $\calb$) for some $a$.
This is precisely the relation ``is accessible from'' discussed
in the previous paragraph.
\begin{lem}
\label{lem:length}Let $x,y,x_{i},x_{(i)}$ be elements of a free-commutative
basis. Then 
\begin{enumerate}
\item $l\left(x_{1}\dots x_{a}\right)=l\left(x_{1}\right)+\dots+l\left(x_{a}\right)$; 
\item For any summand $x_{(1)}\otimes\dots\otimes x_{(a)}$ in $\Delta^{[a]}(x)$,
$l\left(x_{(1)}\right)+\dots+l\left(x_{(a)}\right)\geq l(x)$; 
\item if $x\rightarrow y$, then $l(y)\geq l(x)$. 
\end{enumerate}
\end{lem}
\begin{proof}
(i) is clear from the definition of length.

Prove (ii) by induction on $l(x)$. Note that the claim is vacuously
true if $x$ is a generator, as each $l\left(x_{(i)}\right)\geq0$,
and not all $l\left(x_{(i)}\right)$ may be zero. If $x$ factorises
non-trivially as $x=st$, then, as $\Delta^{[a]}(x)=\Delta^{[a]}(s)\Delta^{[a]}(t)$,
it must be the case that $x_{(i)}=s_{(i)}t_{(i)}$, for some $s_{(1)}\otimes\dots\otimes s_{(a)}$
in $\Delta^{[a]}(s)$, $t_{(1)}\otimes\dots\otimes t_{(a)}$ in $\Delta^{[a]}(t)$.
So $l\left(x_{(1)}\right)+\dots+l\left(x_{(a)}\right)=l\left(s_{(1)}\right)+\dots+l\left(s_{(a)}\right)+l\left(t_{(1)}\right)+\dots+l\left(t_{(a)}\right)$
by (i), and by inductive hypothesis, this is at least $l(s)+l(t)=l(x)$.

(iii) follows trivially from (i) and (ii): if $x\rightarrow y$, then
$y=x_{(1)}\dots x_{(a)}$ for a term $x_{(1)}\otimes\dots\otimes x_{(a)}$
in $\Delta^{[a]}(x)$. So $l(y)=l\left(x_{(1)}\right)+\dots+l\left(x_{(a)}\right)\geq l(x)$. 
\end{proof}
Here is the algebraic fact which causes unidirectionality; the proof
is four paragraphs below. 
\begin{prop}
\label{prop:unidirectional}Let $\calh$ be a Hopf algebra with free-commutative
basis $\calb$, where all coproduct structure constants $\eta_{x}^{wz}$
are non-negative. Then the relation $\rightarrow$ defines a partial
order on $\calb$, and the partial-ordering by length refines this
partial-order: if $x\rightarrow y$ and $x\neq y$, then $l(x)<l(y)$.
Furthermore, for any integer $a$ and any $x\in\calb$, 
\[
\Psi^{a}(x)=a^{l(x)}x+\sum_{l(y)>l(x)}\alpha_{xy}y
\]
for some $\alpha_{xy}\geq0$. 
\end{prop}
The probability consequence is immediate from Definition \ref{defn: better-defition-of-hpmc}
of a Hopf-power Markov chain:
\begin{cor}
\label{cor:unidirectional}Let $\{X_{m}\}$ be the $a$th Hopf-power
Markov chain on a free-commutative basis $\calbn$. Then 
\[
P\{X_{m+1}=x|X_{m}=x\}=a^{l(x)-n},
\]
and $P\{X_{m+1}=y|X_{m}=x\}$ is non-negative only if $l(y)\geq l(x)$.\qed
\end{cor}
In other words, if the states are totally ordered to refine the partial-ordering
by length, then the transition matrices are upper-triangular with
$a^{l-n}$ on the main diagonal. In particular, states with length
$n$ are absorbing - which also follows from the stationary distribution
expressions in Theorem \ref{thm:hpmc-stationarydistribution}. These
states are precisely the products of elements of $\calb_{1}$.
\begin{proof}[Proof of Proposition \ref{prop:unidirectional}]
 It is easier to first prove the expression for $\Psi^{a}(x)$. Suppose
$x$ has factorisation into generators $x=c_{1}c_{2}\dots c_{l(x)}$.
As $\calh$ is commutative, $\Psi^{a}$ is an algebra homomorphism,
so $\Psi^{a}(x)=\Psi^{a}\left(c_{1}\right)\dots\Psi^{a}\left(c_{l(x)}\right)$.
Recall from Section \ref{sec:The-Eulerian-Idempotent} that $\bard(c)=\Delta(c)-1\otimes c-c\otimes1\in\bigoplus_{i=1}^{\deg(c)-1}\calh_{i}\otimes\calh_{\deg(c)-i}$,
in other words, $1\otimes c$ and $c\otimes1$ are the only terms
in $\Delta(c)$ which have a tensor-factor of degree 0. As $\Delta^{[3]}=(\iota\otimes\Delta)\Delta$,
the only terms in $\Delta^{[3]}(c)$ with two tensor-factors of degree
0 are $1\otimes1\otimes c$, $1\otimes c\otimes1$ and $c\otimes1\otimes1$.
Inductively, we see that the only terms in $\Delta^{[a]}(c)$ with
all but one tensor-factor having degree 0 are $1\otimes\dots\otimes1\otimes c,1\otimes\dots\otimes1\otimes c\otimes1,\dots,c\otimes1\otimes\dots\otimes1$.
So $\Psi^{a}(c)=ac+\sum_{l(y)>1}\alpha_{cy}y$ for generators $c$,
and $\alpha_{cy}\geq0$ by the hypothesis that all coproduct structure
constants are non-negative. As $\Psi^{a}(x)=\Psi^{a}\left(c_{1}\right)...\Psi^{a}\left(c_{l}\right)$,
and length is multiplicative (Lemma \ref{lem:length}.i), the expression
for $\Psi^{a}(x)$ follows.

It is then clear that $\rightarrow$ is reflexive and antisymmetric.
Transitivity follows from the power rule: if $x\rightarrow y$ and
$y\rightarrow y'$, then $y$ appears in $\Psi^{a}(x)$ for some $a$
and $y'$ appears in $\Psi^{a'}(y)$ for some $a'$. So $y'$ appears
in $\Psi^{a'}\Psi^{a}(x)=\Psi^{a'a}(x)$. (The non-negativity of coproduct
structure constants ensures that the $y'$ term in $\Psi^{a'a}(x)$
cannot cancel out due to contributions from an intermediary different
from $x'$.) \end{proof}
\begin{rem*}
It is possible to adapt the above arguments to Hopf algebras with
a (noncommutative) free basis $\calb=\left\{ S_{1}S_{2}\dots S_{k}|k\in\mathbb{N},S_{i}\in\calc\right\} $
(see Theorem \ref{thm:diagonalisation}.B). This shows that, for $x\in\calb$,
all terms in $\Psi^{a}(x)$ are either a permutation of the factors
of $x$, or have length greater than that of $x$. In particular,
for the associated Markov chain, the probability of going from $x$
to some permutation of its factors (as opposed to a state of greater
length, from which there is no return to $x$) is $a^{l(x)-\deg(x)}$.
However, it is easier to deduce such information by working in the
\emph{abelianisation} of the underlying Hopf algebra; that is, quotient
it by commutators $xy-yx$, which would send the free basis $\calb$
to a free-commutative basis. By Theorem \ref{thm:hpmc-projection},
such quotienting corresponds to a projection of the Markov chain.
\end{rem*}
Here are two more technical results in this spirit, which will be
helpful in Section \ref{sub:Altrightefns} for deducing a triangularity
feature of the eigenfunctions.
\begin{lem}
\label{lem:unidirectionalproduct}Let $x,x_{i},y_{i}$ be elements
of a free-commutative basis, with respect to which all coproduct structure
constants are non-negative. If $x=x_{1}\dots x_{k}$ and $x_{i}\rightarrow y_{i}$
for each $i$, then $x\rightarrow y_{1}\dots y_{k}$. \end{lem}
\begin{proof}
For readability, take $k=2$ and write $x=st,\ s\rightarrow s',\ t\rightarrow t'$.
By definition of the relation $\rightarrow$, it must be that $s'=s_{(1)}\dots s_{(a)}$
for some summand $s_{(1)}\otimes\dots\otimes s_{(a)}$ of $\bard^{[a]}(s)$.
Likewise $t'=t_{(1)}\dots t_{(a')}$ for some $a'$. Suppose $a>a'$.
Coassociativity implies that $\Delta^{[a]}(t)=(\iota\otimes\dots\otimes\iota\otimes\Delta^{[a-a']})\Delta^{[a']}(t)$,
and $t_{(a')}\otimes1\otimes\dots\otimes1$ is certainly a summand
of $\Delta^{[a-a']}(t_{(a')})$, so $t_{(1)}\otimes\dots\otimes t_{(a')}\otimes1\otimes\dots\otimes1$
occurs in $\Delta^{[a]}(t)$. So, taking $t_{(a'+1)}=\dots=t_{(a)}=1$,
we can assume $a=a'$. Then $\Delta^{[a]}(x)=\Delta^{[a]}(s)\Delta^{[a]}(t)$
contains the term $s_{(1)}t_{(1)}\otimes\dots\otimes s_{(a)}t_{(a)}$.
Hence $\Psi^{a}(x)$ contains the term $s_{(1)}t_{(1)}\dots s_{(a)}t_{(a)}$,
and this product is $s't'$ by commutativity. (Again, this instance
of $s't'$ in $\Psi^{a}(x)$ cannot cancel out with another term in
$\Psi^{a}(x)$ because the coproduct structure constants are non-negative.)\end{proof}
\begin{lem}
\label{lem:etatogenerators} Let $x,y$ be elements of a free-commutative
basis, with respect to which all coproduct structure constants are
non-negative. Suppose $y$ has factorisation into generators $y=c_{1}\dots c_{l}$.
If $x\rightarrow y$ then a coproduct structure constant of the form
$\eta_{x}^{c_{\sigma(1)},\dots,c_{\sigma(l)}}$ (for some $\sigma\in\Sl$)
is nonzero.\end{lem}
\begin{proof}
If $x\rightarrow y$, then, for some $a$, there is a term $x_{(1)}\otimes\dots\otimes x_{(a)}$
in $\coproda(x)$ with $x_{(1)}\dots x_{(a)}=y$. So each $x_{(i)}$
must have factorisation $x_{(i)}=\prod_{j\in B_{i}}c_{j}$ for some
set partition $B_{1}|\dots|B_{a}$ of $\{1,2,\dots,l\}$. In other
words, there is some permutation $\sigma\in\Sl$ and some $l_{1},\dots l_{a}\in\mathbb{N}$
such that $x_{(1)}=c_{\sigma(1)}\dots c_{\sigma(l_{1})},x_{(2)}=c_{\sigma(1_{1}+1)}\dots c_{\sigma(l_{1}+l_{2})},\dots,x_{(a)}=c_{\sigma(1_{1}+\dots+l_{a-1}+1)}\dots c_{\sigma(l)}$.
Now $\Delta^{[l_{1}]}(x_{(1)})$ contains the term $c_{\sigma(1)}\otimes\dots\otimes c_{\sigma(l_{1})}$,
and similarly for $\Delta^{[l_{2}]}(x_{(2)}),\dots,\Delta^{[l_{a}]}(x_{(a)})$.
So $\Delta^{[l]}(x)=(\Delta^{[l_{1}]}\otimes\dots\otimes\Delta^{[l_{a}]})\Delta^{[a]}(x)$
contains the term $c_{\sigma(1)}\otimes\dots\otimes c_{\sigma(l)}$.
(This cannot cancel out with another term in $\Delta^{[l]}(x)$ because
the coproduct structure constants are non-negative.) Hence $\eta_{x}^{c_{\sigma(1)},\dots,c_{\sigma(l)}}$
is nonzero.
\end{proof}

\subsection{Probability Estimates from Eigenfunctions \label{sub:Altrightefns}}

The focus of this section is the right eigenfunctions, since they
aid in measuring how far the chain is from being absorbed. But first,
one observation about left eigenfunctions deserves a mention.

Recall that the eigenbasis $\left\{ e(c_{1})\dots e(c_{k})|k\in\mathbb{N},\left\{ c_{1},\dots,c_{k}\right\} \mbox{ a multiset in }\calc\right\} $
from Theorem \ref{thm:diagonalisation}.A is ``length-triangular''
in the sense that $e(c_{1})\dots e(c_{k})=c_{1}\dots c_{k}+$ terms
of higher length; indeed, this allowed the conclusion that such vectors
form a basis. Since the partial-order by $\rightarrow$ refines the
partial-ordering by length, it's natural to wonder if this basis is
moreover ``triangular'' with respect to the $\rightarrow$ partial-order.
Proposition \ref{prop:unidirectionalleftefns} below shows this is
true: if 
\[
\g_{c_{1}\dots c_{k}}(y)=\mbox{coefficient of }y\mbox{ in }\eta(y)e(c_{1})\dots e(c_{k}),
\]
then $\g_{c_{1}\dots c_{k}}(c_{1}\dots c_{k})=\eta(c_{1}\dots c_{k})$,
and $\g_{c_{1}\dots c_{k}}(y)=0$ if $y$ is not accessible from $c_{1}\dots c_{k}$.
\begin{prop}
\label{prop:unidirectionalleftefns}Let $\calb$ be a free-commutative
basis of a graded connected Hopf algebra over $\mathbb{R}$. If $x\in\calb$
has factorisation into generators $x=c_{1}\dots c_{k},$ then 
\[
e(c_{1})\dots e(c_{k})=x+\sum_{\substack{x\rightarrow y\\
y\neq x
}
}\alpha_{xy}y
\]
for some constants $\alpha_{xy}$.\end{prop}
\begin{proof}
The proof of Theorem \ref{thm:diagonalisation}.A already shows that
the coefficient of $x$ in $e(c_{1})\dots e(c_{k})$ is 1, so it suffices
to show that all $y$ that appear in $e(c_{1})\dots e(c_{k})$ must
satisfy $x\rightarrow y$.

First consider the case where $k=1$. By definition of the Eulerian
idempotent, each term $y$ of $e(c_{1})$ appears in $\frac{(-1)^{a-1}}{a}m^{[a]}\bar{\Delta}^{[a]}(c_{1})$
for some $a$, and hence in $\Psi^{a}(c_{1})$, so $c_{1}\rightarrow y$
as required. Now for $k>1$,
\begin{align*}
e(c_{1})\dots e(c_{k}) & =\left(\sum_{\substack{c_{1}\rightarrow c'_{1}}
}\alpha_{c_{1}c_{1}'}c_{1}'\right)\dots\left(\sum_{\substack{c_{k}\rightarrow c'_{k}}
}\alpha_{c_{k}c_{k}'}c_{k}'\right),
\end{align*}
and Lemma \ref{lem:unidirectionalproduct} precisely concludes that
$x=c_{1}\dots c_{k}\rightarrow c_{1}'\dots c_{k}'$.
\end{proof}
And now onto right eigenfunctions. By Proposition \ref{prop:efns}.R,
these come from eigenvectors of the Hopf-power on the dual algebra
$\calhdual$. As $\calh$ is commutative, its dual $\calhdual$ is
cocommutative, so Theorem \ref{thm:diagonalisation}.B generates an
eigenbasis of $\Psi^{a}$ on $\calhdual$ from a basis of primitives
of $\calhdual$, namely by taking symmetrised products. When the state
space basis is free-commutative, a convenient choice of such a basis
of primitives is the duals of the free generating set. Then the eigenfunctions
are simply sums of coproduct structure constants, and this has the
advantage that their calculation do not explicitly involve $\calhdual$.
Their values have a combinatorial interpretation as the numbers of
ways to break $x$ into the constituent ``components'' of $y$,
divided by the number of ways to break $x$ into singletons.
\begin{thm}
\label{thm:ABdual} Let $\calh$ be a Hopf algebra over $\mathbb{R}$
with free-commutative state space basis $\calb$. For each $y\in\calbn$,
define $\f_{y}:\calbn\rightarrow\calbn$ by: 
\[
\f_{y}(x):=\frac{1}{l!Z(y)\eta(x)}\sum_{\sigma\in\Sl}\eta_{x}^{c_{\sigma(1)},\dots,c_{\sigma(l)}}=\frac{1}{l!\eta(x)}\sum_{\sigma\in\mathfrak{S}_{y}}\eta_{x}^{c_{\sigma(1)},\dots,c_{\sigma(l)}}.
\]
Here, $y=c_{1}\dots c_{l}$ is the factorisation of $y$ into generators;
$\eta_{x}^{c_{\sigma(1)},\dots,c_{\sigma(l)}}$ is the coproduct structure
constant, equal to the coefficient of $c_{\sigma(1)}\otimes\dots\otimes c_{\sigma(l)}$
in $\Delta^{[l]}(x)$; $\eta(x)$ is the rescaling function in Definition
\ref{defn:eta}; $Z(y)$ is the size of the stabiliser of the symmetric
group $\Sl$ permuting $(c_{1},\dots,c_{l});$ and $\mathfrak{S}_{y}$
is a set of coset representatives of this stabiliser in $\Sl$. Then
$\f_{y}$ is a right eigenfunction for the $a$th Hopf-power Markov
chain on $\calbn$, with eigenvalue $a^{l(y)-n}$. This right eigenfunction
has a triangular property
\begin{align*}
\f_{y}(x) & =0 & \mbox{if }x\not\rightarrow y;\\
\f_{y}(x) & >0 & \mbox{if }x\rightarrow y;\\
\f_{y}(y) & =\frac{1}{\eta(y)}.
\end{align*}

Furthermore, $\{\f_{y}|y\in\calbn\}$ is a basis of right eigenfunctions
dual to the basis of left eigenfunctions coming from Theorem \ref{thm:diagonalisation}.A
. In other words, if $\g_{c'_{1}\dots c'_{k}}(x)$ is the coefficient
of $x$ in $\eta(x)e(c'_{1})\dots e(c'_{k})$, then 
\[
\sum_{x\in\calbn}\g_{c'_{1}\dots c'_{k}}(x)\f_{y}(x)=\begin{cases}
1 & \quad\mbox{if }y=c'_{1}\dots c'_{k}\mbox{ (i.e. }\left\{ c_{1},\dots,c_{l}\right\} =\left\{ c'_{1},\dots,c'_{k}\right\} \mbox{ as multisets);}\\
0 & \quad\mbox{otherwise.}
\end{cases}
\]

\end{thm}
The proof is delayed until the end of this section. See Equation \ref{eq:easyrightefns}
below for some special cases of this formula. Note that, if $\calh$
is in addition cocommutative, then it is unnecessary to symmetrise
- just set $\f_{y}(x):=\frac{1}{Z(y)\eta(x)}\eta_{x}^{c_{1},\dots,c_{l}}$. 
\begin{example}
\label{ex:unidirectional-rightefn}Recall from the opening of this
chapter the Hopf algebra $\barcalg$ of isomorphism classes of graphs,
whose associated Markov chain models edge-removal. $\barcalg$ is
cocommutative, so the simpler formula $\f_{y}(x):=\frac{1}{Z(y)\eta(x)}\eta_{x}^{c_{1},\dots,c_{l}}$
applies. As remarked in the opening of this chapter, the rescaling
function $\eta(x)=(\deg x)!$ for all $x$, so $\f_{y}(x)=\frac{1}{Z(y)(\deg x)!}\eta_{x}^{c_{1}\dots c_{l}}$.
This example will calculate $\f_{y}(x)$ in the case where $x$ is
``two triangles with one common vertex'' as depicted in Figure \ref{fig:twotriangle},
and $y$ is the disjoint union of a path of length 3 and an edge.
So $c_{1}=P_{3}$, the path of length 3, and $c_{2}=P_{2}$, a single
edge (or vice versa, the order does not matter). Since these are distinct,
$Z(y)=1$. There are four ways to partition the vertex set of $x$
into a triple and a pair so that the respective induced subgraphs
are $P_{3}$ and $P_{2}$. (The triples for these four ways are, respectively:
the top three vertices; the top left, top middle and bottom right;
the top right, top middle and bottom left; and the bottom two vertices
and the top middle.) Thus $\f_{y}(x)=\frac{1}{5!}4$.

\begin{figure}
\centering{}\includegraphics{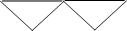} \caption{\label{fig:twotriangle}The ``two triangles with one common vertex''
graph}
\end{figure}

\end{example}
The triangular property of the right eigenfunctions $\f_{y}$ makes
them ideal to use in Proposition \ref{prop:expectationrightefns}.iii
to bound the probability that the chain can still reach $y$. The
result is recorded in Proposition \ref{prop:probboundsrightefns}
below, along with a few variants. (Bounds analogous to those in Part
i hold for any unidirectional Markov chain, since these right eigenfunctions
come from applying Perron-Frobenius to the \emph{minors} of the transition
matrix - that is, the submatrix with the rows and columns corresponding
to states which can reach $y$. However, Part ii requires $x\not\rightarrow y$
for every pair of distinct states $x,y$ whose eigenfunctions $\f_{x},\f_{y}$
have the same eigenvalue.) Remark 1 after \cite[Cor 4.10]{hopfpowerchains}
shows that, for the rock-breaking chain of the present Section \ref{sec:Rock-breaking},
the bound for $y=(2,1,\dots,1)$ is an asymptotic equality. 
\begin{prop}
\label{prop:probboundsrightefns}Let $\{X_{m}\}$ be the $a$th Hopf-power
Markov chain on a free-commutative state space basis $\calbn$. Fix
a state $y\in\calbn$ and let $\f_{y}$ be its corresponding right
eigenfunction as defined in Theorem \ref{thm:ABdual}. Then the probability
that the chain can still reach $y$ after $m$ steps has the following
upper bounds:
\begin{enumerate}
\item 
\[
P\{X_{m}\rightarrow y|X_{0}=x_{0}\}\leq\frac{a^{(l(y)-n)m}\f_{y}(x_{0})}{\min_{x\in\calbn,x\rightarrow y}\f_{y}(x)}=\frac{a^{(l(y)-n)m}\frac{1}{\eta(x_{0})}\sum_{\sigma\in\Sl}\eta_{x_{0}}^{c_{\sigma(1)},\dots,c_{\sigma(l)}}}{\min_{x\in\calbn,x\rightarrow y}\frac{1}{\eta(x)}\sum_{\sigma\in\Sl}\eta_{x}^{c_{\sigma(1)},\dots,c_{\sigma(l)}}}.
\]

\item 
\[
P\{X_{m}\rightarrow y|X_{0}=x_{0}\}=\eta(y)a^{(l(y)-n)m}\f_{y}(x_{0})(1+o(1))\mbox{ as }m\rightarrow\infty.
\]

\item For any starting distribution,
\[
P\{X_{m}\rightarrow y\}\leq\frac{a^{(l(y)-n)m}}{\min_{x\in\calbn,x\rightarrow y}\f_{y}(x)Z(y)}\frac{1}{\eta(y)}\binom{n}{\deg c_{1}\dots\deg c_{l}}.
\]

\end{enumerate}

In each case, $y=c_{1}\dots c_{l}$ is the factorisation of $y$ into
generators.

\end{prop}
Be careful that Parts ii and iii depend upon the scaling of $\f_{y}$
- they need adjustment if used with a right eigenfunction which is
a scalar multiple of $\f_{y}$. See Proposition \ref{prop:probboundseasyrightefns}.
(This problem does not occur for the first bound as that involves
only a ratio of eigenfunction values.)
\begin{proof}
Part i is a straightforward application of Proposition \ref{prop:expectationrightefns}.iii,
a fact of general Markov chains.

To see Part ii, first note that, from the triangularity properties
of $\f_{y}$, the difference of functions $\mathbf{1}_{\{\rightarrow y\}}-\eta(y)\f_{y}$
is non-zero only on $S'_{y}:=\{y'\in\calbn|y'\rightarrow y,y'\neq y\}$.
(Here, $\mathbf{1}_{\{\rightarrow y\}}$ is the indicator function
of being able to reach $y$.) Also by triangularity, such functions
that are non-zero only on $S'_{y}$ are spanned by the eigenfunctions
$\{\f_{y'}|y'\in S'_{y}\}$. Hence the expansion of $\mathbf{1}_{\{\rightarrow y\}}$
into right eigenfunctions has the form 
\[
\mathbf{1}_{\{\rightarrow y\}}=\eta(y)\f_{y}+\sum_{y'\in S'_{y}}\alpha_{y'}\f_{y'}
\]
for some constants $\alpha_{y'}$. By linearity of expectations, as
in Proposition \ref{prop:expectationrightefns}, this implies 
\begin{align*}
P\{X_{m}\rightarrow y|X_{0}=x_{0}\} & =\eta(y)a^{(l(y)-n)m}\f_{y}(x_{0})+\sum_{y'\in S'_{y}}a^{(l(y')-n)m}\alpha_{y'}\f_{y'}(x_{0})\\
 & =\eta(y)a^{(l(y)-n)m}\f_{y}(x_{0})\left(1+\sum_{y'\in S'_{y}}a^{(l(y')-l(y))m}\frac{\alpha_{y'}\f_{y'}(x_{0})}{\eta(y)\f_{y}(x_{0})}\right).
\end{align*}
Now use Proposition \ref{prop:unidirectional}: all $y'\in S'_{y}$
satisfies $y'\rightarrow y$ and $y'\neq y$, which forces $l(y')\leq l(y)$.
So the ratios of eigenvalues $a^{(l(y')-l(y))}$ is less than 1, and
hence the sum tends to zero as $m\rightarrow\infty$.

Now turn to Part iii, the bound independent of the starting state.
It suffices to show that 
\[
Z(y)\f_{y}(x_{0})=\frac{1}{l!\eta(x_{0})}\sum_{\sigma\in\Sl}\eta_{x_{0}}^{c_{\sigma(1)},\dots,c_{\sigma(l)}}\leq\frac{1}{\eta(y)}\binom{n}{\deg c_{1}\dots\deg c_{l}}
\]
for all states $x_{0}\in\calb_{n}$. For any composition $d_{1}+\dots+d_{l}=n$,
coassociativity says that $\Delta^{[n]}=(\Delta^{[d_{1}]}\otimes\dots\otimes\Delta^{[d_{l}]})\Delta^{[l]}$,
so 
\[
\eta(x_{0})=\sum_{\deg(c_{i}')=d_{i}}\eta(c_{1}')\dots\eta(c_{l}')\eta_{x_{0}}^{c_{1}',\dots,c_{l}'}.
\]
All summands on the right hand side are non-negative, so choosing
$d_{i}=\deg(c_{\sigma(i)})$ shows that, for each $\sigma\in\Sl$:
\[
\eta(x_{0})\geq\eta_{x_{0}}^{c_{\sigma(1)},\dots,c_{\sigma(l)}}\eta(c_{1})\dots\eta(c_{l})=\eta_{x_{0}}^{c_{\sigma(1)},\dots,c_{\sigma(l)}}\frac{\eta(y)}{\binom{n}{\deg c_{1}\dots\deg c_{l}}},
\]
using Lemma \ref{lem:etaproduct} for the last equality.
\end{proof}
In many common situations, including all examples in this thesis,
all coproduct structure constants are integral. Then 
\[
\f_{y}(x)=\frac{1}{l!\eta(x)}\sum_{\sigma\in\mathfrak{S}_{y}}\eta_{x}^{c_{\sigma(1)},\dots,c_{\sigma(l)}}\geq\frac{1}{l(y)!\eta(x)},
\]
so replacing $(\min_{x\in\calbn,x\rightarrow y}\f_{y}(x))^{-1}$ with
$l(y)!\max_{x\in\calbn,x\rightarrow y}\eta(x)$ in either inequality
of Proposition \ref{prop:probboundsrightefns} gives a looser but
computationally easier bound.
\begin{example}
\label{ex:probboundsrightefns-graphs}Continue from Example \ref{ex:unidirectional-rightefn}.
Let $y=P_{3}P_{2}$, the disjoint union of a path of length 3 and
an edge, and $x$ be ``two triangles with one common vertex'' as
in Figure \ref{fig:twotriangle}. Example \ref{ex:unidirectional-rightefn}
calculated $\f_{y}(x)$ to be $\frac{4}{5!}$. Then, using the looser
bound in the last paragraph, the probability that, after $m$ steps
of the Hopf-square Markov chain starting at $x$, the graph still
contains three vertices on which the induced subgraph is a path, and
the other two vertices are still connected, is at most $2^{(2-5)m}\frac{4}{5!}2!5!=2^{1-3m}4$.
\end{example}
The previous example of bounding the probability of having three vertices
on which the induced subgraph is a path, and the other two vertices
connected feels a little contrived. It is more natural to ask for
the probability that at least three vertices are still in the same
connected component. This equates to being at a state which can reach
either $P_{3}\bullet^{2}$ or $K_{3}\bullet^{2}$, since the only
connected graphs on three vertices are $P_{3}$, the path of length
3, and $K_{3}$, the complete graph on 3 vertices. Similarly, being
at a state which can reach $P_{2}\bullet^{3}$, the graph with one
edge and three isolated vertices, is synonymous with not yet being
absorbed. So the most important probabilities of the form ``in a
state which can still reach $y$'' are when $y$ has factorisation
$y=c\bullet\dots\bullet$ for some generator $c\neq\bullet$. In this
case, it will be convenient to scale this eigenvector by $\frac{\deg y!}{\deg c!}$.
So abuse notation and write $\f_{c}$ for the eigenvector $\frac{\deg y!}{\deg c!}\f_{y}$
(note that the two notations agree when $y=c$), and extend it to
degrees lower than $\deg(c)$ by declaring it to be the zero function
there. In other words, for all $x\in\calb$:
\begin{align}
\f_{c}(x) & :=\frac{\deg x!}{\deg c!}\f_{c\bullet^{\deg(x)-\deg(c)}}(x)\label{eq:easyrightefns}\\
 & \phantom{:}=\frac{\binom{\deg x}{\deg c}}{\eta(x)(\deg x-\deg c+1)}\left(\eta_{x}^{c,\bullet,\dots,\bullet}+\dots+\eta_{x}^{\bullet,\dots,\bullet,c}\right)\nonumber \\
 & \phantom{:}=\frac{\binom{\deg x}{\deg c}}{\eta(x)(\deg x-\deg c+1)}\sum_{w,z}\eta(w)\eta(z)\eta_{x}^{wcz} & \mbox{if }\deg(x)\geq\deg(c);\nonumber \\
\f_{c}(x) & :=0. & \mbox{if }\deg(x)<\deg(c).\nonumber 
\end{align}
Here, the second equality is by definition of $\f_{y}$, and the third
equality is a consequence of the following coassociativity equation
(which holds for any choice of $d_{i}$ summing to $\deg(x)$)
\[
\sum_{w_{i}\in\calb_{d_{i}}}\eta_{w_{1}}^{z_{1},\dots,z_{i}}\eta_{w_{2}}^{z_{i+1},\dots,z_{j}}\eta_{w_{3}}^{z_{j+1},\dots,z_{a}}\eta_{x}^{w_{1},w_{2},w_{3}}=\eta_{x}^{z_{1},\dots,z_{a}},
\]
in the cases where all but one $z_{i}$ are $\bullet$. The eigenvalue
of $\f_{c}$ is $a^{-\deg c+1}$. These are usually the easiest right
eigenfunctions to calculate, as they behave well with ``recursive
structures'' such as the trees of Section \ref{sec:Tree-Pruning}.
The following Proposition is one general instance of this principle;
it reduces the calculation of $\f_{c}$ to its value on generators.
\begin{prop}
\label{prop:rightefnsofproducts}The right eigenfunction $\f_{c}$
is additive in the sense that 
\[
\f_{c}(xx')=\f_{c}(x)+\f_{c}(x').
\]
\end{prop}
\begin{proof}
This argument is much like that of Lemma \ref{lem:etaproduct} regarding
$\eta(xx')$. Since $\Delta(xx')=\Delta(x)\Delta(x')$, a term in
$\eta_{xx'}^{c,\bullet,\dots,\bullet}+\dots+\eta_{xx'}^{\bullet,\dots,\bullet,c}$
arises in one of two ways: from a term in $\eta_{x}^{c,\bullet,\dots,\bullet}+\dots+\eta_{x}^{\bullet,\dots,\bullet,c}$
and a term in $\eta(x')$, or from a term in $\eta_{x'}^{c,\bullet,\dots,\bullet}+\dots+\eta_{x'}^{\bullet,\dots,\bullet,c}$
and a term in $\eta(x)$. The first way involves a choice of $\deg x'$
tensor-factors amongst $\deg xx'-\deg c+1$ in which to place the
term from $\eta(x')$, and similarly a choice of $\deg x$ positions
for the second way. Hence
\begin{align*}
 & \left(\eta_{x}^{c,\bullet,\dots,\bullet}+\dots+\eta_{x}^{\bullet,\dots,\bullet,c}\right)\\
= & \binom{\deg xx'-\deg c+1}{\deg x'}(\eta_{x}^{c,\bullet,\dots,\bullet}+\dots+\eta_{x}^{\bullet,\dots,\bullet,c})\eta(x')\\
 & \quad+\binom{\deg xx'-\deg c+1}{\deg x}(\eta_{x'}^{c,\bullet,\dots,\bullet}+\dots+\eta_{x'}^{\bullet,\dots,\bullet,c})\eta(x)\\
= & \eta(x)\eta(x')(\deg xx'-\deg c+1)!\deg x'!\deg x!\deg c!(\f_{c}(x)+\f_{c}(x')).
\end{align*}
Combining this with the formula for $\eta(xx')$ in Lemma \ref{lem:etaproduct}
gives 
\begin{align*}
\f_{c}(xx') & =\frac{\binom{\deg xx'}{\deg c}}{\eta(xx')(\deg xx'-\deg c+1)}\left(\eta_{x}^{c,\bullet,\dots,\bullet}+\dots+\eta_{x}^{\bullet,\dots,\bullet,c}\right)\\
 & =\left(\binom{\deg xx'}{\deg x}\eta(x)\eta(x')\right)^{-1}\frac{\binom{\deg xx'}{\deg c}}{(\deg xx'-\deg c+1)}\left(\eta_{x}^{c,\bullet,\dots,\bullet}+\dots+\eta_{x}^{\bullet,\dots,\bullet,c}\right)\\
 & =\f_{c}(x)+\f_{c}(x').
\end{align*}

\end{proof}
Because the $\f_{c}$ are non-trivial multiples of the $\f_{y}$ when
$y\neq c$, the bound in Proposition \ref{prop:rightefnsofproducts}.iii,
which is independent of the starting state, does not apply verbatim.
Here is the modified statement (which uses the fact that $Z(c\bullet^{n-\deg c})=(n-\deg c)!$,
and $\eta(y)=\binom{n}{\deg c}\eta(c)$ as per Lemma \ref{lem:etaproduct}). 
\begin{prop}
\label{prop:probboundseasyrightefns} Let $\{X_{m}\}$ be the $a$th
Hopf-power Markov chain on a free-commutative state space basis $\calbn$.
Let $c$ be a generator of the underlying Hopf algebra $\calh$, and
let $\f_{c}$ be its corresponding right eigenfunction as defined
in Equation \ref{eq:easyrightefns}. Then, for any starting distribution,
the probability that the chain can still reach $c\bullet^{n-\deg c}$
after $m$ steps has the following upper bound:
\[
P\{X_{m}\rightarrow c\bullet^{n-\deg c}\}\leq\frac{a^{(l(y)-n)m}}{\min_{x\in\calbn,x\rightarrow y}\f_{c}(x)(n-\deg c+1)}\frac{1}{\eta(c)}\binom{n}{\deg c}.
\]
\qed
\end{prop}
In the case of (isomorphism classes of) graphs, $\eta_{x}^{c,\bullet,\dots,\bullet}=\dots=\eta_{x}^{\bullet,\dots,\bullet,c}$
is the number of induced subgraphs of $x$ isomorphic to $c$, multiplied
by the number of orders in which to choose the singletons, which is
$(\deg x-\deg c)!$. Recall that $\eta(x)=(\deg x)!$. So 
\[
\f_{c}(x)=\frac{1}{\deg c!}|\{\mbox{induced subgraphs of }x\mbox{ isomorphic to }c\}|.
\]
The analogous statement holds for other species-with-restrictions.
Note that summing these over all connected graphs $c$ on $j$ vertices
gives another right eigenfunction, with eigenvalue $a^{-j+1}$:
\[
\f_{j}(x):=\frac{1}{j!}|\{\mbox{connected induced subgraphs of }G\mbox{ with }j\mbox{ vertices}\}|.
\]
Minor variations on Propositions \ref{prop:expectationrightefns}
and \ref{prop:probboundsrightefns}.i with the $\f_{c}$s and $\f_{j}$s
then imply the following facts. They have an alternative, elementary
derivation: the chance that any one particular connected subgraph
$c$ survives one step of the edge-removal chain is $a^{-\deg c+1}$,
since all vertices of $c$ must receive the same colour. Since expectation
is linear, summing these over all subgraphs of interest gives the
expected number of these subgraphs that survive.
\begin{prop}
\label{prop:probboundsrightefns-graphs}Let $\{X_{m}\}$ be the $a$th
Hopf-power Markov chain on graphs describing edge-removal. Let $c$
be any connected graph. Then
\begin{align*}
 & E\{|\{\mbox{induced subgraphs of }X_{m}\mbox{ isomorphic to }c\}||X_{0}=G\}\\
= & a^{(-\deg c+1)m}|\{\mbox{induced subgraphs of }G\mbox{ isomorphic to }c\}|;
\end{align*}
\begin{align*}
 & P\{X_{m}\mbox{ has a connected component with }\geq j\mbox{ vertices}|X_{0}=G\}\\
\leq & E\{|\{\mbox{connected components of }X_{m}\mbox{ with \ensuremath{\geq}}j\mbox{ vertices}\}||X_{0}=G\}\\
\leq & E\{|\{\mbox{connected induced subgraphs of }X_{m}\mbox{ with }j\mbox{ vertices}\}||X_{0}=G\}\\
= & a^{(-j+1)m}|\{\mbox{connected induced subgraphs of }G\mbox{ with }j\mbox{ vertices}\}|.
\end{align*}
In particular, the case $j=2$ gives 
\begin{align*}
 & P\{X_{m}\mbox{ is not absorbed}|X_{0}=G\}\\
\leq & E\{|\{\mbox{edges in }X_{m}\}||X_{0}=G\}=a^{-m}|\{\mbox{edges in }G\}|.
\end{align*}
\qed\end{prop}
\begin{example}
\label{ex:probboundsrightefns2-graphs}Take $x_{0}$ to be the ``two
triangles with one common vertex'' graph of Figure \ref{fig:twotriangle}
above. It has four induced subgraphs that are paths of length 3 (Example
\ref{ex:unidirectional-rightefn} identified these), and the two obvious
induced subgraphs that are triangles. So the probability of having
a connected component of size at least 3 after $m$ steps of the Hopf-square
Markov chain is less than $2^{-2m}6$, which is also the expected
number of triples that remain connected.\end{example}
\begin{proof}[Proof of Theorem \ref{thm:ABdual}, right eigenfunctions in terms
of coproduct structure constants ]
By definition of the coproduct structure constant, and of the product
structure on the dual Hopf algebra,
\begin{align*}
\f_{y}(x): & =\frac{1}{l!Z(y)\eta(x)}\sum_{\sigma\in\Sl}\eta_{x}^{c_{\sigma(1)},\dots,c_{\sigma(l)}}\\
 & =\frac{1}{l!Z(y)\eta(x)}\sum_{\sigma\in\Sl}c_{\sigma(1)}^{*}\otimes\dots\otimes c_{\sigma(l)}^{*}(\Delta^{[l]}x)\\
 & =\frac{1}{l!Z(y)\eta(x)}\sum_{\sigma\in\Sl}c_{\sigma(1)}^{*}\dots c_{\sigma(l)}^{*}(x).
\end{align*}
So, thanks to Proposition \ref{prop:efns}.R, $\f_{y}$ being a right
eigenfunction of the Hopf-power Markov chain with eigenvalue $a^{l-n}$
equates to $f_{y}:=\frac{1}{l!Z(y)}\sum_{\sigma\in\Sl}c_{\sigma(1)}^{*}\dots c_{\sigma(l)}^{*}$
being an eigenvector of $\Psi^{a}$ on $\calhdual$ with eigenvalue
$a^{l}$. This will follow from the Symmetrisiation Lemma (Theorem
\ref{thm:symlemma}) once it is clear that the $c_{i}^{*}$ are primitive.

To establish that each $c^{*}$ is primitive, proceed by contradiction.
Take a term $w^{*}\otimes z^{*}$ in $\bard(c^{*})=\Delta(c^{*})-1\otimes c^{*}-c^{*}\otimes1$,
with $w,z\in\calb$. Then $\Delta(c^{*})(w\otimes z)$ is non-zero.
Since comultiplication in $\mathcal{H}^{*}$ is dual to multiplication
in $\mathcal{H}$, $\Delta(c^{*})(w\otimes z)=c^{*}(wz)$. Now $\calb$
is free-commutative so $wz\in\calb$, thus $c^{*}(wz)$ is only non-zero
if $wz=c$ . But, by the counit axiom for graded connected Hopf algebras,
$\bard(c^{*})\in\bigoplus_{j=1}^{\deg c-1}\calhdual_{j}\otimes\calhdual_{\deg c-j}$,
so both $w$ and $z$ have strictly positive degree. So $c=wz$ contradicts
the assumption that $c$ is a generator, and hence no term $w^{*}\otimes z^{*}$
can exist in $\bard(c^{*})$, i.e. $\bard(c^{*})=0$.

To see the triangularity properties, note that $\f_{y}(x)$ is non-zero
only if $\eta_{x}^{c_{\sigma(1)},\dots,c_{\sigma(l)}}$ is non-zero
for some $\sigma\in\Sl$, which forces $x\rightarrow c_{\sigma(1)}\dots c_{\sigma(l)}=y$.
Conversely, if $x\rightarrow y$, then by Lemma \ref{lem:etatogenerators}
$\eta_{x}^{c_{\sigma(1)},\dots,c_{\sigma(l)}}$ is non-zero for some
$\sigma$, and all other coproduct structure constants are non-negative,
so $\f_{y}(x)>0$. To show that $\f_{y}(y)=\frac{1}{\eta(y)}$, it
suffices to show that $\sum_{\sigma\in\Sl}\eta_{y}^{c_{\sigma(1)},\dots,c_{\sigma(l)}}=Z(y)$
for each $\sigma\in\Sl$. Rewrite the left hand side using the dual
Hopf algebra: 
\begin{align*}
\sum_{\sigma\in\Sl}\eta_{y}^{c_{\sigma(1)},\dots,c_{\sigma(l)}} & =\left(c_{\sigma(1)}^{*}\otimes\dots\otimes c_{\sigma(l)}^{*}\right)\Delta^{[l]}(y)\\
 & =\left(c_{\sigma(1)}^{*}\otimes\dots\otimes c_{\sigma(l)}^{*}\right)\Delta^{[l]}(c_{1}\dots c_{l})\\
 & =\left(c_{\sigma(1)}^{*}\dots c_{\sigma(l)}^{*}\right)(c_{1}\dots c_{l})\\
 & =\left(\Delta^{[l]}\left(c_{\sigma(1)}^{*}\dots c_{\sigma(l)}^{*}\right)\right)(c_{1}\otimes\dots\otimes c_{l})\\
 & =\left(\Delta^{[l]}(c_{\sigma(1)}^{*})\dots\Delta^{[l]}(c_{\sigma(l)}^{*})\right)(c_{1}\otimes\dots\otimes c_{l}).
\end{align*}
As each $c_{\sigma(i)}^{*}$ is primitive, 
\[
\left(\Delta^{[l]}(c_{\sigma(1)}^{*})\dots\Delta^{[l]}(c_{\sigma(l)}^{*})\right)=\sum_{A_{1}\amalg\dots\amalg A_{l}=\{1,2,\dots,l\}}\sum_{\sigma\in S_{l}}\left(\prod_{i\in A_{1}}c_{\sigma(i)}^{*}\right)\otimes\dots\otimes\left(\prod_{i\in A_{l}}c_{\sigma(i)}^{*}\right).
\]
Hence its evaluation on $c_{1}\otimes\dots\otimes c_{l}$ is 
\begin{align*}
 & \left|\left\{ (A_{1},\dots,A_{l})|A_{1}\amalg\dots\amalg A_{l}=\{1,2,\dots,l\},\prod_{i\in A_{1}}c_{\sigma(i)}^{*}=c_{1}^{*},\dots,\prod_{i\in A_{l}}c_{\sigma(i)}^{*}=c_{l}^{*}\right\} \right|\\
= & |\{\tau\in\Sl|c_{\tau\sigma(i)}=c_{i}\}|\\
= & |\{\tau\in\Sl|c_{\tau(i)}=c_{i}\}|=Z(y).
\end{align*}

The last claim of Theorem \ref{thm:ABdual} is that $\sum_{x\in\calbn}\g_{c'_{1}\dots c'_{k}}(x)\f_{y}(x)=1$
when $y=c'_{1}\dots c'_{k}$ and is 0 otherwise; it follows from this
duality statement that $\f_{y}$ is a basis. First take the case where
$l(y)\neq k$; then $\f_{y}$ and $\g_{c'_{1}\dots c'_{k}}$ are eigenvectors
of dual maps with different eigenvalues, so the required sum must
be zero, by the following simple linear algebra argument (recall that
$\hatk$ is the transition matrix):
\[
a^{k}\sum_{x\in\calbn}\g_{c'_{1}\dots c'_{k}}(x)\f_{y}(x)=\sum_{x,z\in\calbn}\g_{c'_{1}\dots c'_{k}}(z)\hatk(z,x)\f_{y}(x)=a^{l}\sum_{x\in\calbn}\g_{c'_{1}\dots c'_{k}}(x)\f_{y}(x).
\]

So take $l(y)=k$. Recall from Proposition \ref{prop:unidirectionalleftefns}
that $\g_{c'_{1}\dots c'_{k}}(x)$ is non-zero only if $c'_{1}\dots c'_{k}\rightarrow x$,
and earlier in this proof showed that $\f_{y}(x)$ is non-zero only
if $x\rightarrow y$. So the only terms $x$ which contribute to $\sum_{x\in\calbn}\g_{c'_{1}\dots c'_{k}}(x)\f_{y}(x)$
must satisfy $c'_{1}\dots c'_{k}\rightarrow x\rightarrow y$. By Proposition
\ref{prop:unidirectional}, this implies $k=l(c'_{1}\dots c'_{k})\geq l(x)\geq l(y)$
with equality if and only if $c'_{1}\dots c'_{k}=x=y$. As the current
assumption is that $k=l(y)$, no $x$'s contribute to the sum unless
$c'_{1}\dots c'_{k}=y$. In this case, the sum is $\g_{y}(y)\f_{y}(y)=\eta(y)\frac{1}{\eta(y)}=1$.
\end{proof}

\subsection{Probability Estimates from Quasisymmetric Functions\label{sub:Absorption}}

The previous section provided upper bounds for the probabilities that
a Hopf-power Markov chain is ``far from absorbed''. This section
connects the complementary probabilities, of being ``close to absorbed'',
to the following result of Aguiar, Bergeron and Sottile, that the
algebra of quasisymmetric functions (Example \ref{ex:qsym}) is terminal
in the category of combinatorial Hopf algebras with a character. (For
this section, elements of $QSym$ will be in the variables $t_{1},t_{2},\dots$
to distinguish from the states $x$ of the Markov chain.) 
\begin{thm}
\label{thm:qsymisterminal}\cite[Th. 4.1]{abs} Let $\calh$ be a
graded, connected Hopf algebra over $\mathbb{R}$, and let $\zeta:\calh\rightarrow\mathbb{R}$
be a multiplicative linear functional (i.e. $\zeta(wz)=\zeta(w)\zeta(z)$).
Then there is a unique Hopf-morphism $\chi^{\zeta}:\calh\rightarrow QSym$
such that, for each $x\in\calh$, the quasisymmetric function $\chi^{\zeta}(x)$
evaluates to $\zeta(x)$ when $t_{1}=1$ and $t_{2}=t_{3}=\dots=0$
. To explicitly construct $\chi^{\zeta}$, set the coefficient of
the monomial quasisymmetric function $M_{I}$ in $\chi^{\zeta}(x)$
to be the image of $x$ under the composite 
\[
\calh\xrightarrow{\Delta^{[l(I)]}}\calh^{\otimes l(I)}\xrightarrow{\pi_{i_{1}}\otimes\dots\otimes\pi_{i_{l(I)}}}\calh_{i_{1}}\otimes\dots\otimes\calh_{i_{l(I)}}\xrightarrow{\zeta^{\otimes l(I)}}\mathbb{R}.
\]
where, in the middle map, $\pi_{i_{j}}$ denotes the projection to
the subspace of degree $i_{j}$. 
\end{thm}
One motivating example from the authors \cite[Ex. 4.5]{abs} is $\calh=\barcalg$,
the algebra of isomorphism classes of graphs. For a graph $G$, set
$\zeta(G)$ to be 1 if $G$ has no edges, and 0 otherwise. Then $\chi^{\zeta}(G)$
is Stanley's \emph{chromatic symmetric function} \cite{chromaticsymfn}:
the coefficient of $x_{1}^{r_{1}}\dots x_{n}^{r_{n}}$ in $\chi^{\zeta}(G)$
counts the proper colourings of $G$ where $r_{i}$ vertices receive
colour $i$. (A \emph{proper colouring} of $G$ is an assignment of
colours to the vertices of $G$ so that no two vertices on an edge
have the same colour.) Note that $\chi^{\zeta}(G)$ evaluated at $t_{1}=\dots=t_{a}=1,\ t_{a+1}=t_{a+2}=\dots=0$
is then precisely the number of proper colourings of $G$ in $a$
colours (not necessarily using all of them). Equivalently, $\chi^{\zeta}(G)$
evaluated at $t_{1}=\dots=t_{a}=\frac{1}{a},\ t_{a+1}=t_{a+2}=\dots=0$
is the probability that uniformly and independently choosing one of
$a$ colours for each vertex of $G$ produces a proper colouring.
According to the description of the Hopf-power Markov chain on graphs
(Example \ref{ex:graph}), this is precisely the probability of absorption
after a single step. The same is true of other Hopf-power Markov chains
on free-commutative bases. Note that it is enough to consider absorption
in one step because, by the power rule, $m$ steps of the $a$th Hopf-power
Markov chain on a commutative Hopf algebra is equivalent to one step
of the $a^{m}$th Hopf-power Markov chain.

In the results below, $[f]_{1/a}$ denotes evaluating the quasisymmetric
function $f$ at $t_{1}=\dots=t_{a}=\frac{1}{a},\ t_{a+1}=t_{a+2}=\dots=0$.
\begin{prop}[Probability of absorption]
\label{prop:probabsorption}Let $\calb$ be a free-commutative state
space basis of $\calh$, and $\zeta:\calh\rightarrow\mathbb{R}$ be
the indicator function of absorption, extended linearly. (In other
words, $\zeta(x)=1$ if $x$ is an absorbing state, and 0 for other
states $x$.) Then the probability that the $a$th Hopf-power Markov
chain on $\calbn$ is absorbed in a single step starting from $x_{0}$
is 
\[
\sum_{y:l(y)=n}\hatkan(x_{0},y)=\left[\frac{n!}{\eta(x_{0})}\chi^{\zeta}(x_{0})\right]_{1/a}.
\]

\end{prop}
It is natural to ask whether $\chi^{\zeta}$ will analogously give
the probability of landing in some subset $\caly$ of states if $\zeta$
is the indicator function on $\caly$. The obstacle is that such a
$\zeta$ might not be multiplicative. The first theorem below gives
one class of $\caly$s for which $\zeta$ is clearly multiplicative,
and the second indicates the best one can hope for in a completely
general setting, when $\calb$ might not even be free-commutative. 
\begin{thm}
\label{thm:probabsorption2}Let $\calb$ be a free-commutative state
space basis of $\calh$, and $\calc'$ a subset of the free generators.
Let $\zeta:\calh\rightarrow\mathbb{R}$ be the multiplicative linear
functional with $\zeta(c)=\frac{\eta(c)}{(\deg c)!}$ if $c\in\calc'$,
and $\zeta(c)=0$ for other free generators $c$. Then, for the $a$th
Hopf-power Markov chain $\{X_{m}\}$ on $\calbn$,
\[
P\{\mbox{all factors of }X_{1}\mbox{ are in }\calc'|X_{0}=x_{0}\}=\left[\frac{(\deg x_{0})!}{\eta(x_{0})}\chi^{\zeta}(x_{0})\right]_{1/a}.
\]

\end{thm}

\begin{thm}
\label{thm:probabsorption3}Let $\calb$ be any state space basis
of $\calh$, and $\{X_{m}\}$ the $a$th Hopf-power Markov chain on
$\calbn$. Let $\caly\subseteq\calbn$, and $\zeta:\calh\rightarrow\mathbb{R}$
be a multiplicative linear functional satisfying $\zeta(y)>0$ for
$y\in\calbn\cap\caly$, $\zeta(y)=0$ for $y\in\calbn\backslash\caly$.
Then
\[
\left(\min_{y\in\caly}\frac{\eta(y)}{\zeta(y)}\right)\left[\frac{1}{\eta(x_{0})}\chi^{\zeta}(x_{0})\right]_{1/a}\leq P\{X_{1}\in\caly|X_{0}=x_{0}\}\leq\left(\max_{y\in\caly}\frac{\eta(y)}{\zeta(y)}\right)\left[\frac{1}{\eta(x_{0})}\chi^{\zeta}(x_{0})\right]_{1/a}.
\]
\end{thm}
\begin{example}
\label{ex:probabsorption-graphs}Work in $\barcalg$, the algebra
of isomorphism classes of graphs, where $\eta(x)=\deg(x)!$. Let $\calc'=\calb_{1}\amalg\dots\amalg\calb_{j-1}$.
Then the function $\zeta$ of Theorem \ref{thm:probabsorption2} takes
value 1 on graphs each of whose connected components have fewer than
$j$ vertices, and value 0 on graphs with a connected component of
at least $j$ vertices. Then $\left[\chi^{\zeta}(G)\right]_{1/a}$
yields the probability that, after one step of the edge-removal chain
started at $G$, all connected components have size at most $j-1$.\end{example}
\begin{proof}[Proofs of Proposition \ref{prop:probabsorption}, Theorems \ref{thm:probabsorption2}
and \ref{thm:probabsorption3}]
First rewrite the definition of $\chi^{\zeta}$ in terms of coproduct
structure constants: 
\begin{align*}
\chi^{\zeta}(x_{0}) & =\sum_{l=1}^{n}\sum_{\deg(z_{i})>0}\eta_{x_{0}}^{z_{1},\dots,z_{l}}\zeta(z_{1})\dots\zeta(z_{l})M_{(\deg z_{1},\dots,\deg z_{l})}\\
 & =\sum_{l=1}^{\infty}\sum_{\substack{z_{1},\dots,z_{l}\\
\deg(z_{l})>0
}
}\eta_{x_{0}}^{z_{1},\dots,z_{l}}\zeta(z_{1}\dots z_{l})t_{1}^{\deg z_{1}}\dots t_{l}^{\deg z_{l}}.
\end{align*}
So, when $t_{1}=\dots=t_{a}=\frac{1}{a},\ t_{a+1}=t_{a+2}=\dots=0$,
the quasisymmetric function $\chi^{\zeta}(x_{0})$ evaluates to 
\[
\sum_{z_{1},\dots,z_{a}}\eta_{x_{0}}^{z_{1},\dots,z_{a}}\zeta(z_{1}\dots z_{a})a^{-n}=\eta(x_{0})\sum_{y\in\calbn}\hatkan(x_{0},y)\frac{\zeta(y)}{\eta(y)}.
\]

Now, in the setup of Theorem \ref{thm:probabsorption3}, 
\begin{align*}
P\{X_{1}\in\caly|X_{0}=x_{0}\} & =\sum_{y\in\caly}\hatkan(x_{0},y)\\
 & \leq\left(\max_{y\in\caly}\frac{\eta(y)}{\zeta(y)}\right)\frac{1}{\eta(x_{0})}\left(\eta(x_{0})\sum_{y\in\caly}\hatkan(x_{0},y)\frac{\zeta(y)}{\eta(y)}\right)\\
 & =\left(\max_{y\in\caly}\frac{\eta(y)}{\zeta(y)}\right)\left[\frac{1}{\eta(x_{0})}\chi^{\zeta}(x_{0})\right]_{1/a},
\end{align*}
and an analogous argument gives the lower bound. 

In the specialisation of Theorem \ref{thm:probabsorption2}, the character
$\zeta$ has value $\frac{\eta(c)}{(\deg c)!}$ for $c\in\calc'$
and is zero on other generators. By Lemma \ref{lem:etaproduct} on
$\eta$ of products, $\frac{\zeta(y)}{\eta(y)}=\frac{1}{(\deg y)!}$
if all factors of $y$ are in $\calc'$, and is 0 otherwise. Hence
$\left[\chi^{\zeta}(x_{0})\right]_{1/a}$ is precisely 
\[
\frac{\eta(x_{0})}{(\deg y)!}\sum_{y}\hatkan(x_{0},y),
\]
summing over all $y$ whose factors are in $\calc'$. Proposition
\ref{prop:probabsorption} is then immediate on taking $\calc'=\calb_{1}$.
\end{proof}

\section{Rock-Breaking\label{sec:Rock-breaking}}

This section investigates a model of rock-breaking, one of two initial
examples of a Hopf-power Markov chain in \cite[Sec. 4]{hopfpowerchains},
which gives references to Kolmogorov's study of similar breaking models.
The states of this Markov chain are \emph{partitions} $\lambda=(\lambda_{1},\dots,\lambda_{l})$,
a multiset of positive integers recording the sizes of a collection
of rocks. (It is best here to think of the \emph{parts} $\lambda_{i}$
as unordered, although the standard notation is to write $\lambda_{1}\geq\lambda_{2}\geq\lambda_{l(\lambda)}$.)
In what follows, $|\lambda|:=\lambda_{1}+\dots+\lambda_{l(\lambda)}$
is the total size of the rocks in the collection $\lambda$, and the
number of rocks in the collection is $l(\lambda)$, the \emph{length}
of the partition. $Z(\lambda)$ is the size of the stabiliser of $\mathfrak{S}_{l(\lambda)}$
permuting the parts of $\lambda$. If $a_{i}(\lambda)$ is the number
of parts of size $i$ in $\lambda$, then $Z(\lambda)=\prod_{i}a_{i}(\lambda)!$.
For example, if $\mu=(2,1,1,1)$, then $|\mu|=5$, $l(\mu)=4$ and
$Z(\mu)=6$.

At each step of the Markov chain, each rock breaks independently into
$a$ pieces whose sizes follow a symmetric multinomial distribution.
(This may result in some pieces of zero size.) Section \ref{sub:Constructing-rock-breaking}
phrases this process as the Hopf-power Markov chain on the homogeneous
symmetric functions $\{h_{\lambda}\}$. Sections \ref{sub:Right-Eigenfunctions-rock-breaking}
and \ref{sub:Left-Eigenfunctions-rock-breaking} then leverage the
machinery of Section \ref{sub:Altrightefns} and Chapter \ref{chap:hpmc-Diagonalisation}
to deduce a full right and left eigenbasis respectively. These eigenbases
correspond (up to scaling) to the power sum symmetric functions $\{p_{\lambda}\}$,
so the explicit expressions for the eigenfunctions recover well-known
formulae for the change-of-basis between $\{h_{\lambda}\}$ and $\{p_{\lambda}\}$.
Section \ref{sub:matrix-rocks} gives a numerical example of the transition
matrix and full eigenbases, for the case $n=4$.

\subsection{Constructing the Chain\label{sub:Constructing-rock-breaking}}

The goal of this section is to interpret the Hopf-power Markov chain
on the homogeneous symmetric functions $\{h_{\lambda}\}$ as independent
multinomial breaking. (\cite{hopfpowerchains} took instead the elementary
symmetric functions $\{e_{\lambda}\}$ as their state space basis,
which is equivalent as there is a Hopf-involution on $\Lambda$ exchanging
$\{h_{\lambda}\}$ and $\{e_{\lambda}\}$ \cite[Sec. 7.6]{stanleyec2}.
This thesis chooses to use $\{h_{\lambda}\}$ because its dual basis
is $\{m_{\lambda}\}$, the monomial symmmetric functions, while the
dual of $\{e_{\lambda}\}$ is less studied.)

Recall from Example \ref{ex:symmetricfn} that, as an algebra, $\Lambda$
is the subalgebra of the power series algebra $\mathbb{R}[[x_{1},x_{2},\dots]]$
generated by 
\[
h_{(n)}:=\sum_{i_{1}\leq\dots\leq i_{n}}x_{i_{1}}\dots x_{i_{n}},
\]
which has degree $n$. There is a large swathe of literature on $\Lambda$
- the standard references are \cites[Chap. 7]{stanleyec1}[Chap. 1]{macdonald} .
Only two facts are essential for building the present chain: first,
\[
h_{\lambda}:=h_{(\lambda_{1})}\dots h_{(\lambda_{l(\lambda)})}
\]
 is a basis for $\Lambda$; second, the coproduct satisfies $\Delta(h_{(n)})=\sum_{i=0}^{n}h_{(i)}\otimes h_{(n-i)}$,
with the convention $h_{(0)}=1$. It follows from the compatibility
axiom of Hopf algebras that 
\[
\Delta(h_{\lambda})=\sum_{i_{1},\dots,i_{l}=0}^{i_{j}=\lambda_{j}}h_{(i_{1},\dots,i_{l})}\otimes h_{(\lambda_{1}-i_{1},\dots,\lambda_{l}-i_{l})}.
\]
(Here, it is not necessarily true that $i_{1}\geq i_{2}\geq\dots\geq i_{l}$.
This is one instance where it is useful to think of the parts as unordered.)
Then it is obvious that $\{h_{\lambda}\}$ is a state space basis
- the product and coproduct structure constants of $\Lambda$ with
repsect to $\{h_{\lambda}\}$ are non-negative, and $\bard(h_{\lambda})\neq0$
if $\deg(\lambda)>1$. In the sequel, it will be convenient to write
$\lambda$ for $h_{\lambda}$. For example, the above equation translates
in this notation to 
\[
\Delta(\lambda)=\sum_{i_{1},\dots,i_{l}=0}^{i_{j}=\lambda_{j}}(i_{1},\dots,i_{l})\otimes(\lambda_{1}-i_{1},\dots,\lambda_{l}-i_{l}).
\]

Recall that Theorem \ref{thm:threestep} gives a three-step interpretation
of a Hopf-power Markov chain. To apply this to the chain on $\{h_{\lambda}\}$,
it is first necessary to compute the rescaling function $\eta$. A
simple induction shows that 
\begin{equation}
\Delta^{[r]}((n))=\sum_{i_{1}+\dots+i_{r}=n}(i_{1})\otimes\dots\otimes(i_{r}),\label{eq:coprodh}
\end{equation}
so
\[
\bard^{[n]}((n))=(1)\otimes\dots\otimes(1),
\]
and $\eta((n))=1$. Lemma \ref{lem:etaproduct} then shows that 
\[
\eta(\lambda)=\binom{|\lambda|}{\lambda_{1}\dots\lambda_{l}}\eta((\lambda_{1}))\dots\eta((\lambda_{l}))=\binom{|\lambda|}{\lambda_{1}\dots\lambda_{l}}.
\]

Note that $\{h_{\lambda}\}$ is a free-commutative basis, so, by Theorem
\ref{thm:independence}, each rock in the collection breaks independently.
Thus it suffices to understand the chain starting at $(n)$. By Equation
\ref{eq:coprodh}, the coproduct structure constant $\eta_{(n)}^{\mu^{1},\dots,\mu^{a}}=1$
if $\mu^{1},\dots,\mu^{a}$ are all partitions of single parts with
$|\mu^{1}|+\dots|\mu^{a}|=n$, and is 0 for all other $a$-tuples
of partitions. As a result, the three-step description of Theorem
\ref{thm:threestep} simplifies to:
\begin{enumerate}[label=\arabic*.]
\item Choose $i_{1},\dots,i_{a}$ according to a symmetric multinomial
distribution.
\item Choose the $a$-tuple of one part partitions $(i_{1}),\dots,(i_{a})$,
some of which may be the zero partition.
\item Move to $(i_{1},\dots,i_{a})$.
\end{enumerate}
Thus each rock breaks multinomially. Section \ref{sub:matrix-rocks}
below displays the transition matrix for the case $a=2$ and $n=4$,
describing binomial breaking of rocks of total size four.

\subsection{Right Eigenfunctions\label{sub:Right-Eigenfunctions-rock-breaking}}

Begin with the simpler eigenfunctions $\f_{(j)}$ for $j>1$, defined
in Equation \ref{eq:easyrightefns} to be 
\[
\f_{(j)}(\lambda):=\binom{|\lambda|}{j}\frac{\eta_{\lambda}^{(j),(1),\dots,(1)}}{\eta(\lambda)}.
\]
By Proposition \ref{prop:rightefnsofproducts}, these eigenfunctions
satisfy $\f_{(j)}(\lambda)=\f_{(j)}(\lambda_{1})+\dots+\f_{(j)}(\lambda_{l})$;
since $\eta_{(n)}^{(j),(1),\dots,(1)}=1$ and $\eta((n))=1$, it follows
that 
\[
\f_{(j)}(\lambda)=\sum_{i=1}^{l}\binom{\lambda_{i}}{j}.
\]
The corresponding eigenvalue is $a^{-j+1}$.

Recall from Section \ref{sub:Altrightefns} that the main use of these
eigenfunctions is to measure how far the chain is from being absorbed.
For the rock-breaking chain, this measure takes the form of ``expected
number of large rocks''. Note that each part of $\lambda$ of size
$j$ or greater contributes at least 1 to $\f_{(j)}(\lambda)$; a
simple application of Proposition \ref{prop:expectationrightefns}
then gives the Proposition below. The analogous result for the more
general Markov chain of removing edges from graphs is Proposition
\ref{prop:probboundsrightefns-graphs}, from which this also follows
easily.
\begin{prop}
\label{prop:probbounds-rocks}Let $\{X_{m}\}$ denote the rock-breaking
chain. Then, for any $j>1$,
\begin{align*}
 & P\{X_{m}\mbox{ contains a rock of size }\geq j|X_{0}=\lambda\}\\
\leq & E\{|\{\mbox{rocks of size }\geq j\mbox{ in }X_{m}\}||X_{0}=\lambda\}\\
\leq & a^{(-j+1)m}\sum_{i=1}^{l}\binom{\lambda_{i}}{j}.
\end{align*}
In particular, the case $j=2$ shows 
\[
P\{X_{m}\neq(1,1,\dots,1)\}\leq a^{-m}\sum_{i=1}^{l}\binom{\lambda_{i}}{j}.
\]
 \qed
\end{prop}
Theorem \ref{thm:ABdual} gives this formula for the full right eigenbasis:
\begin{thm}
\label{thm:rightefns-rocks}A basis $\{\f_{\mu}\}$ of right eigenfunctions
of the rock-breaking chain is 
\[
\f_{\mu}(\lambda):=\frac{1}{\binom{|\lambda|}{\lambda_{1}\dots\lambda_{l(\lambda)}}}\sum\frac{1}{Z(\mu^{1})\dots Z(\mu^{l(\lambda)})}
\]
where the sum is over all $l(\lambda)$-tuples of partitions $\{\mu^{j}\}$
such that $\mu^{j}$ is a partition of $\lambda_{j}$ and the disjoint
union $\amalg_{j}\mu^{j}=\mu$, and $Z(\mu^{j})$ is the size of the
stabiliser of $\mathfrak{S}_{l(\mu^{j})}$ permuting the parts of
$\mu^{j}$. In particular, $\f_{\mu}(\mu)=\left(\binom{|\mu|}{\mu_{1}\dots\mu_{l(\lambda)}}\right)^{-1}$,
and $\f_{\mu}(\lambda)$ is positive if $\mu$ is a refinement of
$\lambda$, and is otherwise 0. The corresponding eigenvalue is $a^{l(\mu)-n}$. 

From this right eigenfunction formula, one can recover the expansion
of the power sums in terms of monomial symmetric functions \cite[Prop. 7.7.1]{stanleyec2}:
\[
p_{\mu}=Z(\mu)\sum_{\lambda}\binom{|\lambda|}{\lambda_{1}\dots\lambda_{l(\lambda)}}\f_{\mu}(\lambda)m_{\lambda}=\sum_{\amalg\mu^{i}=\mu}\frac{Z(\mu)}{Z(\mu^{1})\dots Z(\mu^{l})}m_{(|\mu^{1}|,\dots,|\mu^{l}|)}.
\]

\end{thm}
Here is an illustration of how to compute with this formula; the proof
will follow.
\begin{example}
Take $\mu=(2,1,1,1),\ \lambda=(3,2)$. Then the possible $\{\mu^{j}\}$
are 
\begin{alignat*}{2}
\mu^{1} & =(2,1), & \qquad\mu^{2} & =(1,1);\\
\mu^{1} & =(1,1,1), & \qquad\mu^{2} & =(2).
\end{alignat*}
Hence 
\[
\f_{\mu}(\lambda)=\frac{1}{\binom{5}{3}}\left(\frac{1}{2}+\frac{1}{3}\right)=\frac{1}{12}.
\]

The full basis of right eigenfunctions for the case $n=4$ is in Section
\ref{sub:matrix-rocks}.\end{example}
\begin{proof}
For concreteness, take $l(\lambda)=2$ and $l(\mu)=3$. Then the simplification
of Theorem \ref{thm:ABdual} for cocommutative Hopf algebras gives
\[
\f_{\mu}(\lambda)=\frac{1}{Z(\mu)\eta(\lambda)}\eta_{\lambda}^{(\mu_{1}),(\mu_{2}),(\mu_{3})}=\frac{1}{Z(\mu)\binom{|\lambda|}{\lambda_{1}\ \lambda_{2}}}\eta_{\lambda}^{(\mu_{1}),(\mu_{2}),(\mu_{3})},
\]
where the parts of $\mu$ are ordered so $\mu_{1}\geq\mu_{2}\geq\mu_{3}.$
To calculate the coproduct structure constant $\eta_{\lambda}^{(\mu_{1}),(\mu_{2}),(\mu_{3})}$,
recall that 
\[
\Delta^{[3]}(\lambda)=\Delta^{[3]}\left(\lambda_{1}\right)\Delta^{[3]}\left(\lambda_{2}\right)=\sum_{\substack{i_{1}+j_{1}+k_{1}=\lambda_{1}\\
i_{2}+j_{2}+k_{2}=\lambda_{2}
}
}(i_{1},i_{2})\otimes(j_{1},j_{2})\otimes(k_{1},k_{2}).
\]
So $\eta_{\lambda}^{(\mu_{1}),(\mu_{2}),(\mu_{3})}$ enumerates the
sextuples $\left(i_{1},j_{1},k_{1},i_{2},j_{2},k_{2}\right)$ such
that $i_{1}+j_{1}+k_{1}=\lambda_{1}$, $i_{2}+j_{2}+k_{2}=\lambda_{2}$,
and $i_{1}$ and $i_{2}$ are $\mu_{1}$ and $0$ in either order,
and similarly for $j_{1},j_{2}$ and $k_{1},k_{2}$. Set $\mu^{1}:=(i_{1},j_{1},k_{1})$,
$\mu^{2}=(i_{2},j_{2},k_{2})$; then these sextuples are precisely
the case where $|\mu^{1}|=\lambda_{1}$, $|\mu^{2}|=\lambda_{2}$,
and the disjoint union $\mu^{1}\amalg\mu^{2}=\mu$. If the parts of
$\mu$ are distinct (i.e. $\mu_{1}>\mu_{2}>\mu_{3})$, then one can
reconstruct a unique sextuple from such a pair of partitions: if $\mu^{1}$
has a part of size $\mu_{1}$, then $i_{1}=\mu_{1}$ and $i_{2}=0$;
else $\mu^{2}$ has a part of size $\mu_{1}$, and $i_{2}=\mu_{1}$,
$i_{1}=0$; and similarly for $j_{1},j_{2},k_{1},k_{2}$. If, however,
$\mu_{1}=\mu_{2}>\mu_{3}$, and $\mu^{1},\mu^{2}$ both have one part
of the common size $\mu_{1}=\mu_{2}$, then there are two sextuples
corresponding to $(\mu^{1},\mu^{2})$: both $i_{1}=j_{2}=\mu_{1}$,
$i_{2}=j_{1}=0$ and $i_{2}=j_{1}=\mu_{1}$, $i_{1}=j_{2}=0$ are
possible. In general, this multiplicity is the product of multinomial
coefficients 
\[
\prod_{i}\binom{a_{i}(\mu)}{a_{i}(\mu^{1})\dots a_{i}(\mu^{l(\lambda)})},
\]
where $a_{i}(\mu)$ is the number of parts of $\mu$ of size $i$.
Since $\prod_{i}a_{i}(\mu)!=Z(\mu)$, the expression in the theorem
follows.

Now show that $p_{\mu}=Z(\mu)\sum_{\lambda}\binom{|\lambda|}{\lambda_{1}\dots\lambda_{l(\lambda)}}\f_{\mu}(\lambda)m_{\lambda}$.
Theorem \ref{thm:ABdual} and Proposition \ref{prop:efns}.R constructs
$\f_{\mu}(\lambda)$ as $\frac{1}{Z(\mu)\eta(\lambda)}\left[(\mu_{1})^{*}\dots(\mu_{l})^{*}\right](\lambda)$,
or the coefficient of $\lambda^{*}$ in $\frac{1}{Z(\mu)\eta(\lambda)}(\mu_{1})^{*}\dots(\mu_{l})^{*}$.
Viewing the algebra of symmetric functions as its own dual via the
Hall inner product, $\lambda^{*}$ is the monomial symmetric function
$m_{\lambda}$. So $\f_{\mu}(\lambda)$ is the coefficient of $m_{\lambda}$
in 
\[
\frac{1}{Z(\mu)\eta(\lambda)}m_{(\mu_{1})}\dots m_{(\mu_{l})}=\frac{1}{Z(\mu)\binom{|\lambda|}{\lambda_{1}\dots\lambda_{l(\lambda)}}}p_{\mu}.
\]

\end{proof}

\subsection{Left Eigenfunctions\label{sub:Left-Eigenfunctions-rock-breaking}}

Applying Theorem \ref{thm:diagonalisation}.A to the rock-breaking
chain, taking the single-part partitions as the free generating set,
gives the following basis of left eigenfunctions.
\begin{thm}
\label{thm:leftefns-rocks}A basis $\{\g_{\mu}\}$ of left eigenfunctions
of the rock-breaking chain is
\[
\g_{\mu}(\lambda)=(-1)^{l\left(\mu\right)-l(\lambda)}\binom{|\lambda|}{\lambda_{1}\dots\lambda_{l(\mu)}}\sum\frac{\left(l\left(\lambda^{1}\right)-1\right)!\dots\left(l\left(\lambda^{l(\mu)}\right)-1\right)!}{Z(\lambda^{1})\dots Z(\lambda^{l(\mu)})}
\]
where the sum is over all $l(\mu)$-tuples of partitions $\{\lambda^{j}\}$
such that $\lambda^{j}$ is a partition of $\mu_{j}$ and the disjoint
union $\amalg_{j}\lambda^{j}=\lambda$, and $Z(\lambda^{j})$ is the
size of the stabiliser of $\mathfrak{S}_{l(\lambda^{j})}$ permuting
the parts of $\lambda^{j}$. In particular, $\g_{\mu}(\mu)=\binom{|\mu|}{\mu_{1}\dots\mu_{l(\lambda)}}$,
and $\g_{\mu}(\lambda)$ is non-zero only if $\lambda$ is a refinement
of $\mu$. The corresponding eigenvalue is $a^{l(\mu)-n}$.

From this left eigenfunction formula, one can recover the expansion
of the power sums in terms of complete symmetric functions:
\begin{align*}
p_{\mu} & =\mu_{1}\dots\mu_{l}\sum_{\lambda}\frac{1}{\binom{|\lambda|}{\lambda_{1}\dots\lambda_{l(\lambda)}}}\g_{\mu}(\lambda)h_{\lambda}\\
 & =\sum_{r}(-1)^{l\left(\mu\right)-r}\mu_{1}\dots\mu_{l}\sum_{|\lambda^{j}|=\mu_{j}}\frac{\left(l\left(\lambda^{1}\right)-1\right)!\dots\left(l\left(\lambda^{r}\right)-1\right)!}{Z(\lambda^{1})\dots Z(\lambda^{r})}h_{\amalg\lambda^{j}}.
\end{align*}

\end{thm}
As previously, here is a calculational example. 
\begin{example}
Take $\lambda=(2,1,1,1),\ \mu=(3,2)$. Then the possible $\{\lambda^{j}\}$
are 
\begin{alignat*}{2}
\lambda^{1} & =(2,1), & \qquad\lambda^{2} & =(1,1);\\
\lambda^{1} & =(1,1,1), & \qquad\lambda^{2} & =(2).
\end{alignat*}
Hence 
\[
\g_{\mu}(\lambda)=(-1)^{2}\binom{5}{2}\left(\frac{1!1!}{2!}+\frac{2!0!}{3!}\right)=\frac{25}{3}.
\]

The full basis of left eigenfunctions for the case $n=4$ is in Section
\ref{sub:matrix-rocks}.\end{example}
\begin{proof}
By Theorem \ref{thm:diagonalisation}.A and and Proposition \ref{prop:efns}.L,
\begin{align*}
\g_{\mu}(\lambda) & =\mbox{coefficient of }\lambda\mbox{ in }\eta(\lambda)e\left((\mu_{1})\right)\dots e\left((\mu_{l(\mu)})\right)\\
 & =\mbox{coefficient of }\lambda\mbox{ in }\binom{|\lambda|}{\lambda_{1}\dots\lambda_{l(\lambda)}}e\left((\mu_{1})\right)\dots e\left((\mu_{l(\mu)})\right).
\end{align*}
Every occurrence of $\lambda$ in $e\left((\mu_{1})\right)\dots e\left((\mu_{l(\mu)})\right)$
is a product of a $\lambda^{1}$ term in $e((\mu_{1}))$, a $\lambda^{2}$
term in $e((\mu_{2}))$, etc., for some choice of partitions $\lambda^{j}$
with $|\lambda^{j}|=\mu_{j}$ for each $j$, and $\amalg_{j}\lambda^{j}=\lambda$.
Hence it suffices to show that the coefficient of a fixed $\lambda^{j}$
in $e((\mu_{j}))$ is 
\[
\frac{(-1)^{l(\lambda^{j})-1}(l(\lambda^{j})-1)!}{Z(\lambda^{j})}.
\]
Recall that 
\begin{align*}
e((\mu_{j})) & =\sum_{r\geq1}\frac{(-1)^{r-1}}{r}m^{[r]}\bard^{[r]}((\mu_{j}))\\
 & =\sum_{r\geq1}\frac{(-1)^{r-1}}{r}\sum_{\substack{i_{1}+\dots i_{r}=\mu_{j}\\
i_{1},\dots,i_{r}>0
}
}(i_{1},\dots i_{r}),
\end{align*}
so $\lambda^{j}$ only appears in the summand with $r=l(\lambda^{j})$.
Hence the required coefficient is $\frac{(-1)^{l(\lambda^{j})-1}}{l(\lambda^{j})}$
multiplied by the number of distinct orderings of the parts of $\lambda^{j}$,
which is $\frac{l(\lambda^{j})!}{Z(\lambda^{j})}$. 

To deduce the $h_{\lambda}$-expansion of $p_{\mu}$, recall from
above that $\frac{\g_{\mu}(\lambda)}{\eta(\lambda)}$ is the coefficient
of $h_{\lambda}$ in the symmetric function $e\left(h_{(\mu_{1})}\right)\dots e\left(h_{(\mu_{l})}\right)$.
Since the algebra of symmetric functions is cocommutative, the Eulerian
idempotent map $e$ is a projection onto the subspace of primitives.
So $e(h_{(n)})$ is a primitive symmetric function of degree $n$.
But, up to scaling, the power sum $p_{(n)}$ is the only such symmetric
function, so $e(h_{(n)})$ is necessarily $\alpha_{n}p_{(n)}$ for
some number $\alpha_{n}$. Thus $\frac{\g_{\mu}(\lambda)}{\eta(\lambda)}$
is the coefficient of $h_{\lambda}$ in $\alpha_{\mu_{1}}\dots\alpha_{\mu_{l}}p_{\mu}$,
and it suffices to show that $\alpha_{n}=\frac{1}{n}$.

As usual, let $f_{\mu},g_{\mu}$ be the symmetric functions inducing
the eigenfunctions $\f_{\mu},\g_{\mu}$ respectively. Then $\langle f_{\mu},g_{\mu}\rangle=\sum_{\lambda}\f_{\mu}(\lambda)\g_{\mu}(\lambda)$,
where the left hand side is the Hall inner product. By Theorem \ref{thm:ABdual},
the right hand side is 1 for all $\mu$. Take $\mu=(n)$, then 
\[
n\alpha_{n}=\alpha_{n}\langle p_{(n)},p_{(n)}\rangle=\langle f_{(n)},g_{(n)}\rangle=1,
\]
so $\alpha_{n}=\frac{1}{n}$ as desired.\end{proof}
\begin{rem*}
This calculation is greatly simplified for the algebra of symmetric
functions, compared to other Hopf algebras. The reason is that, for
a generator $c$, it is in general false that all terms of $m^{[a]}\bard^{[a]}(c)$
have length $a$, or equivalently that all tensor-factors of a term
of $\bard^{[a]}(c)$ are generators. See the fourth summand of the
coproduct calculation in Figure \ref{fig:coproduct-graphs} for one
instance of this, in the Hopf algebra of graphs. Then terms of length
say, three, in $e(c)$ may show up in both $m^{[2]}\bard^{[2]}(c)$
and $m^{[3]}\bard^{[3]}(c)$, so determining the coefficient of this
length three term in $e(c)$ is much harder, due to these potential
cancellations in $e(c)$. Hence much effort \cite{fisher,lrstructure,mrstructure}
has gone into developing cancellation-free expressions for primitives,
as alternatives to $e(c)$.
\end{rem*}

\subsection{Transition Matrix and Eigenfunctions when $n=4$\label{sub:matrix-rocks}}

The Hopf-square Markov chain on partitions of four describes independent
binomial breaking of a collection of rocks with total size four. Its
transition matrix $K_{2,4}$ is the following matrix:
\[
\begin{array}{cccccc}
 & (4) & (3,1) & (2,2) & (2,1,1) & (1,1,1,1)\\
(4) & \frac{1}{8} & \frac{1}{2} & \frac{3}{8} & 0 & 0\\
(3,1) & 0 & \frac{1}{4} & 0 & \frac{3}{4} & 0\\
(2,2) & 0 & 0 & \frac{1}{4} & \frac{1}{2} & \frac{1}{4}\\
(2,1,1) & 0 & 0 & 0 & \frac{1}{2} & \frac{1}{2}\\
(1,1,1,1) & 0 & 0 & 0 & 0 & 1
\end{array}.
\]

Its basis of right eigenfunctions, as determined by Theorem \ref{thm:rightefns-rocks},
are the columns of the following matrix:
\[
\begin{array}{cccccc}
 & \f_{(4)} & \f_{(3,1)} & \f_{(2,2)} & \f_{(2,1,1)} & \f_{(1,1,1,1)}\\
(4) & 1 & 1 & \frac{1}{2} & \frac{1}{2} & \frac{1}{24}\\
(3,1) & 0 & \frac{1}{4} & 0 & \frac{1}{4} & \frac{1}{24}\\
(2,2) & 0 & 0 & \frac{1}{6} & \frac{1}{6} & \frac{1}{24}\\
(2,1,1) & 0 & 0 & 0 & \frac{1}{12} & \frac{1}{24}\\
(1,1,1,1) & 0 & 0 & 0 & 0 & \frac{1}{24}
\end{array}.
\]

Its basis of left eigenfunctions, as determined by Theorem \ref{thm:leftefns-rocks},
are the rows of the following matrix:
\[
\begin{array}{cccccc}
 & (4) & (3,1) & (2,2) & (2,1,1) & (1,1,1,1)\\
\g_{(4)} & 1 & -4 & -3 & 12 & -6\\
\g_{(3,1)} & 0 & 4 & 0 & -12 & 8\\
\g_{(2,2)} & 0 & 0 & 6 & -12 & 6\\
\g_{(2,1,1)} & 0 & 0 & 0 & 12 & -12\\
\g_{(1,1,1,1)} & 0 & 0 & 0 & 0 & 24
\end{array}.
\]

\section{Tree-Pruning\label{sec:Tree-Pruning}}

This section examines the Hopf-power Markov chain whose underlying
Hopf algebra is the Connes-Kreimer algebra of rooted trees. This is
one of many Hopf algebras arising from quantum field theory during
the surge in the relationship between the two fields in the late 1990s.
Its definition as a Hopf algebra first appeared in \cite{cktrees,cktrees2},
though they note that it is essentially the same data as the Butcher
group for Runge-Kutta methods of solving ordinary differential equations
\cite{rungekutta}. A textbook exposition of the use of trees in Runge-Kutta
methods is in \cite[Chap. 3]{diffeqnbook}.

In his thesis, Foissy \cite{noncommutativecktrees1,noncommutativecktrees2,noncommutativecktrees3}
constructs a noncommutative version of the Connes-Kreimer Hopf algebra,
which turns out to be isomorphic to $\mathbf{PBT}$, the Loday-Ronco
Hopf algebra of planar binary trees \cite{lodayroncotrees}. \cite{grhiscommutative,polynomialrealisation2}
then relate it (and its dual $YSym$) to other Hopf algebras of trees,
and well-known Hopf algebras coming from polynomial realisations.

The main purpose of this example is to illustrate how to interpret
the chain and to calculate simple right eigenfunctions and probability
bounds using the ``recursive structure'' of trees. The exposition
below should serve as a prototype for studying Hopf-power Markov chains
on other Hopf algebras of trees.

\subsection{The Connes-Kreimer Hopf algebra\label{sub:CKtrees}}

A \emph{tree} is a connected graph (unlabelled) without cycles; a
tree $T$ is \emph{rooted} if it has a distinguished vertex $\Root(T)$.
(The embedding of a tree in the plane - e.g. whether an edge runs
to the left or the right - is immaterial). A \emph{rooted forest}
is a disjoint union of rooted trees - so each of its components has
a root. All trees and forests in this section are rooted unless specified
otherwise. Following \cite{cktrees2}, all diagrams below will show
$\Root(T)$ as the uppermost vertex, and edges will flow downwards
from a \emph{parent} to a \emph{child}. A \emph{leaf} is a vertex
with no children. More rigorous definitions of these and related terms
are in \cite[Sec. 300]{diffeqnbook}; the trees here he calls ``abstract
trees'' as their vertices are not labelled.

Some non-standard notation (see Example \ref{ex:treedefns} below):
$\deg(T)$ is the number of vertices in the tree $T$. A tree $T'$
is a subtree of $T$ if the subgraph which $T$ induces on the vertex
set of $T'$ is connected. Denote this by $T'\subseteq T$. Subtrees
containing $\Root(T)$ are \emph{trunks}; otherwise, the root of $T'$
is the vertex which was closest to $\Root(T)$. If $v$ is a vertex
of $T'$, written $v\in T'$, then $\desc_{T'}(v)$ is the number
of \emph{descendants} of $v$ in $T'$, including $v$ itself, and
$\anc_{T'}(v)$ is the number of \emph{ancestors }of $v$ in $T'$,
including $v$ itself.

Two families of graphs are of special interest here: let $P_{n}$
be the \emph{path} of degree $n$, where all but one vertex has precisely
one child, and $Q_{n}$ be the \emph{star} of degree $n$, where the
root has $n-1$ children, and all non-root vertices have no children.
(Again this notation is non-standard.) In line with the Hopf algebra
notation in previous chapters, $\bullet$ indicates the unique tree
with one vertex. 
\begin{example}
\label{ex:treedefns}Let $T$ be the tree in Figure \ref{fig:treedefns}.
(The vertex labels are not part of the tree data, they are merely
for easy reference.) Then $\deg(T)=5$. Vertex $t$ has two children,
namely $u$ and $v$; these are both leaves. The star $Q_{3}$ is
a subtree of $T$ in two ways: from the vertices $\{r,s,t\}$, for
which $r$ is the root, and from the vertices $\{t,u,v\}$, for which
$t$ is the root. Only the first of these is a trunk. If $T'$ is
this first copy of $Q_{3}$, then $\desc_{T'}(r)=3,\desc_{T'}(t)=1$.
The ancestors of $u$ are $t$ and $r$, so $\anc_{T}(v)=2$. 

\begin{figure}
\begin{centering}
\includegraphics[scale=0.5]{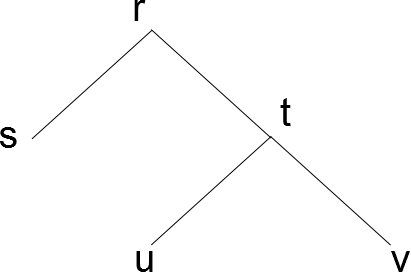}
\par\end{centering}

\caption{The tree $[\bullet Q_{3}]$\label{fig:treedefns}}
\end{figure}

\end{example}
Most results concerning the tree-pruning Markov chain will have an
inductive proof, and the key to such arguments is this: given a tree
$T\neq\bullet$, let $T_{1},\dots,T_{f}$ be the connected components
of $T$ after removing the root. (The ordering of the $T_{i}$ are
immaterial.) Following \cite{diffeqnbook}, write $T:=[T_{1}\dots T_{f}]$;
his Table 300(I) demonstrates how to write every tree in terms of
$\bullet$ (which he calls $\tau$) and repeated applications of this
operator. For example, $Q_{3}=[\bullet\bullet]$, $P_{3}=[\bullet[\bullet]]$,
and $P_{n}=[P_{n-1}]$. The degree 5 tree in Figure \ref{fig:treedefns}
is $[\bullet Q_{3}]=[\bullet[\bullet\bullet]]$.

\cite[Sec. 5]{treefactorial} then defines the \emph{tree factorial}
recursively:
\[
\bullet!=1,\quad T!=\deg(T)T_{1}!\dots T_{f}!.
\]
\cite{diffeqnbook} calls this the ``density'' $\gamma(T)$ and
gives the following equivalent non-recursive expression:
\begin{prop}
\label{prop:treefactorial} \cite[Thm. 301A.c]{diffeqnbook} 
\[
T!=\prod_{v\in T}\desc_{T}(v)
\]
\end{prop}
\begin{proof}
When $T=\bullet$, this is immediate. For $T\neq\bullet$, each non-root
vertex $v\in T$ is a vertex of precisely one $T_{i}$, and $\desc_{T}(v)=\desc_{T_{i}}(v)$,
so, by inductive hypothesis, 
\begin{align*}
T! & =\deg T\prod_{v_{1}\in T_{1}}\desc_{T_{1}}(v_{1})\dots\prod_{v_{f}\in T_{f}}\desc_{T_{f}}(v_{f})\\
 & =\prod_{v\in T}\desc_{T}(v)
\end{align*}
as the root of $T$ has $\deg T$ descendants.
\end{proof}
It is clear from this alternative expression that $P_{n}!=n!$ (which
inspired this notation) and $Q_{n}!=n$. Note that these are respectively
the largest and and smallest possible values for $T!$.
\begin{example}
\label{ex:treefactorial}Take $T=[\bullet Q_{3}]$ as pictured in
Figure \ref{fig:treedefns}. Then $T!=5\bullet!Q_{3}!=5\cdot1\cdot3=15$.
Note that this is also $\desc_{T}(r)\desc_{T}(s)\desc_{T}(t)\desc_{T}(u)\desc_{T}(v)=5\cdot1\cdot3\cdot1\cdot1$.
\end{example}
Finally we are ready to define the Hopf structure on these trees.
The basis $\calb_{n}$ for the subspace of degree $n$ is the set
of forests with $n$ vertices. The product of two forests is their
disjoint union, thus $\calb$ is a free-commutative basis, and the
corresponding free generating set is the rooted trees. The coproduct
of a tree $T$ is given by 
\[
\Delta(T)=\sum T\backslash S\otimes S,
\]
where the sum runs over all trunks $S$ of $T$, including the empty
tree and $T$ itself, and $T\backslash S$ is the forest produced
by removing from $T$ all edges incident with $S$ (each component
is a \textit{cut branch}). The root of each cut branch is the vertex
which was closest to the root of $T$. Extend this definition multiplicatively
to define the coproduct on forests: $\Delta(T_{1}\amalg\dots\amalg T_{l})=\Delta(T_{1})\dots\Delta(T_{l})$.
Note that the trunk is always connected, but there may be several
cut branches. Hence $\calh$ is noncocommutative. 

It is not hard to derive a recursive formula for the coproduct of
a tree. As above, write $T=[T_{1}\dots T_{f}]$, where $T_{1},\dots,T_{f}$
are the connected components of $T$ after removing the root. Then
each non-empty trunk $S$ of $T$ has the form $S=[S_{1}\dots S_{f}]$
for (possibly empty) trunks $S_{i}$ of each $T_{i}$. The cut branches
$T\backslash S$ are then the disjoint union $T_{1}\backslash S_{1}\amalg\dots\amalg T_{f}\backslash S_{f}$.
So, in Sweedler notation (so $\Delta(T_{i})=\sum_{(T_{i})}(T_{i})_{(1)}\otimes(T_{i})_{(2)}$),
the following holds \cite[Eq. 50, 51]{cktrees2}: 
\begin{equation}
\Delta([T_{1}\dots T_{f}])=T\otimes1+\sum_{(T_{1}),\dots,(T_{f})}(T_{1})_{(1)}\dots(T_{f})_{(1)}\otimes[(T_{1})_{(2)}\dots(T_{f})_{(2)}].\label{eq:coproducttrees}
\end{equation}

\begin{example}
\label{ex:coproduct-cktrees}Figure \ref{fig:coproduct-cktrees} calculates
the coproduct for the tree $[\bullet Q_{3}]$ from Figure \ref{fig:treedefns}
above. Check this using Equation \ref{eq:coproducttrees}. By definition,
$Q_{3}=[\bullet\bullet]$ so 
\[
\Delta(Q_{3})=Q_{3}\otimes1+\bullet^{2}\otimes\bullet+2\bullet\otimes P_{2}+1\otimes Q_{3}.
\]
(This made use of $[\bullet]=P_{2}$.) Then (recall $P_{3}=[P_{2}]$),
\begin{align*}
\Delta([\bullet Q_{3}]) & =[\bullet Q_{3}]\otimes1+\bullet Q_{3}\otimes\bullet+\bullet^{3}\otimes P_{2}+2\bullet^{2}\otimes P_{3}+\bullet\otimes[Q_{3}]\\
 & \hphantom{=}+Q_{3}\otimes P_{2}+\bullet^{2}\otimes Q_{3}+2\bullet\otimes[P_{2}\bullet]+1\otimes[\bullet Q_{3}].
\end{align*}

\begin{figure}
\begin{centering}
\includegraphics[scale=0.35]{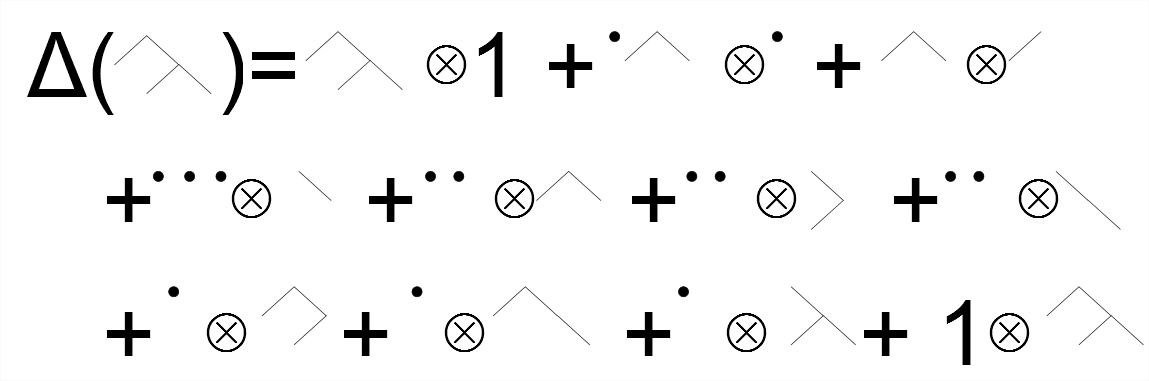}
\par\end{centering}

\caption{Coproduct of $[\bullet Q_{3}]$\label{fig:coproduct-cktrees}}
\end{figure}

\end{example}

\begin{example}
\label{ex:paths-cktrees}Consider $P_{n}$, the path with $n$ vertices.
Its trunks are $P_{i}$, $0\leq i\leq n$, and the sole cut branch
corresponding to $P_{i}$ is $P_{n-i}$. Hence $\Delta(P_{n})=\sum_{i=0}^{n}P_{n-i}\otimes P_{i}$,
which recovers the independent multinomial rock-breaking process of
Section \ref{sec:Rock-breaking}. Equivalently, $h_{(n)}\rightarrow P_{n}$
defines an embedding of the algebra of symmetric functions into $\calh$.
\end{example}

\subsection{Constructing the Chain\label{sub:Constructing-the-Chain-CKtrees}}

To describe the Hopf-power Markov chain on $\calh$, it is necessary
to first calculate the rescaling function $\eta$.
\begin{thm}
\label{thm:eta-cktrees}For a tree $T$, the rescaling function has
the following ``hook-length'' formula 
\[
\eta(T)=\frac{(\deg T)!}{T!}.
\]
\end{thm}
\begin{proof}
Proceed by induction on the number of vertices of $T$. The base case:
$\eta(\bullet)=1=\frac{1!}{1}$. 

Now take $T\neq\bullet$. As previously, write $T=[T_{1}\dots T_{f}]$,
where $T_{1},\dots,T_{f}$ are the connected components of $T$ after
removing the root. View $\Delta^{[n]}$ as $(\iota\otimes\dots\otimes\iota\otimes\Delta)\Delta^{[n-1]}$;
then the rescaling function $\eta$ counts the ways to break $T$
into singletons by pruning the vertices off one-by-one. Each such
sequence of prunings is completely determined by the sequence of prunings
(also one vertex off at a time) induced on each $T_{i}$, and a record
of which $T_{i}$ each of the first $\deg T-1$ vertices came from
(as the last vertex removed is the root). Hence 
\begin{align*}
\eta(T) & =\binom{\deg T-1}{\deg T_{1}\dots\deg T_{f}}\eta(T_{1})\dots\eta(T_{f})\\
 & =(\deg T-1)!\frac{1}{T_{1}!}\dots\frac{1}{T_{f}!}=\frac{(\deg T)!}{T!}.
\end{align*}

\end{proof}
As each tree in a forest breaks independently (Theorem \ref{thm:independence}),
it suffices to understand the Markov chain starting from a tree. The
below will give two descriptions of this: the second (Theorem \ref{thm:chain-cktrees})
is generally more natural, but the first view may be useful for some
special starting states; see Example \ref{ex: chain-cktrees1} for
the case where the starting states are stars $Q_{n}$. Depending on
the starting state, one or the other interpretation may be easier
to implement computationally.

The first interpretation is a straightforward application of the three-step
description (Theorem \ref{thm:threestep}). First take $a=2$. Then,
starting at a tree $T$ of degree $n$, one step of the Hopf-square
Markov chain is:
\begin{enumerate}[label=\arabic*.]
\item Choose $i$ ($0\leq i\leq n$) according to a symmetric binomial
distribution.
\item Pick a trunk $S$ of $T$ of degree $i$ with probability 
\[
\frac{\eta(S)\eta(T\backslash S)}{\eta(T)}=\frac{1}{\binom{n}{i}}\frac{T!}{S!\left(T\backslash S\right)!}=\frac{1}{\binom{n}{i}}\prod_{v\in S}\frac{\desc_{T}(v)}{\desc_{S}(v)}.
\]
(The second equality holds because, for $v\notin S$, $\desc_{T\backslash S}(v)=\desc_{T}(v)$.)
\item Move to $T\backslash S\amalg S$.
\end{enumerate}
Though it may be more mathematically succinct to combine the first
two steps and simply choose a trunk $S$ (of any degree) with probability
$2^{-n}\prod_{v\in S}\frac{\desc_{T}(v)}{\desc_{S}(v)}$, the advantage
of first fixing the trunk size $i$ is that then one only needs to
compute $\desc_{S}(v)$ for trunks $S$ of size $i$, not for all
trunks.
\begin{example}
\label{ex: chain-cktrees1}The star $Q_{n}$ has $\binom{n-1}{i-1}$
trunks isomorphic to $Q_{i}$ ($2\leq i\leq n$), whose cut branches
are respectively $\bullet^{n-i}$. The empty tree and $\bullet$ are
also legal trunks. Since the non-isomorphic trunks all have different
degree, the second step above is trivial: the Hopf-square Markov chain
sees $Q_{n}$ move to $Q_{i}\bullet^{n-i}$ binomially. This corresponds
to marking a corner of a rock and tracking the size of the marked
piece under the rock-breaking process of Section \ref{sec:Rock-breaking}.
Note that this is not the same as removing the leaves of $Q_{n}$
independently, as $Q_{n}$ has $n-1$ leaves, not $n$.
\end{example}
To generalise this interpretation of the $a$th Hopf-power Markov
chain to higher $a$, make use of coassociativity: $\Delta^{[a]}=(\iota\otimes\dots\otimes\iota\otimes\Delta)\Delta^{[a-1]}$.
\begin{enumerate}[label=\arabic*.]
\item Choose the trunk sizes $i_{1},\dots,i_{a}$ (with $i_{1}+\dots+i_{a}=n$)
according to a symmetric multinomial distribution.
\item Choose a trunk $S'_{(2)}$ of $T$ of degree $i_{2}+\dots+i_{a}$,
with probability $\frac{1}{\binom{n}{i_{1}}}\prod_{v\in S'_{(2)}}\frac{\desc_{T}(v)}{\desc_{S'_{(2)}}(v)}$.
\item Choose a trunk $S'_{(3)}$ of $S'_{(2)}$ of degree $i_{3}+\dots+i_{a}$,
with probability $\frac{1}{\binom{n-i_{1}}{i_{2}}}\prod_{v\in S'_{(3)}}\frac{\desc_{S'_{(2)}}(v)}{\desc_{S'_{(3)}}(v)}$.
\item Continue choosing trunks $S'_{(4)},S'_{(5)},\dots S'_{(a)}$ in the
same way, and move to $T\backslash S'_{(2)}\amalg S'_{(2)}\backslash S'_{(3)}\amalg\dots\amalg S'_{(a-1)}\backslash S'_{(a)}\amalg S'_{(a)}$.
\end{enumerate}
Here is the second, more natural description of the tree-pruning chain,
with a Jeu-de-Taquin flavour. Its inductive proof is at the end of
this section.
\begin{thm}
\label{thm:chain-cktrees}One step of the $a$th Hopf-power Markov
chain on rooted forests, starting at a tree $T$ of degree $n$, is
the following process:
\begin{enumerate}[label=\arabic*.]
\item Uniformly and independently assign one of $a$ colours to each vertex
of $T$.
\item If the root did not receive colour $a$, but there are some vertices
in colour $a$, then uniformly select one of these to exchange colours
with the root.
\item Look at the vertices $v$ with $\anc_{T}(v)=2$ (i.e. the children
of the root). Are there any of these which did not receive colour
$a$, but has descendants in colour $a$? Independently for each such
$v$, uniformly choose a vertex $u$ amongst its descendants in colour
$a$, and switch the colours of $u$ and $v$.
\item Repeat step 3 with vertices $v$ where $\anc_{T}(v)=3,4,\dots$ until
the vertices of colour $a$ form a trunk $S_{(a)}$ (i.e. no vertex
of colour $a$ is a descendant of a vertex of a different colour).
\item Repeat steps 2,3,4 with colours $a-1,a-2,\dots,1$ on the cut branches
$T\backslash S_{(a)}$ to obtain $S_{(a-1)}$. ($S_{(a-1)}$ is equivalent
to $S'_{(a-1)}\backslash S'_{(a)}$ in the alternative ``artificial''
description above.)
\item Repeat step 5 to obtain $S_{(a-1)},S_{(a-2)},\dots,S_{(1)}$. Then
move to $S_{(1)}\amalg S_{(2)}\amalg\dots\amalg S_{(a)}$.
\end{enumerate}
\end{thm}
This colour exchange process is very natural if $T$ describes the
structure of an organisation, and if $a=2,$ where colour 1 indicates
the members who leave, and colour 2 the members that stay. Then the
recolourings are simply the promotion of members to fill deserted
positions, with the assumption that the highest positions are replaced
first, and that all members working under the departing member are
equally qualified to be his or her replacement. \cite[Sec. 1]{jdt}
describes a related algorithm in a similar way.
\begin{example}
\label{ex: chain-cktrees2}Take $T=[\bullet Q_{3}]$, as labelled
in Figure \ref{fig:treedefns}: the root $r$ has two children $s$
and $t$, and $t$ has two children $u$ and $v$. Set $a=2$, and
let usual typeface denote colour 1, and \textbf{boldface} denote colour
2. Suppose step 1 above resulted in $r\mathbf{s}t\mathbf{u}v$. The
root did not receive colour 2, so, by step 2, either $s$ or $u$
must exchange colours with $r$. With probability $\frac{1}{2}$,
$u$ is chosen, and the resulting recolouring is $\mathbf{rs}tuv$.
As $\{r,s\}$ is a trunk of $T$, no more colour switching is necessary,
and the chain moves to $Q_{3}P_{2}$. If instead $s$ had exchanged
colours with $r$, then the recolouring would be $\mathbf{r}st\mathbf{u}v$.
Now step 3 is non-trivial, as $\anc_{T}(t)=1$, and $t$ is not in
colour 2, whilst its descendant $u$ is. Since $u$ is the only descendant
of $t$ in colour 2, $t$ must switch colours with $u$, resulting
in $\mathbf{r}s\mathbf{t}uv$. In this case, the chain moves to $(\bullet^{3})P_{2}$.
\end{example}

\begin{proof}[Proof of Theorem \ref{thm:chain-cktrees}, more natural description
of the chain]
Let $S'_{(a)}\subseteq S'_{(a-1)}\subseteq\dots\subseteq S'_{(1)}=T$
be nested trunks, and write $S_{(j)}$ for the cut branches $S'_{(j)}\backslash S'_{(j+1)}$.
The goal is to show that, after all colour exchanges, 
\[
P\{S_{(j)}\mbox{ ends up with colour }j\mbox{ for all }j\}=a^{-n}\prod_{j=1}^{a}\prod_{v\in S'_{(j)}}\frac{\desc_{S'_{(j)}}(v)}{\desc_{S'_{(j+1)}}(v)},
\]
as this is the probability given by the previous, more artificial,
description. Let $a'$ be maximal so that $S_{(a')}\neq\emptyset$,
so $a'$ is the last colour which appears.

The key is to condition on the colouring of $T$ after the root acquires
colour $a'$ (in the generic case where $a'=a$, this will be after
step 2). Call this colouring $\chi$, and notice that it can be any
colouring where the root has colour $a'$, and $\deg S_{(j)}$ vertices
have colour $j$. To reach this colouring after step 2, one of two
things must have happened: either the starting colouring was already
$\chi$, or some vertex $v$ that has colour $k\neq a'$ in $\chi$
originally had colour $a'$, and the root had colour $k$, and these
colours were switched in step 2. For the second scenario, there are
$\deg T-\deg S_{(a')}$ possible choices of $v$, and the chance that
the root switched colours with $v$ is $\frac{1}{\deg S_{(a')}}$.
So 
\[
P\{\mbox{colouring after step 2 is }\chi\}=a^{-n}\left(1+\frac{\deg T-\deg S_{(a')}}{\deg S_{(a')}}\right)=a^{-n}\frac{\deg T}{\deg S_{(a')}},
\]
which depends only on $\deg S_{(a')}$, the number of vertices with
the ``last used colour'' in $\chi$, and not on which colour $\chi$
assigns each specific vertex. Consequently, 
\begin{align}
 & P\{S_{(j)}\mbox{ ends up with colour }j\mbox{ for all }j\}\nonumber \\
= & \sum_{\chi}P\{S_{(j)}\mbox{ ends up with colour }j\mbox{ for all }j|\mbox{colouring after step 2 is }\chi\}\nonumber \\
 & \phantom{\sum_{\chi}}\quad\times P\{\mbox{colouring after step 2 is }\chi\}\nonumber \\
= & \sum_{\chi}P\{S_{(j)}\mbox{ ends up with colour }j\mbox{ for all }j|\mbox{colouring after step 2 is }\chi\}\left(a^{-n}\frac{\deg T}{\deg S_{(a')}}\right).\label{eq:treecolourconditioning}
\end{align}

To calculate the sum on the right hand side, proceed by induction
on $\deg T$. Write $T=[T_{1}\dots T_{f}]$ as usual, and let $\chi_{i}$
be the induced colourings on the $T_{i}$. Then, because all colour
exchanges after step 2 are between a non-root vertex and its descendant,
\begin{align*}
 & \sum_{\chi}P\{S_{(j)}\mbox{ ends up with colour }j\mbox{ for all }j|\mbox{colouring after step 2 is }\chi\}\\
= & \prod_{i=1}^{f}\sum_{\chi_{i}}P\{S_{(j)}\cap T_{i}\mbox{ ends up with colour }j\mbox{ for all }j|\mbox{starting colouring is }\chi_{i}\}.
\end{align*}
Now note that each starting colouring of $T_{i}$ has probability
$a^{-\deg T_{i}}$, so, for each $i$, 
\begin{align*}
 & P\{S_{(j)}\cap T_{i}\mbox{ ends up with colour }j\mbox{ for all }j\}\\
= & \sum_{\chi}P\{S_{(j)}\cap T_{i}\mbox{ ends up with colour }j\mbox{ for all }j|\mbox{starting colouring is }\chi_{i}\}\\
 & \phantom{\sum_{\chi}}\quad\times P\{\mbox{starting colouring is }\chi_{i}\}\\
= & a^{-\deg T_{i}}\sum_{\chi}P\{S_{(j)}\cap T_{i}\mbox{ ends up with colour }j\mbox{ for all }j|\mbox{starting colouring is }\chi_{i}\}.
\end{align*}
By inductive hypothesis, the left hand side is 
\[
a^{-\deg T_{i}}\prod_{j=1}^{a}\prod_{v\in S'_{(j)}\cap T_{i}}\frac{\desc_{S'_{(j)}\cap T_{i}}(v)}{\desc_{S'_{(j+1)}\cap T_{i}}(v)}=a^{-\deg T_{i}}\prod_{j=1}^{a}\prod_{v\in S'_{(j)}\cap T_{i}}\frac{\desc_{S'_{(j)}}(v)}{\desc_{S'_{(j+1)}}(v)}.
\]
So, returning to (\ref{eq:treecolourconditioning}),
\begin{align*}
 & P\{S_{(j)}\mbox{ ends up with colour }j\mbox{ for all }j\}\\
= & \sum_{\chi}P\{S_{(j)}\mbox{ ends up with colour }j\mbox{ for all }j|\mbox{colouring after step 2 is }\chi\}\left(a^{-n}\frac{\deg T}{\deg S_{(a')}}\right)\\
= & \prod_{i=1}^{f}\sum_{\chi_{i}}P\{S_{(j)}\cap T_{i}\mbox{ ends up with colour }j\mbox{ for all }j|\mbox{starting colouring is }\chi_{i}\}\left(a^{-n}\frac{\deg T}{\deg S_{(a')}}\right)\\
= & \left(\prod_{i=1}^{f}\prod_{j=1}^{a}\prod_{v\in S'_{(j)}\cap T_{i}}\frac{\desc_{S'_{(j)}}(v)}{\desc_{S'_{(j+1)}}(v)}\right)\left(a^{-n}\frac{\deg T}{\deg S_{(a')}}\right)\\
= & a^{-n}\left(\prod_{j=1}^{a}\prod_{v\in S'_{(j)}\cap(\cup T_{i})}\frac{\desc_{S'_{(j)}}(v)}{\desc_{S'_{(j+1)}}(v)}\right)\frac{\deg T}{\deg S_{(a')}}\\
= & a^{-n}\prod_{j=1}^{a}\prod_{v\in S'_{(j)}}\frac{\desc_{S'_{(j)}}(v)}{\desc_{S'_{(j+1)}}(v)},
\end{align*}
since the root is the only vertex not in any $T_{i}$, and it is necessarily
in $S'_{(a')}$. 
\end{proof}

\subsection{Right Eigenfunctions\label{sub:Right-Eigenfunctions-CKtrees}}

The aim of this section is to apply Proposition \ref{prop:probboundseasyrightefns}
to the special right eigenfunctions $\f_{C}$ ($C$ a tree) to bound
the probability that the tree-pruning Markov chain can still reach
$C\bullet\dots\bullet$ after a large number of steps. ($C$ is in
capital here in contrast to Section \ref{sub:Altrightefns} as lowercase
letters typically indicate vertices of trees.) Observe that being
able to reach $C\bullet\dots\bullet$ is equivalent to containing
$C$ as a subtree.

More non-standard notation: for a vertex $v$ in a forest $T$, let
$\Anc_{T}(v)$ denote the set of ancestors of $v$ in $T$, including
$v$ itself. So $\Anc_{T}(v)$ comprises the vertices on the path
from $v$ to the root of the connected component of $T$ containing
$v$, including both endpoints.
\begin{thm}
\label{thm:rightefn-cktrees}Let $C\neq\bullet$ be a tree, and $T$
a forest. Then the right eigenfunction $\f_{C}$, of eigenvalue $a^{-\deg C+1}$,
is 
\[
\f_{C}(T)=\frac{1}{\deg C!}\sum_{C\subseteq T}\left(\left(\prod_{v\in\Anc_{T}(\Root(C))}\frac{\desc_{T}(v)}{\desc_{T}(v)-\deg C+1}\right)\left(\prod_{v\in C,v\neq\Root(C)}\desc_{T}(v)\right)\right),
\]
where the sum is over all subtrees of $T$ isomorphic to $C$, though
not necessarily with the same root. Moreover,
\[
\frac{C!\left|\left\{ C\subseteq T\right\} \right|}{\deg C!\deg C}\leq\f_{C}(T)\leq\binom{n'}{\deg C}\frac{\left|\left\{ C\subseteq T\right\} \right|}{(n'-\deg C+1)}
\]
 where $\left|\left\{ C\subseteq T\right\} \right|$ is the number
of subtrees of $T$ isomorphic to $C$ (not necessarily with the same
root), and $n'$ is the degree of the largest component of $T$.
\end{thm}
The proof is fairly technical, so it is at the end of this section. 

\begin{rems*}
$ $

\begin{enumerate}[label=\arabic*.]
\item The second product in the expression for $\f_{C}(T)$ is not $C!$,
since the product is over vertices of $C$, but the count is of the
descendants in $T$. 
\item The denominators $\desc_{T}(v)-\deg C+1$ are positive, since, if
$v\in\Anc_{T}(\Root(C))$, then all vertices of $C$ are descendants
of $v$.
\item The lower bound above is sharp: let $C=[Q_{3}]$, $T=[Q_{3}P_{n-4}]$.
Then $\f_{C}(T)=\frac{1}{4!}\frac{n}{n-3}\cdot1\cdot1\cdot2=\frac{1}{12}\frac{n}{n-3}$,
which has limit $\frac{1}{12}=\frac{8\cdot1}{4!4}=\frac{C!\left|\left\{ C\subseteq T\right\} \right|}{\deg C!\deg C}$,
equal to the above lower bound, as $n\rightarrow\infty$.
\item The upper bound above is attained whenever $C$ and $T$ are both
paths. In this case, the contribution to $\f_{C}(T)$ from the copy
of $C$ whose root is distance $n-i$ from $\Root(T)$ ($0\leq i\leq n-\deg C$)
is 
\begin{align*}
 & \frac{1}{\deg C!}\frac{n}{n-\deg C+1}\frac{n-1}{n-\deg C}\dots\frac{i}{i-\deg C+1}(i-1)\dots(i-\deg C+1)\\
= & \frac{1}{\deg C!}n(n-1)\dots(n-\deg C+2)=\binom{n}{\deg C}\frac{1}{(n-\deg C+1)}.
\end{align*}

\end{enumerate}
\end{rems*}

Combining these bounds on $\f_{C}(T)$ with Proposition \ref{prop:expectationrightefns}.ii
gives the first of the two probability bounds below. The second result
uses the universal bound of Proposition \ref{prop:probboundseasyrightefns}.
\begin{cor}
\label{cor:probboundsrightefns-cktrees}Let $\{X_{m}\}$ be the $a$th
Hopf-power tree-pruning chain, started at a forest $T$. Write $n'$
for the degree of the largest component of $T$. Then
\begin{align*}
 & P\{X_{m}\supseteq C|X_{0}=T\}\\
\leq & E\{|\{\mbox{subtrees of }X_{m}\mbox{ isomorphic to }C\}||X_{0}=T\}\\
\leq & \frac{a^{(-\deg C+1)m}\deg C!\deg C}{C!}\f_{C}(T)\\
\leq & \frac{a^{(-\deg C+1)m}\deg C!\deg C}{C!}\binom{n}{\deg C}\frac{\left|\left\{ C\subseteq T\right\} \right|}{(n-\deg C+1)}.
\end{align*}
Besides, for any starting distribution on forests of $n$ vertices,
\[
P\{X_{m}\supseteq C\}\leq a^{-\deg C+1}\deg C\binom{n}{\deg C}.
\]
 \qed\end{cor}
\begin{example}
\label{ex:probboundsrightefns-cktrees}Here is a demonstration of
how to calculate with the formulae. Take $T=[\bullet Q_{3}]$ as in
Figure \ref{fig:treedefns}, and calculate $\f_{Q_{3}}(T)$. As noted
in Example \ref{ex:treedefns}, $T$ has two subgraphs isomorphic
to $Q_{3}$, namely that spanned by $\{r,s,t\}$ and by $\{t,u,v\}$.
The set of ancestors $\Anc_{T}[\Root(Q_{3})]$ is solely $r$ for
the first copy of $Q_{3}$, and for the second copy of $Q_{3}$, it
is $\{r,t\}$. Hence 
\begin{align*}
\f_{Q_{3}}(T) & =\frac{1}{\deg Q_{3}!}\left(\frac{\desc_{T}(r)}{\desc_{T}(r)-\deg Q_{3}+1}\desc_{T}(s)\desc_{T}(t)\right.\\
 & \left.\quad+\frac{\desc_{T}(r)}{\desc_{T}(r)-\deg Q_{3}+1}\frac{\desc_{T}(t)}{\desc_{T}(z)-\deg Q_{2}+1}\desc_{T}(u)\desc_{T}(v)\right)\\
 & =\frac{1}{6}\left(\frac{5}{3}\cdot1\cdot3+\frac{5}{3}\frac{3}{1}\cdot1\cdot1\right)=\frac{5}{3}.
\end{align*}
So, after $m$ steps of the Hopf-square pruning chain started at $T,$
the probability that there is still a vertex with at least two children
is at most $\frac{2^{-2m}3!2}{3}\frac{5}{3}=2^{-2m}\frac{20}{3}$.
\end{example}

\begin{example}
\label{ex:probboundsrightefns-cktreesp}Specialise Corollary \ref{cor:probboundsrightefns-cktrees}
to $C=P_{j}$, a path with $j$ vertices. The copies of $P_{j}$ in
a tree $T$ are in bijection with the vertices of $T$ with at least
$j$ ancestors, by sending a path to its ``bottommost'' vertex (the
one furthest from the root). There can be at most $\deg T-j+1$ vertices
with $j$ or more ancestors, so by Corollary \ref{cor:probboundsrightefns-cktrees},
\begin{align*}
{\normalcolor {\normalcolor }} & P\{X_{m}\mbox{ has a vertex with}\geq j\mbox{ ancestors}|X_{0}=T\}\\
\leq & E\{|\{\mbox{vertices of }X_{m}\mbox{ with}\geq j\mbox{ ancestors}\}||X_{0}=T\}\\
\leq & \frac{a^{(-j+1)m}j}{(\deg T-j+1)}\binom{\deg T}{j}|\{\mbox{vertices of }T\mbox{ with}\geq j\mbox{ ancestors}\}|\\
\leq & a^{(-j+1)m}j\binom{\deg T}{j}.
\end{align*}
This result holds for any starting state $T$. In the particular case
where $T$ is the path $P_{n}$, this shows that, for the multinomial
rock-breaking process of Section \ref{sec:Rock-breaking} started
at a single rock of size $n$, 
\[
P\{X_{m}\mbox{ contains a piece of size }\geq j|X_{0}=(n)\}\leq a^{(-j+1)m}j\binom{n}{j},
\]
which is looser than the bound in Proposition \ref{prop:probbounds-rocks}
by a factor of $j$.
\end{example}

\begin{example}
\label{ex:probboundsrightefns-cktreesq}Take $C=Q_{j}$ , the star
with $j$ vertices. Then $X_{m}\supseteq Q_{j}$ if and only if $X_{m}$
has a vertex with at least $j-1$ children. Each vertex with $d$
children is responsible for $\binom{d}{j-1}$ copies of $Q_{j}$,
so the two bounds in Corollary \ref{cor:probboundsrightefns-cktrees}
are 
\begin{align*}
{\normalcolor {\normalcolor }} & P\{X_{m}\mbox{ has a vertex with }\geq j-1\mbox{ children}|X_{0}=T\}\\
\leq & E\{|\{\mbox{vertices of }X_{m}\mbox{ with}\geq j-1\mbox{ children}\}||X_{0}=T\}\\
\leq & E\{|\{\mbox{subtrees of }X_{m}\mbox{ isomorphic to }Q_{j}\}||X_{0}=T\}\\
\leq & \frac{a^{(-j+1)m}j!}{\deg T-j+1}\binom{\deg T}{j}\sum_{v\in T}\binom{|\{\mbox{children of }v\}|}{j-1},
\end{align*}
\[
P\{X_{m}\mbox{ has a vertex with }\geq j-1\mbox{ children}\}\leq a^{(-j+1)m}j\binom{\deg T}{j}.
\]
The first bound is tighter if $T$ has high degree compared to $j$,
and has few vertices with at least $j$ children.\end{example}
\begin{proof}[Proof of Theorem \ref{thm:rightefn-cktrees}]
The following inductive argument proves both the expression for $\f_{C}(T)$
and the upper bound. To then obtain the lower bound, note that, for
any vertex $v$, $\frac{\desc_{T}(v)}{\desc_{T}(v)-\deg C+1}\geq1$,
and for a subtree $C\subseteq T$, 
\[
\prod_{v\in C,v\neq\Root(T)}\desc_{T}(v)\geq\prod_{v\in C,v\neq\Root(T)}\desc_{C}(v)=\frac{C!}{\deg C}.
\]

To simply notation, write $C_{T}!$ for $\prod_{v\in C,v\neq\Root(T)}\desc_{T}(v)$,
since $C_{C}!=\frac{1}{\deg C}C!$. First, reduce both the expression
for $\f_{C}(T)$ and the upper bound to the case when $T$ is a tree:
the claimed expression for $\f_{C}(T)$ is additive in the sense of
Proposition \ref{prop:rightefnsofproducts}, and $\binom{n}{\deg C}\frac{1}{n-\deg C+1}=n(n-1)\dots(n-\deg C+2)$
is increasing in $n$. By definition of $\f_{C}$ in Equation \ref{eq:easyrightefns}
and the calculation of $\eta(T)$ in Theorem \ref{thm:eta-cktrees},
the goal is to prove 
\begin{align}
 & \eta_{T}^{C,\bullet,\dots,\bullet}+\eta_{T}^{\bullet,C,\bullet,\dots,\bullet}+\dots+\eta_{T}^{\bullet,\dots,\bullet,C}\label{eq:rightefns-cktrees}\\
= & \frac{(\deg T-\deg C+1)!}{T!}\sum_{C\subseteq T}\left(\left(\prod_{v\in\Anc_{T}(\Root(C))}\frac{\desc_{T}(v)}{\desc_{T}(v)-\deg C+1}\right)C_{T}!\right)\nonumber \\
\leq & |\{C\subseteq T\}|\frac{\deg T!}{T!}.\nonumber 
\end{align}
The key is again to write $T=[T_{1}\dots T_{f}]$ and induct on degree.
(The base case: when $T=\bullet$, both sides are zero, as there are
no copies of $C$ in $T$ since $C\neq\bullet$.) The left hand side
of (\ref{eq:rightefns-cktrees}) counts the ways to prune $T$ successively
so that it results in one copy of $C$ and singletons. Divide this
into two cases: $\eta_{T}^{\bullet,\dots,\bullet,C}$ counts the successive
pruning processes where $C\ni\Root(T)$; the sum of the other coproduct
structure constants in (\ref{eq:rightefns-cktrees}) counts the successive
pruning processes where $C\not\ni\Root(T)$, so $C\subseteq T_{i}$
for some $i$. The inductive proof below handles these cases separately,
to show that
\begin{align}
 & \eta_{T}^{\bullet,\dots,\bullet,C}\label{eq:rightefns-rooted}\\
= & \frac{(\deg T-\deg C+1)!}{T!}\sum_{\substack{C\subseteq T\\
C\ni\Root(T)
}
}\left(\left(\prod_{v\in\Anc_{T}(\Root(C))}\frac{\desc_{T}(v)}{\desc_{T}(v)-\deg C+1}\right)C_{T}!\right)\nonumber \\
\leq & |\{C\subseteq T|C\ni\Root(T)\}|\frac{\deg T!}{T!};\nonumber 
\end{align}
\begin{align}
 & \eta_{T}^{C,\bullet,\dots,\bullet}+\dots+\eta_{T}^{\bullet,\dots,C,\bullet}\label{eq:rightefns-unrooted}\\
= & \frac{(\deg T-\deg C+1)!}{T!}\sum_{\substack{C\subseteq T\\
C\not\ni\Root(T)
}
}\left(\left(\prod_{v\in\Anc_{T}(\Root(C))}\frac{\desc_{T}(v)}{\desc_{T}(v)-\deg C+1}\right)C_{T}!\right)\nonumber \\
\leq & |\{C\subseteq T|C\not\ni\Root(T)\}|\frac{\deg T!}{T!}.\nonumber 
\end{align}
Adding these together then gives (\ref{eq:rightefns-cktrees}).

The argument for (\ref{eq:rightefns-unrooted}) is simpler (though
it relies on (\ref{eq:rightefns-cktrees}) holding for $T_{1},\dots,T_{f}$).
The ways to successively prune $T$ into singletons and one copy of
$C$ not containing $\Root(T)$ correspond bijectively to the ways
to prune some $T_{i}$ into singletons and one copy of $C$ (which
may contain $\Root(T_{i})$) and all other $T_{j}$ into singletons,
keeping track of which $T_{j}$ was pruned at each step. Hence, writing
$d_{i}$ for $\deg T_{i}$, 
\begin{align*}
 & \eta_{T}^{C,\bullet,\dots,\bullet}+\dots+\eta_{T}^{\bullet,\dots,C,\bullet}\\
= & \sum_{i}\binom{\deg T-\deg C}{d_{1}\dots d_{i-1}\ d_{i}-\deg C+1\ d_{i+1}\dots d_{f}}\left(\eta_{T_{i}}^{C,\bullet,\dots,\bullet}++\dots+\eta_{T_{i}}^{\bullet,\dots,\bullet,C}\right)\prod_{j\neq i}\eta(T_{j}).
\end{align*}
 Use Theorem \ref{thm:eta-cktrees} and the inductive hypothesis of
(\ref{eq:rightefns-cktrees}) to substitute for $\eta(T_{j})$ and
$\eta_{T_{i}}^{C,\bullet,\dots,\bullet}++\dots+\eta_{T_{i}}^{\bullet,\dots,\bullet,C}$
respectively:
\begin{align}
 & \eta_{T}^{C,\bullet,\dots,\bullet}+\dots+\eta_{T}^{\bullet,\dots,C,\bullet}\label{eq:rightefns-unrooted2}\\
= & \frac{(\deg T-\deg C)!}{T_{1}!\dots T_{f}!}\sum_{i}\sum_{\substack{C\subseteq T_{i}\\
C\not\ni\Root(T)
}
}\left(\left(\prod_{v\in\Anc_{T_{i}}(\Root(C))}\frac{\desc_{T_{i}}(v)}{\desc_{T_{i}}(v)-\deg C+1}\right)C_{T_{i}}!\right),\nonumber 
\end{align}
and
\begin{align*}
 & \eta_{T}^{C,\bullet,\dots,\bullet}+\dots+\eta_{T}^{\bullet,\dots,C,\bullet}\\
\leq & \frac{(\deg T-\deg C)!}{T_{1}!\dots T_{f}!}\sum_{i}\frac{d_{i}!}{(d_{i}-\deg C+1)!}|\{C\subseteq T_{i}\}|.
\end{align*}
To deduce the equality in (\ref{eq:rightefns-unrooted}), first rewrite
the fraction outside the sum in (\ref{eq:rightefns-unrooted2}) as
\[
\frac{\deg T}{\deg T-\deg C+1}\frac{(\deg T-\deg C+1)!}{T!}.
\]
Then it suffices to show that 
\begin{align*}
 & \frac{\deg T}{\deg T-\deg C+1}\sum_{i}\sum_{\substack{C\subseteq T_{i}\\
C\not\ni\Root(T)
}
}\left(\left(\prod_{v\in\Anc_{T_{i}}(\Root(C))}\frac{\desc_{T_{i}}(v)}{\desc_{T_{i}}(v)-\deg C+1}\right)C_{T_{i}}!\right)\\
= & \sum_{\substack{C\subseteq T\\
C\not\ni\Root(T)
}
}\left(\left(\prod_{v\in\Anc_{T}(\Root(C))}\frac{\desc_{T}(v)}{\desc_{T}(v)-\deg C+1}\right)C_{T}!\right).
\end{align*}
Now note that, for $C\subseteq T_{i}$, $\Anc_{T}(\Root(C))=\Anc_{T_{i}}(\Root(C))\cup\Root(T)$.
For each $v\in\Anc_{T_{i}}(\Root(C))$ , $\desc_{T_{i}}(v)=\desc_{T}(v)$,
and for $v=\Root(T)$, $\frac{\desc_{T_{i}}(v)}{\desc_{T}(v)-\deg C+1}=\frac{\deg T}{\deg T-\deg C+1}$.
As for the inequality: for each $i$, $d_{i}=\deg T_{i}\leq\deg T-1$,
so 
\[
\frac{(\deg T-\deg C)!d_{i}!}{(d_{i}-\deg C+1)!}=d_{i}(d_{i}-1)\dots(d_{i}-\deg C+2)(\deg T-\deg C)!\leq(\deg T-1)!.
\]

Now turn to the case where $C\ni\Root(T)$. Then $\Anc_{T}(\Root(C))=\Root(T)$,
so (\ref{eq:rightefns-rooted}) simplifies to 
\begin{equation}
\eta_{T}^{\bullet,\dots,\bullet,C}=\frac{(\deg T-\deg C)!}{T!}\sum_{\substack{C\subseteq T\\
C\ni\Root(T)
}
}\prod_{v\in C}\desc_{T}(v)\leq|\{C\subseteq T|C\ni\Root(T)\}|\frac{\deg T!}{T!}.\label{eq:rightefns-rooted2}
\end{equation}
Here $C\ni\Root(T)$ means that $C\not\subseteq T_{i}$ for any $i$,
hence a proof based on $T=[T_{1}\dots T_{f}]$ will need to consider
several $C$'s (in contrast to the previous paragraph when $C$ did
not contain $\Root(T)$). Note first that (\ref{eq:rightefns-rooted2})
does hold for $C=\emptyset$ (both sides are zero) and $C=\bullet$
(the formula for $\eta(T)$ as in Theorem \ref{thm:eta-cktrees})
- these cases are not part of the theorem, but are useful for the
proof. For $C\neq\emptyset,\bullet$, write $C=[C_{1}\dots C_{f'}]$;
necessarily $f'\leq f$ or there would be no copy of $C$ in $T$
with $C\ni\Root(T)$ (then both sides of (\ref{eq:rightefns-rooted2})
are zero). For ease of notation, let $C{}_{f'+1}=\dots=C{}_{f}=\emptyset$.
Recall that $\eta_{T}^{\bullet,\dots,\bullet,C}$ counts the number
of ways to successively prune vertices from $T$ to leave $C$. This
is equivalent to successively pruning vertices from each $T_{i}$
to leave $C{}_{1}$ in some $T_{\sigma(1)}$, $C{}_{2}$ in some $T_{\sigma(2)}$,
etc, and keeping track of which $T_{i}$ was pruned at each step.
Thus
\[
\eta_{T}^{\bullet,\dots,\bullet,C}=\sum_{\sigma}\binom{\deg T-\deg C}{\deg T_{\sigma(1)}-\deg C{}_{1}\dots\deg T_{\sigma(f)}-\deg C{}_{f}}\eta_{T_{\sigma(1)}}^{\bullet,\dots,\bullet,C_{1}}\dots\eta_{T_{\sigma(f)}}^{\bullet,\dots,\bullet,C_{f}},
\]
where the sum is over one choice of $\sigma\in S_{f}$ for each distinct
multiset of pairs $\left\{ (C_{1},T{}_{\sigma(1)}),\dots,(C_{f},T{}_{\sigma(f)})\right\} $.
The inductive hypothesis of (\ref{eq:rightefns-unrooted}) for $(C_{i},T_{\sigma(i)})$
then yields 
\begin{align*}
\eta_{T}^{\bullet,\dots,\bullet,C} & =\sum_{\sigma}(\deg T-\deg C)!\prod_{i}\sum_{\substack{C{}_{i}\subseteq T_{\sigma(i)}\\
C_{i}\ni\Root(T_{\sigma(i)})
}
}\frac{1}{T_{\sigma(i)}!}\prod_{v\in C_{i}}\desc_{T_{\sigma(i)}}(v)\\
 & =\frac{\deg T}{T!}\sum_{\substack{C\subseteq T\\
C\ni\Root(T)
}
}\frac{(\deg T-\deg C)!}{\deg T!}\prod_{v\in C,v\neq\Root(T)}\desc_{T}(v),
\end{align*}
since $\Root(T)=C\backslash\amalg C_{i}$, and for each $v\in C_{i}$,
$\desc_{T_{\sigma(i)}}(v)=\desc_{T}(v)$. To conclude the equality
in (\ref{eq:rightefns-rooted2}), simply absorb the factor of $\deg T$
at the front into the product as $\desc_{T}(\Root(T))$. Also by the
inductive hypothesis,
\[
\eta_{T}^{\bullet,\dots,\bullet,C}\leq\sum_{\sigma}(\deg T-\deg C)!\prod_{i}\frac{\deg T_{\sigma(i)}!}{(\deg T_{\sigma(i)}-\deg C{}_{i})!T_{\sigma(i)}!}|\{C_{i}\subseteq T_{\sigma(i)}|C_{i}\ni\Root(T_{\sigma(i)})\}|
\]
 Now $\frac{\deg T_{\sigma(i)}!}{(\deg T_{\sigma(i)}-\deg C{}_{i})!}$
enumerates the ways to choose $\deg C_{i}$ ordered objects amongst
$\deg T_{\sigma(i)}$; choosing such objects for each $i$ is a subset
of the ways to choose $\deg C-1$ objects from $\deg T-1$. Hence
\begin{align*}
\eta_{T}^{\bullet,\dots,\bullet,C} & \leq\frac{(\deg T-1)!}{(\deg T-\deg C)!}\sum_{\sigma}(\deg T-\deg C)!\prod_{i}\frac{1}{T_{\sigma(i)}!}|\{C_{i}\subseteq T_{\sigma(i)}|C_{i}\ni\Root(T_{\sigma(i)})\}|\\
 & =\frac{\deg T!}{T!}|\{C\subseteq T|C\ni\Root(T)\}|
\end{align*}
as claimed.
\end{proof}

\chapter{Hopf-power Markov Chains on Cofree Commutative Algebras\label{chap:cofree}}

\chaptermark{Chains on Cofree Commutative Algebras}

Sections \ref{sec:Riffle-Shuffling} and \ref{sec:Descent-Sets} study
in detail respectively the chains of riffle-shuffling and of the descent
set under riffle-shuffling. These arise from the shuffle algebra and
the algebra of quasisymmetric functions, which are both cofree and
commutative.

\section{Riffle-Shuffling\label{sec:Riffle-Shuffling}}

Recall from Chapter \ref{chap:Introduction} the Gilbert-Shannon-Reeds
model of riffle-shuffling of a deck of cards: cut the deck binomially
into two piles, then choose uniformly an interleaving of the two piles.
The first extensive studies of this model are \cite[Sec. 4]{shufflepersialdous}
and \cite{bd}. They give explicit formulae for all the transition
probabilities and find that $\frac{3}{2}\log n$ shuffles are required
to mix a deck of $n$ distinct cards. More recently, \cite{shufflenondistinct}
derives the convergence rate for decks of repeated cards, which astonishingly
depends almost entirely on the total number of cards and the number
of distinct values that they take. The number of cards of each value
hardly influences the convergence rate.

One key notion introduced in \cite{bd} is the generalisation of the
GSR model to $a$-handed shuffles, which cuts the deck into $a$ piles
multinomially before uniformly interleaving. As Example \ref{ex:shuffle-threestep}
showed, $a$-handed shuffling is exactly the $a$th Hopf-power Markov
chain on the shuffle algebra $\calsh$, with respect to its basis
of words. In $\calsh$, the product of two words is the sum of their
interleavings (with multiplicity), and the coproduct of a word is
the sum of its deconcatenations - see Example \ref{ex:shufflealg}.
As mentioned in Section \ref{sec:Combinatorial-Hopf-algebras}, $\calsh$
has a multigrading: for a sequence of non-negative integers $\nu$,
the subspace $\calsh_{\nu}$ is spanned by words where 1 appears $\nu_{1}$
times, 2 appears $\nu_{2}$ times, etc. The Hopf-power Markov chain
on $\calsh_{\nu}$ describes shuffling a deck of \emph{composition}
$\nu$, where there are $\nu_{i}$ cards with face value $i$. For
example, $\nu=(1,1,\dots,1)$ corresponds to a deck where all cards
are distinct, and $\nu=(n-1,1)$ describes a deck with one distinguished
card, as studied in \cite[Sec. 2]{biasedshufflesamisoundpersi}. Work
with the following partial order on deck compositions: $\nu\geq\nu'$
if $\nu_{i}\geq\nu'_{i}$ for all $i$. Write $|\nu|$ for the sum
of the entries of $\nu$ - this is the number of cards in a deck with
composition $\nu$. For a word $w$, let $\deg(w)$ denote its corresponding
deck composition (this is also known as the \emph{evaluation} $\ev(w)$),
and $|w|=|\deg w|$ the total number of cards in the deck. For example,
$\deg((1233212))=(2,3,2)$, and $|1233212|=7$. Since the cards behave
equally independent of their values, there is no harm in assuming
$\nu_{1}\geq\nu_{2}\geq\cdots$. In other words, it suffices to work
with $\calh_{\nu}$ for partitions $\nu$, though what follows will
not make use of this reduction.

A straightforward application of Theorem \ref{thm:hpmc-stationarydistribution}
shows that the stationary distribution of riffle-shuffling is the
uniform distribution, for all powers $a$ and all deck compositions
$\nu$.

Sections \ref{sub:rightefns-Shuffling} and \ref{sub:leftefns-Shuffling}
construct some simple right and left eigenfunctions using Parts B$'$
and A$'$ of Theorem \ref{thm:diagonalisation} respectively, and
Section \ref{sub:Duality-shuffle} gives a partial duality result.
All this relies on the Lyndon word terminology of Section \ref{sec:Lyndon-Words}.
Much of the right eigenfunction analysis is identical to \cite[Sec. 5]{hopfpowerchains},
which studies inverse shuffling as the Hopf-power Markov chain on
the free associative algebra; the left eigenfunction derivations here
are new. In the case of distinct cards, these right and left eigenfunctions
have previously appeared in \cites[Sec. 4]{shuffleefns1}[Th. 3.6]{shuffleefns2}
and \cite{johnpike} respectively. All these examine the time-reversal
of riffle-shuffling in the context of walks on hyperplane arrangements
and their generalisation to left regular bands.

\subsection{Right Eigenfunctions\label{sub:rightefns-Shuffling}}

Recall from Proposition \ref{prop:efns}.R that the right eigenfunctions
of a Hopf-power Markov chain come from diagonalising $\Psi^{a}$ on
the dual of the underlying Hopf algebra. For the case of riffle-shuffling,
this dual is the free associative algebra (Example \ref{ex:shufflealg-dual}),
with concatentation product and deshuffling coproduct. The word basis
of the free associative algebra fits the hypothesis of Theorem \ref{thm:diagonalisation}.B$'$.
All single letters have degree 1, so there is no need to apply the
Eulerian idempotent, which simplifies the algorithm a little. In order
to achieve $\sum_{w}\f_{w}(v)\g_{w}(v)=1$ for some $w$, with the
left eigenbasis $\g_{w}$ in Section \ref{sub:leftefns-Shuffling}
below, it will be necessary to divide the output of Theorem \ref{thm:diagonalisation}.B$'$
by an extra factor $Z(w)$, the size of the stabiliser of $\mathfrak{S}_{k(w)}$
permuting the Lyndon factors of $w$. For example, $(31212)$ has
Lyndon factorisation $(3\cdot12\cdot12)$, and the stabiliser of $\mathfrak{S}_{3}$
permuting these factors comprises the identity and the transposition
of the last two elements, so $Z((31212))=2$. 

Coupling this rescaled version of Theorem \ref{thm:diagonalisation}.B'
with Proposition \ref{prop:efns}.R, $\f_{w'}(w)$ is the coefficient
of $w$ in:
\begin{alignat*}{2}
f_{w'} & :=w' & \quad & \mbox{if }w'\mbox{ is a single letter};\\
f_{w'} & :=f_{u_{1}}f_{u_{2}}-f_{u_{2}}f_{u_{1}} & \quad & \mbox{if }w'\mbox{ is Lyndon with standard factorisation }w'=u_{1}\cdot u_{2};\\
f_{w'} & :=\frac{1}{Z(w')k!}\sum_{\sigma\in\sk}f_{u_{\sigma(1)}}\dots f_{u_{\sigma(k)}} & \quad & \mbox{if }w'\mbox{ has Lyndon factorisation }w'=u_{1}\cdot\dots\cdot u_{k}.
\end{alignat*}
(The second line is a recursive definition for the \emph{standard
bracketing}.) A visual description of $\f_{w'}(w)$ is two paragraphs
below.

$\f_{w'}$ is a right eigenfunction of eigenvalue $a^{-|w'|+k(w')}$,
where $k(w')$ is the number of Lyndon factors of $w'$. Since the
$\f_{w'}$ form an eigenbasis, the multiplicity of the eigenvalue
$a^{-|\nu|+k}$ when shuffling a deck of composition $\nu$ is the
number of words of degree $\nu$ with $k$ Lyndon factors. This has
two consequences of note. Firstly, when $\nu=(1,1,\dots,1)$, this
multiplicity is $c(|\nu|,k)$, the signless Stirling number of the
first kind. Its usual definition is the number of permutations of
$|\nu|$ objects with $k$ cycles, which is easily equivalent \cite[Prop. 1.3.1]{stanleyec1}
to the number of words of $\deg(\nu)$ with $k$ record minima. (The
letter $i$ is a \emph{record minima} of $w$ if all letters appearing
before $i$ in $w$ are greater than $i$.) This is the eigenvalue
multiplicity because a word with distinct letters is Lyndon if and
only if its first letter is minimal, so the Lyndon factors of a word
with distinct letters start precisely at the record minima.

Secondly, for general $\nu$, the eigenvalues $1,a^{-1},\dots,a^{-|\nu|+1}$
all occur. Each eigenfunction $\f_{w'}$ of eigenvalue $a^{-|\nu|+k}$
corresponds to a word of degree $\nu$ with $k$ Lyndon factors, or
equivalently, $k$ Lyndon words whose degrees sum to $\nu$. One way
to find such a $k$-tuple is to choose a Lyndon word of length $|\nu|-k+1$
in which letter $i$ occurs at most $\nu_{i}$ times, and take the
remaining $k-1$ letters of $\nu$ as singleton Lyndon factors. How
to construct the non-singleton Lyndon factor depends on $\nu$ and
$k$: if $|\nu|-k>\nu_{1}$, one possibility is the smallest $|\nu|-k+1$
values in increasing order. For $|\nu|-k\leq\nu_{1}$, take the word
with $|\nu|-k$ 1s followed by a 2. 

As for the eigenvectors, \cite[Sec.~2]{descentalg} and \cite{trees}
provide a way to calculate them graphically, namely via \emph{decreasing
Lyndon hedgerows}. For a Lyndon word $u$ with standard factorisation
$u=u_{1}\cdot u_{2}$, inductively draw a rooted binary tree $T_{u}$
by taking $T_{u_{1}}$ as the left branch and $T_{u_{2}}$ as the
right branch. Figure \ref{fig:lyndonhedgerows1} shows $T_{(13245)}$
and $T_{(1122)}$.

\begin{figure}
\begin{centering}
\includegraphics[scale=0.4]{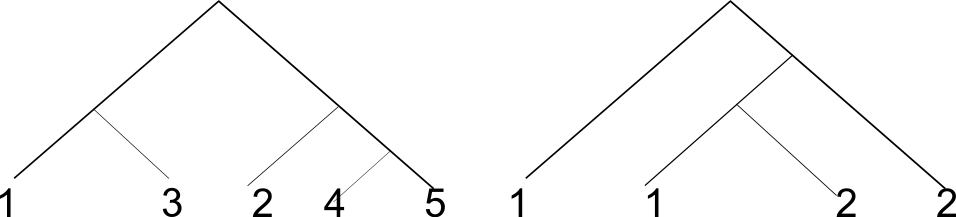}
\par\end{centering}

\caption{\label{fig:lyndonhedgerows1}The trees $T_{(13245)}$ and $T_{(1122)}$}
\end{figure}

For a Lyndon word $u$, it follows from the recursive definition of
$f_{u}$ above that $\f_{u}(w)$ is the signed number of ways to exchange
the left and right branches at some vertices of $T_{u}$ so that the
leaves of $T_{u}$, reading from left to right, spell out $w$ (the
sign is the parity of the number of exchanges required). For example, 
\begin{itemize}
\item $\f_{(13245)}((25413))=1$ since the unique way to rearrange $T_{(13245)}$
so the leaves spell $(25413)$ is to exchange the branches at the
root and the lowest interior vertex; 
\item $\f_{(13245)}((21345))=0$ since in all legal rearrangements of $T_{(13245)}$,
2 appears adjacent to either 4 or 5, which does not hold for $(21345)$; 
\item $\f_{(1122)}((1221))=0$ as there are two ways to make the leaves
of $T_{(1122)}$ spell $(1221)$: either exchange branches at the
root, or exchange branches at both of the other interior vertices.
These two rearrangements have opposite signs, so the signed count
of rearrangements is 0. 
\end{itemize}
Now for general $w'$ with Lyndon factorisation $w=u_{1}\cdot\dots\cdot u_{k}$,
set $T_{w'}$ to be simply $T_{u_{1}},T_{u_{2}},\dots,T_{u_{k}}$
placed in a row. So $T_{(35142)}$ is the hedgerow in Figure \ref{fig:lyndonhedgerows2}.

\begin{figure}
\begin{centering}
\includegraphics[scale=0.4]{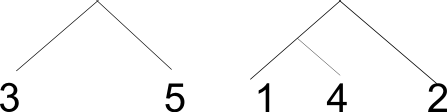}
\par\end{centering}

\caption{\label{fig:lyndonhedgerows2}The Lyndon hedgerow $T_{(35142)}$}
\end{figure}

Again $\f_{w'}(w)$ is the signed number of ways to rearrange $T_{w'}$
so the leaves spell $w$, divided by $k!$. Now there are two types
of allowed moves: exchanging the left and right branches at a vertex
(as before), and permuting the trees of the hedgerow. The latter move
does not come with a sign. Thus $\f_{(35142)}((14253))=\frac{1}{2!}(-1)$,
as the unique rearrangement of $T_{(35142)}$ which spells $(14253)$
requires transposing the trees and permuting the branches of 3 and
5. The division by $Z(w')$ in the definition of $\f_{w'}$ means
that, if $w'$ has a repeat Lyndon factor, then the multiple trees
corresponding to this repeated factor are not distinguished, and transposing
them does not count as a valid rearrangement. So if $w'=(31212)=(3\cdot12\cdot12)$,
then $\f_{w'}((12312))=\frac{1}{3!}$.

Writing $\backw$ for the reverse of $w$, this graphical calculation
method shows that $\f_{w'}(\backw)$ and $\f_{w'}(w)$ differ only
in possibly a sign, since switching branches at every interior vertex
and arranging the trees in the opposite order reverses the word spelt
by the leaves. The number of interior vertices of a tree is one fewer
than the number of leaves, hence the sign change is $(-1)^{|w'|-k(w')}$,
which depends only on the corresponding eigenvalue. In conclusion,
\begin{prop}
\label{prop:invariance-under-reversal}Let $\backw$ denote the reverse
of $w$. Then, if $\f$ is any right eigenfunction of $a$-handed
shuffling with eigenvalue $a^{j}$ then $\f(w)=(-1)^{j}\f(\backw)$.\qed
\end{prop}
Let $u$ be a Lyndon word. In similar notation abuse as in Section
\ref{sub:Altrightefns}, write $\f_{u}(w)$ for the sum of $\f_{u}$
evaluated on all consecutive subwords of $w$ whose degree is $\deg u$
(i.e. on all consecutive subwords of $w$ whose constituent letters
are those of $u$). For example, in calculating $\f_{(12)}((1233212))$,
the relevant subwords are $\underline{12}33212$, $1233\underline{21}2$
and $12332\underline{12}$, so $\f_{(12)}((1233212))=\f_{(12)}((12))+\f_{(12)}((21))+\f_{(12)}((12))=1-1+1=1$.
It is clear from the graphical calculation of eigenfunctions that,
on any subspace $\calsh_{\nu}$ with $\nu\geq\deg(u)$, this new function
$\f_{u}$ is $(|\nu|-|u|+1)!\f_{w'}$ where $w'$ has degree $\nu$
and $u$ is its only non-singleton Lyndon factor. The corresponding
eigenvalue is $a^{-|u|+1}$. For the example above, $\f_{(12)}=(7-2+1)!\f_{(3322121)}$,
since $(3322121)$ has Lyndon factorisation $(3\cdot3\cdot2\cdot2\cdot12\cdot1)$.
The pointwise products of certain $\f_{u}$s are also right eigenfunctions,
see Proposition \ref{prop:rightefnmultiplicative-shuffle} at the
end of this section.
\begin{example}
\label{ex:2rightefn-shuffle} Take the simplest case of $u=(ij)$,
where $i<j$. Then $\f_{(ij)}(w)$ is the number of consecutive subwords
$(ij)$ occurring in $w$, subtract the occurrences of $(ji)$ as
a consecutive subword. In particular, if $w$ has distinct letters,
then 
\[
\f_{(ij)}(w)=\begin{cases}
1, & \mbox{if }(ij)\text{ occurs as a consecutive subword of \ensuremath{w};}\\
-1, & \mbox{if }(ji)\text{ occurs as a consecutive subword of \ensuremath{w};}\\
0, & \text{otherwise.}
\end{cases}
\]
The corresponding eigenvalue is $\frac{1}{a}$.
\end{example}
Summing the $\f_{(ij)}$ over all pairs $i<j$ gives another right
eigenfunction $\f_{\backslash}$, also with eigenvalue $\frac{1}{a}$:
$\f_{\backslash}(w)$ counts the increasing 2-letter consecutive subwords
of $w$, then subtracts the number of decreasing 2-letter consecutive
subwords. These subwords are respectively the ascents and descents
of $w$, so denote their number by $\asc(w)$ and $\des(w)$ respectively.
Note that reversing $w$ turns an ascent into a descent, so $\f_{\backslash}(w)=-\f_{\backslash}(\backw)$,
as predicted by Proposition \ref{prop:invariance-under-reversal}.
If $w$ has all letters distinct then the non-ascents are precisely
the descents; this allows Proposition \ref{prop:2rightefn-shuffle}
below to express $\f_{\backslash}$ solely in terms of $\des(w)$.
(This explains the notation $\f_{\backslash}$. Descents typically
receive more attention in the literature than ascents.) The claims
regarding expected values follow from Proposition \ref{prop:expectationrightefns}.i.
\begin{prop}
\label{prop:2rightefn-shuffle}The function $\f_{\backslash}:\calb_{\nu}\rightarrow\mathbb{R}$
with formula 
\[
\f_{\backslash}(w):=\asc(w)-\des(w)
\]
is a right eigenfunction of $a$-handed shuffling of eigenvalue $\frac{1}{a}$.
Hence, if $X_{m}$ denotes the deck order after $m$ shuffles, 
\[
E\{\asc(X_{m})-\des(X_{m})|X_{0}=w_{0}\}=a^{-m}(\asc(w_{0})-\des(w_{0})).
\]
If $\nu=(1,1,\dots,1)$, then $\f_{\backslash}$ is a multiple of
the ``normalised number of descents'':
\[
\f_{\backslash}(w):=-2\left(\des(w)-\frac{n-1}{2}\right).
\]
So, if a deck of distinct cards started in ascending order (i.e. $\des(w_{0})=0$),
then 
\[
E\{\des(X_{m})|X_{0}=w_{0}\}=(1-a^{-m})\left(\frac{n-1}{2}\right).
\]
\qed
\end{prop}
Similar analysis applies to Lyndon words with three letters:
\begin{example}
\label{ex:3rightefn-shuffle}Fix three letters $i<j<k$. There are
two Lyndon words with three distinct letters: $(ijk)$ and $(ikj)$.
Their standard factorisations are $(i\cdot jk)$ and $(ik\cdot j)$,
so
\begin{itemize}
\item $\f_{(ijk)}$ counts the consecutive subwords $(ijk)$ and $(kji)$
with weight 1, and $(ikj)$ and $(jki)$ with weight -1;
\item $\f_{(ikj)}$ counts the consecutive subwords $(ikj)$ and $(jki)$
with weight 1, and $(kij)$ and $(jik)$ with weight -1.
\end{itemize}

By inspection, $\f_{(ijk)}=\f_{(ij)}\f_{(jk)}+\f_{(ik)}\f_{(jk)}$,
$\f_{(ikj)}=-\f_{(ik)}\f_{(jk)}+\f_{(ik)}\f_{(ij)}$. (This is unrelated
to Proposition \ref{prop:rightefnmultiplicative-shuffle}.) These
have eigenvalue $a^{-2}$.

\end{example}
When all cards in the deck are distinct, certain linear combinations
of these again have a neat interpretation in terms of well-studied
statistics on words. The table below lists the definition of the four
relevant statistics in terms of 3-letter consecutive subwords, and
the (non-standard) notation for their number of occurrences in a given
word $w$.

\begin{tabular}{|c|c|c|}
\hline 
peak & $\peak(w)$ & middle letter is greatest\tabularnewline
\hline 
valley & $\vall(w)$ & middle letter is smallest\tabularnewline
\hline 
double ascent & $\aasc(w)$ & letters are in increasing order\tabularnewline
\hline 
double descent & $\ddes(w)$ & letters are in decreasing order\tabularnewline
\hline 
\end{tabular}

For example, if $w=(1233212)$, then $\vall(w)=\aasc(w)=\ddes(w)=1$
and $\peak(w)=0$.
\begin{prop}
\label{prop:3rightefn-shuffle}The function $\f_{\wedge\vee}:\calb_{\nu}\rightarrow\mathbb{R}$
with formula 
\[
\f_{\wedge\vee}(w):=\peak(w)-\vall(w)
\]
is a right eigenfunction of $a$-handed shuffling of eigenvalue $a^{-2}$.
Hence, if $X_{m}$ denotes the deck order after $m$ shuffles, 
\[
E\{\peak(X_{m})-\vall(X_{m})|X_{0}=w_{0}\}=a^{-2m}(\peak(w_{0})-\vall(w_{0})).
\]
If $\nu=(1,1,\dots,1)$, then the following are also right eigenfunctions
of $a$-handed shuffling of eigenvalue $a^{-2}$: 
\begin{align*}
\f_{\wedge}(w) & :=\peak(w)-\frac{n-2}{3};\\
\f_{\vee}(w) & :=\vall(w)-\frac{n-2}{3};\\
\f_{-}(w) & :=\aasc(w)+\ddes(w)-\frac{n-2}{3}.
\end{align*}
So, if a deck of distinct cards started in ascending order (i.e. $\peak(w_{0})=\vall(w_{0})=0$),
then 
\begin{align*}
E\{\peak(X_{m})|X_{0}=w_{0}\}=E\{\vall(X_{m})|X_{0}=w_{0}\} & =(1-a^{-2m})\frac{n-2}{3};\\
E\{\aasc(X_{m})+\ddes(X_{m})|X_{0}=w_{0}\} & =(1+2a^{-2m})\frac{n-2}{3}.
\end{align*}
\end{prop}
\begin{proof}
From Example \ref{ex:3rightefn-shuffle} above, 
\begin{align*}
\sum_{i<j<k}\f_{(ijk)}(w) & =\aasc(w)+\ddes(w)-\peak(w),\\
\sum_{i<j<k}\f_{(ikj)}(w) & =\peak(w)-\vall(w)=\f_{\wedge\vee}(w)
\end{align*}
are right eigenfunctions of eigenvalue $a^{-2}$. If all cards in
the deck are distinct, then
\[
\peak(w)+\vall(w)+\aasc(w)+\ddes(w)=n-2,
\]
so $\f_{\wedge}=\frac{1}{3}\left(\f_{\wedge\vee}-\sum_{i<j<k}\f_{(ijk)}\right)$,
$\f_{\vee}=\frac{-1}{3}\left(\sum_{i<j<k}\f_{(ijk)}+2\f_{\wedge\vee}\right)$,
$\f_{-}=\frac{1}{3}\left(2\sum_{i<j<k}\f_{(ijk)}+\f_{\wedge\vee}\right)$
are also right eigenfunctions. The statements on expectations follow
from Proposition \ref{prop:expectationrightefns}.
\end{proof}
Linear combinations of $\f_{u}$ for Lyndon $u$ with $|u|=4$ provides
right eigenfunctions of eigenvalue $a^{-3}$ which are weighted counts
of consecutive 4-letter subwords of each ``pattern'', but these
are more complicated.

Here is one final fact about right eigenfunctions, deducible from
the graphical calculation:
\begin{prop}
\label{prop:rightefnmultiplicative-shuffle}Let $u_{1},\dots,u_{j}$
be Lyndon words each with distinct letters, such that no letter appears
in more than one $u_{i}$. Then, for any $\nu\geq\deg(u_{1})+\dots+\deg(u_{j})$,
the pointwise product $\f(w):=\f_{u_{1}}(w)\dots\f_{u_{j}}(w)$ is
a right eigenfunction on $\calsh_{\nu}$ of eigenvalue $a^{-|u_{1}|-\dots-|u_{j}|+j}$;
in fact, $\f=(|\nu|-|u_{1}|-\dots-|u_{j}|+j)!\f_{w'}$, where the
only non-singleton Lyndon factors of $w'$ are precisely $u_{1},\dots,u_{j}$,
each occurring exactly once. \qed
\end{prop}
Under these same conditions, the corresponding relationship for the
left eigenfunctions of the following section is $\g_{u_{1}}(w)\dots\g_{u_{j}}(w)=\frac{1}{Z(w')}\g_{w'}$,
where again $w'$ is the word whose only non-singleton Lyndon factors
are precisely $u_{1},\dots,u_{j}$.
\begin{example}
Let $\nu=(2,1,1,1,1)$ and let $w'=(352141)$ which has Lyndon factorisation
$(35\cdot2\cdot14\cdot1)$. The two non-singleton Lyndon factors $(35)$
and $(14)$ combined have distinct letters, so $\f_{w'}=\frac{1}{(6-2-2+2)!}\f_{(35)}\f_{(14)}$.
For instance, $\f_{w'}((114253))=\frac{1}{24}\f_{(35)}((114253))\f_{(14)}((114253))=\frac{1}{24}\left(-1\right)1=-\frac{1}{24}$.
\end{example}

\subsection{Left Eigenfunctions \label{sub:leftefns-Shuffling}}

Now comes a parallel analysis of the left eigenfunctions, which arise
from diagonalising $\Psi^{a}$ on the shuffle algebra $\calsh$. Apply
Theorem \ref{thm:diagonalisation}.A$'$ to the word basis of $\calsh$
and use Proposition \ref{prop:efns}.L to translate the result: if
$w'$ has Lyndon factorisation $u_{1}\cdot\dots\cdot u_{k}$ , then
the left eigenfunction $\g_{w'}(w)=$coefficient of $w$ in $e(u_{1})\dots e(u_{k})$,
where $e$ is the Eulerian idempotent map: 
\[
e(x)=\sum_{r\geq1}\frac{(-1)^{r-1}}{r}m^{[r]}\bar{\Delta}^{[r]}(x).
\]

Again, concentrate on the case where only one of the factors is not
a single letter. For a Lyndon word $u$, let $\g_{u}(w)$ be the sum
of $\g_{u}$ evaluated on all subwords (not necessarily consecutive)
of $w$ whose degree is $\deg u$ (i.e. on all subwords of $w$ whose
constituent letters are those of $u$). For example, the relevant
subwords for calculating $\g_{(12)}((1233212))$ are $\underline{12}33212$,
$\underline{1}233\underline{2}12$, $\underline{1}23321\underline{2}$,
$1\underline{2}332\underline{1}2$, $1233\underline{21}2$, and $12332\underline{12}$.
Because $e(12)=(12)-\frac{1}{2}(1)(2)=\frac{1}{2}((12)-(21))$, it
follows that $\g_{(12)}((1233212))=4\g_{(12)}((12))+2\g_{(12)}((21))=\frac{1}{2}(4-2)=1$.
It follows from the definition of $\g_{w'}$ for general $w'$ that,
on any subspace $\calsh_{\nu}$ with $\nu\geq\deg(u)$, this new function
$\g_{u}=\frac{1}{Z(w')}\g_{w'}$ for $w'$ with degree $\nu$ and
$u$ its only non-singleton Lyndon factor, as was the case with right
eigenfunctions. (Recall that $Z(w')$ is the size of the stabiliser
in $\sk$ permuting the Lyndon factors of $w'$.) The corresponding
eigenvalue is $a^{-|u|+1}$. For the example above, $\g_{(12)}=\frac{1}{2!2!}\g_{(3322121)}$. 
\begin{example}
\label{ex:2leftefn-shuffle}Again, start with $u=(ij)$, with $i<j$.
Because $e(ij)=(ij)-\frac{1}{2}(i)(j)=\frac{1}{2}((ij)-(ji))$, the
left eigenfunction $\g_{(ij)}$ counts the pairs $(i,j)$ with $i$
appearing before $j$, subtracts the number of pairs $(i,j)$ with
$i$ occurring after $j$, then divides by 2. In particular, if $w$
has distinct letters, then 
\[
\g_{(ij)}(w)=\begin{cases}
\frac{1}{2}, & \mbox{if }i\text{ occurs before }j\mbox{ in }w;\\
-\frac{1}{2}, & \mbox{if }i\text{ occurs after }j\mbox{ in }w.
\end{cases}
\]
The corresponding eigenvalue is $\frac{1}{a}$. In general, $\f_{u}$
and $\g_{u}$ do not count the same subwords.
\end{example}
As before, sum the $\g_{(ij)}$ over all pairs $i<j$ to obtain a
more ``symmetric'' left eigenfunction $\g_{\backslash}$, also with
eigenvalue $\frac{1}{a}$: $\g_{\backslash}(w)$ halves the number
of pairs appearing in increasing order in $w$ minus the number of
\emph{inversions} $\inv(w)$, when a pair appears in decreasing order.
These eigenfunctions also feature in \cite[Th. 3.2.1]{johnpike}.
Out of the $\binom{|w|}{2}$ pairs of letters in $w$, there are $\sum_{i}\binom{(\deg w)_{i}}{2}$
pairs of the same letter, and all other pairs must either appear in
increasing order or be an inversion. This explains:
\begin{prop}
\label{prop:2leftefn-shuffle}The function $\g_{\backslash}:\calb_{\nu}\rightarrow\mathbb{R}$
with formula 
\[
\g_{\backslash}(w):=\frac{1}{2}\binom{|\nu|}{2}-\frac{1}{2}\sum_{i}\binom{\nu_{i}}{2}-\inv(w)
\]
is a left eigenfunction of $a$-handed shuffling of eigenvalue $\frac{1}{a}$.
\qed
\end{prop}
There is no terminology for a ``non-consecutive peak'' in the same
way that an inversion is a ``non-consecutive descent'', so it is
not too interesting to derive an analogue of Proposition \ref{prop:3rightefn-shuffle}
from $\g_{(ijk)}$ and $\g_{(ikj)}$.

\subsection{Duality of Eigenfunctions\label{sub:Duality-shuffle}}

Recall from Proposition \ref{prop:efnsdiagonalisation} that explicit
diagonalisation of Markov chains is most useful when the right and
left eigenbases obtained are dual bases. This is almost true of $\{\f_{w}\}$
and $\{\g_{w}\}$: $\sum_{v\in\calsh_{\nu}}\f_{w'}(v)\g_{w}(v)=0$
for the large majority of pairs of distinct words $w$ and $w'$,
but, for $\nu\geq(1,1,1,0,0,\dots)$, there will always be $w\neq w'\in\calsh_{\nu}$
with $\sum_{v}\f_{w'}(v)\g_{w}(v)\neq0$, essentially because of Example
\ref{ex:duality3-shuffle} below. For ease of notation, write the
inner product $\langle\f_{w'},\g_{w}\rangle$ for $\sum_{v}\f_{w'}(v)\g_{w}(v)$
. 
\begin{thm}
\label{thm:duality-shuffle} Let $w,w'$ be words with Lyndon factorisations
$w=u_{1}\cdot\dots\cdot u_{k}$, $w'=u'_{1}\cdot\dots\cdot u'_{k'}$
respectively. Then 
\[
\langle\f_{w'},\g_{w}\rangle=\begin{cases}
0 & \mbox{if }k\neq k';\\
\frac{1}{Z(w')}\sum_{\sigma\in\mathfrak{S}_{k}}\f_{u'_{\sigma(1)}}(u_{1})\dots\f_{u'_{\sigma(k)}}(u_{k})=\frac{1}{Z(w')}\f_{w'}(u_{1}\dots u_{k}) & \mbox{if }k=k'.
\end{cases}
\]
(Note that $u_{1}\dots u_{k}$ is the shuffle product of the Lyndon
factors, not the concatenation, and is therefore not equal to $w$.)
In particular, $\langle\f_{w},\g_{w}\rangle=1$, and $\langle\f_{w'},\g_{w}\rangle$
is non-zero only when there is a permutation $\sigma\in\sk$ such
that $\deg(u'_{\sigma(i)})=\deg(u_{i})$ for each $i$, and each $u_{i}$
is equal to or lexicographically larger than \textup{$u'_{\sigma(i)}$.} \end{thm}
\begin{example}
\label{ex:duality1-shuffle}$\langle\f_{(23113)},\g_{(13123)}\rangle=0$:
the Lyndon factorisations are $(23\cdot113)$, which has degrees $(0,1,1)$
and $(2,0,1)$; and $(13\cdot123)$, which has degrees $(1,0,1)$
and $(1,1,1)$. These degrees do not agree, so the inner product is
0.
\end{example}

\begin{example}
\label{ex:duality2-shuffle}$\langle\f_{(13213)},\g_{(13123)}\rangle=0$:
the Lyndon factorisations are $(132\cdot13)$ and $(13\cdot123)$,
so $\deg(u'_{\sigma(i)})=\deg(u_{i})$ is true for $i=1,2$ if $\sigma$
is the transposition. But $(123)$ is lexicographically smaller than
$(132)$. 
\end{example}

\begin{example}
\label{ex:duality3-shuffle}Using the Lyndon factorisations in the
previous example, $\langle\f_{(13123)},\g_{(13213)}\rangle=\frac{1}{1}\f_{(13)}(13)\f_{(123)}(132)=1\cdot1(-1)=-1$. \end{example}
\begin{proof}[Proof of Theorem \ref{thm:duality-shuffle}]
As usual, write $f_{w'},g_{w}$ for the eigenvectors in the free
associative algebra and the shuffle algebra respectively corresponding
to $\f_{w'}$, $\g_{w}$ under Proposition \ref{prop:efns}. So
\begin{align*}
f_{w'} & =\frac{1}{k'!Z(w')}\sum_{\sigma\in\mathfrak{S}_{k}}f_{u'_{\sigma(1)}}\dots f_{u'_{\sigma(k')}},\\
g_{w} & =e(u_{1})\dots e(u_{k}).
\end{align*}

If $k\neq k'$, so $w$ and $w'$ have different numbers of Lyndon
factors, then $\f_{w'}$ and $\g_{w}$ are eigenfunctions with different
eigenvalues, so from pure linear algebra, $\langle\f_{w'},\g_{w}\rangle=0$.
(A more detailed explanation is in the penultimate paragraph of the
proof of Theorem \ref{thm:ABdual}, at the end of Section \ref{sub:Altrightefns}.)

Now assume $k=k'$. First, take $k=1$, so $w,\ w'$ are both Lyndon.
Then 
\begin{align*}
\langle\f_{w'},\g_{w}\rangle & =f_{w'}\left(e(w)\right)\\
 & =f_{w'}\left(w-\frac{1}{2}m\bard w+\frac{1}{3}m^{[3]}\bard^{[3]}w-\dots\right)\\
 & =f_{w'}(w)-\frac{1}{2}\left(\Delta f_{w'}\right)\left(\bard w\right)+\frac{1}{3}\left(\Delta^{[3]}f_{w'}\right)\left(\bard^{[3]}w\right)-\dots\\
 & =f_{w'}(w).
\end{align*}
The third equality uses that comultiplication in the free associative
algebra is dual to multiplication in $\calsh$, and the last step
is because $f_{w'}$ is primitive, by construction.

For the case $k>1$, the argument is similar to the third paragraph
of the proof of Theorem \ref{thm:ABdual}. 
\begin{align*}
\langle\f_{w'},\g_{w}\rangle & =\left(\frac{1}{k!Z(w')}\sum_{\sigma\in\mathfrak{S}_{k}}f_{u'_{\sigma(1)}}\dots f_{u'_{\sigma(k)}}\right)\left(e(u_{1})\dots e(u_{k})\right)\\
 & =\left(\frac{1}{k!Z(w')}\sum_{\sigma\in\mathfrak{S}_{k}}\Delta^{[k]}f_{u'_{\sigma(1)}}\dots\Delta^{[k]}f_{u'_{\sigma(k)}}\right)\left(e(u_{1})\otimes\dots\otimes e(u_{k})\right),
\end{align*}
as comultiplication in the free associative algebra is dual to multiplication
in $\calsh$. Each $f_{u'_{\sigma(r)}}$ is primitive, so the terms
of $\Delta^{[k]}f_{u'_{\sigma(r)}}$ are all possible ways to have
$f_{u'_{\sigma(r)}}$ in one tensor-factor and 1 in all other tensor-factors.
Hence the right hand side above simplifies to 
\begin{align*}
 & \frac{1}{k!Z(w')}\sum_{\sigma\in\mathfrak{S}_{k}}\sum_{\tau\in\mathfrak{S}_{k}}f_{u'_{\tau\sigma(1)}}\left(e(u_{1})\right)\dots f_{u'_{\tau\sigma(k)}}\left(e(u_{k})\right)\\
= & \frac{1}{Z(w')}\sum_{\sigma\in\mathfrak{S}_{k}}f_{u'_{\sigma(1)}}\left(e(u_{1})\right)\dots f_{u'_{\sigma(k)}}\left(e(u_{k})\right)\\
= & \frac{1}{Z(w')}\sum_{\sigma\in\mathfrak{S}_{k}}f_{u'_{\sigma(1)}}\left(u_{1}\right)\dots f_{u'_{\sigma(k)}}\left(u_{k}\right),
\end{align*}
using the $k=1$ case in the last step. Running this calculation with
$u_{i}$ instead of $e(u_{i})$ reaches the same conclusion, so $\langle\f_{w'},\g_{w}\rangle$
must also equal $f_{w'}(u_{1}\dots u_{k})$, which is $\f_{w'}(u_{1}\dots u_{k})$
by definition (because riffle-shuffling does not require any basis
rescaling via $\eta$).

Clearly $\f_{u'}(u)$ is non-zero only if $u$ and $u'$ have the
same constituent letters, i.e. $\deg(u')=\deg(u)$. Also, \cite[Th. 5.1]{freeliealgs}
claims that, for Lyndon $u'$ and any word $u$, the right eigenfunction
value $\f_{u'}(u)$ is non-zero only if $u$ is lexicographically
larger than or equal to $u'$.

If $w=w'$, then $u_{i}'=u_{i}$ for each $i$, so the condition that
each $u_{i}$ is equal to or lexicographically larger than $u'{}_{\sigma(i)}$
can only hold when $u_{i}=u'{}_{\sigma(i)}$ for all $i$. The set
of $\sigma\in\sk$ which achieves this is precisely the stabiliser
in $\sk$ permuting the $u_{i}$. So 
\[
\langle\f_{w},\g_{w}\rangle=f_{u{}_{1}}\left(u_{1}\right)\dots f_{u{}_{k}}\left(u_{k}\right),
\]
and \cite[Th. 5.1]{freeliealgs} states that $f_{u}(u)=1$ for all
Lyndon words $u$.
\end{proof}

\section{Descent Sets under Riffle-Shuffling\label{sec:Descent-Sets}}

This section applies the Hopf-power Markov chain machinery to the
algebra $QSym$ of quasisymmetric functions (Example \ref{ex:qsym})
to refine a result of Diaconis and Fulman on the Markov chain tracking
the number of descents under riffle-shuffling of a distinct deck of
cards. (Recall that a descent is a high value card directly on top
of a low value card.) The result in question is the following interpretations
of the right and left eigenfunctions $\f_{i}$ and $\g_{i}$ ($0\leq i\leq n-1$): 
\begin{itemize}
\item \cite[Th. 2.1]{amazingmatrix} $\f_{i}(j)$ is the coefficient of
any permutation with $j$ descents in the $i$th Eulerian idempotent;
\item \cite[Cor. 3.2]{amazingmatrix} $\g_{i}(j)$ is the value of the $j$th
Foulkes character of the symmetric group on any permutation with $i$
cycles.
\end{itemize}
\cite{amazingmatrixncsym} recovers these connections using the algebra
$\sym$ of noncommutative symmetric functions, which is dual to $QSym$.

The first result of the present refinement is the existence of an
``intermediate'' chain between riffle-shuffling and the number of
descents, namely the position of descents. (This also follows from
the descent set being a ``shuffle-compatible statistic'', which
\cite{gesseltalk} attributes to Stanley.) Theorem \ref{thm:chain-qsym}
identifies this chain as the Hopf-power Markov chain on the basis
of fundamental quasisymmetric functions $\{F_{I}\}$. For a deck of
$n$ cards, the states of this descent-set chain naturally correspond
to subsets of $n-1$, though it will be more convenient here to instead
associate them to compositions of $n$, recording the lengths between
each pair of descents. A more detailed explanation is in Section \ref{sub:notation-qsym}.
The right and left eigenfunctions for this chain, coming from Theorem
\ref{thm:diagonalisation}.B$'$ and \ref{thm:diagonalisation}.A$'$
respectively, are also labelled by compositions. The subset of eigenfunctions
with interpretations akin to the Diaconis-Fulman result correspond
to non-decreasing compositions $I$, which may be viewed as partitions:
\begin{itemize}
\item (Theorem \ref{thm:gridems}) $\f_{I}(J)$ is the coefficient of any
permutation with descent set $J$ in the Garsia-Reutenauer idempotent
(of the descent algebra) corresponding to $I$;
\item (Theorem \ref{thm:ribbonchar}) $\g_{I}(J)$ is the value of the ribbon
character (of the symmetric group) corresponding to $J$ on any permutation
of cycle type $I$.
\end{itemize}
Instructions for calculating these eigenfunctions are in Sections
\ref{sub:rightefnspartition-qsym} and \ref{sub:leftefnspartition-qsym}
respectively; the computations are entirely combinatorial so they
only require the notation in Section \ref{sub:notation-qsym} below,
and are independent of all other sections. The eigenfunctions for
general compositions are considerably more unwieldly.

The calculation and interpretation of eigenfunctions are but a small
piece in the Diaconis-Fulman collaboration concerning the number of
descents under riffle-shuffling. The first of their series of papers
on the topic proves \cite[Th. 3.3, 3.4]{descentsetchain} that $\frac{1}{2}\log n$
steps are necessary and sufficient to randomise the number of descents.
As an aside, they show that $\log n$ steps are sufficient to randomise
the positions of descents, hence the descent-set Markov chain has
a mixing time between $\frac{1}{2}\log n$ and $\log n$. Their second
paper \cite{descentsetchain2} gives a neat combinatorial explanation
that this number-of-descents Markov chain is the same as the carries
observed while adding a list of numbers, a chain previously studied
by \cite{amazingmatrixoriginal}. \cite{generalisedcarriesshuffle}
finds a carries process which equates to the number of descents under
\emph{generalised riffle-shuffles}. Here the cards can have one of
$p$ colours, and the colours change during shuffling depending on
which pile the cards fall into when the deck is cut. The notion of
descent is modified to take into account the colours of the cards.
The left eigenfunctions of the Markov chain on the number of descents
correspond to a generalisation of Foulkes characters in \cite{generalisedfoulkeschar};
these are characters of wreath products $\nicefrac{\mathbb{Z}}{p\mathbb{Z}}\wr\sn$.
An interesting question for the future is whether the descent set
of generalised riffle-shuffles also forms a Markov chain, with some
refinement of these generalised Foulkes characters describing some
of the left eigenfunctions. 

Returning to the present, the rest of the chapter is organised as
follows: Section \ref{sub:notation-qsym} establishes the necessary
notation. Section \ref{sub:qsym-sym} covers background on the algebra
$QSym$ of quasisymmetric functions and its dual $\sym$, the noncommutative
symmetric functions, necessary for the proofs and for computing the
``messy'' eigenfunctions. Section \ref{sub:chain-qsym} shows that
the descent set is indeed a Markov statistic for riffle-shuffling,
by creating a Hopf morphism $\calsh\rightarrow QSym$ and appealing
to the projection theory of Hopf-power Markov chains (Section \ref{sec:hpmc-projection}).
Sections \ref{sub:rightefnspartition-qsym} and \ref{sub:leftefnspartition-qsym}
detail the right and left eigenfunctions corresponding to partitions,
while Sections \ref{sub:rightefns-qsym} and \ref{sub:leftefns-qsym}
contain the full eigenbasis and the proofs of the relationships to
ribbon characters and Garsia-Reutenauer idempotents. Section \ref{sub:Duality-qsym}
addresses a partial duality between the two eigenbases, recovering
a weak version of a result of Stanley on the probability that a deck
in ascending order acquires a particular descent composition after
$m$ shuffles. Section \ref{sub:matrix-qsym} is an appendix containing
the transition matrix and full eigenbases for the case $n=4$. The
main results of this section previously appeared in the extended abstract
\cite{hpmccompositions}.

\subsection{Notation regarding compositions and descents\label{sub:notation-qsym}}

For easy reference, this section collects all notation relevant to
the rest of this chapter.

A \emph{composition} $I$ is a list of positive integers $\left(i_{1},i_{2},\dots,i_{l(I)}\right)$.
Each $i_{j}$ is a \emph{part} of $I$. The sum $i_{1}+\dots+i_{l(I)}$
is denoted $|I|$, and $l(I)$ is the number of parts in $I$. So
$|(3,5,2,1)|=11$, $l((3,5,2,1))=4$. Forgetting the ordering of the
parts of $I$ gives a multiset $\lambda(I):=\left\{ i_{1},\dots,i_{l(I)}\right\} $.
Clearly $\lambda(I)=\lambda(I')$ if and only if $I'$ has the same
parts as $I$, but in a different order. $I$ is a \emph{partition}
if its parts are non-increasing, that is, $i_{1}\geq i_{2}\geq\dots\geq i_{l(I)}$.

The following two pictorial descriptions of compositions will come
in useful for calculating right and left eigenfunctions respectively.
Firstly, the \emph{diagram} of $I$ is a string of $|I|$ dots with
a division after the first $i_{1}$ dots, another division after the
next $i_{2}$ dots, etc. Next, the \emph{ribbon shape} of $I$ is
a skew-shape (in the sense of tableaux) with $i_{1}$ boxes in the
bottommost row, $i_{2}$ boxes in the second-to-bottom row, etc, so
that the rightmost square of each row is directly below the leftmost
square of the row above. Hence this skew-shape contains no 2-by-2
square. The diagram and ribbon shape of $(3,5,2,1)$ are shown below.

\begin{center}
\begin{tabular}{cc}
$\cdot\cdot\cdot|\cdot\cdot\cdot\cdot\cdot|\cdot\cdot|\cdot$ & \tableau{ & & & & & & & \e \\  & & & & & & \e & \e \\ & & \e & \e & \e & \e & \e \\ \e & \e & \e } \tabularnewline
\end{tabular}
\par\end{center}

There is a natural partial ordering on the collection of compositions
$\{I|\ |I|=n\}$ - define $J\geq I$ if $J$ is a \emph{refinement}
of $I$. Then $I$ is a \emph{coarsening} of $J$.

Given compositions $I,\ J$ with $|I|=|J|$, \cite[Sec. 4.8]{ncsym}
defines the \emph{decomposition of $J$ relative to $I$} as the $l(I)$-tuple
of compositions $\left(J_{1}^{I},\dots,J_{l(I)}^{I}\right)$ such
that $|J_{r}^{I}|=i_{r}$ and each $l(J_{r}^{I})$ is minimal such
that the concatenation $J_{1}^{I}\dots J_{l(I)}^{I}$ refines $J$.
Pictorially, the diagrams of $J_{1}^{I},\dots,J_{l(I)}^{I}$ are obtained
by ``splitting'' the diagram of $J$ at the points specified by
the divisions in the diagram of $I$. For example, if $I=(4,4,3)$
and $J=(3,5,2,1)$, then $J_{1}^{I}=(3,1)$, $J_{2}^{I}=(4)$, $J_{3}^{I}=(2,1)$.

It will be useful to identify the composition $I$ with the word $i_{1}\dots i_{l(I)}$;
then it makes sense to talk of Lyndon compositions, factorisations
into Lyndon compositions, and the other concepts from Section \ref{sec:Lyndon-Words}.
Write $I=I_{(1)}\dots I_{(k)}$ for the Lyndon factorisation of $I$;
so, if $I=(3,5,2,1)$, then $I_{(1)}=(3,5)$, $I_{(2)}=(2)$, $I_{(3)}=(1)$.
$k(I)$ will always denote the number of Lyndon factors in $I$. A
composition $I$ is a partition precisely when all its Lyndon factors
are singletons - this is what simplifies their corresponding eigenfunctions.
$\lambda(I)$ is the multigrading of $I$ as a word, and $l(I)$ is
the integer grading, though neither agrees with the grading $|I|$
on $QSym$ so this view may be more confusing than helpful.

Finally, the \emph{descent set} of a word $w=w_{1}\dots w_{n}$ is
defined to be $D(w)=\left\{ j\in\{1,2,\dots,|w|-1\}|w_{j}>w_{j+1}\right\} $.
As noted earlier, it is more convenient here to consider the associated
composition of $D(w)$. Hence a word $w$ has \emph{descent composition}
$\Des(w)=I$ if $i_{j}$ is the number of letters between the $j-1$th
and $j$th descent,\emph{ }i.e. if $w_{i_{1}+\dots+i_{j}}>w_{i_{1}+\dots+i_{j}+1}$
for all $j$, and $w_{r}\leq w_{r+1}$ for all $r\neq i_{1}+\dots+i_{j}$.
For example, $D(4261)=\{1,3\}$ and $\Des(4261)=(1,2,1)$. Note that
no information is lost in passing from $D(w)$ to $\Des(w)$, as the
divisions in the diagram of $\Des(w)$ indicate the positions of descents
in $w$.

\subsection{Quasisymmetric Functions and Noncommutative Symmetric Functions\label{sub:qsym-sym}}

Recall from Example \ref{ex:qsym} the algebra $QSym$ of quasisymmetric
functions: it is a subalgebra of the algebra of power series in infinitely-many
commuting variables $\{x_{1},x_{2},\dots\}$ spanned by the \emph{monomial
quasisymmetric functions} 
\[
M_{I}=\sum_{j_{1}<\dots<j_{l(I)}}x_{j_{1}}^{i_{1}}\dots x_{j_{l(I)}}^{i_{l(I)}}.
\]
The basis runs over all compositions $I=(i_{1},\dots,i_{l(I)})$.
This, however, is not the state space basis of the Markov chain of
interest; that basis is the \emph{fundamental quasisymmetric functions}
\[
F_{I}=\sum_{J\geq I}M_{J}
\]
where the sum runs over all partitions $J$ refining $I$. $QSym$
inherits a grading and a commutative algebra structure from the algebra
of power series, so $\deg(M_{I})=\deg(F_{I})=|I|$. \cite{duality}
extends this to a Hopf algebra structure using the ``alphabet doubling''
coproduct: take two sets of variables $X=\{x_{1},x_{2},\dots\}$,
$Y=\{y_{1},y_{2},\dots\}$ that all commute, and totally-order $X\cup Y$
by setting $x_{i}<x_{j}$ if $i<j$, $y_{i}<y_{j}$ if $i<j$, and
$x_{i}<y_{j}$ for all $i,j$. Then, if $F(x,y)$ denotes the quasisymmetric
function $F$ applied to $X\cup Y$, and $F(x,y)=\sum_{i}G_{i}(x)H_{i}(y)$,
then $\Delta(F)=\sum_{i}G_{i}\otimes H_{i}$. For example, $\Delta(M_{i})=M_{i}\otimes1+1\otimes M_{i}$,
and 
\[
\Delta(M_{I})=\sum_{j=0}^{l(I)}M_{(i_{1},i_{2},\dots,i_{j})}\otimes M_{(i_{j+1},\dots,i_{l(I)})}.
\]
The graded dual Hopf algebra of $QSym$ is $\sym$, the algebra of
noncommutative symmetric functions. (Some authors call this $NSym$.
Beware that there are several noncommutative analogues of the symmetric
functions, such as $NCSym$, and these are not all isomorphic.) A
comprehensive reference on this algebra is \cite{ncsym} and its many
sequels. The notation here follows this tome, except that all indices
of basis elements will be superscripts, to distinguish from elements
of $QSym$ which use subscripts. The duality of $\sym$ and $QSym$
was first established in \cite[Th. 2.1]{duality}.

\cite[Sec. 2]{goodsymexplanation} frames $\sym$ under the polynomial
realisation viewpoint previously discussed in Section \ref{sub:Polyrealisation}.
The construction starts with the power series algebra in infinitely-many
noncommuting variables. For simplicity, write the word $(i_{1}\dots i_{l})$
for the monomial $x_{i_{1}}\dots x_{i_{l}}$; so, for example, $(12231)$
stands for $x_{1}x_{2}^{2}x_{3}x_{1}$. As an algebra, $\sym$ is
a subalgebra of this power series algebra generated by 
\[
S^{(n)}:=\sum_{w:\Des(w)=(n)}w,
\]
the sum over all words of length $n$ with no descent. For example,
\begin{align*}
S^{(1)} & =(1)+(2)+(3)+\dots=x_{1}+x_{2}+x_{3}+\dots;\\
S^{(2)} & =(11)+(12)+(13)+\dots+(22)+(23)+\dots.
\end{align*}
The algebra $\sym$ inherits a concatenation product from the full
power series algebra, and the alphabet doubling trick endows $\sym$
with the coproduct 
\[
\Delta(S^{(n)})=\sum_{i=0}^{n}S^{(i)}\otimes S^{(n-i)}.
\]
For any composition $I=(i_{1},\dots,i_{l})$, define the \emph{complete
noncommutative symmetric functions} 
\[
S^{I}:=S^{(i_{1})}\dots S^{(i_{l})}=\sum_{w:\Des(w)\leq I}w.
\]
A moment's thought will convince that $\{S^{I}\}$ is linearly independent.
So $\{S^{I}\}$ is a free basis in the sense of Theorem \ref{thm:diagonalisation}.B$'$;
it is analogous to the $\{h_{\lambda}\}$ basis of the symmetric functions.
Indeed, the abelianisation map from the noncommutative power series
ring to $\mathbb{R}[[x_{1},x_{2},\dots]]$ (i.e. allowing the variables
$x_{i}$ to commute) sends each $S^{(n)}$ to $h_{(n)}$, and consequently
sends $S^{I}$ to $h_{\lambda(I)}$. The basis $\{S^{I}\}$ is dual
to the monomial quasisymmetric functions $\{M_{I}\}$.

The dual basis to the fundamental quasisymmetric functions $\{F_{I}\}$
is the \emph{ribbon noncommutative symmetric functions} $\{R^{I}\}$:
\[
R^{I}:=\sum_{w:\Des(w)=I}w.
\]

One more basis will be useful in the ensuing analysis. \cite[Eq. 26]{ncsym}
defines $\frac{\Phi^{(n)}}{n}$ to be the coefficient of $t^{n}$
in the formal power series $\log(1+\sum_{i>0}S^{(i)}t^{i})$. Equivalently,
\[
\Phi^{(n)}:=ne(S^{(n)})=n\sum_{I}\frac{(-1)^{l(I)}}{l(I)}\sum_{w:\Des(w)\leq I}w,
\]
where $e$ is the Eulerian idempotent map. This is a noncommutative
analogue of the relationship $e(h_{(n)})=\frac{1}{n}p_{(n)}$, established
in Section \ref{sub:Left-Eigenfunctions-rock-breaking}. Noncommutativity
of the underlying variables means that there is sadly no formula for
the $\Phi^{(n)}$ quite as convenient as $p_{(n)}=x_{1}^{n}+x_{2}^{n}+\dots$.
Then the \emph{power sum noncommutative symmetric functions of the
second kind} are 
\[
\Phi^{I}:=\Phi^{(i_{1})}\dots\Phi^{(i_{l})}.
\]

\cite{ncsym} details explicitly the change-of-basis matrices of these
and other bases in $\sym$; these will be extremely useful in Sections
\ref{sub:rightefnspartition-qsym} and \ref{sub:rightefns-qsym} for
determining the right eigenfunctions of the associated Markov chain.

\subsection{The Hopf-power Markov chain on $QSym$\label{sub:chain-qsym}}

Solely from the above definitions of the fundamental quasisymmetric
functions, the product and the coproduct, it is unclear what process
the Hopf-power Markov chain on $\{F_{I}\}$ might represent. The key
to solving this mystery is the following Hopf morphism, which sends
any word with distinct letters to the fundamental quasisymmetric function
indexed by its descent set.
\begin{thm}
\label{thm:shuffletoqsym}There is a morphism of Hopf algebras $\theta:\calsh\rightarrow QSym$
such that, if $w$ is a word with distinct letters, then $\theta(w)=F_{\Des(w)}$.
\end{thm}
The proof is at the end of this section. Applying the Projection Theorem
for Hopf-power Markov Chains (Theorem \ref{thm:hpmc-projection})
to the map $\theta$ shows that:
\begin{thm}
\label{thm:chain-qsym} The Hopf-power Markov chain on the fundamental
quasisymmetric functions $\{F_{I}\}$ tracks the descent set under
riffle-shuffling of a distinct deck of cards. In particular, the descent
set is a Markovian statistic of riffle-shuffling of a distinct deck
of cards.
\end{thm}
In order to keep the algebra in the background, the subsequent sections
will refer to this chain simply as the Hopf-power Markov chain on
compositions, and the states of the chain will be labelled by compositions
$I$ instead of the corresponding quasisymmetric functions $F_{I}$.
This is similar to the notation of Section \ref{sec:Rock-breaking}.
\begin{proof}[Proof of Theorem \ref{thm:chain-qsym}]
Follow the notation of the Projection Theorem and write $\calb$
for the word basis of the shuffle algebra, and $\barcalb$ for the
fundamental quasisymmetric functions. Then, for any $\nu$ where each
$\nu_{i}$ is 0 or 1, $\calb_{\nu}$ consists of words with distinct
letters, so the map $\theta$ from Theorem \ref{thm:shuffletoqsym}
satisfies $\theta(\calb_{\nu})=\barcalb_{|\nu|}$. Moreover, $\theta$
sends all single letters to $F_{1}=\barcalb_{1}$. Hence the conditions
of the Projection Theorem hold, and its application proves the result.
\end{proof}

\begin{proof}[Proof of Theorem \ref{thm:shuffletoqsym}]
 By \cite[Th. 4.1]{abs}, $QSym$ is the terminal object in the category
of combinatorial Hopf algebras equipped with a multiplicative character.
So, to define a Hopf morphism to $QSym$, it suffices to define the
corresponding character $\zeta$ on the domain. By \cite[Th. 6.1.i]{freeliealgs},
the shuffle algebra is freely generated by Lyndon words, so any choice
of the values of $\zeta$ on Lyndon words extends uniquely to a well-defined
character on the shuffle algebra. For Lyndon $u$, set 
\begin{equation}
\zeta(u)=\begin{cases}
1 & \mbox{if }u\mbox{ has all letters distinct and has no descents};\\
0 & \mbox{otherwise.}
\end{cases}\label{eq:defzeta}
\end{equation}
I claim that, consequently, (\ref{eq:defzeta}) holds for all words
with distinct letters, even if they are not Lyndon. Assuming this
for now, \cite[Th. 4.1]{abs} defines 
\[
\theta(w)=\sum_{I:|I|=|w|}\zeta(w_{1}\cdot\dots\cdot w_{i_{1}})\zeta(w_{i_{1}+1}\cdot\dots\cdot w_{i_{1}+i_{2}})\dots\zeta(w_{i_{l(I)-1}+1}\cdot\dots\cdot w_{n})M_{I}.
\]
If $w$ has distinct letters, then every consecutive subword $w_{i_{1}+\dots+i_{j}+1}\cdot\dots\cdot w_{i_{1}\dots+i_{j+1}}$
of $w$ also has distinct letters, so 
\[
\zeta(w_{1}\cdot\dots\cdot w_{i_{1}})\dots\zeta(w_{i_{l(I)-1}+1}\cdot\dots\cdot w_{n})=\begin{cases}
1 & \mbox{if }\Des(w)\leq I;\\
0 & \mbox{otherwise.}
\end{cases}
\]
Hence $\theta(w)=\sum_{\Des(w)\leq I}M_{I}=F_{\Des(w)}$.

Now return to proving the claim that (\ref{eq:defzeta}) holds whenever
$w$ has distinct letters. Proceed by induction on $w$, with respect
to lexicographic order. \cite[Th. 6.1.ii]{freeliealgs}, applied to
a word $w$ with distinct letters, states that: if $w$ has Lyndon
factorisation $w=u_{1}\cdot\dots\cdot u_{k}$, then the product of
these factors in the shuffle algebra satisfies 
\[
u_{1}\dots u_{k}=w+\sum_{v<w}\alpha_{v}v
\]
 where $\alpha_{v}$ is 0 or 1. The character $\zeta$ is multiplicative,
so 
\begin{equation}
\zeta(u_{1})\dots\zeta(u_{k})=\zeta(w)+\sum_{v<w}\alpha_{v}\zeta(v).\label{eq:zeta}
\end{equation}
If $w$ is Lyndon, then the claim is true by definition; this includes
the base case for the induction. Otherwise, $k>1$ and there are two
possibilities:
\begin{itemize}
\item None of the $u_{i}$s have descents. Then the left hand side of (\ref{eq:zeta})
is 1. Since the $u_{i}$s together have all letters distinct, the
only way to shuffle them together and obtain a word with no descents
is to arrange the constituent letters in increasing order. This word
is Lyndon, so it is not $w$, and, by inductive hypothesis, it is
the only $v$ in the sum with $\zeta(v)=1$. So $\zeta(w)$ must be
0.
\item Some Lyndon factor $u_{i}$ has at least one descent. Then $\zeta(u_{i})=0$,
so the left hand side of (\ref{eq:zeta}) is 0. Also, no shuffle of
$u_{1},\dots,u_{k}$ has its letters in increasing order. Therefore,
by inductive hypothesis, all $v$ in the sum on the right hand side
have $\zeta(v)=0$. Hence $\zeta(w)=0$ also.
\end{itemize}
\end{proof}
\begin{rems*}
$ $

\begin{enumerate}[label=\arabic*.]
\item From the proof, one sees that the conclusion $\theta(w)=F_{\Des(w)}$
for $w$ with distinct letters relies only on the value of $\zeta$
on Lyndon words with distinct letters. The proof took $\zeta(u)=0$
for all Lyndon $u$ with repeated letters, but any other value would
also work. Alas, no definition of $\zeta$ will ensure $\theta(w)=F_{\Des(w)}$
for all $w$:
\[
\theta((11))=\frac{1}{2}\theta((1)(1))=\frac{1}{2}\theta(1)\theta(1)=\frac{1}{2}M_{1}^{2}\neq F_{2}.
\]

\item The map $\theta$ is inspired by, but ultimately mathematically unrelated
to, the polynomial realisation of $\sym$. Dualising the algebra embedding
$\sym\subseteq\calsh^{*}$ gives a coalgebra map $\theta':\calsh\rightarrow QSym$,
with $\theta'(w)=F_{\Des(w)}$ for all $w$, but this is not a Hopf
algebra map. Mysteriously and miraculously, if all letters occurring
in $v$ and $w$ together are distinct, then $\theta'(vw)=\theta'(v)\theta'(w)$,
and doctoring the image of $\theta'$ on words with repeated letters
can make this true for all $v,w$. I have yet to find another combinatorial
Hopf algebra $\calh$ where the coalgebra map $\theta':\calsh\rightarrow\calh$
dual to a polynomial realisation $\calhdual\subseteq\calsh^{*}$ satisfies
$\theta'(vw)=\theta'(v)\theta'(w)$ for a large class of $v,w\in\calsh$.
\end{enumerate}
\end{rems*}

\subsection{Right Eigenfunctions Corresponding to Partitions\label{sub:rightefnspartition-qsym}}

Throughout this subsection, let $I$ be a partition. That is, $i_{1}\geq i_{2}\geq\dots\geq i_{l(I)}$.
Set $n=|I|$.

All right eigenfunctions are essentially built from the function
\[
\f(J):=\frac{1}{|J|}\frac{(-1)^{l(J)-1}}{\binom{|J|-1}{l(J)-1}}.
\]
Note that $\f(J)$ depends only on $|J|$ and $l(J)-1$, which are
respectively the number of dots and the number of divisions in the
diagram of $J$.

Theorem \ref{thm:gridems} below gives the formula for $\f_{I}$,
the right eigenfunction corresponding to the partition $I$, in terms
of $\f$. The proof is at the end of the following section, after
establishing the full eigenbasis. The scaling of these eigenfunctions
differs from that in Theorem \ref{thm:diagonalisation}.B$'$ in order
to connect them to the idempotents $E_{I}$ defined by \cite[Sec. 3]{descentalg},
of the descent algebra. (The descent algebra is the subalgebra of
the group algebra $\mathbb{Z}\sn$ spanned by sums of permutations
with the same descent sets. Hence each $E_{I}$ is a linear combination
of permutations, where permutations with the same descent set have
the same coefficient.)
\begin{thm}
\label{thm:gridems} With $\f$ as defined above, the function 
\begin{align*}
\f_{I}(J) & :=\frac{1}{l(I)!}\sum_{I':\lambda(I')=\lambda(I)}\prod_{r=1}^{l(I')}\f(J_{r})\\
 & \phantom{:}=\frac{1}{l(I)!i_{1}\dots i_{l(I)}}\sum_{I':\lambda(I')=\lambda(I)}\prod_{r=1}^{l(I')}\frac{(-1)^{l(J_{r})-1}}{\binom{|J_{r}|-1}{l(J_{r})-1}},
\end{align*}
is a right eigenfunction of eigenvalue $a^{l(I)-n}$ of the ath Hopf-power
Markov chain on compositions. (Here $(J_{1},\dots,J_{l(I')})$ is
the decomposition of $J$ with respect to $I'$.) The numbers $\f_{I}(J)$
appear as coefficients in the Garsia-Reutenauer idempotent $E_{I}$:
\[
E_{I}=\sum_{\sigma\in\mathfrak{S}_{n}}\f_{I}(\Des(\sigma))\sigma.
\]
(Here, $\Des(\sigma)$ is the descent composition of the word whose
$i$th letter is $\sigma(i)$ - that is, the word given by $\sigma$
in one-line notation.)\end{thm}
\begin{rem*}
The sum of $E_{I}$ across all $I$ with $i$ parts is the $i$th
Eulerian idempotent, in which the coefficients of a permutation $\sigma$
depend only on its number of descents. Hence $\sum_{l(I)=i}\f_{I}$
is a right eigenfunction of eigenvalue $a^{i-n}$ whose value depends
only on $l(J)$, the number of parts. The $n$ such eigenfunctions
descend to the right eigenbasis of \cite[Th. 2.1]{amazingmatrix}
for the number of descents under riffle-shuffling.
\end{rem*}
Here is a more transparent description of how to calculate $\f_{I}(J)$: 
\begin{enumerate}[label=\arabic*.]
\item  Split the diagram of $J$ into pieces whose numbers of dots are
the parts of $I$, with multiplicity. 
\item Calculate $\f$ on each piece of $J$ by counting the number of dots
and divisions and multiply these $\f$ values together. 
\item Sum this number across all decompositions of $J$ in step 1, then
divide by $l(I)!$.
\end{enumerate}
Note that $\f$ itself, when restricted to compositions of a fixed
size, is a right eigenfunction, that corresponding to the partition
with a single part. Its eigenvalue is $a^{1-n}$, the smallest possible.
\begin{example}
\label{ex:rightefnspartition-qsym}Here's how to apply the algorithm
above to calculate $\f_{(4,4,3)}((3,5,2,1))$. There are three relevant
decompositions of $(3,5,2,1)$: 
\[
\cdot\cdot\cdot|\cdot\ \cdot\cdot\cdot\cdot\ \cdot\cdot|\cdot\quad\cdot\cdot\cdot|\cdot\ \cdot\cdot\cdot\ \cdot|\cdot\cdot|\cdot\quad\cdot\cdot\cdot\ \cdot\cdot\cdot\cdot\ \cdot|\cdot\cdot|\cdot
\]
so 
\[
\f_{(4,4,3)}((3,5,2,1))=\frac{1}{3!}\left(\frac{-1}{4\binom{3}{1}}\frac{1}{4}\frac{-1}{3\binom{2}{1}}+\frac{-1}{4\binom{3}{1}}\frac{1}{3}\frac{1}{4\binom{3}{2}}+\frac{1}{3}\frac{1}{4}\frac{1}{4\binom{3}{2}}\right)=\frac{7}{5184}.
\]

\end{example}
Note that $\f((1))=1$, so pieces of size one do not contribute to
step 2 of the algorithm above. This observation simplifies the calculation
of $\f_{(i_{1},1,1,\dots,1)}(J)$, in a similar way to the $\f_{u}$
of Section \ref{sub:rightefns-Shuffling}: $\f_{(i_{1},1,1,\dots,1)}(J)$
is the sum of $\f$ evaluated on the ``subcompositions'' of $J$
formed by $i_{1}$ consecutive dots. In other words, $\f_{(i_{1},1,1,\dots,1)}$
is the weighted enumeration of ``patterns'' of length $i_{1}$,
where pattern $J$ has weight $\frac{f(J)}{(n-i_{1}+1)!}$. In the
similar notational abuse as Section \ref{sub:rightefns-Shuffling},
call this eigenfunction $\f_{(i)}$. (The parallels end here: products
of $\f_{(i)}$ are not eigenfunctions, that fact is particular to
riffle-shuffling.)

Each right eigenfunction $\f_{I}$ has a lift to the riffle-shuffle
chain: that is, the function $\tilde{\f_{I}}(w):=\f_{I}(\Des(w))$
for words $w$ with distinct letters is a right eigenfunction for
riffle-shuffling. (This is a general fact about projections of Markov
chains and is unrelated to Hopf algebras, see \cite[Lem. 12.8.ii]{markovmixing}).
As divisions correspond to descents, $\tilde{\f}_{(i)}$ is a weighted
enumeration of ``up-down-patterns'' of length $i$.
\begin{example}
\label{ex:2rightefn-qsym}Take $i=2$, then each subcomposition is
either $(2)$ or $(1,1)$. Since $\f((2))=\frac{1}{2}$ and $\f((1,1))=-\frac{1}{2}$,
the right eigenfunction $\f_{(2)}$ counts a non-divison with weight
$\frac{1}{2(n-1)!}$ and a division with weight $\frac{-1}{2(n-1)!}$.
Since the number of non-divisions and the number of divisions sum
to $n-1$, 
\[
\f_{(2)}(J)=\frac{1}{(n-1)!}\left(\frac{|J|-1}{2}-\left(l(J)-1\right)\right).
\]
It will follow from the full eigenbasis description of Theorem \ref{thm:rightefns-qsym}
that this is the unique right eigenvector of eigenvalue $\frac{1}{a}$,
up to scaling. Its lift $\widetilde{\f}_{(2)}$ to the riffle-shuffling
chain is (a multiple of) the ``normalised number of descents'' eigenvector
of Proposition \ref{prop:2rightefn-shuffle}: $\widetilde{\f}_{(2)}=\frac{1}{2(n-1)!}\f_{\backslash}$.
\end{example}

\begin{example}
\label{ex:3rightefn-qsym}Take $i=3$. Then $\f((3))=\f((1,1,1))=\frac{1}{3}$,
$\f((2,1))=\f((1,2))=-\frac{1}{6}$, so
\resizebox{\textwidth}{!}{  
\begin{minipage}{\textwidth}   
\begin{align*} 
\f_{(3)}(J) & =\frac{1}{3(n-2)!}\left(\#\mbox{ (two consecutive non-divisions)}+\#\mbox{ (two consecutive divisions)}\phantom{\frac{1}{2}}\right.\\  & \left.\phantom{=}-\frac{1}{2}\#\mbox{(division followed by non-division)}-\frac{1}{2}\#\mbox{(non-division followed by division)}\right).
\end{align*}   
\end{minipage} }
%
 The associated eigenvalue is $\frac{1}{a^{2}}$. Its lift $\widetilde{\f}_{(3)}$
to the riffle-shuffling chain is 
\begin{align*}
\widetilde{\f}_{(3)}(w) & =\frac{1}{3(n-2)!}\left(\aasc(w)+\ddes(w)-\frac{1}{2}\vall(w)-\frac{1}{2}\peak(w)\right)\\
 & =\frac{1}{2(n-2)!}\f_{-}
\end{align*}
in the notation of Proposition \ref{prop:3rightefn-shuffle}.
\end{example}

\subsection{A full Basis of Right Eigenfunctions\label{sub:rightefns-qsym}}

When $I$ is not a partition, the calculation of $\f_{I}(J)$ is very
similar to the previous three-step process, except that, in the last
step, each summand is weighted by $\f_{I}^{\calsh}(I')$, the value
on $I'$ of the right eigenfunction $\f_{I}^{\calsh}$ of riffle-shuffling.
\begin{thm}
\label{thm:rightefns-qsym}Given a composition $I=(i_{1},\dots,i_{l})$
with $k(I)$ Lyndon factors, define the function 
\[
\f_{I}(J):=\frac{1}{i_{1}\dots i_{l(I)}}\sum_{I':\lambda(I')=\lambda(I)}\f_{I}^{\calsh}(I')\prod_{r=1}^{l(I')}\frac{(-1)^{l(J_{r})-1}}{\binom{|J_{r}|-1}{l(J_{r})-1}},
\]
where $(J_{1},\dots,J_{l(I')})$ is the decomposition of $J$ relative
to $I'$, and $\f_{I}^{\calsh}$ is the right eigenfunction of riffle-shuffling
corresponding to the word $i_{1}\dots i_{l}$, as explained in Section
\ref{sub:rightefns-Shuffling}. Then $\{\f_{I}|\left|I\right|=n,\ I\mbox{ has }k\mbox{ Lyndon factors}\}$
is a basis of right $a^{k-n}$-eigenfunctions for the $a$th Hopf-power
Markov chain on compositions. 
\end{thm}
The proof is at the end of this section.
\begin{example}
\label{ex:rightefns-qsym}Take $I=(1,2,1)$ and $J=(3,1)$. Using
the decreasing Lyndon hedgerows technique of Section \ref{sub:rightefns-Shuffling},
one finds that $\f_{(1,2,1)}^{\calsh}((1,1,2))=\frac{1}{2},\ \f_{(1,2,1)}^{\calsh}((2,1,1))=-\frac{1}{2}$,
and $\f_{(1,2,1)}^{\calsh}$ is zero on all other compositions. The
decomposition of $(3,1)$ relative to $(1,1,2)$ and $(2,1,1)$ are
$((1),(1),(1,1))$ and $((2),(1),(1))$ respectively. Putting all
this information into the formula in Theorem \ref{thm:rightefns-qsym}
above yields 
\[
\f_{(1,2,1)}((3,1))=\frac{1}{1\cdot2\cdot1}\left(\frac{1}{2}\cdot1(-1)-\frac{1}{2}\cdot1\cdot1\right)=-\frac{1}{2}.
\]
The full right eigenbasis for the case $n=4$, as specified by Theorem
\ref{thm:rightefns-qsym}, is tabulated in Section \ref{sub:matrix-qsym}.
\end{example}
The following property of the right eigenfunctions will be useful
for proving Proposition \ref{cor:probdescentset}. It essentially
says that, if the starting state is the one-part partition, then only
the right eigenfunctions corresponding to partitions are relevant.
When interpreting this chain on compositions as the descent-set chain
under riffle-shuffling, this scenario corresponds to starting the
deck in ascending order.
\begin{prop}
\label{prop:rightefns-qsym}If $I$ is a partition, then $\f_{I}((n))=\frac{1}{Z(I)i_{1}\dots i_{l}}$,
the proportion of permutations in $\mathfrak{S}_{n}$ with cycle type
$I$. For all other $I$, $\f_{I}((n))=0$.\end{prop}
\begin{proof}
First note that the decomposition of $(n)$ relative to any composition
$I$ is $((i_{1}),\dots,(i_{l(I)}))$, so 
\[
\f_{I}((n))=\frac{1}{i_{1}\dots i_{l(I)}}\sum_{I':\lambda(I')=\lambda(I)}\f_{I}^{\calsh}(I').
\]
Recall from Section \ref{sub:rightefns-Shuffling} that $k(I)!\f_{I}^{\calsh}(I')$
is the signed number of ways to rearrange the decreasing Lyndon hedgerow
$T_{I}$ so the leaves spell $I'$. So $k(I)!\sum_{I':\lambda(I')=\lambda(I)}\f_{I}^{\calsh}(I')$
is the total signed number of rearrangements of $T_{I}$. If $I$
is a partition, then $T_{I}$ consists only of singletons, so the
rearrangements of $T_{I}$ are exactly the orbit of $\mathfrak{S}_{k(I)}$
permuting the Lyndon factors of $I$, and these all have positive
sign. Writing $Z(I)$ for the size of the stabiliser of this $\sk$
action, it follows that 
\[
\f_{I}((n))=\frac{1}{i_{1}\dots i_{l(I)}}\frac{1}{k(I)!}\frac{k(I)!}{Z(I)}=\frac{1}{Z(I)i_{1}\dots i_{l(I)}}.
\]
By \cite[Prop. 1.3.2]{stanleyec1}, $Z(I)i_{1}\dots i_{l(I)}$ is
the size of the centraliser in $\mathfrak{S}_{n}$ of a permutation
with cycle type $I$, so its reciprocal is the proportion of permutations
with cycle type $I$.

If $I$ is not a partition, then $I$ has a Lyndon factor which is
not a single part. So $T_{I}$ has an internal vertex, allowing the
following ``signed involution'' trick: exchanging the branches at
this vertex gives a bijection between rearrangements of opposite signs.
So $\sum_{I':\lambda(I')=\lambda(I)}\f_{I}^{\calsh}(I')=0$.
\end{proof}

\begin{proof}[Proof of Theorem \ref{thm:rightefns-qsym}, the full right eigenbasis]
By Proposition \ref{prop:efns}.R, the right eigenfunctions of the
Hopf-power Markov chain on compositions come from the eigenvectors
of the Hopf-power map on the dual Hopf algebra $\sym$. $\sym$ is
cocommutative and has the complete noncommutative symmetric functions
$S^{I}$ as a word basis, so Theorem \ref{thm:diagonalisation}.B$'$
applies. Specifically, use the alternate formulation of the eigenvectors
in the ensuing Remark 3 involving the right eigenfunctions $\f^{\calsh}$of
riffle-shuffling, and input the result into Proposition \ref{prop:efns}.R.
The resulting basis of right eigenfunctions for the descent-set chain
is 
\[
\f{}_{I}(J):=\sum_{I':\lambda(I')=\lambda(I)}\f_{I}^{\calsh}(I')e(S^{(i'_{1})})\dots e(S^{(i'_{l(I)})})\mbox{ evaluated at }F_{J}.
\]
(Recall that $e$ is the Eulerian idempotent map.) Since the basis
of ribbon noncommutative symmetric functions $\left\{ R^{J}\right\} $
is the dual basis to the fundamental quasisymmetric functions $\{F_{J}\}$,
the above is equivalent to 
\[
\f{}_{I}(J)=\mbox{coefficient of }R^{J}\mbox{ in }\sum_{I':\lambda(I')=\lambda(I)}\f_{I}^{\calsh}(I')e(S^{(i'_{1})})\dots e(S^{(i'_{l(I)})})
\]
Now Section \ref{sub:qsym-sym} defines $\Phi^{(n)}$ to be $ne(S^{(n)})$,
so 
\begin{align*}
\f{}_{I}(J) & =\mbox{coefficient of }R^{J}\mbox{ in }\sum_{I':\lambda(I')=\lambda(I)}\f_{I}^{\calsh}(I')\frac{\Phi^{I'}}{i'_{1}\dots i'_{l(I)}}\\
 & =\mbox{coefficient of }R^{J}\mbox{ in }\frac{1}{i_{1}\dots i_{l(I)}}\sum_{I':\lambda(I')=\lambda(I)}\f_{I}^{\calsh}(I')\Phi^{I'},
\end{align*}
and \cite[Cor. 4.28]{ncsym} gives the coefficient of $R^{J}$ in
$\Phi^{I'}$ as 
\[
\prod_{r=1}^{l(I')}\frac{(-1)^{l(J_{r})-1}}{\binom{|J_{r}|-1}{l(J_{r})-1}}.
\]

\end{proof}

\begin{proof}[Proof of Theorem \ref{thm:gridems}, right eigenfunctions corresponding
to partitions]
 Fix a partition $I$. The decreasing Lyndon hedgerow $T_{I}$ consists
only of singletons, so, for any $I'$ with $\lambda(I')=\lambda(I)$,
there is only one rearrangement of $T_{I}$ spelling $I'$, and it
has positive sign. So $\f_{I}^{\calsh}(I')=\frac{1}{k(I)!}=\frac{1}{l(I)!}$.

\cite[Sec. 3.3]{ncsym2} then states that $\sum\f_{I}(J)R_{J}$ is
the image of $E_{I}$ under their map $\alpha$ from the descent algebra
to $\sym$ sending $\sum_{\sigma:\Des(\sigma)=I}\sigma$ to the ribbon
noncommutative symmetric function $R_{I}$. As this map is injective,
it must be that $E_{I}=\sum_{\sigma\in\mathfrak{S}_{n}}\f_{I}(\Des(\sigma))\sigma.$
\end{proof}

\subsection{Left Eigenfunctions Corresponding to Partitions\label{sub:leftefnspartition-qsym}}

Throughout this section, let $I$ be a partition with $|I|=n$. The
left eigenfunctions $\g_{I}$ are most concisely defined using some
representation theory of the symmetric group $\sn$, although their
calculation is completely combinatorial. \cite[Sec. 3.5.2]{snreps}
describes a representation of $\sn$ for each skew-shape with $n$
boxes; denote by $\chi^{J}$ the character of such a representation
whose skew-shape is the ribbon shape of $J$.
\begin{thm}
\label{thm:ribbonchar} Let $I$ be a partition. Define $\g_{I}(J):=\chi^{J}(I)$,
the character of $\mathfrak{S}_{n}$ associated to the ribbon shape
$J$ evaluated at a permutation with cycle type $I$. Then $\g_{I}$
is a left eigenfunction of the $a$th Hopf-power Markov chain on compositions
with eigenvalue $a^{l(I)-n}$. \end{thm}
\begin{rem*}
Here's how to recover from this the left eigenfunctions of the chain
tracking the number of descents. As observed in \cite[Th. 1.3.1.3]{johnpike},
any left eigenfunction $\g$ of a Markov chain induces a left eigenfunction
$\bar{\g}$ on its projection, by summing over the values of $\g$
on its preimage. Here, this construction gives 
\[
\bar{\g}_{I}(j)=\sum_{l(J)=j}\chi^{J}(I),
\]
 and $\sum_{l(J)=j}\chi^{J}$ is by definition the Foulkes character.
Hence these induced eigenfunctions are precisely those calculated
in \cite[Cor. 3.2]{amazingmatrix}.
\end{rem*}
By Theorem \ref{thm:ribbonchar}, calculating the eigenfunctions $\g_{I}$
for partitions $I$ amounts to evaluating characters of the symmetric
group, for which the standard method is the Murnaghan-Nakayama rule.
This rule simplifies when the character in question corresponds to
a ribbon shape; as noted in \cite[Rem. 3.5.18]{snreps}, finding $\chi^{J}(I)$
requires the following: 
\begin{enumerate}[label=\arabic*.]
\item  Find all possible ways of filling the ribbon shape of $J$ with
$i_{1}$ copies of 1, $i_{2}$ copies of 2, etc. such that all copies
of each integer are in adjacent cells, and all rows and columns are
weakly increasing. 
\item Let $l_{r}$ be the number of rows containing $r$. Sum over all the
fillings found in step 1, weighted by $(-1)^{\Sigma(l_{r}-1)}$.\end{enumerate}
\begin{example}
\label{ex:leftefnspartition-qsym}Calculating $\g_{(4,4,3)}((3,5,2,1))$
requires filling the ribbon shape of $(3,5,2,1)$ with four copies
of 1, four copies of 2 and three copies of 3, subject to the constraints
in step 1 above. Observe that the top square cannot be 1, because
then the top four squares must all contain 1, and the fifth square
from the top must be equal to or smaller than these. Similarly, the
top square cannot be 3, because then the top three squares are all
3s, but the fourth must be equal or larger. Hence 2 must fill the
top square, and the only legal way to complete this is \[ \mbox{\tableau{ & & & & & & & 2 \\ & & & & & & 2 & 2 \\ & & 1 & 1 & 1 & 1 & 2 \\ 3 & 3 & 3 }} \]
so 
\[
\g_{(4,4,3)}((3,5,2,1))=(-1)^{(0+2+0)}=1.
\]

\end{example}

\begin{example}
\label{ex:leftefnspartitionn-qsym}There is only one way to fill any
given ribbon shape with $n$ copies of 1, so 
\[
\g_{(n)}(J)=(-1)^{l(J)}.
\]

\end{example}
Next, take $I=(1,1,\dots,1)$. Then $\g_{(1,1,\dots,1)}$ has eigenvalue
$a^{n-n}=1$, so $\g_{(1,1,\dots,1)}$ is a multiple of the stationary
distribution. (The full left eigenbasis of Theorem \ref{thm:leftefns-qsym}
will show that the stationary distribution is unique). Following the
algorithm for $\g_{I}(J)$ above, $\g_{(1,1,\dots,1)}$ is the signed
enumeration of fillings of the ribbon shape of $J$ by $1,2,\dots,n$,
each appearing exactly once. Reading the fillings from bottom left
to top right gives a word of degree $(1,1,\dots,1)$ whose descent
composition is exactly $J$. In conclusion:
\begin{cor}
\label{cor:stationarydistribution-qsym}The stationary distribution
for the Hopf-power Markov chain on compositions is 
\[
\pi(J)=\frac{1}{n!}\left|\left\{ w|\left|w\right|=n,\ \deg(w)=(1,1,\dots,1),\ \Des(w)=J\right\} \right|.
\]
In other words, the stationary probability of $J$ is the proportion
of permutations with descent composition $J$.\qed
\end{cor}
This also follows from the stationary distribution of riffle-shuffling
being the uniform distribution.

\subsection{A full Basis of Left Eigenfunctions\label{sub:leftefns-qsym}}

The definition of the full basis of left eigenfunctions involve an
obscure basis of $QSym$, which \cite[Cor. 2.2, Eq. 2.12]{duality}
defines as the following sum of monomial quasisymmetric functions:
\[
P_{I}:=\sum_{J\leq I}\left(l(I_{1})!\dots l(I_{l(J)})!\right)^{-1}M_{J}
\]
Here the sum runs over all compositions $J$ coarser than $I$, and
$\left(I_{1},\dots,I_{l(J)}\right)$ is the decomposition of $I$
relative to $J$. (This may be related to the basis of \cite{qsymfreecommbasis}.)
Also recall that $e$ is the Eulerian idempotent map
\[
e(x)=\sum_{r=1}^{\deg x}\frac{(-1)^{r-1}}{r}m^{[r]}\bard^{[r]}(x).
\]

\begin{thm}
\label{thm:leftefns-qsym} Given a composition $I$ with Lyndon factorisation
$I=I_{(1)}\dots I_{(k)}$, define the function 
\[
\g_{I}(J):=\mbox{coefficient of }F_{J}\mbox{ in }e\left(P_{I_{(1)}}\right)\dots e\left(P_{I_{(k)}}\right).
\]
Then $\{\g_{I}|\left|I\right|=n,\ I\mbox{ has }k\mbox{ Lyndon factors}\}$
is a basis of left $a^{k-n}$-eigenfunctions for the $a$th Hopf-power
Markov chain on compositions. \end{thm}
\begin{example}
\label{ex:leftefns-qsym}Take $I=(1,2,1),\: J=(3,1)$. Then $I_{(1)}=(1,2)$,
$I_{(2)}=(1)$, so $\g_{I}$ has eigenvalue $a^{-2}$, and is described
by $e(P_{(1,2)})e(P_{(1)})$. Now 
\begin{align*}
e(P_{(1,2)}) & =e\left(\frac{1}{1!1!}M_{(1,2)}+\frac{1}{2!}M_{(3)}\right)\\
 & =\left(M_{(1,2)}-\frac{1}{2}M_{(1)}M_{(2)}\right)+\frac{1}{2}M_{(3)}\\
 & =\frac{1}{2}(M_{(1,2)}-M_{(2,1)}),
\end{align*}
and
\[
e(P_{(1)})=e(M_{(1)})=M_{(1)}.
\]
So 
\begin{align*}
e(P_{1,2})e(P_{1}) & =\frac{1}{2}(M_{(1,2)}-M_{(2,1)})M_{(1)}\\
 & =\frac{1}{2}(2M_{(1,1,2)}-2M_{(2,1,1)}+M_{(1,3)}-M_{(3,1)})\\
 & =\frac{1}{2}(F_{(1,1,2)}-F_{(2,1,1)}+F_{(1,3)}-F_{(3,1)}).
\end{align*}
Hence $\g_{(1,2,1)}((3,1))=-\frac{1}{2}$. The full left eigenbasis
for $n=4$ is documented in Section \ref{sub:matrix-qsym}.\end{example}
\begin{proof}[Proof of Theorem \ref{thm:leftefns-qsym}, the full left eigenbasis]
By Proposition \ref{prop:efns}.L and Theorem \ref{thm:diagonalisation}.A$'$,
it suffices to show that there is a (non-graded) algebra isomorphism
$\calsh\rightarrow QSym$ sending the word $(i_{1}\dots i_{l})$ to
$P_{(i_{1},\dots,i_{l})}$. This is the content of \cite[Cor. 2.2]{duality}.
The main idea of the proof goes as follows: the scaled power sum of
the second kind $\{\frac{1}{i_{1}\dots i_{l}}\Phi^{I}\}$ (which they
call $\{P_{I}^{*}\}$) form a free basis for $\sym$, and $\frac{1}{i}\Phi^{(i)}$
is primitive, so $\frac{1}{i_{1}\dots i_{l}}\Phi^{I}\rightarrow(i_{1}\dots i_{l})$
is a Hopf-isomorphism from $\sym$ to the free associative algebra.
Dualising this map gives a Hopf-isomorphism $\calsh\rightarrow QSym$.
\cite[Cor. 2.2]{duality} gives a generating function proof that the
image of $(i_{1}\dots i_{l})$ under this map is indeed $P_{(i_{1},\dots,i_{l})}$
as defined in the theorem.
\end{proof}

\begin{proof}[Proof of Theorem \ref{thm:ribbonchar}, left eigenfunctions corresponding
to partitions]
If $I$ is a partition, then its Lyndon factors are all singletons,
so 
\[
\g_{I}(J)=\mbox{coefficient of }F_{J}\mbox{ in }e\left(P_{(i_{1})}\right)\dots e\left(P_{(i_{l})}\right).
\]
By definition, $P_{(i_{r})}=M_{(i_{r})}$ and this is primitive, so
$\bard^{[a]}M_{(i_{r})}=0$ for all $a\geq2$, and $e(M_{(i_{r})})=M_{(i_{r})}$.
So $\g_{I}(J)$ is the coefficient of $F_{J}$ in $M_{(i_{1})}\dots M_{(i_{l(I)})}=p_{I}$,
the power sum symmetric function. As $p_{I}$ is a symmetric function
(as opposed to simply quasisymmetric), \cite[Th. 3]{qsym} determines
its coefficient of $F_{J}$ to be the inner product $\langle p_{I},s_{J}\rangle$,
with $s_{J}$ the skew-Schur function associated to the ribbon shape
$J$. By the Murnaghan-Nakayama rule, $\langle p_{I},s_{J}\rangle=\chi^{J}(I)$.
\end{proof}

\subsection{Duality of Eigenfunctions\label{sub:Duality-qsym}}

The eigenfunctions $\{\f_{I}\}$ and $\{\g_{I}\}$ above are ``almost
dual'' in the same sense as the riffle-shuffle eigenfunctions $\{\f_{w}^{\calsh}\}$,
$\{\g_{w}^{\calsh}\}$ of Section \ref{sec:Riffle-Shuffling}, and
this is enough to produce the neat Corollary \ref{cor:probdescentset}.
As before, write $\langle\f_{I'},\g_{I}\rangle$ for $\sum_{J:|J|=n}\f_{I'}(J)\g_{I}(J)$.
\begin{thm}
\label{thm:A'B'duality-qsym} Let $I,I'$ be compositions of $n$.
Then 
\[
\langle\f_{I'},\g_{I}\rangle=\langle\f_{I'}^{\calsh},\g_{I}^{\calsh}\rangle.
\]
In particular, 
\begin{enumerate}
\item $\langle\f_{I},\g_{I}\rangle=1$;
\item if $I$ is a partition and $I'$ is any composition different from
$I$, then $\langle\f_{I'},\g_{I}\rangle=\langle\f_{I},\g_{I'}\rangle=0$;
\item in fact, $\langle\f_{I'},\g_{I}\rangle=0$ unless there is a permutation
$\sigma\in\mathfrak{S}_{k(I)}$ such that $\lambda(I'_{(\sigma(r))})=\lambda(I_{(r)})$
for each $r$, and each $I_{(r)}$ is equal to or lexicographically
larger than \textup{$I'_{(\sigma(r))}$.} (Here, $I=I_{(1)}\dots I_{(k)}$
is the Lyndon factorisation of $I$, and similarly for $I'$.)
\end{enumerate}
\end{thm}
\begin{proof}
Theorem \ref{thm:duality-shuffle}, the partial duality of riffle-shuffle
eigenfunctions, shows that 
\[
\langle\f_{w'}^{\calsh},\g_{w}^{\calsh}\rangle=\frac{1}{Z(w')}\sum_{\sigma\in\mathfrak{S}_{k}}\f_{u'_{\sigma(1)}}^{\calsh}(u_{1})\dots\f_{u'_{\sigma(k)}}^{\calsh}(u_{k}),
\]
where $Z(w')$ is the size of the stabiliser of $\sk$ acting on the
Lyndon factors of $w'$, and $w=u_{1}\cdot\dots\cdot u_{k}$ and $w'=u'_{1}\cdot\dots\cdot u'_{k}$
are Lyndon factorisations. The same argument, with $P_{I_{(r)}}$in
place of $u_{r}$ and $\f_{I'_{(r)}}$ in place of $\f_{u'_{r}}^{\calsh}$,
proves 
\[
\langle\f_{I'},\g_{I}\rangle=\frac{1}{Z(I')}\sum_{\sigma\in\mathfrak{S}_{k}}f_{I'_{(\sigma1)}}(P_{I_{(1)}})\dots f_{I'_{(\sigma k)}}(P_{I_{(k)}}).
\]
So, for the main statement, it suffices to show $f_{I'}(P_{I})=\f_{I'}^{\calsh}(I)$
for Lyndon compositions $I,I'$. Recall that 
\[
f_{I'}=\frac{1}{i'_{1}\dots i'_{l(I)}}\sum_{J':\lambda(J')=\lambda(I')}\f_{I'}^{\calsh}(J')\Phi^{J'}.
\]
Now the basis $\{P_{I}\}$ was designed to be the dual basis to $\{\frac{1}{i_{1}\dots i_{l}}\Phi^{I}\}$,
so, when evaluating $f_{I'}$ at $P_{I}$, the only summand that contributes
is $J'=I$. So indeed $f_{I'}(P_{I})=\f_{I'}^{\calsh}(I)$ .

Parts (i) and (iii) then follow from the analogous statements of Theorem
\ref{thm:duality-shuffle}. To deduce Part (ii), note that the Lyndon
factors of a partition $I$ are its parts, so the condition $\lambda(I'_{(\sigma(r))})=\lambda(I_{(r)})$
reduces to $\lambda(I'_{(\sigma(r))})=(i_{(r)})$. Hence $\langle\f_{I'},\g_{I}\rangle$
or $\langle\f_{I},\g_{I'}\rangle$ is nonzero only if all Lyndon factors
of $I'$ are singletons, which forces $I'$ to also be a partition.
Then the condition $i'_{(\sigma(r))}=i_{(r)}$ implies $I'=I$.
\end{proof}
If $I,\ I'$ are both partitions, then the interpretations of Theorems
\ref{thm:gridems} and \ref{thm:ribbonchar} translate Part ii of
the previous Theorem to:
\begin{cor}
\label{cor:gridems-ribbonchar}Let $\chi^{J}$ be the character corresponding
to the ribbon shape $J$, and $E_{\lambda}(J)$ be the coefficient
of any permutation with descent composition $J$ in the Garsia-Reutenauer
idempotent $E_{\lambda}$. Then 
\[
\sum_{J}\chi^{J}(\sigma)E_{\lambda}(J)=\begin{cases}
1 & \mbox{if }\sigma\mbox{ has cycle type }\lambda;\\
0 & \mbox{otherwise.}
\end{cases}
\]
\qed
\end{cor}
There is another consequence of Theorem \ref{thm:A'B'duality-qsym}.ii
that is more relevant to the riffle-shuffle Markov chain:
\begin{cor}
\label{cor:probdescentset}Let $\{X_{m}\}$ be the Markov chain of
$a$-handed riffle-shuffling for a deck of $n$ distinct cards, starting
in ascending order. Then 
\[
P\{\Des(X_{m})=J\}=\frac{1}{n!}\sum_{\sigma\in\sn}a^{m(-n+l(\sigma))}\chi^{J}(\sigma),
\]
where $l(\sigma)$ is the number of cycles of $\sigma$. 
\end{cor}
This also follows from \cite[Th. 2.1]{stanleyshuffleqsym}. In the
present notation, his theorem reads 
\[
P\{Y_{1}=w|Y_{0}=(12\dots n)\}=F_{\Des(w^{-1})}(t_{1},t_{2},\dots)
\]
where $\{Y_{m}\}$ is the biased riffle-shuffling chain: cut the deck
multinomially with parameter $(t_{1},t_{2},\dots)$ and interleave
the piles uniformly as before. The usual $a$-handed shuffle is the
case where $t_{1}=t_{2}=\dots=t_{a}=\frac{1}{a}$, $t_{a+1}=t_{a+2}=\dots=0$.
So, letting $[g]_{1/a}$ denote the evaluation of the quasisymmetric
function $g$ at $t_{1}=\dots=t_{a}=\frac{1}{a},\ t_{a+1}=t_{a+2}=\dots=0$
as in Section \ref{sub:Absorption}, 
\[
P\{\Des(X_{1})=J\}=\left[\sum_{w\in\sn:\Des(w)=J}F_{\Des(w^{-1})}\right]_{1/a}.
\]
According to \cite[Th. 7.19.7]{stanleyec2}, $\sum_{w\in\sn:\Des(w)=J}F_{\Des(w^{-1})}=s_{J}$,
the skew-Schur (symmetric) function of ribbon shape $J$. And checking
on the power sums $p_{\lambda}$ shows that the linear map of evaluating
a symmetric function of degree $n$ at $t_{1}=\dots=t_{a}=\frac{1}{a},\ t_{a+1}=t_{a+2}=\dots=0$
is equivalent to taking its inner product with $\frac{1}{n!}\sum_{\sigma\in\sn}a^{-n+l(\sigma)}p_{\lambda(\sigma)}$,
where $\lambda(\sigma)$ is the cycle type of $\sigma$. So 
\[
P\{\Des(X_{1})=J\}=\frac{1}{n!}\sum_{\sigma\in\sn}a^{-n+l(\sigma)}\langle p_{\lambda(\sigma)},s_{J}\rangle=\frac{1}{n!}\sum_{\sigma\in\sn}a^{-n+l(\sigma)}\chi^{J}(\sigma).
\]
The case $m>1$ follows from the power rule, as $m$ iterations of
$a$-handed shuffling is equivalent to one $a^{m}$-handed shuffle.

Below is the proof of Corollary \ref{cor:probdescentset} using the
diagonalisation of the descent-set chain.
\begin{proof}
Write $K$ for the transition matrix of the descent-set chain under
riffle-shuffling. Then the left hand side is $K^{m}((n),J)$, which,
by the change of coordinates in Proposition \ref{prop:efnsdiagonalisation},
is equal to 
\[
\sum_{I}a^{m(-n+l(I))}\f_{I}((n))\breve{\g_{I}}(J),
\]
where $\{\breve{\g_{I}}\}$ is the basis of left eigenfunctions dual
to the right eigenbasis $\{\f_{I}\}$. By Proposition \ref{prop:rightefns-qsym},
$\f_{I}((n))$ is 0 unless $I$ is a partition, in which case $\f_{I}((n))$
is the proportion of permutations in $\mathfrak{S}_{n}$ with cycle
type $I$. So 
\[
K^{m}((n),J)=\sum_{\sigma\in\sn}a^{m(-n+l(\sigma))}\frac{1}{n!}\breve{\g_{\lambda(\sigma)}}(J),
\]
where $\lambda(\sigma)$ denotes the cycle type of $\sigma$. For
a partition $I$, Theorem \ref{thm:A'B'duality-qsym} asserts that
$\langle\f_{I},\g_{I}\rangle=1$ and $\langle\f_{I'},\g_{I}\rangle=0$
for any composition $I'$ different from $I$ - this means $\breve{\g_{I}}=\g_{I}$
when $I$ is a partition. So 
\[
K^{m}((n),J)=\sum_{\sigma\in\sn}a^{m(-n+l(\sigma))}\frac{1}{n!}\g_{\lambda(\sigma)}(J),
\]
and the conclusion follows from Theorem \ref{thm:ribbonchar} relating
the left eigenfunctions to the ribbon characters.
\end{proof}
There is an intermediate statement, stronger than this Corollary and
deducible from Stanley's theorem: 
\[
P\{\Des(Y_{1})=J\}=\sum_{w\in\sn:\Des(w)=J}F_{\Des(w^{-1})}=\frac{1}{n!}\sum_{\sigma\in\sn}\chi^{J}(\sigma)p_{\sigma}.
\]
I conjecture that this can be proved independently of Stanley's result
via an analogous diagonalisation of the descent-set Markov chain under
biased riffle-shuffling. (It is not hard to define ``biased Hopf-powers''
to study deformations of the chains in this thesis, but I will not
discuss it here as eigenbasis algorithms for these chains are still
in development.)

\subsection{\label{sub:matrix-qsym}Transition Matrix and Eigenfunctions when
$n=4$}

The Hopf-square Markov chain on compositions of four describes the
changes in descent set under the GSR riffle-shuffle of four distinct
cards. By explicit calculation of $m\Delta$ for the fundamental quasisymmetric
functions of degree four, the transition matrix $K_{2,4}$ is the
following matrix multiplied by $\frac{1}{16}$:
\[
\begin{array}{ccccccccc}
 & (4) & (1,3) & (3,1) & (2,2) & (1,2,1) & (2,1,1) & (1,1,2) & (1,1,1,1)\\
(4) & 5 & 3 & 3 & 4 & 1 & 0 & 0 & 0\\
(1,3) & 1 & 5 & 2 & 3 & 2 & 1 & 2 & 0\\
(3,1) & 1 & 2 & 5 & 3 & 2 & 2 & 1 & 0\\
(2,2) & 1 & 2 & 2 & 6 & 3 & 1 & 1 & 0\\
(1,2,1) & 0 & 1 & 1 & 3 & 6 & 2 & 2 & 1\\
(2,1,1) & 0 & 1 & 2 & 2 & 3 & 5 & 2 & 1\\
(1,1,2) & 0 & 2 & 1 & 2 & 3 & 2 & 5 & 1\\
(1,1,1,1) & 0 & 0 & 0 & 1 & 4 & 3 & 3 & 5
\end{array}.
\]

Its basis of right eigenfunctions, as determined by Theorem \ref{thm:rightefns-qsym},
are the columns of the following matrix:
\[
\begin{array}{ccccccccc}
 & E_{(4)} & \f_{(1,3)} & E_{(3,1)} & E_{(2,2)} & \f_{(1,2,1)} & E_{(2,1,1)} & \f_{(1,1,2)} & E_{(1,1,1,1)}\\
(4) & \frac{1}{4} & 0 & \frac{1}{3} & \frac{1}{8} & 0 & \frac{1}{4} & 0 & \frac{1}{24}\\
(1,3) & -\frac{1}{12} & \frac{1}{2} & \frac{1}{12} & -\frac{1}{8} & \frac{1}{2} & \frac{1}{12} & -1 & \frac{1}{24}\\
(3,1) & -\frac{1}{12} & -\frac{1}{2} & \frac{1}{12} & -\frac{1}{8} & -\frac{1}{2} & \frac{1}{12} & -1 & \frac{1}{24}\\
(2,2) & -\frac{1}{12} & 0 & -\frac{1}{6} & \frac{1}{8} & 0 & \frac{1}{12} & 2 & \frac{1}{24}\\
(1,2,1) & \frac{1}{12} & 0 & -\frac{1}{6} & \frac{1}{8} & 0 & -\frac{1}{12} & -2 & \frac{1}{24}\\
(2,1,1) & \frac{1}{12} & \frac{1}{2} & \frac{1}{12} & -\frac{1}{8} & -\frac{1}{2} & -\frac{1}{12} & 1 & \frac{1}{24}\\
(1,1,2) & \frac{1}{12} & -\frac{1}{2} & \frac{1}{12} & -\frac{1}{8} & \frac{1}{2} & -\frac{1}{12} & 1 & \frac{1}{24}\\
(1,1,1,1) & -\frac{1}{4} & 0 & \frac{1}{3} & \frac{1}{8} & 0 & -\frac{1}{4} & 0 & \frac{1}{24}
\end{array}.
\]

Its basis of left eigenfunctions, as determined by Theorem \ref{thm:leftefns-qsym},
are the rows of the following matrix:
\[
\begin{array}{ccccccccc}
 & (4) & (1,3) & (3,1) & (2,2) & (1,2,1) & (2,1,1) & (1,1,2) & (1,1,1,1)\\
\chi(4) & 1 & -1 & -1 & -1 & 1 & 1 & 1 & -1\\
\g_{(1,3)} & 0 & \frac{1}{2} & -\frac{1}{2} & 0 & 0 & \frac{1}{2} & -\frac{1}{2} & 0\\
\chi(3,1) & 1 & 0 & 0 & -1 & -1 & 0 & 0 & 1\\
\chi(2,2) & 1 & -1 & -1 & 1 & 1 & -1 & -1 & 1\\
\g_{(1,2,1)} & 0 & \frac{1}{2} & -\frac{1}{2} & 0 & 0 & -\frac{1}{2} & \frac{1}{2} & 0\\
\chi(2,1,1) & 1 & 1 & 1 & 1 & -1 & -1 & -1 & -1\\
\g_{(1,1,2)} & 0 & -\frac{1}{12} & -\frac{1}{12} & \frac{1}{6} & \frac{1}{6} & \frac{1}{12} & \frac{1}{12} & 0\\
\chi(1,1,1,1) & 1 & 3 & 3 & 5 & 5 & 3 & 3 & 1
\end{array}.
\]

\appendix

 \printbibliography[heading=bibintoc]
\end{document}